\documentclass{amsart}
\usepackage{amssymb}
\usepackage{amscd}
\usepackage{amsfonts}
\usepackage{latexsym}
\usepackage[all]{xy}
\usepackage{verbatim}

\newtheorem{theorem}{Theorem}[section]
\newtheorem{lemma}[theorem]{Lemma}
\newtheorem{proposition}[theorem]{Proposition}

\newtheorem{corollary}[theorem]{Corollary}
\theoremstyle{definition}
\newtheorem{definition}[theorem]{Definition}
\newtheorem{example}[theorem]{Example}
\newtheorem{examples}[theorem]{Examples}

\newtheorem{remark}[theorem]{Remark}
\newtheorem{principles}[theorem]{Principles}

\newtheorem{rems}[theorem]{Remarks}

\newcommand{\Tr}{\text{Tr}}
\newcommand{\id}{\text{id}}

\newcommand{\FPdim}{\text{FPdim}}

\newcommand{\Ker}{\text{Ker\,}}
\newcommand{\Der}{\text{Der}}

\newcommand{\Fun}{\text{Fun}}
\newcommand{\End}{\text{End}}
\renewcommand{\Vec}{\text{Vec}}

\newcommand{\Hom}{\text{Hom}}
\newcommand{\uHom}{\underline{\text{Hom}}}
\newcommand{\Ad}{\text{Ad}}
\newcommand{\Aut}{\text{Aut}}

\newcommand{\Ind}{\text{Ind}}
\newcommand{\Int}{\text{Int}}
\newcommand{\Rep}{\text{Rep}}

\newcommand{\rev}{\text{rev}}

\newcommand{\Vect}{\text{Vec}}

\newcommand{\BN}{\mathbb{N}}
\newcommand{\ev}{\text{ev}}
\newcommand{\coev}{\text{coev}}
\newcommand{\op}{\text{op}}
\newcommand{\oP}{\text{op}}
\newcommand{\additive}{{$k$-linear }}
\newcommand{\additively}{{$k$-linearly }}

\newcommand{\Func}{\text{Func\,}}
\newcommand{\Repr}{\text{Repr}}
\newcommand{\VecT}{\text{Vect}}
\newcommand{\fD}{{\mathfrak D}}
\newcommand{\fP}{{\mathfrak P}}
\newcommand{\fL}{{\mathfrak L}}

\newcommand{\Ob}{\text{Ob}}

\newcommand{\B}{\mathcal{B}}
\newcommand{\C}{\mathcal{C}}
\newcommand{\K}{\mathcal{K}}
\newcommand{\D}{\mathcal{D}}

\newcommand{\cP}{\mathcal{P}}
\newcommand{\E}{\mathcal{E}}
\newcommand{\F}{\mathcal{F}}
\newcommand{\Z}{\mathcal{Z}}

\renewcommand{\L}{\mathcal{L}}
\newcommand{\M}{\mathcal{M}}
\newcommand{\A}{\mathcal{A}}

\newcommand{\T}{\mathcal{T}}

\newcommand{\PB}{\mathcal{PB}}
\newcommand{\PM}{\mathcal{PM}}

\renewcommand{\O}{\mathcal{O}}

\newcommand{\be}{\mathbf{1}}
\newcommand{\fg}{{\mathfrak g}}
\newcommand{\fa}{{\mathfrak a}}

\newcommand{\fc}{{\mathfrak c}}
\newcommand{\fh}{{\mathfrak h}}

\newcommand{\G}{\widetilde{G}}

\renewcommand{\be}{\mathbf{1}}

\newcommand{\BZ}{{\mathbb Z}}
\newcommand{\BQ}{{\mathbb Q}}
\newcommand{\BC}{{\mathbb C}}
\newcommand{\BF}{{\mathbb F}}

\newcommand{\mI}{{\mathcal I}}
\newcommand{\bt}{\boxtimes}
\newcommand{\ot}{\otimes}
\newcommand{\cyc}{cyc}
\newcommand{\kcyc}{{k_{\cyc}}}
\newcommand{\kRcyc}{{k_{\mathbb{R}\cyc}}}

\newcommand{\iso}{\buildrel{\sim}\over{\longrightarrow}}

\newcommand{\x}{x}
\newcommand{\q}{\mathbb{Q}^{\times}_{>0}/(\mathbb{Q}^{\times}_{>0})^2}

\begin{document}
\title{On braided fusion categories I}

\author{Vladimir Drinfeld}
\address{V.D.: Department of Mathematics,
University of Chicago, Chicago, IL 60637, USA}
\email{drinfeld@math.uchicago.edu}

\author{Shlomo Gelaki}
\address{S.G.: Department of Mathematics, Technion-Israel Institute of
Technology, Haifa 32000, Israel}
\email{gelaki@math.technion.ac.il}

\author{Dmitri Nikshych}
\address{D.N.: Department of Mathematics and Statistics,
University of New Hampshire,  Durham, NH 03824, USA}
\email{nikshych@math.unh.edu}

\author{Victor Ostrik}
\address{V.O.: Department of Mathematics,
University of Oregon, Eugene, OR 97403, USA}
\email{vostrik@math.uoregon.edu}

\dedicatory{To Yuri Ivanovich Manin with admiration}

\begin{abstract}
We introduce a new notion of the core of a braided fusion category.
It allows to separate the part of a braided fusion category that does not come from finite groups.
We also give a comprehensive and self-contained exposition of the known results
on braided fusion categories without assuming them pre-modular or non-degenerate.
The guiding heuristic principle of our work is an analogy between 
braided fusion categories and Casimir Lie algebras.  
\medskip

{\em Mathematics Subject Classification (2000): 18D10, 16W30.}

{\em Keywords: Tensor category, braided tensor category, equivariantization.}
\end{abstract}

\maketitle

\setcounter{tocdepth}{1}
\tableofcontents
\newpage

\section{Introduction}

This work is the first one in a series of articles on the structure of braided fusion categories.
It is concerned with the foundations  of the theory. In the future articles we are going to study
certain classes of braided fusion categories (e.g., nilpotent and group-theoretical ones)
and the structure of the core of any braided fusion category.

\subsection{Terminology and notation}  \label{terminology}

In this paper we work over an algebraically closed field $k$ of characteristic $0$.
A {\em $k$-linear category\,} is an additive category in which each $\Hom$ group is
equipped with a structure of vector space over $k$ so that the composition of morphisms
is $k$-linear. A {\em tensor category\,}  (or {\em monoidal $k$-linear category\,}) is a
$k$-linear category $\C$ with a structure of monoidal category such that the
bifunctor $\ot : \C \times \C \to \C$ is bilinear on morphisms.

By a {\em fusion category} we mean a $k$-linear 
semisimple rigid tensor category $\C$ with finitely many isomorphism classes of
simple objects, finite-dimensional spaces of morphisms, and such
that the unit object $\be$ of $\C$ is simple.  We refer the reader to
\cite{ENO} for a general theory of fusion categories. 
A fusion category is  {\em braided} if it is equipped with a natural isomorphism
$c_{X,Y}: X\ot Y \simeq Y \ot X$ satisfying the hexagon axioms, see \cite{JS, BK}.

The simplest example of a fusion category is a {\em pointed} category, i.e., 
such that all its simple objects are invertible.  A pointed fusion category
is equivalent to $\Vec_G^\omega$, i.e., the category of finite-dimensional $G$-graded vector
spaces  with the associativity constraint given by some cocycle
$\omega\in Z^3(G,\, k^\times)$ (here $G$ is a finite group).

\subsection{Braided versus pre-modular}  \label{whybraided}
In the literature most of the results on braided fusion categories are formulated under the
assumption that the braided category is {\em pre-modular,\,} i.e., equipped with
a {\em spherical structure\,} (see \S\ref{pivotal}). Moreover, many results are formulated
for {\em modular\,} categories, i.e., pre-modular categories satisfying a non-degeneracy
condition. However, the spherical structure is irrelevant for the questions considered in
this work and for many other purposes. So we found it necessary to formulate and prove 
the generalizations of some known results without assuming the existence of a spherical structure 
(although it seems to exist in all known examples of fusion categories). We also tried to
avoid non-degeneracy assumptions whenever possible.

\subsection{The core of a braided fusion category}
This is the main new notion of this work. Before discussing it in more details in \S\ref{Main: core},
let us explain informally why we think  it is important.

\subsubsection{Informal discussion} \label{core-informal}
The problem of classifying {\em all\,} braided fusion categories is probably unreasonable
because it contains the subproblem of classifying all finite groups. So given a braided fusion
category $\C$ it is natural to try to separate the ``part" of $\C$ that ``comes from finite groups"
from the remaining ``part", $\C_1$. The notion of core allows to give a precise definition of $\C_1$.
Namely, $\C_1$ is the adjoint category of the core of $\C$ (the definition of
adjoint category is given in \S\ref{adcat}). It turns out that if the global dimension
(or the Frobenius-Perron dimension) of $\C_1$ is odd then $\C_1$ is a tensor product
of braided fusion categories which are {\em simple\,} (i.e., not pointed, 
not symmetric, and without nontrivial fusion subcategories). We hope that a similar (but more complicated) 
statement can be proved without the oddness assumption.

\subsubsection{Definition and main results} \label{Main: core}

Let $\E=\Rep(G) \subset \C$ be a maximal Tannakian subcategory of a braided fusion category $\C$.
The core of $\C$ corresponding to $\E$ is the braided fusion category 
\begin{equation}
\label{intro core}
\mbox{Core}_\E(\C)  := \E'\bt_\E \Vec
\end{equation}
obtained by  the de-equivariantization of the centralizer of $\E$, see Definition~\ref{core corresponding to E}.
This is parallel to the notion of the core of a pre-metric Lie coalgebra, see 
Principle~\ref{intro principles}(iii) and \S\ref{warmup}.
The above category  $\mbox{Core}_\E(\C)$ carries an action of $G$ by braided 
autoequivalences. Let $\Gamma_\E$
denote the image of $G$ in the group of isomorphism classes of braided 
autoequivalences of  $\mbox{Core}_\E(\C)$.

In Theorem~\ref{canonical modularization} we show that {\em the equivalence class of the pair 
$(\mbox{Core}_\E(\C),\, \Gamma_\E)$ does not depend on the choice  of $\E$,} i.e., is an invariant of  
$\C$. We hope that this new invariant will be useful in the study of braided
fusion categories. 
In a joint work \cite{ENO4} with P.~Etingof  we will explain
how to reconstruct a braided fusion category from its core and some finite group data.
An example of such reconstruction is provided in Theorem~\ref{main1} which says
that non-degenerate  braided fusion categories with trivial core
are precisely the centers  of pointed categories. 

We also show that the core is {\em weakly anisotropic}, i.e., has no 
non-trivial Tannakian subcategories stable under all braided autoequivalences.
To demonstrate that the class of weakly anisotropic braided fusion categories
is reasonable, we formulate Theorems~\ref{5wapt0} 
and \ref{5ordmaintheor}, which will be proved in a subsequent article.
Theorem~\ref{5ordmaintheor} describes the structure of a weakly anisotropic braided fusion category 
$\C$ satisfying a certain condition, which holds automatically if the Frobenius-Perron dimension
of $\C$ is odd.

\subsection{Some other results}

\subsubsection{Characterization of the centers of pointed categories}
\label{Main: centers}
A fusion subcategory $\L$ of a braided fusion category $\C$ is said to be 
{\em Lagrangian} if $\L$ is Tannakian and coincides with its centralizer.

In Theorem~\ref{main1} we prove that  a non-degenerate braided fusion category $\D$
is equivalent to the center of a pointed fusion category if and only if 
it contains a Lagrangian subcategory. More precisely,
we show that there is an equivalence between the groupoid of Lagrangian subcategories of $\D$  
and the groupoid of ``realizations'' of  $\D$ as the center of a pointed fusion category, 
see Theorem~\ref{more precise}. This agrees with the heuristic Principles~\ref{intro principles}.

\subsubsection{Invariance of the central charge of a pre-modular category}
\label{Main: cecnarge}

Theorem~\ref{tau after mod} relates the Gauss sum  of a pre-modular category $\C$ 
with that of the fiber category  $\E'\bt_\E \Vec$ corresponding to a Tannakian subcategory $\E\subset \C$.  
We prove in Corollary~\ref{zariad} that the central charges of $\C$ and $\E'\bt_\E \Vec$ coincide.

\subsection{The classical-quantum analogy} \label{heuristic pr}
In this subsection (which can be skipped by the skeptical reader)
we formulate some heuristic principles.

\subsubsection{Casimir Lie algebras and pre-metric Lie coalgebras}  \label{classical}
Fix a ground field $K$ (which has nothing to do with $k$). 
We do not assume that $\mbox{char }K=0$
(e.g., $K$ can be finite), but for safety assume that $\mbox{char }K>2$.
All vector spaces and algebras will be over $K$.
\begin{definition}  \label{Casimirdef}
A {\em Casimir Lie algebra\,} is a pair $(\fg ,t)$, where
$\fg$ is a finite-dimensional Lie algebra and $t\in\fg\otimes\fg$
is symmetric and $\fg$-invariant. A {\em pre-metric Lie coalgebra\,} is
a finite-dimensional Lie coalgebra $C$ equipped with a quadratic form $q:C\to K$
invariant with respect to the coadjoint action of the Lie algebra $C^*$.
\end{definition}

If $(\fg ,t)$ is a Casimir Lie algebra then $t$ can be viewed as
an invariant quadratic form on the dual vector space $\fg^*$, so
$(\fg^*,t)$ is a pre-metric Lie coalgebra. Thus the notions of Casimir Lie
algebra and pre-metric Lie coalgebra are essentially equivalent.

A subspace $V\subset\fg^*$ is a Lie {\em subcoalgebra\,} if
$V^*$ is a quotient {\em algebra\,} of $\fg$. This happens if and only
if $V$ is stable under the coadjoint action of $\fg$. Therefore if
$V\subset\fg^*$ is a subcoalgebra so is $V^{\perp}\subset\fg^*$.

\begin{definition}   \label{metricLiedef}
A {\em metric Lie coalgebra\,} is
a pre-metric Lie coalgebra $(C,q)$ such that the quadratic form $q$ is non-degenerate.
A {\em metric Lie algebra\,} is a Casimir Lie algebra $(\fg ,t)$ such that
the tensor $t$ is non-degenerate.
\end{definition}

If $(\fg ,t)$ is a metric Lie algebra then $t$ gives rise
to an invariant non-degenerate quadratic form on $\fg$ (a ``metric").
Clearly the notions of metric Lie
algebra and metric Lie coalgebra are essentially equivalent.

\begin{definition}   \label{Maninpairdef}
A subalgebra $\fa$ of a metric Lie
algebra $\fg$ is said to be {\em Lagrangian\,} if it is isotropic and 
$\dim\fa =\frac{1}{2}\dim\fg$. A {\em Manin pair\,} is a pair consisting of
a metric Lie algebra $\fg$ and a Lagrangian subalgebra $\fa\subset\fg$.
A Manin pair is said to be {\em pointed\,} if $\fa$ is an ideal. 
\end{definition}

\begin{rems}    \label{Maninpairrem}
\begin{enumerate}
\item[(i)] If a Lagrangian subalgebra $\fa\subset\fg$
is an ideal then $\fa$ is abelian because $[\fa , \fa]$ is orthogonal to every element of $\fg$.
\item[(ii)] It is easy to show that for any finite-dimensional Lie algebra $\fc$, pointed
Manin pairs $(\fg ,\fa )$ with  a fixed isomorphism $\fg/\fa\iso\fc$
are classified by $H^3(\fc ,K)$. On the other hand, pointed fusion categories (see \S\ref{terminology}) 
are classified by pairs $(G,\omega)$, where $G$ is a finite group and $\omega\in H^3(G,k^{\times})$.
\end{enumerate}
\end{rems}

\subsubsection{Heuristic principles}   \label{heursubsubsec}
\begin{principles}    \label{intro principles}
\begin{enumerate}
\item[(i)] 
Braided fusion categories are ``quantum" analogs of
Casimir  Lie algebras (and, by duality, of pre-metric Lie coalgebras).  
Pointed braided fusion categories are analogs of  abelian Casimir Lie algebras.
Non-degenerate braided fusion categories are 
analogs of metric Lie algebras.  Symmetric fusion categories are analogs
of Casimir Lie algebras with $t=0$.
\item[(ii)]
Fusion subcategories of braided fusion categories are analogs of Lie subcoalgebras
in pre-metric Lie coalgebras. The centralizer of a fusion subcategory (see \S\ref{Sect3}) 
is an analog of  the orthogonal complement of a Lie subcoalgebra (here ``complement"
means ``complement in the Lie coalgebra").
\item[(iii)]
The de-equivariantization procedure (see \S\ref{isotsub}) is analogous  to taking the quotient 
of a pre-metric Lie coalgebra by a subcoalgebra contained in the kernel of the quadratic form. 
Passing from a braided fusion category $\C$ to its core (see \S\ref{Main: core} and \S\ref{coresec}) is 
analogous  to passing from a pre-metric Lie coalgebra $C$ to the
pre-metric Lie coalgebra $\mathfrak{a}^\perp/\mathfrak{a}$, where $\mathfrak{a}\subset C$ is a maximal isotropic subcoalgebra. 
\item[(iv)]
Fusion categories (resp. pointed fusion categories) are ``quantum" analogs of 
Manin pairs (resp. pointed Manin pairs).
Passing from a fusion category to its center is a ``quantum" analog of
passing from a Manin pair $(\fg ,\fa )$ to the metric Lie algebra $\fg$.
\end{enumerate}
\end{principles}

These principles go back to the article \cite{Dr}, in which 
Manin pairs (which are more or less the same as {\em Lie quasi-bialgebras\,}) appear as 
classical limits of quasi-Hopf algebras and pre-metric Lie algebras appear as 
classical limits of quasitriangular quasi-Hopf algebras. 

In this article Manin pairs are used only to explain why
the characterization of the centers of pointed categories given in
Theorems~\ref{main1}-\ref{more precise} is natural from the point of view
of the classical-quantum
analogy. 
In fact, there is a useful notion of {\em quantum Manin pair\,}
\cite{MNO}. This is a pair $(\C ,A)$, where
$\C$ is a non-degenerate braided fusion category  and $A$ is a commutative
algebra  in $\C$ satisfying certain conditions (namely, $A$ is connected,
etale, and $\FPdim (A)^2 = \FPdim (\C)$).
It is proved in \cite{MNO} that for fixed $\C$, the groupoid formed by
such algebras $A$ is canonically equivalent
to the groupoid of realizations of $\C$ as a center of a fusion category.
This justifies the name ``quantum Manin pair".

\begin{remark}  \label{Lie-rings}
To get a closer analogy with braided fusion categories, instead of pre-metric Lie algebras
one can consider {\em Casimir Lie rings}. By this we mean a finite Lie ring $\fg$
(for safety, of odd order) whose Pontryagin dual  $\fg^*:=\Hom (\fg ,k^{\times})$ is equipped
with a $\fg$-invariant symmetric bi-additive pairing $\fg^*\times\fg^*\to k^{\times}$.
\end{remark}

\subsection{Organization}

In \S\ref{Prelim} we review the background material from the theory of fusion
categories. For a braided fusion category $\C$ without a fixed spherical structure the notion
of $S$-matrix is not defined. Instead, we introduce a similar notion of $\tilde{S}$-matrix.
We define $\C$ to be non-degenerate if its $\tilde{S}$-matrix is. In Proposition~\ref{Muger1prime}
we formulate two other conditions equivalent to the non-degeneracy of~$\C$.
One of them (the triviality of M\"uger's center $\C'$) is well known
in the case of pre-modular categories (see \cite{Br}).

M\"uger's notion of centralizer plays a central role in this work (this is natural in view
of Principle~\ref{intro principles}(ii)).
In \S\ref{Sect3} we extend M\"uger's results on centralizers from \cite{Mu2, Mu4}
to braided fusion categories which are not necessarily non-degenerate or pre-modular.
We also introduce and study the notion of {\em projective centralizer.}

\S\ref{isotsub} deals with  two mutually inverse constructions:
{\em equivariantization} and {\em de-equivariantization}. We follow 
the ideas of Brugui\`{e}res \cite{Br} and M\"uger \cite{Mu1} who showed
that a datum of a fusion category with an action of a finite group $G$
is equivalent to a datum  of a fusion category $\C$ over the Tannakian 
category $\Rep(G)$ such that the composition $\Rep(G)\to \Z(\C)\to \C$
is fully faithful. The passage from an action of $G$ to a category over $\Rep(G)$ 
is by taking the category of $G$-equivariant objects;
the passage in the opposite direction (i.e., the de-equivariantization) 
consists of taking the fiber category, see \S\ref{alternat} for an explanation
of our terminology.  This correspondence  also specializes to group actions  
on braided and pre-modular categories, see Theorem~\ref{princprinc}.
In the case of pre-modular categories the de-equivariantization procedure
coincides with the {\em modularization} introduced in \cite{Br, Mu1}. 
The above correspondence also  allows to study the structure of braided fusion 
categories  from two  complementary  points of view: one based on analysis of 
Tannakian  subcategories and another based on group actions on braided fusion categories, 
see  \S\ref{equi}.  In \S\ref{orbi} we present a description (due to Kirillov.~Jr.\ \cite{Ki} 
and M\"uger \cite{Mu3}) of the structure of braided fusion categories containing
a subcategory $\Rep(G)$ as equivariantizations of {\em braided $G$-crossed categories}.
The latter notion was introduced by Turaev \cite{Tuor}. 
At the end of \S\ref{isotsub} we establish the characterization of the centers 
of pointed fusion categories explained in \S\ref{Main: centers} above.

\S\ref{coresec} is the main new contribution of this article. 
Here we define the notion of the {\em core} of a braided fusion category $\C$ (see
\S\ref{Main: core} above) and study its properties.
We also introduce the class of {\em weakly anisotropic} braided fusion categories
(i.e., those without non-trivial Tannakian subcategories stable under all braided autoequivalences)
and explain why it is reasonable to study the structure of such categories, see
Theorems~\ref{5wapt0}  and and \ref{5ordmaintheor}.

In \S\ref{sectiongauss} we discuss the Gauss sums of pre-modular categories,
which are generalizations of the classical Gauss sums of pre-metric groups.  
We prove the invariance property of 
the central charge of a pre-modular category (see \S\ref{Main: cecnarge} above).
If a braided fusion category $\C$ is not equipped with a spherical structure the notion
of Gauss sum is not defined; however, if $\C$ is non-degenerate and its Frobenius-Perron
dimension is integral we introduce in \S\ref{elementxc} a modification of this notion,
which we call {\em Gauss-Frobenius-Perron sum.}

Appendix~\ref{app on metric groups} contains basic results on quadratic forms 
on finite abelian groups (most of them are well-known).
Appendix~\ref{Is} is a summary of known facts about Ising categories (i.e., non-pointed 
fusion categories of Frobenius-Perron dimension $4$) and braided structures on them.

In Appendix~\ref{appensix Alt-Br} we recall a characterization of centralizers
in braided fusion categories  due to Altsch\"uler and Brugui\`eres \cite{AB}.
This is a key technical result  used in our proofs in \S\ref{Sect3}.

In Appendix~\ref{likbez} we give  an alternative proof of the theorem due to
Joyal and Street \cite{JS}, which says that every pre-metric
group comes from a pointed braided fusion category.

In Appendix~\ref{grcatsect} we introduce the notion of extension of a group by a
gr-category (a.k.a. categorical group). We show that in the case of 
pointed categories equipped with a faithful grading Turaev's notion of
braided $G$-crossed category amounts to the notion of
a {\em central\,} extension of a group by a braided
gr-category, which was introduced in \cite{CC}.

In Appendix~\ref{fusionring} we prove that a full subcategory of a fusion category 
which is closed under tensor product is automatically rigid.

\subsection{Acknowledgments} The research of V.~Drinfeld was supported by NSF grants
DMS-0401164 and DMS-0701106.
The research of D.~Nikshych was supported by 
the NSA grant H98230-07-1-0081 and the NSF grants DMS-0200202 and DMS-0800545.   
The research of V.~Ostrik was supported by NSF grant DMS-0602263. He is also
grateful to the Institute for Advanced Study where part of his work was done.
S.~Gelaki is grateful to the departments of mathematics at the University of New Hampshire 
and MIT for their warm hospitality
during his Sabbatical. The authors thank Alexander Beilinson, Pavel Etingof,
Alexander Kirillov, Jr., and Michael M\"uger for useful discussions and advice.

\section{Preliminaries}  \label{Prelim}

In this section we recall some definitions and results on fusion categories and 
braided fusion categories following mainly \cite{BK, De, ENO, GN, JS}.
Throughout the paper we work over an algebraically closed field
$k$ of characteristic $0$.

\subsection{Generalities on fusion categories}  \label{geneneralities}
The definitions of fusion category and pointed fusion category 
were given in the introduction.

By a {\em fusion subcategory} of a fusion category $\C$ we
understand a full tensor subcategory $\C'\subset\C$ such that if $X\in\C$
is isomorphic to a direct summand of an object of $\C'$ then $X\in\C'$.
As an additive category, such $\C'$ is generated by some of the simple
objects of $\C$. It is known (see Corollary \ref{fusion sub} in Appendix~\ref{fusionring}) 
that a fusion subcategory of a fusion category is rigid; therefore it is itself a fusion category.

Given a fusion category $\C$ and  a collection of objects $S$ in $\C$,
there is a smallest fusion subcategory containing $S$. It is called the
fusion subcategory {\it generated\,} by $S$. 

\begin{definition}  \label{image}
Let $\C$ and $\D$ be fusion categories and $F: \C\to \D$ a tensor
functor.  Its {\em image\,} $F(\C)$ is defined to be the fusion subcategory
of $\D$ generated by $F(X)$, $X\in\C$. The functor $F$ is called {\em surjective} if $F(\C) =\D$.
\end{definition}

Note that as an additive category, $F(\C)$ is generated
by those simple objects $Y$ in $\D$ which are contained in $F(X)$ for some simple 
$X$ in $\C$.

For a fusion category $\C$ let $\O(\C)$ denote the set of isomorphism classes
of simple objects. $\C$ is said to be {\it trivial\,} if $\O(\C)$ has only one element
(the class of the unit object); in this case $\C$ is equivalent to $\Vect$, i.e., 
the fusion category of finite dimensional vector spaces over $k$.

\subsection{The adjoint category of a fusion category $\C$}  \label{adcat}
By definition, this is the fusion subcategory $\C_{ad}\subset\C$ 
generated by $X\ot X^*$, $X\in \O(\C)$.

\begin{example}  \label{ex-ad}
Let $G$ be a finite group and let $\C=\Rep(G)$ be the category of finite-dimensional
representations of $G$.
Then $\C_{ad} = \Rep(G/Z(G))$, where $Z(G)$ is the center of $G$.
\end{example}

\subsection{The universal grading} \label{gradsubsec}
A {\it grading\,} of a fusion category
$\C$ by a group $G$ is a map $\deg :\O(\C)\to G$ with the following property: for any
simple objects $X,Y,Z\in\C$ such that $X\otimes Y$ contains $Z$ one has
$\deg Z=\deg X\cdot\deg Y$. The name ``grading" is also used for the corresponding
decomposition $\C = \oplus_{g\in G}\, \C_g$, where
$\C_g\subset\C$ is the full additive
subcategory generated by simple objects of degree $g\in G$. The subcategory
$\C_1$ corresponding to $g=1$ is a fusion subcategory; it is called the
 {\it trivial component\,} of the grading. A grading is said to be {\em trivial\,} if $\C_1=\C$.
 It is said to be {\em faithful\,} if the map $\deg :\O(\C)\to G$ is surjective.
The set of all gradings of $\C$ by $G$ depends functorially
on $G$, so there is a notion of {\it universal grading.} As explained in \S3.2 of \cite{GN},
a universal grading exists. The corresponding group $G$ will be denoted by $U_{\C}$.
The following statements are also proved in \S3.2 of \cite{GN} (see
\cite[Lemma 3.4, Theorem 3.5, and Proposition 3.9]{GN}).

\begin{proposition}\label{un-grad} 
\begin{enumerate}
\item[(i)]The universal grading $\deg :\O(\C)\to U_{\C}$ is faithful.
\item[(ii)]The trivial component of the universal grading equals $\C_{ad}$.
Every fusion subcategory  $\D\subset\C$ containing $\C_{ad}$
has the form $\D =\bigoplus\limits_{g\in G}\, \C_g$ for some subgroup $G\subset U_\C$.
\item[(iii)] The group of tensor automorphisms of the identity functor $\id_{\C}$ is
canonically isomorphic to $\Hom (U_{\C},k^{\times}$). The homomorphism
$f_{\psi}\in\Hom (U_{\C},k^{\times})$ corresponding to  $\psi\in\Aut_\otimes(\id_\C)$
is defined as follows: assigning to a simple $X\in\C$ the element  $\psi_X\in\Aut (X)=k^{\times}$
one gets a grading of $\C$ by $k^{\times}$, and $f_{\psi}:U_{\C}\to k^{\times}$ corresponds to 
this grading by the universality of $U_{\C}$.
\end{enumerate}
\end{proposition}

\begin{remark} \label{reformul}
The second part of (ii) means that
if $\D\subset\C$ is a  fusion subcategory containing $\C_{ad}$
then for each $g\in U_{\C}$ either $\D \supset \C_g$ or $\D \cap \C_g=0$.
\end{remark}

Gradings $d_1: \O(\C)\to G_1$ and $d_2: \O(\C)\to G_2$ are said to be {\em equivalent\,}
if $d_2=f\circ d_1$ for some isomorphism $f:G_1\iso G_2$.

\begin{corollary} \label{trivcomp}
There is a one-to-one correspondence between fusion subcategories $\D \subset \C$ 
containing $\C_{ad}$ and subgroups $G\subset U_\C$, namely
$$\D \mapsto G_\D:= \{g\in U_\C\mid \D \cap \C_g \neq 0\},\quad 
G\mapsto \D_G:=\bigoplus\limits_{g\in G}\, \C_g.$$
\end{corollary}

\begin{proof}
Use Proposition \ref{un-grad}(i-ii).
\end{proof}

Let $\D\subset\C$ be a fusion subcategory. Among all gradings of $\C$ trivial on $\D$ there is
a universal one, namely, the composition $\O(\C) \to U_\C \to   U_\C/G$, where $G\subset U_{\C}$
is the normal subgroup generated by the image of $U_{\D}$.

\begin{corollary}\label{trivial component}
There is a one-to-one correspondence between equivalence classes of
faithful gradings of $\C$ and fusion subcategories $\D \subset \C$ 
containing $\C_{ad}$ such that the subgroup $G_{\D}\subset U_{\C}$ is normal.
Namely, one associates to $\D$ the universal grading of $\C$ trivial on $\D$; one associates to a grading its trivial component.
\end{corollary}

\begin{proof}
For every normal subgroup $G\subset U_\C$ the composition $\O(\C) \to U_\C \to   U_\C/G$
is a faithful grading. By the universality of the grading $\O(\C)\to U_\C$, one thus gets a one-to-one correspondence between normal subgroups of $U_\C$  and 
equivalence classes of faithful gradings of $\C$. Now use Corollary~\ref{trivcomp}.
\end{proof}

\begin{remark} \label{automat}
If $\C$ is braided then $U_\C$ is abelian. So in this case the normality condition in Corollary
\ref{trivial component} holds automatically.
\end{remark}

\subsection{Dimensions}
Let $\C$ be a fusion category. The tensor product in $\C$ induces a ring structure
on the Grothendieck ring $K_0(\C)$. 

There are two different but related notions of dimension in $\C$. One of them takes
values in $\mathbb{R}$, the other one takes values in $k$.

\subsubsection{Frobenius-Perron dimension} \label{fpd}
This is a homomorphism $\FPdim :K_0(\C )\to\mathbb{R}$
such that $\FPdim (X)\ge 0$ for all 
$X\in\C$. According to \S8 of \cite{ENO}, there exists one and only one such homomorphism
and in fact, $\FPdim (X)\ge 1$ for all nonzero $X\in\C$.
Moreover, the numbers $\FPdim (X)$, $X\in\C$, are algebraic integers in some
cyclotomic field (see \cite[Corollaries 8.53-8.54]{ENO})).

The  {\it Frobenius-Perron dimension of a fusion category\,} $\C$ is defined by
$\FPdim(\C ):=\sum_{X\in \O(\C)}\FPdim(X)^2$. This is a positive cyclotomic
algebraic integer in $\mathbb{R}$. Moreover, according to \cite{O3}, the ideal
in the ring of algebraic integers generated by $\FPdim(\C )$ is stable under
$Gal (\bar\BQ /\BQ )$.

\begin{example}  \label{Hopf}
Let $H$ be a finite dimensional semisimple Hopf algebra over $k$ and $\C:=\Rep(H)$, 
i.e., $\C$ is the category
of finite dimensional representations of $H$. We have the usual dimension homomorphism
$\dim_k :K_0(\C )\to\mathbb{Z}$. Since $\dim_kX>0$ for all nonzero $X\in\C$ we see that
$\FPdim(X) = \dim_k(X)$ for all $X\in\C$. Therefore $\FPdim(\Rep(H)) = \dim_k(H)$.
\end{example}

\subsubsection{Traces, squared norms, and global dimension} \label{squared}

Let $\C$ be a fusion category and
let $\psi$ be any  natural (not necessarily tensor) isomorphism between
the identity and the functor of taking the second dual. Such an isomorphism
exists in any fusion category.
It is completely determined by a collection of isomorphisms 
$\psi_X: X \xrightarrow{\sim} X^{**},
\, X\in \O(\C)$.   For any object $X$ of $\C$ let 
\[
\ev_X: X^*\ot X \to \be \quad \mbox{and}\quad \coev_X: \be \to X\ot X^*
\]
denote the evaluation and coevaluation maps of $X$.
For any morphism $f: X \to X$ we can define two traces, depending on $\psi$:
\begin{eqnarray}
\label{2 traces +}
\Tr_+(f) &:& \be \xrightarrow{\text{coev}_X} X \ot X^* \xrightarrow{\psi_X f\ot \id_{X^*}}
X^{**}\ot X^* \xrightarrow{\text{ev}_{X^*}} \be, \\
\label{2 traces -}
\Tr_-(f) &:& \be \xrightarrow{\text{coev}_{X^*}} X^*\ot X^{**} \xrightarrow{ \id_{X^*} \ot f\psi_X^{-1}}
X^*\ot X  \xrightarrow{\text{ev}_X} \be.
\end{eqnarray}

Let us denote $d_+(X) = \Tr_+(\id_X)$ and $d_-(X) = \Tr_-(\id_X)$.
Note that $d_\pm(X)\neq 0$ for all simple $X$, see e.g., \cite{ENO} or \cite{BK}.
The dimensions $d_\pm(X)$ are of course non-canonical since they depend 
on the choice of $\psi$. However,  for any $X\in \O(\C)$ we have  
a canonical quantity 
\begin{equation}
\label{dim1}
|X|^2:= d_+(X) d_-(X), 
\end{equation}
known as M\"uger's {\em squared norm} of $X$ (the definition is
due to M.~M\"uger \cite{Mu-I}, the name and notation were introduced 
in \S2 of \cite{ENO}).
The squared norm  $|X|^2$ is independent of the choice of $\psi$.

Note that for any morphism $f: X\ot Y \to X\ot Y $
we can define morphisms
\begin{equation*}
\id_X \ot \Tr_+(f) :
X \xrightarrow{\text{coev}_Y} X \ot Y \ot Y^* \xrightarrow{(\id_X \ot \psi_Y)(f\ot \id_{Y^*}) }
X \ot Y^{**} \ot Y^*  \xrightarrow{\text{ev}_{Y^*}} X, 
\end{equation*}
\begin{equation*}
\Tr_- \ot \id_Y (f) :  Y\xrightarrow{\text{coev}_{X^*} } X^*\ot X^{**} \ot Y
\xrightarrow{(id_{X^*}\ot f) (\psi_X^{-1} \ot \id_Y)} X^*\ot X \ot Y \xrightarrow{\text{ev}_X } Y.
\end{equation*} 
Note that $\Tr_+(\Tr_- \ot \id_Y )(f)= \Tr_- (\id_X \ot \Tr_+)(f)$. We will denote
the latter by $\Tr_- \ot \Tr_+(f)$.

Let $\kcyc\subset k$  be the maximal cyclotomic extension of
$\mathbb{Q}$ in $k$, i.e., the subfield generated by all roots of unity.
The automorphism of $\kcyc$ that takes each root of unity to its inverse
is called {\em complex conjugation\,} and denoted by $z\mapsto\bar z$.
Define $\kRcyc :=\{ z\in \kcyc | \bar z=z\}$; this is the maximal totally real
subfield of $\kcyc$.

It follows from  \cite[Remark 3.1]{ENO} and
\cite[Corollary 8.53]{ENO} that for every $X\in\O(\C)$  its squared norm $|X|^2$
is a totally positive algebraic integer in $\kRcyc$.
According to \cite{O3}, if $\C$ is braided (but not in general) the ideal
in the ring of algebraic integers generated by the  squared norm of any simple object of $X$
is stable under $Gal (\bar\BQ /\BQ )$.

The {\em categorical dimension\,} (also known as  {\em global dimension\,}) of a fusion category $\C$ 
is defined to be the sum of squared norms of its simple objects:
\begin{equation}
\label{categorical dimC}
\dim(\C) = \sum_{X\in O(\C)}\, d_+(X) d_-(X) = \sum_{X\in O(\C)}\, |X|^2,
\end{equation}
where $d_\pm$ are the dimensions introduced above. The categorical dimension
is a totally positive  algebraic integer in $\kRcyc$ which
does not depend on a choice of $\psi$. Moreover, according to \cite{O3}, the ideal
in the ring of algebraic integers generated by $\dim(\C )$ is stable under
$Gal (\bar\BQ /\BQ )$.

\subsubsection{Pivotal structures and categorical dimensions}  \label{pivotal}
By definition, a {\em pivotal structure} on a fusion category $\C$  is a {\em tensor} isomorphism 
$\psi$ between the identity functor $\id_\C: \C\to\C$ to the tensor functor $X\mapsto X^{**}$, $X\in\C$.
If pivotal structures on $\C$ exist (conjecturally, this is always the case) then the set
of all such structures is a torsor over the group of tensor automorphisms of $\id_{\C}$,
which equals $\Hom (U_{\C},k^{\times})$ by Proposition \ref{un-grad}(iii).

Let us fix a pivotal structure $\psi$ on $\C$ and let $d_\pm$ denote 
the dimensions defined using $\psi$  as in \S\ref{squared} above.
It was explained in \cite[\S2]{ENO} that
\[
d_+(X^*)= \overline{d_+(X)} = d_-(X) \qquad
\mbox{ for any object } X \mbox{ in } \C
\]
and that $d_+$  defines a ring homomorphism $d_+ :K_0(\C )\to\kcyc$.
When no confusion is possible we  write $d$ instead of $d_+$.
The element $d(X)\in\kcyc$ is called the {\it categorical} (or {\it quantum})  
{\it dimension\,} of $X\in K_0(\C )$. 

The above homomorphism $d :K_0(\C )\to\kcyc$ depends 
on the pivotal structure as follows: if the pivotal structure is composed with the
tensor automorphism of $\id_C$ corresponding to $f:U_C\to k^{\times}$ then 
$d(X)$ is multiplied by $f(\deg X)$ for every $X\in\O (\C )$. 
Here $\deg : O (\C )\to U_C$ is the universal grading. In particular, a pivotal
structure is uniquely determined by the corresponding homomorphism
$d :K_0(\C )\to\kcyc$.

A pivotal structure is said to be {\it spherical\,} if $d_+(X) = d_-(X)$ for all $X$,
or, equivalently, if  the corresponding homomorphism
$d :K_0(\C )\to\kcyc$ satisfies the relation $d(X^*)=d(X)$. This
is equivalent to the condition
\begin{equation} 
d_+(X)\in \kRcyc  \quad\mbox{for all } X.
\end{equation}
If spherical structures on $\C$
exist then they form a torsor over $\Hom (U_C,\mu_2)$, where 
$\mu_2:=\{ 1,-1\}\subset k^{\times}$.
Note that in the case of a spherical structure equality \eqref{dim1} can be rewritten as
\begin{equation}  \label{dim3}
d_+(X)^2=|X|^2.
\end{equation}
A {\em spherical fusion category\,}  is a fusion category with a spherical structure.

\begin{remark}\label{spherical trace}
If $\C$ is a pivotal category then for any $X\in\C$ and $f\in\End X$ one has the numbers
$\Tr_{\pm} (f)\in k$ defined by formulas \eqref{2 traces +}-\eqref{2 traces -}. If $\C$ is spherical then
$\Tr_+ (f)=\Tr_- (f)$ (if $X$ is simple and $f=\id_X$ this is clear from the equlaity 
$\Tr_{\pm} (f)=d_{\pm}(X)$, and the general case follows).
So if $\C$ is spherical one writes $\Tr (f)$ instead of $\Tr_{\pm} (f)$.
\end{remark}

\subsubsection{Spherical extension of a fusion category}  
\label{spherical extension}
According to \cite{ENO, ENO2}, in any fusion category $\C$ there is a canonical isomorphism 
$g$ between the tensor functors $\id_{\C}$ and  the fourth duality 
functor $X \mapsto X^{****}$, $X\in\C$. 
The spherical extension, or {\em sphericalization} of $\C$ (denoted by $\C^{sph}$)
is defined to be the fusion category whose objects are  pairs $(X, f_X)$, where $X\in \O(\D)$ and  
$f_X: X \to X^{**}$ is an isomorphism such that $f_X^{**}f_X = g_X$,
and a morphism from $(X, f_X)$ to $(X', f_{X'})$ is a morphism $h: X \to X'$ in $\C$
such that $ f_{X'} h = h^{**} f_X$.
Note that for any simple $X$ there are exactly two choices of $f_X$ which differ one from another
by a sign. The category $\C^{sph}$ is a spherical fusion category; 
its spherical structure is given by $f_X: (X, f_{X}) \to  (X^{**}, f_{X^{**}})$.
Note that $\dim(\C^{sph}) = 2 \dim(\C)$.

\begin{remark}  
\label{sphericrem}
The spherical extension 
is a particular case of the {\em equivariantization\,} construction described in \S\ref{equiv2tion}.
Namely, any fusion category $\C$ is equipped with a canonical action of $\mathbb{Z}/2\mathbb{Z}$
such that the non-trivial element of $\mathbb{Z}/2\mathbb{Z}$ takes each $X\in\C$
to $X^{**}$. The fusion category $\C^{sph}$ is nothing but the category of equivariant
objects of $\C$ with respect to this action.
\end{remark}

\subsection{Properties of tensor functors}

\begin{proposition}
\label{tens func}
Let $\C,\, \D$ be fusion categories and let $F: \D \to \C$ be a tensor functor.
\begin{enumerate}
\item[(i)] $F$ is fully faithful if and only if the map $\Hom_\D(X,\, \be)\to \Hom_\C(F(X),\, \be)$ 
is an isomorphism for all $X\in\D$.
\item[(ii)] If  $\FPdim(\C) =\FPdim(\D)$ and $F$ is fully faithful then $F$ is an equivalence.
\item[(iii)] If  $\FPdim(\C) =\FPdim(\D)$ and $F$ is surjective then $F$ is an equivalence.
\end{enumerate}
\end{proposition}
\begin{proof}
To prove (i) observe that for all objects $X,Y$ in  $\D$ we have
$\Hom_\D(X,\, Y) =\Hom_\D(X\ot Y^*,\, \be)$ and
$\Hom_\C(F(X),\, F(Y)) = \Hom_\C(F(X\ot Y^*),\, \be)$.

Statements (ii) and (iii) are proved in \cite[Propositions 2.19, 2.20]{EO}.
\end{proof}

\begin{lemma}
\label{tens fun from product}
Let $\C$ be a fusion category, $\D_1,\D_2\subset\C$ be fusion subcategories,
and $F:\D_1\bt\D_2\to\C$ be a tensor functor whose restriction to
$\D_i$ is the natural embedding $\D_i\hookrightarrow\C,\, i=1,2$. Then $F$ is 
fully faithful if and only if $\D_1\cap\D_2 =\Vec$.
\end{lemma}
\begin{proof}
Let $X_1\bt X_2$ be a simple object of $\D_1\bt\D_2$. 
Then $\Hom_\C(F(X_1\bt X_2), \be) =  \Hom_\C (X_1\ot X_2,\, \be)= 
\Hom_\C (X_1,\, X_2^*)$. This is zero for all $X_1\bt X_2 \neq \be\bt\be$
if and only if the intersection of $\D_1$ and $\D_2$ is trivial.
Hence the statement follows from  Proposition~\ref{tens func}(i).
\end{proof}

\subsection{Pseudo-unitarity}\label{ps-unitarity}

Let $\C$ be a fusion category and $\varphi :\kRcyc\to\mathbb{R}$
a field embedding.

\begin{definition}
$\C$ is {\em $\varphi$-pseudounitary\,} if 
$\varphi (\dim(\C ))=\FPdim(\C )$. 
\end{definition}

\begin{remark}
According to Proposition 8.21 of \cite{ENO}, for any field embedding $\varphi :\kRcyc\to\mathbb{R}$
one has
\begin{equation}  \label{comparison}
\varphi (|X|^2)\le \FPdim (X)^2 \quad\mbox{for all } X\in\O (\C ).
\end{equation}  
Therefore $\varphi (\dim(\C ))\le\FPdim(\C )$, and $\varphi$-pseudounitarity means
that
\begin{equation}  \label{ps-un}
\varphi (|X|^2)=\FPdim (X)^2 \quad\mbox{for all } X\in\O (\C ).
\end{equation}  
\end{remark}

\begin{definition}
A spherical structure on $\C$ is {\em $\varphi$-positive\,}
if the corresponding homomorphism $d : K_0(\C )\to \kRcyc$ satisfies
\[
\varphi (d(X))>0 \quad\mbox{for all } X\in\O (\C ).
\]
\end{definition}

By \eqref{dim3}, a $\varphi$-positive spherical structure is unique if it exists.

\begin{proposition}\label{positive} 
A $\varphi$-pseudounitary fusion category has a (unique) $\varphi$-positive 
spherical structure.
\end{proposition}
This is \cite[Proposition 8.23]{ENO}.
By \eqref{dim3} and \eqref{ps-un}, for the $\varphi$-positive spherical structure on a
pseudounitary fusion category $\C$ one has 
\begin{equation}  \label{dimFPdim}
\varphi (d(X))=\FPdim (X).
\end{equation}

\subsection{Integral and weakly integral fusion categories} \label{FPZ}
\begin{definition}   \label{intdef}
A fusion category $\C$ is {\em weakly integral\,} if $\FPdim (\C )\in\mathbb{Z}$.
We say that $\C$ is {\em integral\,} if $\FPdim (X)\in\mathbb{Z}$ for all $X\in\C$.
\end{definition}

 For weakly integral fusion categories the notion of $\varphi$-pseudounitarity does not depend
on $\varphi :\kRcyc\to\mathbb{R}$, so we will write ``pseudounitarity" instead of 
``$\varphi$-pseudounitarity". 
Moreover, according to the next proposition, in this situation pseudounitarity
is automatic. Set $\mathbb{Q}^{\times}_{>0}:=\{ x\in \mathbb{Q}^{\times}|x>0\}$.

\begin{proposition}\label{FPint1} 
Let $\C$ be a weakly integral fusion category. Then
\begin{enumerate}
\item[(i)] $\C$ is pseudounitary, i.e., $\dim(\C )=\FPdim(\C )$;
\item[(ii)] for every $X\in\O (\C )$ one has $|X|^2=\FPdim (X)^2\in\mathbb{Z}$;
\item[(iii)] the map $\deg :\O (\C )\to\mathbb{Q}^{\times}_{>0}/(\mathbb{Q}^{\times}_{>0})^2$
that takes $X\in\O (\C )$ to the image of $\FPdim (X)^2$
in $\mathbb{Q}^{\times}_{>0}/(\mathbb{Q}^{\times}_{>0})^2$ is
a grading.
\end{enumerate}
\end{proposition}

 This is a combination of \cite[Propositions 8.24, 8.27]{ENO} and \cite[Theorem 3.10]{GN}.
 
\begin{rems}  \label{intrem}
\begin{enumerate}
\item[(a)] The trivial component of the grading from Proposition \ref{FPint1}(iii) is 
the maximal integral fusion subcategory of $\C$.
\item[(b)] By Proposition \ref{FPint1}(i), a weakly integral fusion category $\C$ 
has integral  global dimension $\dim(\C )$. The converse is false, see  
\cite[Remark 8.26]{ENO} for a counterexample.
\item[(c)] Ising fusion categories (see Appendix \ref{Is}) are weakly integral but not integral.
\end{enumerate}
\end{rems}
 
\begin{corollary}\label{subwi}
A fusion subcategory of a weakly integral fusion category is weakly integral.
\end{corollary}

\begin{proof} This is an immediate consequence of Definition \ref{intdef} and Proposition \ref{FPint1} (ii).
\end{proof}

\begin{proposition}\label{dimgradi}
If a fusion category $\C$ is faithfully graded by a group $G$
then $\FPdim(\C )=|G|\cdot \FPdim(\C_1)$, where $\C_1\subset\C$ is the 
trivial component of the grading. 
\end{proposition}

This is  \cite[Theorem 8.20]{ENO}. Later we will prove the analog of Proposition~\ref{dimgradi}
for categorical dimensions (see Corollary~\ref{dimgradcat}).

Combining Proposition \ref{dimgradi}
with Proposition \ref{FPint1}(iii) and using the fact that any
finite subgroup of $\mathbb{Q}^{\times}_{>0}/(\mathbb{Q}^{\times}_{>0})^2$ is
a 2-group one gets the following statement from \cite{GN}.

\begin{corollary}   \label{odd}
A weakly integral fusion category of odd Frobenius-Perron dimension is integral.
\end{corollary}

\begin{proposition}   \label{FPint2}
The following properties of a fusion category $\C$ are equivalent:
\begin{enumerate}
\item[(i)] $\C$ is integral;
\item[(ii)] there exists a faithful $k$-linear functor $F:\C\to\Vect$
such that $F(X\otimes Y)\simeq F(X)\otimes F(Y)$ for all $X,Y\in\C$;
\item[(iii)] $\C$ is equivalent to the fusion category of finite dimensional 
representations of a finite dimensional quasi-Hopf algebra $H$.

If $\C$ satisfies (i)-(iii) then $\FPdim (X)=\dim_kF(X)$ for all $X\in\C$,
and therefore $\FPdim (\C )=\dim H$.
\end{enumerate}
\end{proposition} 

This is  \cite[Theorem 8.33]{ENO}. The main argument of the proof was
explained in Example \ref{Hopf} above. Note that the algebra $H$ from
Proposition \ref{FPint2} is automatically semisimple because the category $\C$ is
assumed to be semisimple.

Combining Propositions \ref{positive} and  \ref{FPint1} with formula 
\eqref{dimFPdim} we get the following statement.
\begin{corollary}    \label{FPint3}
A weakly integral fusion category has a unique 
$\varphi$-positive spherical structure for each field embedding $\varphi :\kRcyc\to\mathbb{R}$. 
For this structure one has $d(X)=\FPdim (X)$.
\end{corollary}

\begin{remark}\label{nophi}
For integral fusion categories the notion of $\varphi$-positivity does not
depend on $\varphi$, so one can write ``positive" instead of  $\varphi$-positive.
The homomorphism $d :K_0(\C )\to\mathbb{Z}$ corresponding to
a positive spherical structure equals $\FPdim :K_0(\C )\to\mathbb{Z}$. 
\end{remark}

\begin{example}\label{point sph} Let $\C$ be a pointed category. Then $\FPdim(X)=1$
for any $X\in \O(\C)$, so $\C$ is integral. In this case it is very easy to describe the positive 
spherical structure from Remark \ref{nophi}. Namely, note that for an invertible object $X$ 
the map $\coev_X: \be \to X\otimes X^*$ is an isomorphism,
hence invertible. Now let $\delta_X: X\to X^{**}$ be the composition 
$$
X \xrightarrow{\id_X\otimes \coev_{X^*}} X\otimes X^*\otimes X^{**}
\xrightarrow{\coev_X^{-1} \otimes \id_{X^{**}}}X^{**}.
$$
Then it is easy to see that $\delta_{X\otimes Y}=\delta_X\otimes \delta_Y$ for $X,Y \in \O(\C)$ 
and hence $\delta$
(extended to $\C$ by linearity) is a pivotal structure on $\C$. It is easy to check that for this structure we
have $d(X)=1$ for $X\in \O(\C)$, so this is the positive spherical structure from 
Corollary \ref{FPint3} and Remark \ref{nophi}. 
\end{example}

\bigskip

\subsection{Braided fusion categories and non-degeneracy}
\subsubsection{$\tilde{S}$-matrix of a braided fusion category}
\label{tildeS matrix}

Recall that a {\em braided} tensor  category $\C$ is a  tensor
category equipped with a natural isomorphism $c_{X,Y}: X\otimes Y\simeq
Y\otimes X$ satisfying the hexagon axioms, see \cite{JS, BK}.

Recall that in  \S\ref{squared} given a natural isomorphism
$\psi_X: X \xrightarrow{~} X^{**}$ on $\C$ we defined traces
$\Tr_\pm$ and dimensions $d_\pm(X), \, X\in \O(\C)$, which depend on $\psi$.
Namely, replacing $\psi$ by $\psi\circ \alpha$, where 
$\alpha =\{\alpha_X\}_{X\in \O(\C)}$ is an automorphism of the identity
functor of $\C$, changes $d_\pm(X)$ to $\alpha_X^{\pm 1} d_\pm(X)$.

The following definition was suggested to us by Pavel Etingof.

\begin{definition}
Let $\C$ be a braided fusion category. Define a matrix 
$\tilde{S}=\{\tilde s_{XY}\}_{X,Y\in \O(\C)}$ by
\begin{equation}
\label{def tilde S}
\tilde{s}_{XY} := \dfrac{\Tr_- \ot \Tr_+(c_{Y,X}c_{X,Y})} {d_-(X) d_+(Y)}.
\end{equation}
\end{definition}

Note that  if one replaces $\psi$ by $\psi\circ \alpha$ as above 
then both the numerator and the denominator of \eqref{def tilde S} are
multiplied by $\alpha_X^{-1} \alpha_Y$, so that $\tilde{s}_{XY}$
does not change.
Thus, the matrix $\tilde{S}$ is an invariant of a braided fusion category $\C$
and does {\em not} depend on the choice of $\psi$.

\begin{definition} \label{defnondeg}
A braided fusion category $\C$ is said to be {\em non-degenerate} if
the corresponding  matrix $\tilde{S}$ is non-degenerate.
\end{definition}

Properties of a braided fusion category equivalent to its non-degeneracy 
are listed in Proposition~\ref{Muger1prime} below.

\subsubsection{Pre-modular and modular categories} \label{premod}

\begin{definition} (\cite{Br}) 
A {\em pre-modular} category is a braided fusion category
equipped with a spherical structure.
\end{definition} 

Let us reformulate this definition in terms of {\em twists.}
Recall that in any tensor category $\C$ the functor $X\mapsto X^{**}$ has a natural structure
of tensor functor. A braided structure on $\C$ provides a natural morphism
$u_X: X\to X^{**}$ defined as the composition
\begin{equation}
\label{Drinfeld iso}
X \xrightarrow{\id_X \ot \text{coev}_{X^*}}  X\otimes X^*\otimes X^{**}
\xrightarrow{ c_{X,X^*} \ot \id_{X^{**}}}  X^*\otimes X \otimes X^{**}
\xrightarrow{\text{ev}_{X} \ot  \id_{X^{**}}} X^{**},
\end{equation} 
where  $\text{ev}_X: X^* \ot X \to \be $ and $\text{coev}_X: \be \to X \ot X^*$ 
are the evaluation and coevaluation morphisms; moreover, it is known that
$u_X$ is an isomorphism. In general, the isomorphism $u_X$ is not tensor; instead, it satisfies
\begin{equation}  \label{nontensor}
u_X\otimes u_Y = u_{X\otimes Y} c_{Y,X}c_{X,Y},
\end{equation}
see e.g. \cite[Lemma 2.2.2]{BK}.
Any isomorphism $\psi_X:X\simeq X^{**}$ functorial in $X\in\C$ can be written as 
\begin{equation} \label{deltatheta}
\psi_X= u_X \theta_X, \quad \theta\in\Aut (\id_\C ).
\end{equation}
It follows from \eqref{nontensor} that
$\psi$ is a tensor isomorphism (i.e., a pivotal structure on $\C$) if and only if 
\begin{equation}
\label{balancing ax}
\theta_{X\ot Y} = (\theta_X\ot\theta_Y) c_{Y,X} c_{X,Y}.
\end{equation}
An automorphism $\theta: \id_\C\to \id_\C$ satisfying \eqref{balancing ax} is called a
{\em twist\,} (or a {\em balancing transformation\,}). It is known (e.g., see \cite[2.2.10, 2.3.15]{BK})
that a pivotal structure is  spherical if an only if the corresponding twist $\theta=\psi u^{-1}$
has the property
\begin{equation} \label{ribbon}
\theta_X^*=\theta_{X^*} \mbox{ for all } X\in \C .
\end{equation}
A twist $\theta$ satisfying \eqref{ribbon} is called a {\em ribbon structure.}
So {\em a pre-modular
category $\C$ is the same as a braided fusion category endowed with a ribbon structure\,}
(actually this is the original definition of pre-modular category from \cite{Br}; our definition
is taken from \cite[2.1]{Mu2}).

\medskip
\noindent {\bf Convention.}
For $X\in \O(\C)$ we will often abuse notation and consider $\theta_X$ as an element of $k^\times$
via the isomorphism $\Hom(X,X)\simeq k$. 
\medskip

Let $\C$ be a pre-modular category with a spherical structure  $\psi$.
Let $\Tr$ and $d$ denote the trace and dimension corresponding to $\psi$.

The {\em $S$-matrix} $S=\{s_{XY}\}_{X,Y\in \O(\C)}$ of $\C$ 
is defined  by 
\begin{equation}
\label{def S matrix}
s_{XY} = \Tr(c_{Y,X}c_{X,Y}),
\end{equation}
see \cite{T, BK,Mu2} (we follow here the convention of \cite{Mu2}
which is different from that of \cite{BK}). We have
\[
\tilde{s}_{XY} = \dfrac{s_{XY}}{d(X) d(Y)},
\]
i.e., the $\tilde{S}$-matrix can be thought of as a normalized $S$-matrix.
Note that $S$ and $\tilde{S}$ have the same rank.

\begin{remark}
Unlike the  $\tilde{S}$-matrix, the $S$-matrix of $\C$ depends on the choice
of isomorphism $\psi:X \xrightarrow{\sim} X^{**}$ (which is chosen
to be a spherical structure in \eqref{def S matrix}). A canonical alternative 
to this is to use $u$ from \eqref{Drinfeld iso} instead of $\psi$ in  \eqref{def S matrix}.
This would result in replacing $s_{XY}$  by $\theta_X^{-1} \theta_Y^{-1} s_{XY}$.
\end{remark}

\begin{definition} (\cite{T}) A pre-modular category $\C$ is said to be {\em modular} if its
$S$-matrix is non-degenerate. 
\end{definition}

Clearly, a pre-modular category is modular if and only  if it is non-degenerate in the 
sense of Definition~\ref{defnondeg}.

Recall that for a spherical fusion category $\C$, its object $Y\in \C$ and $f\in \End Y$ we defined
its trace $\Tr(f)\in k$, see Remark \ref{spherical trace}. In particular this applies in the case
of a pre-modular category $\C$, $Y=X\ot X$ for some $X\in \C$ and $f=c_{X,X}^{-1}$. 
We will need the following well known fact.

\begin{proposition} \label{traceofbraiding}
Let $X$ be a simple object of a pre-modular category 
and $\theta_X\in k^\times$ the corresponding twist. 
Then 
\begin{equation}\label{traceofc}
\theta_X^{-1}d(X)=
\Tr(c_{X,X}^{-1}).
\end{equation}
\end{proposition}

\begin{remark} 
One also has $\theta_Xd(X)=\Tr(c_{X,X})$. This follows from formula \eqref{traceofc} 
applied to the tensor category $\C$ equipped with the braiding $\tilde c_{X,Y}:=c_{Y,X}^{-1}$
(it is well known that the twist corresponding to this braiding equals $\theta^{-1}$).
\end{remark}

\begin{proof}
By definition,  
\begin{equation} \label{first_formula}
\theta_X^{-1}d(X)=\Tr(\psi_X^{-1}u_X), 
\end{equation}
where $\psi_X: X \xrightarrow{\sim} X^{**}$ is the spherical structure and
$u_X: X\iso X^{**}$ is defined by~\eqref{Drinfeld iso}.
Consider the following diagram, whose first row is identical to \eqref{Drinfeld iso}:
\begin{equation*}
\xymatrix{
X \ar[rr]^(.4){\id_X\ot \coev_{X^*}}\ar[d]^{\id_X}&&X\ot X^*\ot X^{**}\ar[rr]^{c_{X,X^*}\ot \id_{X^{**}}}
\ar[d]^{c_{X,X^*\ot X^{**}}}&&X^*\ot X\ot X^{**}\ar[rr]^(.65){\ev_{X}\ot \id_{X^{**}}}
\ar[d]^{\id_{X^*\ot X\ot X^{**}}}&&X^{**} \ar[d]^{\id_{X^{**}}}\\
X \ar[rr]^(.4){\coev_{X^*}\ot \id_X}&&X^*\ot X^{**}\ot X\ar[rr]^{\id_{X^*}\ot c^{-1}_{X,X^{**}}}&&X^*\ot X\ot X^{**}\ar[rr]^(.65){\ev_X\ot \id_{X^{**}}}&&X^{**} }.
\end{equation*}
The left square commutes by naturality of the braiding, the middle one commutes by the hexagon 
axiom, and the commutativity of the right square is obvious. Hence the diagram commutes.  
By definition, $u_X$ is the composition in the top row.  It equals the composition in the bottom row,
i.e., 
\begin{equation} \label{new uX}
u_X = (\text{ev}_{X} \ot  \id_{X^{**}})  \, (\id_{X^*}\ot c_{X,X^{**}}^{-1}) \, (\text{coev}_{X^*}\ot \id_X).
\end{equation}
Using the naturality of the braiding and $\psi_X$ we obtain
\begin{equation} \label{new uX times psi}
\psi_X^{-1}u_X = (\text{ev}_{X} \ot  \id_{X}) \, (\psi_{^*X}\ot c_{X,X}^{-1}) \, (\text{coev}_{^*X}\ot \id_X).
\end{equation}
Next, observe that for  any morphism $f: X \to X$ in a spherical fusion category its trace 
equals the composition
\begin{equation} \label{2nd dual trace}
\be \xrightarrow{\text{coev}_{{}^*X}} {}^*X \ot X \xrightarrow{\psi_{{}^*X} \ot f} X^* \ot X \xrightarrow{\ev_X} \be.
\end{equation}
This follows by comparing \eqref{2nd dual trace} with \eqref{2 traces -}.

Now using \eqref{first_formula},  \eqref{new uX times psi}
and \eqref{2nd dual trace}, we see that the l.h.s. of \eqref{traceofc} equals the composition
\begin{equation*}
\be \xrightarrow{\text{coev}_{{}^*X\ot ^*X}} {}^*X \ot {}^*X \ot X \ot X \xrightarrow{\psi_{ {}^*X \ot {}^*X  } \ot c_{X,X}^{-1}}
 {}^*X \ot {}^*X \ot X \ot X  \xrightarrow{\ev_{X\ot X}} \be.
\end{equation*}

The r.h.s. of \eqref{traceofc} equals the same composition. To see this, apply formula
\eqref{2nd dual trace} to $X\ot X$ instead of $X$.
\end{proof}

\subsubsection{The Verlinde formula}  
\label{Qtrace}

Let  $\C$ be a pre-modular category.
Note that the $S$-matrix of $\C$ is symmetric
and satisfies $s_{X^* Y^*} = s_{XY}$ for all $X,Y\in \O(\C)$.
Applying $\Tr$ to both sides of formula \eqref{balancing ax} we obtain
\begin{equation}
\label{sij}
s_{XY} = \theta_X^{-1}\theta_Y^{-1} \sum_{Z\in \O(\C)}
\, N_{XY}^Z \theta_Z d(Z),
\end{equation}
for all $X,Y,Z\in \O(\C)$, 
where $N_{XY}^Z$ is the multiplicity of $Z$ in $X\ot Y$. 

The elements of the $S$-matrix also satisfy the following condition,
see \cite[Theorem 3.1.12]{BK}, \cite[Lemma 2.4 (iii)]{Mu2}:
\begin{equation}
\label{ss= Nds}
s_{XY} s_{XZ} = d(X) \sum_{W\in \O(\C)} \, N_{YZ}^W s_{XW},\quad X,Y,Z\in \O(\C).
\end{equation}
When $S$ is invertible, equation \eqref{ss= Nds} is equivalent to the Verlinde formula
\cite[Theorem 3.1.14]{BK}.

\subsection{Center of a fusion category}\label{center}
For any fusion category $\C$ its {\em center}
$\mathcal{Z}(\C)$ is defined as the category whose objects are
pairs $(X, c_{-,X})$, where $X$ is an object of $\C$ and
$c_{-,X}$ is a natural family of isomorphisms $c_{V, X} : V\ot X
\simeq X \ot V$, $V\in\C$, satisfying a certain
compatibility condition, see \cite[Defini\-tion~XIII.4.1]{Kass}. 
The tensor category $\Z(\C)$ is equipped with a braiding, see 
{\em loc. cit.}
It is proved in
\cite[Theorem 2.15, Proposition 8.12]{ENO} that $\Z(\C)$ is a fusion
category and
\begin{equation}
\label{dimZ}
\dim(\Z(\C))=\dim(\C)^2, \quad \FPdim(\Z(\C))=\FPdim(\C)^2.
\end{equation}

There is an obvious forgetful tensor functor $\Z(\C)\to \C$.
The following fact was proved in \cite[Proposition 3.39(i)]{EO}.

\begin{proposition}
\label{Forg surj}
The forgetful functor $\Z(\C)\to \C$ is surjective.
\end{proposition}

It is known that the center of a spherical fusion category $\C$ has an
(obviously unique) spherical structure such that the forgetful functor
$\Z(\C )\to\C$ commutes with spherical structures
(see \cite{Kass, Mu4}). Thus $\Z(\C)$ is a pre-modular category. In fact,
$\Z(\C)$ is modular
(this is proved in {\em loc. cit.} under the assumption
$\dim(\C)\ne 0$, which holds automatically by \cite[Theorem 2.3]{ENO}).

\begin{rems} \label{dimFX}
\begin{enumerate}
\item[(i)] 
If a fusion category $\C$ is $\varphi$-pseudounitary in the sense of
\S\ref{ps-unitarity} then so
is $\mathcal{Z}(\C)$ (this follows from \eqref{dimZ}). If a spherical
structure on $\C$ is
$\varphi$-positive in the sense of \S\ref{ps-unitarity} then so is the
corresponding spherical structure on
$\mathcal{Z}(\C)$ (to see this, note that since the forgetful functor
$F:\Z(\C )\to\C$ commutes with spherical structures one has $d(X)=d(F(X))$
for all $X\in\Z(\C)$).
\item[(ii)] 
If a fusion category $\C$ is integral in the sense of Definition \ref{intdef} then so is $\Z(\C)$. 
This is because the function 
$\FPdim : K_0(\Z (\C))\to {\mathbb R}$
equals the composition $f:K_0(\Z (\C))\to K_0(\C)\stackrel{\FPdim}{\longrightarrow} {\mathbb R}$ 
(note that $f$ is a homomorphism and $f(X)>0$ for all nonzero $X\in\Z(\C)$,
i.e., $f$ has the properties that uniquely characterize the homomorphism
$\FPdim : K_0(\Z (\C))\to {\mathbb R}$ according to \S\ref{fpd}). The converse statement is
also true: if $\Z(\C)$ is integral then so is $\C$. This follows from \cite[Corollary 8.36]{ENO}
and Proposition~\ref{Forg surj}.
\end{enumerate}
\end{rems}

We will need the following result in \S \ref{elementxc}.

\begin{proposition} \label{centerweaklyintegral}
Let $\C$ be a fusion category with $\Z(\C)$ weakly integral. Then $\C$ is weakly integral.
\end{proposition}

\begin{proof} Using \eqref{dimZ} we see that $\FPdim(\C\boxtimes \C)=\FPdim(\Z(\C))\in \BZ$, so
the category $\C\boxtimes \C$ is weakly integral. Now $\C =\C \boxtimes \be \subset \C \boxtimes
\C$ and we are done by Corolla\-ry~\ref{subwi}.
\end{proof}

\subsection{Center of a braided fusion category}  \label{centre}
Let $\C$ be a braided fusion category with braiding $c$. 
Let $\C^{\oP}$ be the same fusion category equipped with the braiding $\tilde c_{X,Y}:=c_{Y,X}^{-1}$.
Then the assignment $X\mapsto (X,\, c_{-,X})$ extends to a braided functor $\C\to \Z(\C)$.
Similarly, the assignment $X \mapsto (X,\, \tilde c_{-,X})$ extends to a braided functor $\C^{\oP}\to\Z(\C)$.
Both functors are fully faithful. They combine together
into a single braided tensor functor $G: \C \boxtimes \C^{\oP}\to \Z(\C)$
(see Proposition~\ref{Muger1prime}(iii) and Remark~\ref{newrems}(a) below). 
In Proposition \ref{Muger1prime} we will prove that $G$ is an equivalence if and only if $\C$ is 
non-degenerate. For pre-modular categories this is due to M.~M\"uger \cite{Mu4}.

\subsection{Pointed braided categories and metric groups}\label{CGQ}

\subsubsection{Quadratic forms}\label{Quadra}
Let $G$ and $B$ be abelian groups. We will use the multiplicative notation for the
group operations. Following, e.g., \cite{JS} we define a {\em quadratic form\,} on
$G$ with values in $B$ to be a map $q: G\to B$ 
such that $q(g^{-1})=q(g)$ and the symmetric function $b(g,h):=\frac{q(gh)}{q(g)q(h)}$ 
is bimultiplicative, i.e., $b(g_1g_2,h)= b(g_1,h)b(g_2,h)$ for all $g_1, g_2, h\in G$.
We say that $b:G\times G\to B$ is the bimultiplicative form associated to $q$. If $B=k^{\times}$ for
some field $k$ we will also say that $b:G\times G\to k^{\times}$ is the 
{\em bicharacter associated to\,} $q$. We say that $q$ is {\em non-degenerate\,} if $b$ is.

\begin{rems}  \label{easyrems}
\begin{enumerate}
\item[(i)] 
It is easy to show that if $q$ is a quadratic form in our sense then $q(g^n)=q(g)^{n^2}$ 
for all $g\in G$, $n\in \BZ$. Therefore our notion of quadratic form is equivalent to
Bourbaki's notion of a quadratic form on a $\mathbb{Z}$-module.
\item[(ii)] Not every quadratic form $q:G\to B$ can be represented as $q(g)=\beta (g,g)$,
where $\beta :G\times G\to B$ is bimultiplicative. Counterexample: $G=\BZ/2\BZ$,
$B=\mathbb{C}^{\times}$, $q(n)=i^{n^2}$.
\item[(iii)] Let $q: G\to B$ be a quadratic form and $H\subset G$ the kernel of the corresponding
bimultiplicative form. Then the restriction of $q$ to $H$ is a homomorphism $\chi :H\to B$
such that $\chi (h)^2=1$ for all $h\in H$. Moreover, $q$ is the pullback of a quadratic form
$\tilde q:G/H_0\to B$, where $H_0:=\Ker\chi$. 
\end{enumerate}
\end{rems}

\subsubsection{The category of pre-metric groups}
Now fix an algebraically closed field  $k$ of characteristic 0 and define
$G^*:=\Hom (G,k^{\times})$.

\begin{definition} A {\em pre-metric group\,} is a pair $(G,q)$ where $G$ is a finite abelian group 
and $q:G \to k^\times$ is a quadratic form. A {\em morphism of pre-metric groups\,} 
$(G_1,q_1)\to (G_2,q_2)$ is a homomorphism $\phi: G_1\to G_2$ such that 
$q_2(\phi(g))=q_1(g)$ for all $g\in G_1$.
\end{definition}

Pre-metric groups form a category, which will be denoted by $\PM_k$.

\subsubsection{Metric groups}\label{metra}
Let $G$ be a finite abelian group. A quadratic form $q:G\to k^{\times}$ is said to be 
{\em non-degenerate\,} if the associated bicharacter $b:G\times G\to k^{\times}$ is 
(i.e., if the homomorphism $G\to G^*$ corresponding to $b$ is an isomorphism).

\begin{remark}  \label{warning}
If $|G|$ is odd then by Remark \ref{easyrems}(iii) every quadratic form $q:G\to k^{\times}$ 
comes from a non-degenerate quadratic form on a quotient of $G$. This is not true if 
$G=\mathbb{Z}/2\mathbb{Z}$ and $q(n)=(-1)^n$.
\end{remark}

\begin{definition} 
A {\em metric group} is a pre-metric group $(G,q)$ such that $q$
is non-degenerate.
\end{definition}

\subsubsection{Pointed braided categories and pre-metric groups}   \label{pbcpmg}
Let $\PB_k$ denote the 1-categorical truncation of the 2-category of pointed braided fusion
categories. In other words, $\PB_k$ is the category whose objects are pointed braided fusion
categories and whose morphisms are braided functors up to braided isomorphism.

Let $\C$ be a pointed braided fusion category. Then isomorphism classes of simple objects of $\C$
form a finite abelian group $G$. If $g\in G$ let $q(g)\in k^{\times}$ denote the braiding 
$c_{XX}\in\Aut (X\ot X)=k^{\times}$, where $X$ is a simple object of $\C$ whose isomorphism class equals
$g$. It is easy to show that $q:G\to k^{\times}$ is a quadratic form. By associating to $\C$ the 
pre-metric group $(G,q)$ one gets a functor $F:\PB_k\to\PM_k$.

\begin{proposition}\label{braided metric}
The functor $F$ is an equivalence.
\end{proposition}

This proposition is proved in \cite[\S 3]{JS} using some results on the homology of
Eilenberg - Mac Lane spaces. An elementary computational proof is given in
 \cite[Proposition 2.5.1]{Q}. In Appendix \ref{likbez} we give a proof of the essential 
 surjectivity of $F$.

\subsubsection{The pre-modular category $\C(G,q)$}\label{thetaq}
Let $(G,q)$ be a pre-metric group. By Proposition \ref{braided metric}, it corresponds
to a braided fusion category. Any such braided fusion category will be denoted by
$\C(G,q)$. So $\C(G,q)$ is defined up to a braided equivalence, which is unique up to 
(non-unique) isomorphism. 

Example \ref{point sph} provides a spherical structure and
therefore a pre-modular structure on the category $\C(G,q)$; we use the same symbol
$\C(G,q)$ for this pre-modular category. The corresponding twist is given by 
\begin{equation}\label{tw}
\theta_g=q(g),
\end{equation}
see \eqref{traceofc}.

The fact that $\theta_g=q(g)$ satisfies \eqref{balancing ax} means that if $X,Y$ are simple objects
of $\C(G,q)$ and $\bar X,\bar Y\in G$ are their isomorphism classes then
\begin{equation}\label{cc=b}
c_{Y,X} c_{X,Y}=b(\bar X,\bar Y),
\end{equation}
where $b:G\times G\to k^{\times}$ is the bicharacter associated with $q$. So
the $S$-matrix of the category $\C(G,q)$ is given by $S_{gh}=b(g,h)$, $g,h\in G$.
Thus {\em the pre-modular category $\C(G,q)$ is modular if
and only if the pre-metric group $(G,q)$ is a metric group.}

\subsubsection{The pre-modular category $\C(G,q,\chi )$}\label{genpointed}
In the situation of \S\ref{thetaq} let $\chi :G\to k^{\times}$ be a character such that $\chi^2=1$.
It defines a tensor automorphism $\sigma_{\chi}$ of the identity functor $\C(G,q)\to\C(G,q)$, namely
$\sigma_{\chi}$ acts on a simple object $X$ as multiplication by $\chi (\bar X)$, where $\bar X$
is the class of $X$ in $G$. Define a pre-modular category $\C(G,q,\chi )$ as follows: as a braided fusion category it equals $\C(G,q)$, 
but the spherical structure on $\C(G,q, \chi )$ differs from the one considered in \S\ref{thetaq} by $\sigma_{\chi}$. 
Thus one gets a bijection between
isomorphism classes of triples $(G,q,\chi )$ and pointed pre-modular categories up to
pre-modular equivalence.

The simple object of $\C(G,q,\chi )$ with isomorphism class $g\in G$ has dimension
$\chi (g)$. The twist in $\C(G,q,\chi )$ is given by $\theta_g=q(g)\chi (g)$. The 
$S$-matrix of $\C(G,q,\chi )$ is given by $S_{gh}=b(g,h)\chi (g)\chi (h)$, $g,h\in G$.

\subsection{Symmetric fusion categories}\label{Deligne}
A braided fusion category $\C$ is called {\em symmetric} if $\C'=\C$. Equivalently,
$c_{Y,X}c_{X,Y} =\id_{X\ot Y}$ for all objects $X,Y \in \C$; in this case the braiding $c$ is
called {\em symmetric}. 

The category $\Rep(G)$ of representations of a finite group $G$ equipped with its standard
symmetric braiding $\sigma_{XY}(x\otimes y):=y\otimes x$
is an example of a symmetric fusion category. Deligne  \cite[Corollaire 0.8]{De}
proved that any symmetric fusion category is equivalent to a "super" generalization of $\Rep(G)$. 
Namely, let $G$ be a finite group and let $z\in G$
be a central element such that $z^2=1$. Then the fusion category $\Rep(G)$ has a braiding
$\sigma'$ defined as follows:
\[
\sigma'_{XY}(x\otimes y)=(-1)^{mn}y\otimes x \,\mbox{ if }\,
x\in X, \; y\in Y, \; zx=(-1)^{m}x,\; zy=(-1)^{n}y.
\]
The fusion category $\Rep(G)$ equipped with this braiding will be denoted by $\Rep(G,z)$.
Equivalently, $\Rep(G,z)$ can be described as a full subcategory of the category of super-representations of $G$; 
namely, $\Rep(G,z)$ consists of those super-representations $V$ on 
which $z$ acts by the {\em parity automorphism\,} 
(i.e., $zv=v$ if $v\in V$ is even and $zv=-v$ if $v\in V$ is odd). 

\begin{example} Let $G=\BZ/2\BZ$ and $z$ be the nontrivial element of $G$. Then
$\Rep(G,z)$ is the category of super-vector spaces $s\Vec$.
\end{example}

Recall (see \cite{De}) that for a symmetric fusion category $\C$ a {\em super fiber functor} is
a braided tensor functor $F: \C \to s\Vec$. For a super fiber functor $F: \C \to s\Vec$ let
$G_F$ be the group of tensor automorphisms of $F$. Let $z_F\in G_F$ denote the parity
automorphism of $F$ (i.e., for each $c\in\C$ the automorphism of $F(c)\in s\Vec$ corresponding
to $F(c)$ is the parity automorphism). It is clear that $z_F$ is central and $(z_F)^2=1$. 
A super fiber functor $F: \C \to s\Vec$ has an obvious lifting $\tilde F: \C \to \Rep(G_F,z_F)$. 

\begin{theorem} {\em (\cite{De})} \label{Deligne1} Let $\C$ be a symmetric fusion category.
Then
\begin{enumerate}
\item[(i)] there exist super fiber functors $\C \to s\Vec$;
\item[(ii)]  all super fiber functors $\C \to s\Vec$ are isomorphic to each other;
\item[(iii)]  for any super fiber functor $F: \C \to s\Vec$ the functor $\tilde F: \C \to \Rep(G_F,z_F)$
constructed above is an equivalence.
\end{enumerate}
\end{theorem}

\begin{remark} 
Theorem \ref{Deligne1} shows that the isomorphism class of a pair
$(G,z)$ is uniquely determined by the symmetric fusion category $\Rep(G,z)$.
\end{remark} 

\begin{example} 
Let $\C$ be a pointed braided category corresponding to a pre-metric group
$(\Gamma ,q)$ (see \S\ref{CGQ}). Such $\C$ is symmetric if and only if the 
bicharacter $b$ associated to $q$ is trivial. Then $q:\Gamma\to k^{\times}$ is a character 
such that $q^2=1$, and $\C\simeq\Rep(G,z)$ with 
$G=\Gamma^*:=\Hom (\Gamma,k^{\times})$, $z=q$.
\end{example}

The following statement immediately follows from Theorem \ref{Deligne1} and Example~\ref{Hopf}.

\begin{corollary} \label{symint}
A symmetric fusion category is integral in the sense of Defini\-tion~\ref{intdef}.
\end{corollary}

A symmetric fusion category $\C$ has a spherical structure such that the corresponding twist 
(see \S\ref{premod}) equals the identity. In general, this spherical structure is not positive: 
it is easy to see that if $\C =\Rep(G,z)$ then the categorical dimension of $X\in\C$ equals the 
{\em super-}dimension of the super-vector space $X$ (i.e., the trace of the operator $z$ acting on the vector space $X$), so it can be negative. 

\begin{definition}\label{tannaka}
A symmetric fusion category $\C$ is said to be {\em Tannakian} 
if it satisfies the following equivalent conditions: 
\begin{enumerate}
\item[(i)] the spherical structure on $\C$ defined above is positive;
\item[(ii)] $\C$ admits a {\em fiber functor}, i.e., a  braided tensor functor $\C\to\Vec$;
\item[(iii)]  the essential image of any super fiber functor $F: \C \to s\Vec$ is contained in
$\Vec \subset s\Vec$;
\item[(iv)] there exists a finite group $G$ such that $\C$ is equivalent to $\Rep(G)$ as a braided fusion
category.
\end{enumerate}
\end{definition} 

\begin{example}
A pointed braided category corresponding to a pre-metric group $(G,\, q)$ is Tannakian
if and only if $q=1$.
\end{example}

\begin{lemma}   \label{tan-eq}
If $\C_1, \C_2$ are symmetric categories
and $F: \C_1 \to \C_2$ is a surjective braided functor then
$\C_1$ is Tannakian if and only if $\C_2$ is.
\end{lemma}
\begin{proof}
If $\C_2$ saitisfies property (ii) from Definition~\ref{tannaka}
then so does $\C_1$. By surjectivity, if $\C_1$ satisfies (iii) then
so does $\C_2$.
\end{proof}

Now let $\C$ be a general symmetric fusion category. By Corollary \ref{symint},
Proposition \ref{FPint3}, and Remark \ref{nophi}, $\C$
has a unique positive spherical structure. The corresponding 
twist $\theta$ is a tensor  automorphism of the identity functor $\id_{\C}$, so by Proposition 
\ref{un-grad}(iii) it corresponds to a $\mu_2$-grading of 
$\C$. Let $\C_1\subset\C$ be the trivial component of this grading, i.e., $\C_1$ is the full subcategory
of objects $X\in\C$ such that $\theta_X=\id_X$. In other words, $\C_1$ is the maximal Tannakian 
subcategory of $\C$.

\begin{corollary}  \label{deligne} 
\begin{enumerate}
\item[(i)] If $\FPdim(\C)$ is odd then $\C_1=\C$. In other words, a symmetric fusion category
of odd Frobenius-Perron dimension is Tannakian.
\item[(ii)] In general, either $\C_1=\C$ or $\FPdim(\C_1)=\frac12\FPdim(\C)$.
\end{enumerate}
\end{corollary}
\begin{proof} 
This follows from Proposition~\ref{dimgradi}. One can also use Theorem \ref{Deligne1} to
reduce the statement to the case $\C =\Rep(G,z)$. In this case one easily checks that the twist
$\theta_X$ equals $z|_X$, $\C_1=\Rep(G/\langle z\rangle)\subset \Rep(G)$, $\FPdim(\C)=|G|$,
and  $\FPdim(\C_1)=|G/\langle z\rangle |$.
\end{proof}

\subsection{The regular algebra of a Tannakian category}  \label{subsec-reg}
Let $\E$ be a Tannakian category (see Definition
\ref{tannaka}). By Theorem \ref{Deligne1} (iii), a choice of a fiber functor $F: \E \to \Vec$
defines a finite group $G_F:=\Aut F$ and a braided equivalence $\tilde F :\E\simeq \Rep(G_F)$.
Let $Fun(G_F)$ be the algebra of functions $G_F\to k$; the group $G_F$
acts on $Fun(G_F)$ via left translations, so $Fun(G_F)$ is a commutative algebra in the category
$\Rep(G_F)$. Therefore $A_{\E,F}:=\tilde F^{-1}Fun(G_F)$ is a commutative algebra in $\E$.

\begin{definition}\label{cca-drinf}
The algebra $A_{\E,F}$ constructed above is called the {\em regular algebra} of $(\E ,F)$.
\end{definition}

In practice, we will often denote it by $A_{\E}$ and call it the regular algebra in $\E$.
To some extent, this abuse of language and notation is justified because the isomorphism class of $A_{\E,F}$ does not depend on $F$ (recall that by Theorem \ref{Deligne1}(ii) all fiber functors are isomorphic).

\begin{remark}  \label{regular-rem}
One has a natural isomorphism
\begin{equation}  \label{regalg} 
\Hom (V, A_{\E,F}):=\Hom (F(V), \be ), \quad V\in\E ,
\end{equation}
where $\be \in \Vec$ is the unit object. In other words, $A_{\E,F}=I( \be)$, where 
$I: \Vec \to \E$ is right adjoint  to $F$. One can use \eqref{regalg} as a definition of
the regular algebra and describe the algebra structure on $A_{\E,F}$ as follows.
Since $F$ is a tensor functor for any $V_1,V_2\in\E$ one gets a map 
\[
\Hom (F(V_1), \be )\ot\Hom (F(V_2), \be )\to\Hom (F(V_1)\ot F(V_2), \be )=\Hom (F(V_1\ot V_2),\be ),
\]
i.e., a natural map  $\Hom (V_1, A_{\E,F})\ot \Hom (V_2, A_{\E,F})\to\Hom (V_1\ot V_2, A_{\E,F})$. 
Such a natural map  is the same as a morphism $A_{\E,F}\ot A_{\E,F}\to A_{\E,F}$. This is
the algebra structure on $A_{\E,F}$.
\end{remark}

For any algebra $A$ in a monoidal category the algebra $A$ itself and $A\ot A$ both have
natural structures of $A$-bimodules. Moreover, the multiplication morphism $A\ot A\to A$
is a morphism in the category of $A$-bimodules.
We will often use the following property of the regular algebra.

\begin{proposition} \label{splitting} 
Let $A=A_{\E,F}$ be the regular algebra in a Tannakian fusion category $\E$. Then the multiplication
morphism $A\ot A\to A$ admits a unique  $A$-bimodule splitting. 
In particular, the bimodule $A$ is a direct summand of the bimodule $A\ot A$.  \hfill\qedsymbol
\end{proposition} 

\section{Centralizers in and non-degeneracy of braided fusion categories} \label{Sect3}
\subsection{The subject of this section}
The main goal of this section is to extend the results of M\"uger \cite{Mu1, Mu2, Mu3}
and Altsch\"uler-Brugui\`eres \cite{AB} on centralizers to the setting of arbitrary 
braided fusion categories (not necessarily non-degenerate or pre-modular ones). 
We also prove some other results  (e.g., see Theorem~\ref{Muger2-FPdim} and the 
theory of  projective  centralizers in \S\ref{projcentsubsec}).

In terms of the the heuristic Principle~\ref{intro principles}(ii),
the M\"uger-Altsch\"uler-Brugui\-\`eres theory can be viewed as a
``quantization" of the following obvious proposition.

\begin{proposition}
\label{heuristic centr}
Let $(C,q)$ be a pre-metric Lie coalgebra (see Definition~\ref{Casimirdef}) and 
$\mathfrak{a}\subset C$ a Lie subcoalgebra.
\begin{enumerate}
\item[(i)] $\mathfrak{a}$ is a metric Lie coalgebra 
if and only if $\mathfrak{a}  \cap \mathfrak{a}^\perp = 0$. In particular, $C$
is a metric Lie coalgebra if and only if $C^\perp=0$.
\item[(ii)] $\dim \mathfrak{a} + \dim \mathfrak{a}^\perp =
\dim C + \dim C^\perp$ and $\mathfrak{a}^{\perp\perp} = \mathfrak{a} + C^\perp$.
In particular, if $\mathfrak{a} \supset C^\perp$  then
$\mathfrak{a}^{\perp\perp} = \mathfrak{a}$.
\item[(iii)] Assume that  $\mathfrak{a}$ is a metric Lie coalgebra. Then $C\simeq \mathfrak{a} \oplus \mathfrak{a}^\perp$
and $C$ is a metric Lie coalgebra if and only if $\mathfrak{a}^\perp$ is.  \hfill\qedsymbol
\end{enumerate}
\end{proposition}

The ``quantum'' analogs of the statements of Proposition~\ref{heuristic centr} are
Theorem~\ref{Muger1} and Corollary~\ref{a friend of Muger1} for (i), Theorem~\ref{Muger2}
and Corollary~\ref{Muger-2'} for (ii), and Theorem~\ref{Muger3} for (iii).

The notion of projective  centralizer from \S\ref{projcentsubsec} also has a ``classical" analog:
namely, given a subcoalgebra $\mathfrak{a}$ of  a Lie subcoalgebra $C=\fg^*$ one can
form the subcoalgebra $[\fg ,\mathfrak{a}]^{\perp}$.

\subsection{Centralizers}\label{centra-drinf}

In \cite{Mu1, Mu2} M\"uger introduced the following definition.

\begin{definition}
\label{def centralizer}
Objects $X$ and $Y$ of a braided fusion category $\C$ are said to
{\em centralize\,}  each other if
\begin{equation} \label{monodromy-drinf} 
c_{Y,X}c_{X,Y} =\id_{X\ot Y}.
\end{equation}
\end{definition}
The {\em centralizer\,} $\D'$ of a fusion subcategory $\D\subset\C$ is defined to
be the full subcategory of objects of $\C$ that centralize each object of $\D$.
It is easy to see that $\D'$ is a fusion subcategory. 
Clearly, $\D$ is symmetric if and only if $\D\subset\D'$.

Recall from Definition~\ref{defnondeg} that  we defined the $\tilde{S}$-matrix
$\tilde{S}=\{\tilde{s}_{XY}  \}_{X,Y\in \O(\C)}$ of a braided fusion category $\C$ by
\[
\tilde{s}_{XY} =\dfrac{\Tr_-\ot \Tr_+(c_{Y,X} c_{X,Y})}{d_-(X)d_+(Y)}.
\]
Here traces  $\Tr_\pm$ and dimensions $d_\pm$ depend on the choice
of a natural isomorphism $\psi: X\xrightarrow{\sim} X^{**}$,
but $\tilde{S}$ does not. Recall that  $\C$ is called {\em non-degenerate} 
if and only if $\tilde{S}$ is non-degenerate.

For every $X\in \O(\C)$  the squared norm $|X|^2=d_+(X)d_-(X)$ is independent
from the choice of $\psi$. Recall also that the sum 
\[
\dim(\C):= \sum_{X\in \O(\C)}\, |X|^2
\]
is non-zero  by \cite{ENO} (see \S\ref{squared}).

A characterization of the centralizer of a fusion subcategory of a pre-modular
fusion category  in terms of its $S$-matrix was given by 
D.~Altsch\"uler and A.~Brugui\`eres in \cite{AB}. In  Appendix~\ref{appensix Alt-Br} 
we prove a generalization of their result for a braided (not necessarily spherical)
fusion category $\C$. Namely, we prove in Proposition~\ref{AB criterion} that
for an object $V\in \O(\C)$ we have
\begin{equation}
\label{O(D')}
V\in \O(\D') \Longleftrightarrow \tilde{s}_{YV} = 1 \mbox{ for all } Y \in \O(\D),
\end{equation}
and 
\begin{equation}
\label{not O(D')}
V\not\in \O(\D') \Longleftrightarrow   \sum_{Y\in \O(\D)}\, |Y|^2\tilde{s}_{YV} =0 
\Longleftrightarrow   \sum_{Y\in \O(\D)}\, \tilde{s}_{VY} |Y|^2 =0.
\end{equation}

\begin{example} \label{pointedcase}
Let $(G,q)$ be a pre-metric group and $H\subset G$ a subgroup. If $\C$ is a pointed braided
category corresponding to a pre-metric group $(G,q)$ and $\mathcal{K}$ corresponds to
a subgroup $H\subset G$ then $\mathcal{K}'$ corresponds to $H^{\perp}$. This immediately follows from \eqref{cc=b}.
\end{example}

Now we will formulate the main properties of the centralizers. 
In the case of {\em pointed\,} braided categories (see Example \ref{pointedcase}) 
Theorems \ref{Muger1}, \ref{Muger2}, and \ref{Muger3} below amount
to standard properties of orthogonal complements.

Given a fusion subcategory $\D \subset \C$ we introduce an equivalence relation on $\O(\C)$ 
as follows. We say that objects $Y, Z \in \O(\C)$ are equivalent if there is an object $W$
in $\D'$ such that  $Y$ is contained in $Z \ot W$.  It is easy to see that this
is indeed an equivalence relation on $\O(\C)$. The full abelian subcategory of
$\C$ generated by an equivalence class of this relation will be called 
a {\em $\D'$-component} of $\C$. This terminology
is justified by the fact that $Y$ and $Z$ are equivalent if and only if
they belong to the same indecomposable component of a $\D'$-module category $\C$ 
(see \cite{O1} for relevant definitions). The following result is a generalization
of \cite[Theorem 3.1]{Br} and \cite[Theorem 3.2]{Mu2}.

For a fusion subcategory $\D$ of a braided fusion ctageory $\C$ let us consider
the following submatrix of the $\tilde{S}$-matrix of $\C$: 
\begin{equation}
\label{restricted tilde S}
\tilde{S}_\D :=\{ \tilde{s}_{XY} \}_{X\in \O(\D),Y\in \O(\C)}.
\end{equation}
Note that $\tilde{S} = \tilde{S}_\C$.

\begin{theorem}   
\label{Muger1}
Let $\C$ be a braided fusion category and let $\D \subset \C$ be a fusion subcategory.
The rank of $\tilde{S}_\D$  is equal to the number of $\D'$-components of $\C$. 
In particular, $\C$ is non-degenerate if and only if $\C'=\Vec$.
\end{theorem}

The proof is given in Section \ref{proof Muger1}.

\begin{corollary}
\label{a friend of Muger1}
Let $\C$ be a braided fusion category. A fusion subcategory $\D\subset\C$ is a
non-degenerate braided category   if and only if 
$\D\cap\D'=\Vec$. 
$\square$
\end{corollary} 

\begin{corollary}
\label{object-component bijection}
Let  $\D$ be a fusion subcategory of a non-degenerate braided fusion category $\C$.
Then the number of $\D'$-components of $\C$ is equal to $|\O(\D)|$.
\end{corollary}
\begin{proof}
Since $\C$ is non-degenerate, the rows of $\tilde{S}_D$ must be linearly independent.
\end{proof}

Now we collect various conditions equivalent to the non-degeneracy of a braided
fusion category.

\begin{proposition}\label{Muger1prime}
The following conditions are equivalent for a braided fusion category $\C$:
\begin{enumerate}
\item[(i)] $\C$ is non-degenerate (see Definition \ref{defnondeg});
\item[(ii)] $\C'=\Vec$;
\item[(iii)] Let $G:\C \boxtimes \C^{\oP}\to \Z (\C)$ 
be the unique braided functor whose restrictions to $\C$ and $\C^{\oP}$ are the braided functors from 
\S\ref{centre}, then $G$ is an equivalence.
\end{enumerate}
\end{proposition}
\begin{proof}
The equivalence (i)$\Leftrightarrow$(ii) is already proved in Theorem \ref{Muger1}. So we
just need to prove the equivalence (ii)$\Leftrightarrow$(iii). 
We have  $\FPdim(\Z(\C))=\FPdim(\C)^2=
\FPdim(\C \boxtimes \C^{\oP})$, see \eqref{dimZ}.
Hence, by Proposition~\ref{tens func}(ii)  it suffices to prove that $\C'=\Vec$
if and only if $G$ is fully faithful. By Lemma~\ref{tens fun from product},
the latter condition holds if and only if the images of $\C$
and $\C^{\oP}$ in $\Z(\C)$ have trivial intersection . 
But this intersection is precisely  $\C'$.
\end{proof}

\begin{rems}   \label{newrems}
\begin{enumerate}
\item[(a)] In Proposition \ref{Muger1prime}(iii) the existence of $G$ is clear because the fusion subcategories 
$\C, \C^{\oP}\hookrightarrow\Z (\C)$ centralize each other.
\item[(b)] Condition (iii) from Proposition \ref{Muger1prime} is known 
as {\em factorizability.}
\end{enumerate}
\end{rems}

\begin{corollary} \label{centerfactorizable}
Let $\C$ be a fusion category. Then its center $\Z (\C)$ is non-degenerate.
\end{corollary}

\begin{proof} It is proved in \cite[Proposition 4.4]{ENO2} that the center of a fusion category 
satisfies condition (iii) from Proposition \ref{Muger1prime}.
\end{proof}

The next theorem is a generalization of the result of M.~M\"uger 
\cite[Theorem 3.2]{Mu2} on  modular categories.

\begin{theorem}       \label{Muger2}
Let $\C$ be a braided fusion category and let $\B,\,\D\subset\C$ be fusion
subcategories. 
\begin{enumerate}
\item[(i)]
We have
\begin{equation}
\label{Dim B,D}
\dim(\B \cap \D') \dim(\D) = \dim(\D \cap \B') \dim(\B),
\end{equation}
in particular,
\begin{equation} 
\label{Dim D dim D'}
\dim(\D) \dim(\D') = \dim(\C) \dim(\D \cap \C').
\end{equation}
\item[(ii)]
If $\D\supset \C'$ then 
\begin{equation}
\label{K''=K}
\D''=\D.
\end{equation}
\end{enumerate}
\end{theorem}
The proof is given in Section \ref{proof Muger2}. 

For a pair of fusion subcategories $\mathcal{A},\,\B \subset \C$ let  $\mathcal{A} \vee \B$ denote 
the smallest fusion subcategory of $\C$ containing $\mathcal{A}$ and $\B$.

\begin{corollary}  \label{Muger-2'}
If $\D\subset\C$ is any fusion subcategory then $\D'' =\D \vee \C'$.
\end{corollary} 
\begin{proof}
Let $\widetilde{\D} = \D \vee \C'$.
Then $\D'=\widetilde{\D}'$, so $\D''=\widetilde{\D}''=\widetilde{\D}$ by Theorem \ref{Muger2}(ii).
\end{proof}

\begin{corollary}
\label{diamond cat dims}
Let $\C$ be a braided fusion category and let $\mathcal{A},\, \B \subset\C$ be fusion subcategories.
Then
\begin{equation}
\label{cat dim  diamond eqn}
\dim(\mathcal{A})\dim(\B) = \dim(\mathcal{A}\vee \B)  \dim(\mathcal{A}\cap \B). 
\end{equation}
\end{corollary}
\begin{proof}
We can assume that $\C$ is non-degenerate, for otherwise we replace $\C$ by $\Z(\C)$
which is non-degenerate by  Corollary~\ref{centerfactorizable}. Using 
equation \eqref{Dim D dim D'} twice and equation \eqref{Dim B,D} with $\D =\mathcal{A'}$ we have
\begin{equation*}
 \dim(\mathcal{A}\vee \B)  
=\frac{\dim(\C)}{\dim(\mathcal{A}' \cap \B')} 
= \frac{ \dim(\C) \dim(\B) } { \dim(\mathcal{A}') \dim(\mathcal{A} \cap \B)} 
= \frac{ \dim(\mathcal{A}) \dim(\B) } {\dim(\mathcal{A} \cap \B) }. \qedhere
\end{equation*} 
\end{proof}

The following result is an easy consequence of Theorems \ref{Muger1} and \ref{Muger2}(ii).
In the case when $\C$ is modular it is due to M.~M\"uger \cite{Mu2, Mu4}. 

\begin{theorem}        \label{Muger3}
Let $\C$ be a braided fusion category and $\K\subset\C$ a fusion subcategory. Suppose that 
the braided category $\K$ is non-degenerate. Then 
\begin{enumerate}
\item[(i)] the natural braided functor $\K\bt\K'\to\C$ is an equivalence;
\item[(ii)] $\K'$ is non-degenerate if and only if $\C$ is.
\end{enumerate}
\end{theorem}
\begin{proof}
By Theorem~\ref{Muger1}, since $\K$ is non-degenerate $\K \cap \K' =\Vec$.
Let $\D$ be the fusion subcategory of $\C$ generated by $\K$ and $\K'$.
The tensor functor $\K \boxtimes \K' \to \D$ defined by $X \boxtimes Y \mapsto X\ot Y$
is braided. It is an equivalence by Lemma~\ref{tens fun from product}
(this fact is also proved in \cite[Proposition 7.7]{Mu4}). 

By Corollary~\ref{Muger-2'}, $\K''$ is the smallest subcategory of $\C$ containing
$\K$ and $\C'$. Therefore,  $\K'' \simeq \K \bt \C' \subset \K \bt \K' \simeq \D$
and $\K'' \cap \K' =\C'$.  In view of Theorem~\ref{Muger1}
this proves (ii). Applying Theorem~\ref{Muger2}(ii) to $\D$ we have
$\D = \D'' = (\K'' \cap \K')' = \C'' =\C$, proving (i).
 \end{proof}

The next theorem is a modification  of equation \eqref{Dim D dim D'} of 
Theorem~\ref{Muger2}(i)  with categorical dimensions replaced by Frobenius-Perron.

\begin{theorem}        \label{Muger2-FPdim}
Let $\C$ be a braided fusion category and let 
$\D\subset\C$ be a fusion subcategory.
Then
\begin{equation}
\label{FPDim D FPdim D'}
\FPdim(\D)\FPdim(\D')=\FPdim(\C)\FPdim(\D\cap \C').
\end{equation}
\end{theorem}

The proof is given in Section \ref{Muger2-FPdimproof}.

\subsection{Projective centralizers} \label{projcentsubsec}
As before, let $\C$ be a braided fusion category. 

\subsubsection{The notion of projective centralizer}  \label{projcentnot}

\begin{lemma} \label{projcentlem1}
Let $X\in \C$. For each $\lambda \in k^\times$ let $\D_{\lambda}$ be the full subcategory of
objects $Y\in\D$ such that $c_{Y,X}c_{X,Y}=\lambda\cdot\id_{X\otimes Y}$. 
\begin{enumerate}
\item[(i)] If $Y\in\D_{\lambda}$ and $Z\in\D_{\mu}$ then $Y\ot\Z\in\D_{\lambda\mu}$;
\item[(ii)] If $Y\in\D_{\lambda}$ then $Y^*\in\D_{\lambda^{-1}}$;
\item[(iii)] The category 
\begin{equation}  \label{k-grading}
\D = \bigoplus_{\lambda\in k^\times} \D_{\lambda}
\end{equation}
is a fusion subcategory of $\C$.
\end{enumerate}
\end{lemma}
\begin{remark}
According to (i), the decomposition \eqref{k-grading} is a grading of $\D$.
\end{remark}
\begin{proof}
Statement (i) is immediate from the hexagon axiom of braiding.
To prove statement (ii)  let $X,Y$ be objects of $\C$ and consider the following diagram
\begin{equation}
\label{Kassel diagram}
\xymatrix{
X \ot Y \ar[rr]^(.4){\id_X \ot \coev_Y \ot \id_Y}  \ar[d]_{\coev_Y \ot \id_{X\ot Y}} &&
X\ot Y \ot Y^*\ot Y  \ar[rr]^(.6){\id_{X \ot Y} \ot \ev_Y} &&
X\ot Y \ar[d]^{c_{Y,X}^{-1}} \\
Y\ot Y^*  \ot  X\ot Y \ar[urr]^(.4){c_{Y\ot Y^*, X}\ot \id_Y} \ar[rr]_{\id_Y\ot c_{Y^*,X}\ot \id_Y} &&
Y\ot X \ot Y^*\ot Y   \ar[u]_{c_{Y,X}\ot \id_{Y^*\ot Y}}  \ar[rr]_(.6){\id_{Y\ot X}\ot \ev_Y} && Y\ot X.
}
\end{equation}
 
Here, as usual, $\ev_Y$ and  $\coev_Y$ denote the evaluation and coevaluation morphisms of $Y$
and $c$ denotes the braiding of $\C$.  We suppress the associativity and unit constraints.
The left triangle and the right square commute by naturality of the braiding, and the middle triangle 
commutes by the hexagon axiom.  Hence the diagram \eqref{Kassel diagram} is commutative,
i.e., 
\begin{equation}
\label{Kassel 1}
c_{Y,X}^{-1} = (\id_{Y\ot X}\ot \ev_Y) (\id_Y\ot c_{Y^*,X}\ot \id_Y) (\coev_Y \ot \id_{X\ot Y}).
\end{equation}
Similarly, we have
\begin{equation}
\label{Kassel 2}
c_{X,Y} = (\id_{Y\ot X}\ot \ev_Y) (\id_Y\ot c_{X, Y^*}^{-1} \ot \id_Y) (\coev_Y \ot \id_{X\ot Y}).
\end{equation}
If $Y^*\in\D_{\lambda}$ (i.e., $c_{Y^*,X} = \lambda c_{X, Y^*}^{-1}$) then by 
\eqref{Kassel 1}-\eqref{Kassel 2}, $c_{Y,X}^{-1}=\lambda c_{X,Y}$, i.e., $Y\in\D_{\lambda^{-1}}$.
This is equivalent to (ii).

Finally, (iii) follows from (i) and (ii) combined.
\end{proof}

\begin{rems}
\begin{enumerate}
\item[(i)]  Identites \eqref{Kassel 1} and \eqref{Kassel 2} are well known. Our proofs of them
are the same as  in \cite[Proposition XIV.3.1]{Kass}).
\item[(ii)] Statement (iii) of Lemma~\ref{projcentlem1} can be deduced directly from (i)
because a fusion subcategory of a fusion category
is rigid (see Corollary~\ref{fusion sub} in Appendix~\ref{fusionring}). 

\end{enumerate}
\end{rems}

\begin{definition} \label{projcentdef1}
We say that simple objects $X,Y\in\C$ 
{\em projectively centralize\,}  each other if 
$c_{Y,X}c_{X,Y}=\lambda\cdot\id_{X\otimes Y}$ for some $\lambda \in k^\times$.
If $X$ and $Y$ are arbitrary objects of $\C$ we say that they projectively centralize
each other if every simple component of $X$ projectively centralizes every simple component of $Y$.
We say that full subcategories $\C_1,\C_2\subset\C$ projectively centralize each other if each
object of $\C_1$ projectively centralizes each object of $\C_2$.
\end{definition}

\begin{definition} \label{projcentdef2}
The {\em projective centralizer of an object $X\in\C$\,} is the
full subcategory of objects of $\C$ projectively centralizing $X$. The {\em projective centralizer of a full
subcategory $\D\subset\C$\,} is the full subcategory of objects of $\C$ projectively centralizing
each object of $\D$.
\end{definition}

Lemma \ref{projcentlem1} implies that the projective centralizer of an object of $\C$ (or of a full subcategory $\D\subset\C$) is a fusion subcategory.

\subsubsection{Centralizers and projective centralizers} 
\begin{lemma} \label{EO/2}
Let $Y\in\C$, $Y\ne 0$, $A:=Y\otimes Y^*$. Equip $A\in\C$ with the algebra structure 
$\id_Y\ot\ev_Y\ot\id_Y: Y\ot Y^*\ot Y\ot Y^*\to Y\ot Y^*$. Define a functor $F$ from $\C$ to the category of $A$-modules in $\C$ by $F(X):=Y\ot X$.
Then $F$ is fully faithful.
\end{lemma}

\begin{remark}
In fact, $F$ is an equivalence, see, e.g., \cite[Lemma 3.28]{EO}. 
\end{remark}

\begin{proof}
Since $Y\ne 0$ any object of $\C$ is a direct summand of $Y^*\ot X$ for some $X\in\C$. So it suffices
to show that for any $X,X'\in\C$ the map 
\begin{equation}   \label{Ytensor}
\Hom (Y^*\ot X,X')\to\Hom_A(Y\ot Y^*\ot X,Y\ot X'), \quad   \varphi\mapsto\id_Y\ot\varphi
\end{equation} 
is an isomorphism. In fact, the r.h.s. of \eqref{Ytensor} equals
$\Hom_A(A\ot X, Y\ot X')=\Hom (X, Y\ot X')$
and the map \eqref{Ytensor} is the usual isomorphism $\Hom (Y^*\ot X,X')\iso\Hom (X, Y\ot X')$.
\end{proof}

\begin{proposition} \label{projective}
For any simple objects $X,Y\in \C$ the following conditions are equivalent:

(i) $X$ centralizes $Y\otimes Y^*$;

(ii) $X\otimes X^*$ centralizes $Y$;

(iii) $X$ and $Y$ projectively centralize each other.
\end{proposition}

\begin{proof} It suffices to show that (i) and (iii) are equivalent. The implication (iii)$\Rightarrow$(i)
follows from Lemma \ref{projcentlem1}(i-ii). 

Let us deduce (iii) from (i). We have to show that $c_{X,Y}c_{Y,X}\in k$. Consider $A:=Y\otimes Y^*$ as an algebra in $\C$. Then $Y\otimes X$ has an $A$-module structure, namely, 
$f\ot\id_X:A\ot Y\ot X\to Y\ot X$,
where $f:A\ot Y=Y\ot Y^*\ot Y\to Y$, $f=\id_Y\ot\ev_Y$.
By Lemma~\ref{EO/2}, $\End_A(Y\otimes X)=\End (X)=k$. So to prove that
$c_{X,Y}c_{Y,X}\in k$ it suffices to show that $c_{X,Y}c_{Y,X}$ is an endomorphism of $Y\otimes X$ in the category of $A$-modules, which means that the diagram 
\begin{equation}\label{endom}
\xymatrix{A\ot Y\ot X\ar[rr]^{\hspace{20pt} f\ot \id_X} \ar[d]_{\id_{A}\ot c_{X,Y}c_{Y,X}
}&&Y\ot X\ar[d]^{c_{X,Y}c_{Y,X}}\\ A\ot Y\ot X\ar[rr]^{\hspace{20pt}f\ot \id_X}&&Y\ot X}
\end{equation}
commutes. The functoriality of the braiding implies that the diagram
\begin{equation} \label{endom2}
\xymatrix{A\ot Y\ot X\ar[rr]^{\hspace{20pt} f\ot \id_X} \ar[d]_{c_{X,A\ot Y}
c_{A\ot Y,X}
}&&Y\ot X\ar[d]^{c_{X,Y}c_{Y,X}}\\ A\ot Y\ot X\ar[rr]^{\hspace{20pt}f\ot \id_X}&&Y\ot X}
\end{equation}
commutes, so it remains to show that the diagrams \eqref{endom} and \eqref{endom2} are equal.
By the hexagon axiom,
$c_{X,A\ot Y}=(\id_{A}\ot c_{X,Y})\circ (c_{X,A}\ot \id_Y)$ and
$c_{A\ot Y,X}=(c_{A,X}\ot \id_Y)\circ (\id_{A}\ot c_{Y,X})$.
Finally, $c_{X,A}c_{A,X}=\id_{X\ot A}$ by (i).
\end{proof}

\subsubsection{Application: the centralizer of the adjoint subcategory}
In \S\ref{adcat} we defined the adjoint category $\mathcal{K}_{ad}$ of a fusion category $\K$.
The next proposition describes $(\mathcal{K}_{ad})'$, where $\K$ is a fusion subcategory of
a braided category $\C$. To formulate it, we need the following definition from \cite{GN}.

\begin{definition}  \label{Kco}
Let $\K$ be a fusion subcategory  of a fusion category $\C$. The {\em commutator\,} of $\K$ 
is the fusion subcategory $\mathcal{K}^{co}\subset\C$ generated 
by  all simple objects $X\in\C$ such that $X\ot X^*\in \mathcal{K}$.
\end{definition}

Equivalently, $\mathcal{K}^{co}$ is the biggest fusion subcategory $\mathcal{R}\subset\C$ such that
$\mathcal{R}_{ad}\subset\K$. Clearly $(\mathcal{K}^{co})_{ad} \subset \mathcal{K} \subset
(\mathcal{K}_{ad})^{co}$. 

\begin{example}
If $\C=\Rep(G)$, where $G$ is
a finite group, then any fusion subcategory $\mathcal{K}$ of $\C$
is of the form $\mathcal{K}=\Rep(G/N)$ for some normal subgroup
$N$ of $G$, and $\mathcal{K}^{co} = \Rep(G/[G, N])$ (see \cite{GN}). 
\end{example}

The first part of the next statement was proved in \cite[Proposition 6.6]{GN} in the case where
$\C$ is pseudounitary.

\begin{proposition}\label{comm-drinf}
Let $\mathcal{K}$ be a fusion subcategory of a 
braided fusion category $\C$.  Then $(\mathcal{K}_{ad})'=(\mathcal{K}')^{co}$.
In fact, both $(\mathcal{K}_{ad})'$ and $(\mathcal{K}')^{co}$ are equal to the projective centralizer
of $\K$.
\end{proposition}

\begin{proof}
By definition,
the simple objects of $(\mathcal{K}_{ad})'$ are those $X\in \O(\C)$ which
centralize $Y\otimes Y^*$ for any $Y\in \O(\K)$. Similarly, the simple objects of $(\mathcal{K}')^{co}$
are those $X\in \O(\C)$ for which $X\otimes X^*$ centralizes any $Y\in \O(\K)$. Now use
Proposition \ref{projective}. 
\end{proof}

\begin{corollary} \label{comm-cor}
Let $\mathcal{K}$ be a fusion subcategory of a 
braided fusion category $\C$. Then $(\K')_{ad}\subset (\K^{co})'$.
If $\C$ is non-degenerate then $(\K')_{ad}=(\K^{co})'$.
\end{corollary}

\begin{proof} 
By Proposition \ref{comm-drinf},
$((\K')_{ad})'=(\mathcal{K}'')^{co}\supset\K^{co}$, so  $(\K')_{ad}\subset (\K^{co})'$.

If $\C$ is non-degenerate then $\C'=\Vec$, so by Theorem \ref{Muger2}(ii), $\D''=\D$
for any fusion subcategory $\D\subset\C$. Therefore in the non-degenerate case
Proposition \ref{comm-drinf} immediately implies the equality $(\K')_{ad}=(\K^{co})'$.
\end{proof}

\begin{corollary} \label{adpt}
If $\C$ is non-degenerate then $(\C_{ad})'$ equals the maximal pointed subcategory $\C_{pt}\subset\C$,
i.e., the fusion subcategory generated by the invertible objects of $\C$. In addition
$(\C_{pt})'=\C_{ad}$.
\end{corollary}

\begin{proof} 
Apply Proposition \ref{comm-drinf} in the case $\K =\C$. Since $\C$ is non-degenerate $\C'=\Vec$, so
we get $(\C_{ad})'=(\Vec)^{co}$. Finally,  $(\Vec)^{co} =\C_{pt}$ (this is immediate from Definition~\ref{Kco}).
The second statement follows immediately from the first one by Theorem \ref{Muger2}(ii).
\end{proof} 

In Corollary~\ref{adpt} we assume that $\C'=\Vec$. To prove a similar statement in the case
$\C'=s\Vec$ (see Proposition~\ref{s-adpt} below), we need the following lemma, which is
due to M.~M\"uger \cite[Lemma 5.4]{Mu1}.

\begin{lemma}  \label{Dimalem}
Let $\C$ be a braided fusion category and $\delta\in\C'$ an invertible object such that the
fusion subcategory of $\C$ generated by $\delta$ is braided equivalent to $s\Vec$. Then
$\delta\ot X\not\simeq X$ for every simple $X\in\C$. 
\end{lemma}

\begin{proof}
In \ref{squared} we defined the dimension $d_+(X)=\Tr_+ (\id_X)$, where $\Tr_+$ depends
on the choice of $\psi_X:X\iso X^{**}$. We will show that if one takes $\psi_X$ to be the
canonical isomorphism $u_X: X\to X^{**}$ from \eqref{Drinfeld iso} then
\begin{equation}   \label{Dimaeq}
d_+( \delta\otimes X)=-d_+( X).
\end{equation}
To prove \eqref{Dimaeq}, note that since $X$ centralizes $\delta$ formula \eqref{nontensor} 
says that $u_{\delta\otimes X} = u_\delta\otimes u_X  = -  id_\delta\otimes  u_X$, and 
\eqref{Dimaeq} follows. Since $d_+( X)\ne 0$ formula \eqref{Dimaeq} implies that 
$\delta\ot X\not\simeq X$.
\end{proof}

\begin{proposition}   \label{s-adpt}
Let $\C$ be a braided fusion category such that $\C'\simeq s\Vec$. Then 
\begin{equation}   \label{eq-s-adpt}
(\C_{ad})'=\C_{pt},\quad (\C_{pt})'=\C_{ad}\vee\C'.
\end{equation}   
\end{proposition}

\begin{proof}
By Proposition~\ref{comm-drinf}, a simple object $X\in\C$ belongs to $(\C_{ad})'$ if and only if
$X\in(\C')^{co}$, i.e., if and only if $X\ot X^*\in\C'$. By Lemma~\ref{Dimalem}, $X\ot X^*$ cannot
contain the nontrivial simple object of  $s\Vec$, so $X\ot X^*\in\C'$ if and only if $X\ot X^*\in\Vec$,
i.e., if and only if $X\in \C_{pt}$. Thus $(\C_{ad})'=\C_{pt}$. Combining this with 
Corollary~\ref{Muger-2'} we see that $(\C_{pt})'=\C_{ad}\vee\C'$.
\end{proof}

\begin{remark}   \label{generalized}
One can prove \eqref{eq-s-adpt} for any braided fusion category $\C$ such that $\C'$ is pointed
and the group $\O (\C')$ is cyclic. We will not need this fact.
\end{remark}

\subsubsection{A useful pairing} \label{useful pairing}
Here is a more explicit version of Corollary \ref{adpt}.
An invertible object $X$ of a braided fusion category $\C$
defines an automorphism $\alpha(X)$ of the identity functor given by the composition
\[
Y \xrightarrow{\id_Y\otimes \text{coev}_X} Y\otimes X\otimes X^*
\xrightarrow{c_{X,Y}c_{Y,X}\otimes \id_{X^*}} Y\otimes X\otimes X^*
\xrightarrow {\id_Y\otimes \text{coev}_X^{-1}} Y.
\]

The following result is well known.

\begin{lemma}\label{alpha}
 Let $X$ be an invertible object of a braided fusion category $\C$.
\begin{enumerate}
\item[(i)] $\alpha(X)$ is a tensor automorphism of the identity functor;
\item[(ii)] the map $\alpha: \O(\C_{pt})\to \Aut_{\otimes}(\id_\C)$ is a homomorphism of groups;
\item[(iii)] for a non-degenerate category $\C$ the homomorphism $\alpha$ is isomorphism.
\end{enumerate}
\end{lemma}
\begin{proof} (i) and (ii) are straightforward. For (iii) note that $X\in \O(\C_{pt})$
is in kernel of $\alpha$ if and only if $X\in \C'$. Thus under assumptions of (iii)
$\alpha$ is injective. Recall that by Proposition \ref{un-grad} (iii) we have
$\Aut_{\otimes}(\id_\C)=\Hom(U_\C,k^\times)$ where $U_\C$ is the universal grading group. 
It is clear that for a braided category $\C$ the group $U_\C$ is commutative and hence
$U_\C=\Hom(\Aut_{\otimes}(\id_\C),k^\times)$. The homomorphism 
$\alpha^\vee: U_\C=\Hom(\Aut_{\otimes}(\id_\C),k^\times)\to \Hom(\O(\C_{pt}),k^\times)$ dual to 
$\alpha$ is surjective and hence defines a faithful grading of $\C$ 
by the group $\Hom(\O(\C_{pt}),k^\times)$, see \S \ref{gradsubsec}.
It is clear from the definitions that the trivial component of this grading is precisely
$(\C_{pt})'$. By Lemma \ref{adpt} we have $(\C_{pt})'=\C_{ad}$ and by Proposition \ref{un-grad} (ii)
$\C_{ad}$ is the trivial component of the universal grading. 
Hence $\alpha^\vee$ and $\alpha$ are isomorphisms
and the lemma is proved.
\end{proof}

Proposition \ref{un-grad}(iii) provides a
canonical isomorphism $\Aut_{\otimes}(\id_\C)=\Hom(U_\C,k^\times)$,
where $U_\C$ is the group of the universal grading. So 
we can consider the homomorphism $\alpha: \O(\C_{pt})\to \Aut_{\otimes}(\id_\C)$ as a pairing 
\begin{equation}  \label{thepairing}
\langle\, ,\,\rangle : \O(\C_{pt})\times U_\C\to k^\times.
\end{equation}
Now Lemma \ref{alpha}(iii) says that for a non-degenerate braided fusion category $\C$
this pairing is perfect. 

\subsubsection{A more general pairing}

\begin{proposition} \label{bigrading}
Suppose that fusion subcategories $\D_1,\D_2\subset\C$ projectively centralize each other.
\begin{enumerate}
\item[(i)] There exists a faithful grading $\D_1=\bigoplus\limits_{g\in G_1}(\D_1)_g$
with trivial component $\D_1\cap (\D_2)'$ and a faithful grading 
$\D_2=\bigoplus\limits_{g\in G_2}(\D_2)_g$
with trivial component $\D_2\cap (\D_1)'$. These gradings are unique up to equivalence.
\item[(ii)] There is a unique pairing $b:G_1\times G_2\to k^{\times}$ such that 
\begin{equation}
c_{X_2X_1}c_{X_1X_2}=b(g_1,g_2)\cdot\id_{X_1\otimes X_2} 
\mbox { if }X_i\in (\D_i)_{g_i}\, , \, g_i\in G_i\, .
\end{equation}
This pairing is non-degenerate. If $\D_1=\D_2$ it is symmetric.
\end{enumerate}
\end{proposition}

Note that if $\C$ is non-degenerate, $\D_1=\C_{pt}$, and $\D_2 =\C$ then $b$ is the
pairing~\eqref{thepairing}.

\begin{proof}
The uniqueness statement from (i) follows from Corollary~\ref{trivial component}. 
The uniqueness statement from (ii) is obvious. The symmetry property follows from uniqueness.

To prove the existence statements, consider the universal grading 
$d_i:\O(\D_i)\to U_{\D_i}$ (see \S\ref{gradsubsec}). By Lemma~\ref{projcentlem1}(i),
there is a unique pairing $\tilde b:U_{\D_1}\times U_{\D_2}\to k^{\times}$ such that
$c_{X_2X_1}c_{X_1X_2}=\tilde b(d_1(X_1),d_2(X_2))\cdot\id_{X_1\otimes X_2}$
for all simple $X_1\in\D_1$, $X_2\in\D_2$. Let $K_1\subset U_{\D_1}$ and $K_2\subset U_{\D_2}$
be the left and right kernel of this pairing. Set $G_i:=U_{\D_i}/K_i$. Then $\tilde b$ induces a 
non-degenerate pairing $b:G_1\times G_2\to k^{\times}$. Clearly $\D_i$ is graded by $G_i$. The trivial
component of the $G_1$-grading of $\D_1$ is generated by those simple objects of $\D_1$
which centralize each object of $\D_2$, so it equals $\D_1\cap\D_2'$. Similarly, the trivial
component of the $G_2$-grading of $\D_2$ equals $\D_2\cap\D_1'$.
\end{proof}

\subsection{Proof of Theorem \ref{Muger1}}\label{proof Muger1}
\label{Proof of Muger-1}

\subsubsection{Idea of the proof}

Given a fusion subcategory $\D \subset \C$
let us define an inner product of functions $f,g: \O(\D)\to k$ by
\begin{equation}
\label{inner product D}
( f,\, g )_\D := \sum_{X\in \O(\D)}\, |X|^2 f(X)g(X^*).
\end{equation}
For any $Y\in \O(\C)$ define  $h_Y: \O(\D)\to k$ by $h_Y(X) =\tilde{s}_{XY}$.

Recall that we introduced an equivalence relation on $\O(\C)$ by taking
$Y,\, Z\in \O(\C)$ to be equivalent if there is $W\in \O(\D')$ such that
$Y$ is contained in $Z\ot W$. We call the equivalence classes of the above
relation {\em $\D'$-components} of $\C$.

The idea of the proof is to show that functions $h_Y$ and $h_Z$
are equal if and only if $Y,\,Z$ belong to the same component 
and are orthogonal with respect to the inner product \eqref{inner product D}
if and only if they belong to different $\D'$-components of $\C$, 
see Lemmas~\ref{YZ same component} and~\ref{hXhY}
below.   This will imply that distinct  columns of $\tilde{S}_\D$ 
are linearly independent and are parameterized by $\D'$-components of $\C$.

\subsubsection{A non-spherical analogue of the Verlinde formula~\eqref{ss= Nds}}
Pick a natural isomorphism $\psi_X: X \xrightarrow{\sim} X^{**},\, X\in \O(\C)$.
For all $Y,\,Z,\,W\in \O(\C)$ consider the linear  automorphism
\[
t_{YZ}^W : \Hom(W,\, Y\ot Z) \to \Hom(W,\, Y\ot Z), \quad
f \mapsto \psi_{Y\ot Z}^{-1} (\psi_Y \ot \psi_Z)f.
\]
The automorphisms $t_{YZ}^W$ measure the failure of $\psi$ to be an isomorphism
of {\em tensor} functors. Let $T_{YZ}^W$ denote the trace of $t_{YZ}^W$. 

\begin{remark}
According to  \cite[3.1]{ENO}, for a certain choice of the isomorphisms $\psi_X$ one
has $T_{YZ}^W\in\BZ$ (we will not need this fact in the sequel). 
When $\psi$ is a pivotal structure (i.e., an isomorphism of tensor functors)
we have $T_{YZ}^W=  N_{YZ}^W$, where $N_{YZ}^W$ denotes the multiplicity of $W$ in $Y\ot Z$.
\end{remark}

The numbers $T_{YZ}^W$ depend on the choice of $\psi$. Namely, 
if one replaces $\psi$ by $\psi\circ \alpha$,  where 
$\alpha =\{\alpha_X\}_{X\in \O(\C)}$ is an automorphism of the identity
functor of $\C$, then $T_{YZ}^W$ gets multiplied by $\tfrac{\alpha_Y\alpha_Z}{\alpha_W}$.
Therefore the numbers
\begin{equation}
\label{tilde N}
\tilde{N}_{YZ}^W:=  \dfrac{d_+(W)}{d_+(Y)d_+(Z)} T_{YZ}^W,\qquad Y,\,Z,\,W\in \O(\C),
\end{equation}
do not depend on the choice of $\psi$. Note that $\tilde{N}_{YZ}^W=0$ when
$N_{YZ}^W=0$.

\begin{lemma}
\label{non-spherical Verlinde}
For all $X,\, Y,\, Z\in \O(\C)$ we have
\begin{equation}
\label{non spherical ss= Nds}
\tilde{s}_{XY} \tilde{s}_{XZ} = \sum_{W\in \O(\C)}\, \tilde{N}_{YZ}^W \tilde{s}_{XW}.
\end{equation}
\end{lemma}
\begin{proof}
Consider the equality
\begin{equation}
\label{braiding triangle}
(c_{Y,X} \ot \id_{Z}) (\id_Y \ot c_{Z,X}c_{X, Z})(c_{X, Y}\ot \id_{Z}) =
c_{Y\ot Z, X} c_{X, Y\ot Z}, \qquad X\in \O(\C).
\end{equation}
Let us apply $\Tr_+(\Tr_-\ot \id_Y \ot  \Tr_+)$ to both sides of \eqref{braiding triangle}.
The left hand side gives $d_-(X)d_+(Y)d_+(Z)  \tilde{s}_{XY} \tilde{s}_{XZ}$. 
For $W\in \O(\C)$ let $\pi_{Y\ot Z}^{(W)}$ be the idempotent endomorphism of $Y\ot  Z$
corresponding to the projection on the maximal multiple of $W$ contained in $Y\ot  Z$.
Then the right hand side of \eqref{braiding triangle} gives 
\[
d_-(X) \sum_{W\in \O(\C)}\, \Tr_+(\id_Y  \ot \Tr_+)(\pi_{Y\ot Z}^{(W)}) \tilde{s}_{XW}
= d_-(X) \sum_{W\in \O(\C)}\,  d_+(W) T_{Y,Z}^W  \tilde{s}_{XW}.
\]
Comparing the results we obtain \eqref{non spherical ss= Nds}.
\end{proof}

\begin{remark}
In the case when $\psi$ is pivotal our proof of Lemma~\ref{non-spherical Verlinde} 
reduces to  the computation given in \cite[Lemma 2.4]{Mu2}.
\end{remark}

\subsubsection{Orthogonality of columns of the submatrix \eqref{restricted tilde S} of $\tilde{S}$-matrix}

\begin{lemma}
\label{YZ same component}
Let $\D \subset \C$ be a fusion subcategory.
Suppose  $Y, Z\in \O(\C)$  are in the same $\D'$-component of $\C$.
Then 
\[
h_Y = h_Z.
\]
\end{lemma}
\begin{proof}
Let us choose $W$ in $\D'$ so that $Y$ is contained in $Z\ot W$.
Let $\pi^{(Y)}_{Z\ot W}: Z\ot W \to Y$  
be a projection from $Z\ot W$ onto $Y$ and let $i^{(Y)}_{Z\ot W}: Y\to Z\ot W$
be an inclusion of $Y$ into $Z\ot W$
such that $p^{(Y)}_{Z\ot W} i^{(Y)}_{Z\ot W} =\id_Y$.
For any $X\in \O(\D)$ we compute
\begin{eqnarray*}
\lefteqn{\tilde{s}_{XY} 
= d_-(X)^{-1} d_+(Y)^{-1} (\Tr_-\ot \Tr_+)(c_{Y,X} c_{X,Y}) }\\
&=& d_-(X)^{-1} d_+(Y)^{-1} (\Tr_-\ot \Tr_+) ( (\id_X \ot p^{(Y)}_{Z\ot W})(c_{Z\ot W, X} c_{X, Z\ot W})(\id_X \ot i^{(Y)}_{Z\ot W})) \\
&=& d_-(X)^{-1} d_+(Y)^{-1} (\Tr_-\ot \Tr_+) ( (\id_X \ot p^{(Y)}_{Z\ot W})(c_{Z, X} c_{X, Z} \ot \id_W)(\id_X \ot i^{(Y)}_{Z\ot W})) \\
&=& d_-(X)^{-1} d_+(Y)^{-1} \left( d_-(Z)^{-1} (\Tr_-\ot \Tr_+) ( c_{Z,X} c_{X,Z} ) \right) \Tr_+( \id_Y) \\
&=& \tilde{s}_{XZ},
\end{eqnarray*}
as required.
\end{proof}

\begin{lemma}
\label{hXhY}
For $Y, Z\in \O(\C)$ we have $( h_Y,\, h_{Z})_\D =0$ if and only if
$Y$ and $Z$ are in different $\D'$-components of $\C$.
\end{lemma}
\begin{proof}
Fix $Y\in \O(\C)$. If $Z\in \O(\C)$
is such that $Y$ and $Z$ are in different
$\D'$-components of $\C$, then
\begin{eqnarray*}
( h_Y,\, h_{Z})_\D
&=&  \sum_{X\in \O(\D)}\, |X|^2 \tilde{s}_{XY} \tilde{s}_{XZ^*} \\
&=&   \sum_{W\in \O(\C)}\, \tilde{N}_{YZ^*}^W 
     \left(  \sum_{X\in \O(\D)}\, |X|^2 \tilde{s}_{XW} \right) \\
&=& \dim(\C) \sum_{W\in \O(\D')}\, \tilde{N}_{YZ^*}^W  =0,
\end{eqnarray*}
where we used Lemma~\ref{non-spherical Verlinde} and 
Proposition~\ref{AB criterion}, see \eqref{not O(D')}.


Assume there is $Z_0$ in the $\D'$-component
of $Y$ such that $(h_Y,\,h_{Z_0})_\D=0$. Then by Lemma~\ref{YZ same component}  
we have $(h_Y,\,h_Z)_\D=0$ for any simple $Z$ from the $\D'$-component of $Y$.
Thus we have seen that for any $Z\in \O(\C)$
\[
\sum_{X\in \O(\D)}\, |X|^2 \tilde{s}_{XY}\tilde{s}_{XZ^*} =( h_Y,\, h_{Z})_\D = 0.
\]
Multiplying the last equation by $|Z|^2$, taking the sum over $Z\in \O(\C)$,
and using Proposition~\ref{AB criterion} (see  \eqref{O(D')}, \eqref{not O(D')}) we obtain
\begin{eqnarray*}
0
&=& \sum_{X\in \O(\D)}\, |X|^2 \tilde{s}_{XY} 
\left( \sum_{Z\in \O(\C)}\, \tilde{s}_{XZ} |Z|^2 \right) \\
&=&  \dim(\C) \sum_{X\in \O(\D \cap \C')}\, |X|^2 \tilde{s}_{XY}\\
&=& \dim(\C) \dim(\D \cap \C'),
\end{eqnarray*}
which is a contradiction, since the categorical dimension of a fusion category
is non-zero by \cite{ENO} (see \S\ref{squared}). 
\end{proof}

It follows that the functions $h_Y$, where $Y$ 
runs through the set of simple objects representing
different $\D'$-components of $\C$, form an orthogonal basis
of the vector space generated by columns of $\tilde{S}_\D$. 
Theorem~\ref{Muger1} is proved.
\subsection{Proof of Theorem \ref{Muger2}}\label{proof Muger2}
Our argument follows that of \cite[Theorem 3.2]{Mu2}, where 
\eqref{Dim D dim D'} and  \eqref{K''=K} were proved
for a pseudo-unitary braided fusion  category $\C$. 
We use Proposition~\ref{AB criterion}
(see \eqref{O(D')} and \eqref{not O(D')}) to compute
\begin{eqnarray*}
\sum_{X\in \O(\D)}\, \sum_{Z\in \O(\B)}\, |X|^2 \tilde{s}_{XZ} |Z|^2
&=&  \sum_{X\in \O(\D)}\, |X|^2  \left(   \sum_{Z\in \O(\B)}\, \tilde{s}_{XZ}  |Z|^2  \right) \\
&=& \dim(\B) \sum_{X\in \O(\B' \cap \D)}\, |X|^2 \\
&=& \dim(\B) \dim(\B' \cap \D).
\end{eqnarray*}
Changing the order of summation in the above computation 
we see that the same expression also equals 
$\dim(\D) \dim(\B \cap \D')$, proving \eqref{Dim B,D}.
Taking $\B =\C$ we obtain \eqref{Dim D dim D'}.
Finally, when $\D \supset \C'$ we have
\[
\dim(\D) \dim(\D') = \dim(\C)\dim(\C') = \dim(\D') \dim(\D''),
\]
whence $\dim(\D) = \dim(\D'')$ and $\D =\D''$, proving \eqref{K''=K}.

\subsection{Proof of Theorem \ref{Muger2-FPdim}}  \label{Muger2-FPdimproof}
We start with the following result.
Let $\C$ be a fusion category.
If $\A,\B\subset\C$ are fusion subcategories let $\A\vee \B$ denote the smallest fusion
subcategory of  $\C$  containing $\A$ and $\B$. Suppose that  
\begin{equation}\label{weakercom} 
X\ot Y\simeq Y\otimes X \mbox{ for all }X\in \O(\A),\, Y\in \O(\B).
\end{equation}
Property \eqref{weakercom} implies that the simple objects of 
$\A\vee \B$ are the simple summands of $X\otimes Y$ with $X\in \O(\A)$ and $Y\in \O(\B)$.

\begin{lemma}
\label{diamond dimensions}
Let $\C,\,\A,\,\B$ be as above. Then
\begin{equation}
\label{AveeB}
\FPdim(\A)\FPdim(\B) = \FPdim(\A \vee \B) \FPdim(\A\cap\B).
\end{equation}
\end{lemma}

\begin{proof}
Recall that the regular element of
$K_0(\C)\otimes_\mathbb{Z} {\mathbb R}$ is
$R_\C =\sum_{X\in \O(\C)}\, \FPdim(X) X$.
It is defined up to a scalar multiple  by the property that
$Y \otimes R_\C  = \FPdim(Y) R_\C$ for all $Y\in\O(\C)$, see \cite[\S 8.2]{ENO}.

Using \cite[Proposition 8.5]{ENO}, one can see that
\begin{equation}
\label{diamond} R_\A \otimes R_\B =a R_{\A\vee \B},
\end{equation}
for some positive $a$. Indeed, take $V=\oplus_{X\in \O(\A), Y\in \O(\B)}\, X\ot Y$.
Then the multiplication matrix of $V$ in $K_0(\A\vee \B)$ has strictly positive
entries and both sides of \eqref{diamond} are Frobenius-Perron eigenvectors
of $V$. Hence, they differ by a positive scalar.

The scalar $a$ equals the multiplicity of the unit
object $\mathbf{1}$ in $R_\A \ot R_\B$, which is the same as the
multiplicity of $\mathbf{1}$ in $\sum_{Z\in \O(\A\cap\B)}\, \FPdim(Z)^2
Z \otimes Z^*$. Hence, $a = \FPdim(\A\cap\B)$. Taking the Frobenius-Perron 
dimensions of both sides of \eqref{diamond} we get the result.
\end{proof}

Now we will prove Theorem \ref{Muger2-FPdim}.
For any fusion category $\C$ and any fusion subcategory $\D\subset\C$
let $\Z_\D(\C)$ be the fusion category whose objects are  
pairs $(X, c_{X, -})$, where $X$ is an object of $\C$ and
$c_{X,-}$ is a natural family of isomorphisms $c_{X, V} : X\ot V
\simeq V \ot X$ for all objects $V$ in $\D$ satisfying the same
compatibility conditions as in the definition of the center of $\C$
(so, in particular, $\Z_\C(\C) =\Z(\C)$). 
Note that  $\Z_\D(\C)$
is dual to the fusion category $\D\bt \C^{\op}$ with respect to its 
module category $\C$, where $\D$ and $\C^{\op}$ act on $\C$ via
the left and right multiplication respectively
(see \cite{Mu-I, O1} for the definitions of
module category and duality for fusion categories). 
So by \cite[Corollary 8.14]{ENO} we have
\begin{equation}\label{firstequality}
\FPdim(\Z_\D(\C)) = \FPdim(\D\bt \C^{\op})= \FPdim(\D) \FPdim(\C).
\end{equation}
Also, by  \cite[Proposition 5.3]{ENO} the forgetful
tensor functor $F: \Z(\C)\to \Z_\D(\C)$ is surjective. 

Now assume that $\C$  is a non-degenerate braided fusion category 
with braiding $c$.  There are two embeddings $X\mapsto (X,\, c_{X,-})$ 
and  $X\mapsto (X,\, c_{-,X}^{-1})$ of $\C$ into $\Z_\D(\C)$.
Their images will be denoted by $\C_+$ and $\C_-$. Clearly $\C_+\cap \C_- = \D'$.
On the other hand, $\C_+=FG(\C\bt \be)$ and $\C_-=FG(\be\bt \C)$, where 
$G: \C \bt \C^\op\iso \Z(\C)$ is the functor from Proposition~\ref{Muger1prime}.
Since $F$ is surjective we see that $\C_+\vee \C_-=\Z_\D(\C)$. 
So applying Lemma~\ref{diamond dimensions} to $\mathcal{A}=\C_+$
and $\mathcal{B}=\C_-$ we obtain
\[
\FPdim(\Z_\D(\C))\FPdim(\D') = \FPdim(\C_+) \FPdim(\C_-)= \FPdim(\C )^2.
\]
By \eqref{firstequality}, this means that
\begin{equation}
\label{FPDim D FPdim D' non-deg}
\FPdim(\D)\FPdim(\D') = \FPdim(\C).
\end{equation}

Now let $\C$ be any (possibly degenerate) braided fusion category.
Let $\D^\wr,\C^\wr$ denote the centralizers of $\D,\C$  in the non-degenerate
braided fusion category $\Z(\C)$. Then  the usual centralizer of $\D$
in $\C$ is  $\D' =\D^\wr\cap \C$ and   $\D\cap \C^\wr =\D \cap \C'$.

We can compute $\FPdim(\D \vee \C^\wr)$
in two different ways. First, by Lemma~\ref{diamond dimensions},
\begin{equation} \label{way1}
\FPdim(\D\vee \C^\wr) = \frac{\FPdim(\D)\FPdim(\C^\wr)}{\FPdim(\D\cap \C^\wr)}
= \frac{\FPdim(\D)\FPdim(\C)}{\FPdim(\D\cap \C')},
\end{equation}
(we have used the equality $\FPdim(\C^\wr) =\FPdim(\C)$,
which follows from \eqref{FPDim D FPdim D' non-deg}). Second,
by \eqref{FPDim D FPdim D' non-deg}  and \eqref{dimZ}
we have
\begin{equation} \label{way2}
\FPdim(\D\vee \C^\wr)
= \frac{\FPdim(\Z(\C))}{\FPdim(\D^\wr\cap \C)}
= \frac{\FPdim(\C)^2}{\FPdim(\D')}
\end{equation}
Comparing \eqref{way1} and \eqref{way2} we get Theorem \ref{Muger2-FPdim}.

\section{Equivariantization and de-equivariantization}
\label{isotsub}

In this section we recall the equivariantization and de-equivariantization
constructions (first in the general setting of \additive categories and then for fusion categories).
The main principles are formulated in \ref{princ1}, \ref{princ2}, \ref{princ3}, and 
Theorem~\ref{Last_principle}.

Unless stated otherwise, we assume that all categories are small (i.e., objects form a set rather
than a class).

\subsection{Equivariantization and de-equivariantization in the linear setting}
\label{linear setting}
\subsubsection{Basic definitions}
As usual, $k$ denotes an algebraically closed field of characteristic 0.

Let $\C$ be a \additive category. 
 Let $\underline{End}(\C)$ denote the category of \additive functors $\C \to \C$.
This is a monoidal \additive category (the tensor product is the composition of functors).

\begin{definition} We say that a monoidal category $\M$ {\em acts on} $\C$ if we
are given a monoidal functor $F: \M \to \underline{End}(\C)$. If $\M$ is a \additive monoidal 
category and $F$ is \additive we say that $\M$ acts on $\C$
{\em \additively}.
\end{definition}

\noindent {\bf Notation:} for a \additive action $F: \M \to \underline{End}(\C)$
and objects $V\in \M$, $X\in \C$
we write $V\otimes X$ instead of $F(V)(X)$. In particular, we write $V\otimes X$ for
the tensor product of a finite-dimensional $k$-vector space $V$ and an object $X\in \C$.

\medskip

For a finite group $G$ let $\underline{G}$ denote the corresponding monoidal category: the
objects of $\underline{G}$ are elements of $G$, the only morphisms are the identities and the tensor
product is given by multiplication in $G$. 

\begin{definition} 
An action of $G$ on $\C$ is the same as an action of $\underline{G}$ on $\C$.
\end{definition}

\begin{remark} 
Let $\Vec_G$ be the fusion category of finite-dimensional $G$-graded $k$-vector spaces 
with the obvious associativity constraint. We have a monoidal functor $\underline{G}\to\Vec_G$
which takes $g\in G$ to the vector space $k$ placed in degree $g$. Thus $\Vec_G$
identifies with the \additive hull of $\underline{G}$. So
 {\em an action of $G$ on $\C$ is the same as a
\additive action of $\Vec_G$ on $\C$}. 
\end{remark}

Assume that a monoidal category $\M$ acts on $k$-linear categories
$\A_1$ and $\A_2$. 
Then one defines the category of {\em module functors\,} from $\A_1$ to 
$\A_2$. Recall (see e.g. \cite{O2}) that a {\em module functor} from $\A_1$ to 
$\A_2$ is a $k$-linear functor $F: \A_1\to \A_2$, which {\em commutes\,} with the action of $\M$,
i.e., $F$ is equipped with isomorhisms $f_{V,X}:F(V\otimes X)\simeq V\otimes F(X)$ functorial 
in $V\in \M$, $X\in \A_1$ so that the diagram
\begin{equation*}
\xymatrix{F(V\otimes W\otimes X)\ar[rr]^{f_{V,W\otimes X}} \ar[rd]_{f_{V\otimes W,X}
}&&V\otimes F(W\otimes X)\ar[ld]^{\id_V\otimes f_{W,X}}\\ &V\otimes W\otimes F(X)&}
\end{equation*}
is commutative for every $V,W\in \M$ and $X\in \A_1$. 

So {\em 
$k$-linear categories equipped with $\M$-action form a 2-category.}

\subsubsection{Main principle} \label{princ1}
Let $\Rep(G)$ denote the monoidal \additive category of finite-dimensional 
representations of $G$. Recall that an additive category $\A$ is said to be {\em Karoubian\,}  if
every idempotent endomorphism of an object 
$a\in\A$ comes from a decomposition $a=a_1\oplus a_2$.
In \S\ref{equivtion}-\ref{enrich2} we recall the following well-known principle
(see \cite{AG,G,FG}):
{\em the datum of a Karoubian \additive category with an action of $G$ is equivalent
to the datum of a Karoubian \additive category with a \additive action of $\Rep(G)$.}
More precisely, in \S\ref{equivtion}-\ref{deequivtion} we recall two constructions (or if you prefer, two 2-functors)
\[
\mbox{\{\additive categories with }G\mbox{-action\}} \rightleftarrows
\mbox{\{\additive categories with }\Rep(G)\mbox{-action\}},
\]
which are called {\em equivariantization\,} and {\em de-equivariantization.} 
Theorem \ref{eqdepre} 
says that these constructions are mutually inverse when applied to Karoubian categories.

\subsubsection{Equivariantization: from $G$-action to $\Rep(G)$-action} \label{equivtion}
Let $G$ be a finite group acting on a Karoubian \additive category $\C$. For any $g\in G$ let $F_g\in 
\underline{\End}(\C)$ be the corresponding functor and for any $g,h\in G$ let $\gamma_{g,h}$ be
the isomorphism $F_g\circ F_h\simeq F_{gh}$ that defines the
tensor structure on the functor $\underline{G}\to \underline{\End}(\C)$. 
A {\em $G$-equivariant object\,} of $\C$ is
an object $X\in \C$ together with isomorphisms $u_g: F_g(X)\simeq X$ such that the diagram
\begin{equation*}
\label{equivariantX}
\xymatrix{F_g(F_h(X))\ar[rr]^{F_g(u_h)} \ar[d]_{\gamma_{g,h}(X)
}&&F_g(X)\ar[d]^{u_g}\\ F_{gh}(X)\ar[rr]^{u_{gh}}&&X}
\end{equation*}
commutes for all $g,h\in G$.
One defines morphisms of equivariant objects to be morphisms in $\C$ commuting with $u_g,\; g\in G$.
The category of $G$-equivariant objects of $\C$ will be denoted by $\C^G$. It is a Karoubian \additive category, which is often called the  {\em equivariantization\,} of $\C$.

The category $\Rep(G)$ acts \additively on the category $\C^G$. Namely, to a
representation $\rho :G\to GL(V)$ and an object $(X,u_g)\in \C^G$
one associates the object $(V\otimes X,u^V_g)$, where $u^V_g: F_g(V\otimes X)\simeq V\otimes X$ is the composition
$$F_g(V\otimes X)\simeq V\otimes F_g(X)\stackrel{\rho(g)\otimes u_g}{\longrightarrow}V\otimes X.$$

\subsubsection{De-equivariantization: from $\Rep(G)$-action to $G$-action}   \label{deequivtion}
Let $A:=Fun (G)$ be the algebra of functions $G\to k$, a.k.a. the regular algebra.
The group $G$ acts on $A$ by left 
translations, so $A$ can be considered as an algebra in the monoidal category $\Rep (G)$.
Now let $\D$ be a Karoubian category with a \additive action of $\Rep(G)$.
Since $\Rep (G)$ acts on $\D$ there is a notion of $A$-module in $\D$. The category of
$A$-modules in $\D$ will be called the {\em de-equivariantization\,} of $\D$. For a while, we
will denote it by $\D_G$ (in \S\ref{alternat} we will introduce alternative notation and terminology). 
Clearly $\D_G$
is a Karoubian \additive category. Note that $G$
acts on $A$ viewed as an algebra in $Rep (G)$ (this is the action of $G$ on $Fun (G)$ by {\em right\,}
translations). So $G$ acts on the category of $A$-modules in $\D$, i.e., on $\D_G$.

\subsubsection{Main result} 

\begin{theorem}  \label{eqdepre} 
The 2-functors $\C \mapsto \C^G$ and $\D \mapsto \D_G$ are
mutually inverse 2-equivalences between the 2-category of Karoubian $k$-linear categories equipped
with $G$-action and the 2-category of Karoubian $k$-linear categories equipped with a
$k$-linear $\Rep(G)$-action. 
\end{theorem}

A proof of this well known theorem is given in \S\ref{forgetequiv}-\ref{enrich2} below.

\begin{remark}
In particular, the theorem says that for any Karoubian $k$-linear category $\C$
equipped with $G$-action there is an equivalence $\C \iso (\C^G)_G$ 
which commutes with the action of $\Vec_G$, and for any Karoubian $k$-linear category 
equipped with a $k$-linear $\Rep(G)$-action there is an equivalence $\D \iso (\D_G)^G$
which commutes with the action of $\Rep(G)$.
\end{remark}

\subsubsection{The forgetful functor $\C^G\to\C$ and the equivalence $\C\iso (\C^G)_G$} 
\label{forgetequiv}
Let $\C$ be a Karoubian \additive category with an action of a finite group $G$.

\begin{lemma} \label{eqtztion}  
\begin{enumerate}
\item[(i)] The forgetful functor $\Phi :\C^G\to \C$ is faithful. 
\item[(ii)] The induction functor
\begin{equation}  \label{induction}
\Ind :\C\to\C^G, \quad \Ind\, (X):=\mathop\oplus\limits_{g\in G} F_g(X)
\end{equation}
is both left and right adjoint to $\Phi$.
\item[(iii)] Each $X\in\C$ can be represented as
a direct summand of $\Phi(Y)$ for some $Y\in\C^G$. Namely, one can take $Y=\Ind\, (X)$.
\item[(iv)] Every $X\in\C^G$ can be represented as a $G$-equivariant direct summand of 
$\Ind\, (\Phi (X))$.
\end{enumerate}
\end{lemma}

\begin{proof}
Statements (i-iii) are clear. To prove (iv), notice that the composition of the adjunctions
$X\to\Ind\, (\Phi (X))\to X$ equals $|G|\cdot\id_X$ (it is here that we use that $k$ has characteristic 0).
\end{proof}

Now we will construct a $G$-equivariant equivalence $F:\C\iso(\C^G)_G$ such that the diagram
\begin{equation}   \label{drawit}
\xymatrix{\C\ar[rr]^F \ar[rd]_{\Ind}&&(\C^G)_G\ar[ld]^{\text{Forget}}\\ &\C^G&}
\end{equation}
commutes. Namely, for $X\in\C$ we define $F(X)$ to be the object 
$\Ind\, (X)\in\C^G$ (see \eqref{induction}) equipped with following action of the
algebra $A:=\Fun (G)$: a function $f\in\Fun (G)$ acts on $F_g(X)\subset\Ind\, (X)$ as 
multiplication by $f(g)$. The functor $F$ commutes with $G$-action
in the obvious way (use the isomorphism $F_g(F_{g'}(X))\iso F_{gg'}(X)$).

\begin{proposition}    \label{eqdeeq}
$F:\C\to (\C^G)_G$ is an equivalence.
\end{proposition}

\begin{proof}
We will construct $F^{-1}$. Let $Y\in (\C^G)_G$, i.e., $Y$ is a $G$-equivariant object of $\C$
with a $G$-equivariant algebra homomorphism $\Fun (G)\to\End (Y)$. For $g\in G$ let $e_g\in Fun (G)$
be the corresponding minimal idempotent, i.e., $e_g$ is the function on $G$ such that $e_g(g)=1$
and $e_g(g')=0$ for $g'\ne g$. Since $\C$ is Karoubian the images of the idempotents $e_g$ in
$\End (Y)$ define a decomposition $Y=\bigoplus\limits_{g\in G}Y_g$. Now set $F^{-1}(Y):=Y_1$,
where $1$ is the the unit of $G$.
\end{proof}

To prove Theorem \ref{eqdepre}, it remains to show that the 2-functor $\C \mapsto \C^G$ is
a 2-equivalence. In \S\ref{enrich2} we give a proof of this fact which was communicated to us 
by A.~Beilinson. It is based on a convenient way of thinking about 
categories equipped with $\Rep (G)$-action, which is explained in \S\ref{enrich1}.

\subsubsection{$\Rep (G)$-actions and $\Repr (G)$-enrichments}\label{enrich1}
Notation: $\Repr (G)$ is the category of all representations of $G$ over $k$
and $\Rep (G)\subset\Repr (G)$ is the full subcategory of finite-dimensional representations; 
$\VecT$ is the category of all vector spaces over $k$ and $\Vec\subset\VecT$ is the full subcategory
of finite-dimensional spaces.

Let $\D$ be a \additive category with a \additive $\Rep (G)$-action. For every $X,Y\in\D$ we have
a contravariant \additive functor
\begin{equation}  \label{contrav}
\Rep (G)\to\VecT ,\quad V\mapsto \Hom_{\D}(V\ot X,Y).
\end{equation}
Since $\Rep (G)$ is semisimple any \additive functor $\Rep (G)\to\VecT$ is representable by an
object of $\Repr (G)$. Let $\uHom (X,Y)\in\Repr (G)$ be the object representing the functor \eqref{contrav}, i.e., 
\begin{equation}  \label{inthom}
\Hom_G(V, \uHom (X,Y)) =\Hom_{\D} (V\ot X,Y).
\end{equation}
Thus we get a $\Repr (G)$-enriched category $\fD$. Clearly $\D$ can be uniquely reconstructed
from $\fD$: the objects are the same, $\Hom_{\D} (X,Y)= \uHom (X,Y)^G$, 
and the action of $\Rep (G)$ on $\D$ is uniquely determined by \eqref{inthom} if it exists, i.e., if
$\fD$ is {\em complete\,} in the following sense.

\begin{definition}
Let $\fD$ be a $\Repr (G)$-enriched category and $\D$  the \additive category with
$\Ob\,\D=\Ob\,\fD$ and $\Hom_{\D}(X,Y)= \uHom_{\fD}(X,Y)^G$. We say that 
$\fD$ is {\em complete\,} if for every $X\in\fD$ and $V\in\Rep (G)$
there exists $V\ot X\in\fD$ such that \eqref{inthom} holds functorially in $Y\in\D$.
We say that $\fD$ is {\em $G$-Karoubian\,} if $\D$ is Karoubian.
\end{definition}

So a Karoubian \additive category $\D$ equipped with a \additive $\Rep (G)$-action is essentially the
same as a $G$-Karoubian complete $\Repr (G)$-enriched category $\fD$.

\subsubsection{$\Repr (G)$-enrichments and $G$-actions}\label{enrich2}
Let $\C$ be a \additive category with $G$-action. In this subsection we view $\C^G$ as a
$\Repr (G)$-enriched category. So the objects of $\C^G$ are the $G$-equivariant objects in $\C$,
for every $X,Y\in\C^G$ the space $\uHom (X,Y)$ equals $\Hom_{\C}(X,Y)$, and the $G$-action on
$\uHom (X,Y)$ comes from the $G$-action on $\C$. If $\C$ is Karoubian then $\C^G$ is
$G$-Karoubian.

On the other hand, if $\fD$ is a $\Repr (G)$-enriched category we can forget the $G$-action on morphisms and get a \additive category, denoted by $\fD_{plain}$. Then we can interpret the
$G$-action on the morphisms of $\fD$ as a $G$-action on $\fD_{plain}$.
Let  $\fD_G$ denote the Karoubi envelope (a.k.a idempotent completion) of $\fD_{plain}$.
Clearly $G$ acts on $\fD_G$.

\begin{theorem}  \label{Sashastheorem}
The 2-functors $\C\mapsto\C^G$ and $\fD\mapsto\fD_G$ are mutually inverse 2-equivalences
between the 2-category of Karoubian
\additive categories equipped with $G$-action and the 2-category of $G$-Karoubian complete
$\Repr (G)$-enriched categories. 
\end{theorem}

\begin{rems}
(i) In particular, Theorem \ref{Sashastheorem} states that
the 2-functor $\C \mapsto \C^G$ is a 2-equivalence.
In view of the results in \S\ref{forgetequiv},
this implies Theorem \ref{eqdepre}.

(ii) Combining Theorems~\ref{Sashastheorem} and \ref{eqdepre} we see that if $\D$ and $\fD$
are related as in \S\ref{enrich1} then  $\fD_G$ identifies with $\D_G$.
\end{rems}

\begin{proof}
(i) Let $\C$ be a Karoubian \additive category equipped with a $G$-action.
The forgetful functor $\C^G\to\C$ identifies $(\C^G)_{plain}$ with the full subcategory of those
objects of $\C$ which admit a $G$-equivariant structure. By Lemma~\ref{eqtztion}(iii), 
the Karoubi envelope of this full subcategory equals $\C$. So $(\C^G)_G=\C$.

(ii) Now let $\fD$ be a $G$-Karoubian complete $\Repr (G)$-enriched category, $\C_0:=\fD_{plain}$,
and $\C=\fD_G$ (so $\C$ is the Karoubi envelope of $\C_0$). We will show that the obvious fully
faithful $G$-equivariant functors $\fD\hookrightarrow\C_0^G\hookrightarrow\C^G$
are equivalences, i.e., $(\fD_G)^G=\fD$.

Our $\fD$ is $G$-Karoubian, so to prove that $\C^G=\fD$ it suffices to represent any $X\in\C^G$ as
a $G$-equivariant direct summand of an object of $\D$. By Lemma~\ref{eqtztion}(iv), 
$X$ can be represented as a $G$-equivariant direct summand of $\Ind (\Phi (X))$, where
$\Phi :\C^G\to\C$ is the forgetful functor
and $\Ind :\C\to\C^G$ is the induction functor.
But $\C$ is the Karoubi envelope of $\C_0$, so $\Phi (X)$ is a direct
summand of some $Y\in\C_0$ and therefore $\Ind (\Phi (X))$ is a $G$-equivariant direct summand
of $\Ind\, Y$. Finally, $\Ind\, Y\in\fD\subset\C^G$ because $\C$ is complete ($\Ind\, Y$ is the tensor
product of $\Fun (G)\in\Rep (G)$ and $Y\in\fD$).
\end{proof}

\subsubsection{The canonical functor $\D\to\D_G$}  \label{Thecanonfunc}
By definition, $\D_G$ is the category of $A$-modules in $\D$, where $A\in\Rep (G)$ is the regular algebra. The forgetful functor $\Psi :\D_G\to\D$ has a left adjoint, namely,
\begin{equation}  \label{free-mod}
\mbox{Free}: \D\to\D_G, \quad X\mapsto A\ot X.
\end{equation}

If $\D=\C^G$ we can identify $\C$ with $\D_G$ using the equivalence $F:\C\iso\D_G$ from
\S\ref{forgetequiv}. Then the induction functor $\Ind :\C\to\C^G$ identifies with $\Psi :\D\to\D_G$
(see diagram \eqref{drawit}). The forgetful functor $\C^G\to\C$ is left adjoint to $\Ind$,
so it identifies with the functor \eqref{free-mod}.

\subsubsection{De-equivariantization: alternative notation and terminology}  \label{alternat}
Let $\D$ be a \additive category with an action of $\E:=\Rep (G)$. We will often
write $\D\boxtimes_{\E}\Vec$ instead of $\D_G$, and we will often call it the
{\em fiber category.} This is especially convenient if the Tannakian category 
$\E$ is the primary object and $G$ appears as the automorphism group
of a fiber functor $\E\to\Vec$.

\begin{rems}
(i) The idea behind the name  ``fiber category'' is that
$\D$ can be viewed as a family of categories over the classifying 
stack of the finite group $G$ corresponding to $\E$
and $\D\boxtimes_{\E}\Vec$ is the fiber of this family.

(ii) The notation $\D\boxtimes_{\E}\Vec$ suggests that the fiber category can be 
interpreted as a tensor product of $\E$-module categories
$\D$ and $\Vec$ (the structure of $\E$-module category on $\Vec$ comes from
the fiber functor $\E\to\Vec$).  It is indeed possible to give such an interpretation
(and to define the general notion of tensor product of module categories).
But in this article we treat the notation $\D\boxtimes_{\E}\Vec$
as a single ``hieroglyph''.
\end{rems}

\subsubsection{Semisimplicity}  \label{semis}
We say that a  \additive category is {\em semisimple\,} if it is equivalent
to a direct sum of finitely many copies of $\Vec$.
\begin{proposition}   \label{semisprop}
Let $\C$ be a Karoubian \additive category with an action of a finite group $G$.
Then $\C^G$ is semisimple if and only if $\C$ is.
\end{proposition}

\begin{proof}  
A \additive category $\A$ is semisimple if and only if it has the following properties:
\begin{enumerate}
\item[(i)] $\A$ is abelian;
\item[(ii)] every monomorphism in $\A$ splits;
\item[(iii)] $\dim\Hom (X,Y)<\infty$ for all $X,Y\in\A$;
\item[(iv)] every increasing sequence of strictly full Karoubian subcategories of $\A$ stabilizes.
\end{enumerate}

Now let $\C$ and $G$ be as in the proposition. By Theorem \ref{eqdepre}, 
$\C=(\C ^G)_G$, i.e., $\C$ is the category of
$\Fun (G)$-modules in $\C^G$. So it is clear that properties (i), (iii), (iv) hold for $\C^G$
if and only if they hold for $\C$. It remains to prove this for (ii).

Suppose that (ii) holds for $\C$. Let $Y\in\C^G$ and let $X\subset Y$ be a $G$-equivariant 
subobject. In $\C$ we have a retraction $r:Y\to X$. By averaging $r$ with respect to $G$ we
get a $G$-equivariant retraction $Y\to X$, i.e., a retraction in $\C^G$.

Now suppose that (ii) holds for $\C^G$. One has the functor $\Ind :\C\to\C^G$, 
$\Ind\, X:=\bigoplus\limits_{g\in G}F_g(X)$.
If $Y\in\C$ and $X\subset Y$ then $\Ind\, X\subset\Ind\, Y$, so
we have a retraction $\Ind\, Y\to\Ind\, X$. Composing it with the embedding 
$X=F_1(X)\hookrightarrow\Ind\, X$ and the projection $\Ind\,  Y\twoheadrightarrow F_1(Y)=Y$ we
get a retraction $Y\to X$.
\end{proof}  

\subsection{Equivariantization and de-equivariantization of tensor categories} \label{equi}
\subsubsection{Basic definitions}   \label{basicfus}
Let $\C$ be a tensor category (in the sense of \S\ref{terminology}). Consider the category $\underline{\Aut}(\C)$, whose objects are
tensor auto-equivalences of $\C$ and whose morphisms are isomorphisms of tensor functors.
$\underline{\Aut}(\C)$ has an obvious structure of monoidal category,
in which the tensor product is the composition of tensor functors.  

Recall that if $G$ is a group  then $\underline{G}$ denotes the corresponding monoidal category: the
objects of $\underline{G}$ are elements of $G$, the only morphisms are the identities and the tensor
product is given by multiplication in $G$. 

\begin{definition} \label{action}
An {\em action} of a group $G$ on a tensor category $\C$ is a monoidal functor
$\underline{G}\to \underline{\Aut}(\C)$. In this situation we also say that $G$ {\em acts} on $\C$.
\end{definition}

If $\C$ is equipped with a braided structure the category
$\underline{\Aut}(\C)$ contains the full subcategory $\underline{\Aut}^{br}(\C)$  consisting of braided equivalences. 

\begin{definition}
We say that an action $\underline{G}\to \underline{\Aut}(\C)$ respects the braided 
structure if the image of $\underline{G}$ lies in $\underline{\Aut}^{br}(\C)$. 
In this case we also say that $G$ acts on $\C$ viewed as
a braided tensor category.
\end{definition} 

We will use the following notion from \cite[2.1]{Be}.

\begin{definition}  \label{centraldef}
A {\em central functor\,} from a braided category $\B$ to a tensor category $\C$ is
a braided functor $\B\to\Z (\C )$.
\end{definition} 

Now let us introduce categorical analogs of the notions of associative unital 
algebra and commutative associative unital algebra over a commutative associative ring $E$.
Let $\E$ be a symmetric tensor category (it will play the role of  $E$).

\begin{definition}  \label{overdef}
A {\em tensor category over $\E$\,} is a tensor category $\C$ equipped with a central functor
$\E\to\C$ (i.e., a braided functor $\E\to\Z (\C )$. A {\em braided tensor category over $\E$\,}  is a 
braided tensor  category $\C$ equipped with a braided functor $\E\to\C'$ (see \S\ref{centra-drinf} for 
the definition of~$\C'$). If a tensor (resp. braided tensor) category $\C$ over $\E$ is a fusion category we say that {\em $\C$ is a fusion (resp. braided fusion) category over $\E$.}
\end{definition}

\begin{remark}
A braided tensor category $\C$ over $\E$ can also be viewed as a tensor category over $\E$:
namely, define $F: \E\to\Z (\C )$ to be the composition 
$\E\to\C'\hookrightarrow\C{\buildrel{\Phi}\over{\longrightarrow}}\Z(\C)$,
where $\Phi$ comes from the braided structure on $\C$. Note that the composition
$\C'\hookrightarrow\C{\buildrel{\Phi}\over{\longrightarrow}}\Z(\C)$ identifies with the composition
$\C'\iso (\C')^{\oP}\hookrightarrow\C^{\oP}\to\Z(\C)$.
\end{remark}

For a fixed group $G$,  tensor categories equipped with $G$-action form a 2-category
with 1-morphisms being tensor functors respecting the action and 2-morphisms being
tensor isomorphisms of tensor functors. Similarly, one has the 2-category of
braided tensor categories endowed 
with a $G$-action respecting the braided structure.

For a fixed symmetric tensor category $\E$, 
tensor categories over $\E$ (or braided tensor categories over $\E$) form a 2-category. 

\subsubsection{Equivariantization} \label{equiv2tion}
Suppose that a group $G$ acts on a tensor category $\C$. 
Let $\C^G$ be the category of $G$-equivariant objects (a.k.a. the equivariantization)
of $\C$, which was defined in \S\ref{equivtion}. Then $\C^G$ has
an obvious  structure of tensor category. We have a tensor functor
\begin{equation}  \label{embed}
\Vec\to\C,
\end{equation}
which canonically decomposes as $\Vec\to\Z (\C)\to\C$. Accordingly, we have a tensor functor 
\begin{equation}  \label{theembed}
\Rep(G)=\Vec^G\to\C^G, 
\end{equation}
which canonically decomposes as $\Rep(G)\to\Z (\C^G)\to\C^G$.
So $\C^G$ has a structure of tensor category over $\Rep(G)$.

Now suppose that $\C$ is a braided tensor category and the action of $G$ on $\C$ 
respects the braiding. Then $\C^G$ is a braided tensor category. Moreover, \eqref{theembed}
is a braided functor, so $\C^G$ is a braided fusion category over $\Rep(G)$.

\subsubsection{Main principle} \label{princ2}
The following result is due to A.~ Bru\-gui\`{e}res \cite{Br} and M.~M\"u\-ger~\cite{Mu1} (they formulated and proved it in the setting of pre-modular categories).

\begin{theorem}  \label{princprinc} 
Let $G$ be a finite group.

\begin{enumerate}
\item[(i)] Equivariantization  defines a 2-equivalence $\C\mapsto\C^G$ between the 2-category of Karoubian tensor categories with an action of $G$ and the 2-category of Karoubian tensor categories 
over $\Rep(G)$.
\item[(ii)] Equivariantization  defines a 2-equivalence $\C\mapsto\C^G$ between the 2-category of
 Karoubian braided tensor categories with an action of $G$ and the 2-category of Karoubian 
 braided tensor categories over $\Rep(G)$.
\item[(iii)] Let $\C$ be a Karoubian tensor category with an action of $G$. Then $\C$ is a fusion
 category if and only if $\C^G$ is a fusion category and the functor \eqref{theembed} is fully faithful.
\end{enumerate}
\end{theorem}
\begin{proof}
To prove (i), we will construct the inverse 2-functor. Let $\D$ be a Karoubian tensor category over 
$\Rep (G)$. Note that $\Z(\D)$ acts by multiplication on $\D$ viewed as a \additive category.
So the tensor functor $\Rep(G)\to \Z(\D)$ defines an action of $\Rep(G)$ on
$\D$ viewed as a \additive category. Therefore by
\S\ref{enrich1}, $\D$ gives rise to a 
$\Repr (G)$-enriched category $\fD$. Let us define on $\fD$ a structure of
$\Repr (G)$-enriched {\em tensor\,} category. We already have the tensor product on
the objects of $\fD$ (which are the same as those of $\D$). To define it on
morphisms, note that for $X_1,X_2,Y_1,Y_2\in\D$ and $V_1,V_2\in\Rep (G)$ the tensor
product map
\[
\Hom_{\D}(V_1\ot X_1,Y_1)\ot\Hom_{\D}(V_2\ot X_2,Y_2)\to\Hom_{\D}(V_1\ot V_2\ot X_1\ot X_2,Y_1\ot Y_2)
\]
can be interpreted as a map
\begin{eqnarray}   \label{enrichedtensprod}
\Hom_G(V_1, \uHom (X_1,Y_1))\ot\Hom_G(V_2, \uHom (X_2,Y_2))\to \\
\to\Hom_G(V_1\ot V_2, \uHom (X_1\ot X_2,Y_1\ot Y_2)). \nonumber
\end{eqnarray}
 
Since \eqref{enrichedtensprod} is functorial in $V_1,V_2\in\Rep (G)$ it yields a $\Repr(G)$-morphism
\[
\uHom (X_1,Y_1)\ot \uHom (X_2,Y_2)\to\uHom (X_1\ot X_2,Y_1\ot Y_2).
\] 
So we get a functor
$\ot :\fD\times\fD\to\fD$. In fact, $\fD$ is a $\Repr (G)$-enriched tensor category, whose associativity
isomorphisms are those of $\D$. Just as in \S\ref{enrich2}, define $\fD_G$ to be
the Karoubi envelope (a.k.a idempotent completion) of the plain (not enriched) tensor category
corresponding to $\fD$. Then $\fD_G$ is a Karoubian tensor category equipped with an action of $G$.
The 2-functor $\fD\mapsto\fD_G$ is inverse to the equivariantization 2-functor $\C\mapsto\C^G$
(see Theorem~\ref{Sashastheorem} and its proof). 

So we have proved (i). The proof of (ii) is similar. 

Let us prove (iii). By definition, a tensor category $\C$ is 
fusion if it is rigid, semisimple in the sense of \S\ref{semis}, and $\be_{\C}$ is simple. 
Now suppose that $\C$ is equipped with an action of $G$. Then $\C^G$ is rigid if and only if
$\C$ is (the ``only if" statement follows from Lemma~\ref{eqtztion}(iii)). By Proposition ~\ref{semisprop},
$\C^G$ is semisimple if and only if $\C$ is. Finally, the simplicity of $\be_{\C}$ is equivalent to the
fully faithfullness of the functor \eqref{embed}, which is equivalent to the fully faithfullness of 
its equivarianti\-za\-tion~\eqref{theembed}.
\end{proof}

\subsubsection{De-equivariantization of tensor categories} \label{deeqfus}
In the proof of Theorem~\ref{princprinc} we described the inverse of the equivariantization 2-functor
$\C\mapsto\C^G$. Here is an alternative description due to \cite{Br,Mu1}.

Let $G$ be a finite group and $\E:=\Rep(G)$. Or equivalently, let $\E$ be a Tannakian
fusion category equipped with a fiber functor $\E\to\Vec$, and let $G$ be the automorphism
group of the fiber functor.  Let $\D$ be a Karoubian tensor category over $\E$. Let $A$ be the regular algebra in $\Rep(G)$ (i.e., $A$ is the algebra of functions on $G$ equipped with the action of $G$ by left translations). The tensor functor $\Rep(G)\to \Z(\D)$ allows to talk about $A$-modules in $\D$. Let 
$\D_G$ be the category of $A$-modules in $\D$. It is a tensor category with respect to the functor
\begin{equation}  \label{tens_A}
M_1\ot_A M_2:=(M_1\ot M_2)\ot_{A\ot A}A, \quad M_1,M_2\in \D_G
\end{equation}
(note that by Proposition \ref{splitting}, the multiplication morphism $A\ot A\to A$ admits an
$(A\ot A)$-module splitting, so the r.h.s. of \eqref{tens_A} makes sense, even though
$\D$ is not assumed to be abelian). In \S\ref{deequivtion} we defined an action of 
$G$ on the $k$-linear category $\D_G$ (it comes from the action of $G$ on $A$ by right translations).
Clearly $G$ acts on $\D_G$ viewed as a tensor category. 

\begin{proposition}  \label{indeed_deeq}
The 2-functor $\D\mapsto\D_G$ from the 2-category of Karoubian tensor categories over 
$\Rep(G)$ to the 2-category of Karoubian tensor categories with $G$-action is inverse to the
2-functor $\C\mapsto\C^G$.
\end{proposition}

\begin{proof}
The equivalence $F=F_{\C}:\C\iso (\C^G)_G$ constructed in \S\ref{forgetequiv} has a natural tensor
structure. (Indeed, if $X,Y\in\C$ then  $F(X\ot Y)$ is the object 
$\bigoplus\limits_{g\in G}F_g(X\ot Y)\in\C^G$ equipped with the natural $A$-module structure and $F(X)\ot_AF(Y)$ is the object 
$\bigoplus\limits_{g\in G}(F_g(X)\ot F_g(Y))\in\C^G$  equipped with the natural
$A$-module structure.)

The equivalence $F_{\C}:\C\iso (\C^G)_G$ is natural in $\C$. 
So the 2-functor $\D\mapsto\D_G$ is left inverse to the
2-functor $\C\mapsto\C^G$. By Theorem~\ref{princprinc}(i), the
2-functor $\C\mapsto\C^G$ is a 2-equivalence, so its left inverse is also a right inverse.
\end{proof}

\begin{definition}
The tensor category $\D_G$ is called the {\em de-equivariantization\,} of $\D$
or the  {\em fiber category.}
\end{definition}

We will often use the notation $\D\bt_{\E}\Vec$ instead of $\D_G$
(the idea behind this notation and the name ``fiber category" was explained in \S\ref{alternat}).

\begin{remark}
By Theorem \ref{princprinc}(iii), $\D_G$ is a fusion category if and only if $\D$ is a fusion category
and the functor $\Rep (G)\to\D$ is fully faithful.. 
\end{remark}

\subsubsection{De-equivariantization of braided tensor categories}       \label{deeqbr}       
Let $\D$ be a braided category over $\E =\Rep (G)$. As explained in \ref{deeqfus},
$\D_G=\D\bt_{\E}\Vec$ is a tensor category with respect to the functor \eqref{tens_A}.

\begin{proposition}  \label{br-indeed_deeq}
\begin{enumerate}
\item[(i)] $\D_G$ has a unique braiding $M_1\ot_A M_2\iso M_2\ot_A M_1$
such that for every $M_1,M_2\in\D_G$ the diagram  
\begin{equation}   \label{braidingdia}
\xymatrix{M_1\ot M_2\ar[rr] \ar[d]
&&M_2\ot M_1\ar[d]\\ M_1\ot_A M_2\ar[rr]&&M_2\ot_A M_1}
\end{equation}
commutes.
\item[(ii)] The 2-functor $\D\mapsto\D_G$ from the 2-category of Karoubian braided
tensor categories over $\Rep(G)$ to the 2-category of Karoubian braided tensor categories with 
$G$-action is inverse to the 2-functor $\C\mapsto\C^G$.
\end{enumerate}
\end{proposition}

\begin{proof}
(i) Uniqueness is clear. Let us prove existence. By Theorem~\ref{princprinc}(ii), we can assume
that $\D=\C^G$ for some braided tensor category $\C$. Then we have a tensor equivalence
$F:\C\iso\D_G$ (see \S\ref{forgetequiv} and the proof of Proposition~\ref{indeed_deeq}).
So we get a braided structure on $\D_G$. To check the commutativity of \eqref{braidingdia},
write $M_i$ as $F(X_i)$, $X_i\in\C$.

(ii) By the construction of the braided structure on $\D_G$, the 2-functor $\D\mapsto\D_G$ is
left inverse to the 2-functor $\C\mapsto\C^G$. By Theorem~\ref{princprinc}(i), the
2-functor $\C\mapsto\C^G$ is a 2-equivalence, so its left inverse is also a right inverse.
\end{proof}

\begin{remark}
For direct proofs of existence in Proposition~\ref{br-indeed_deeq}(i) see
\cite[Proposition 2.2]{P}, \cite[Lemma 3.10]{Mu1}, or the proof of  \cite[Theorem 1.10]{KiO}
(especially, figure 6 from the proof).   
\end{remark}

\subsubsection{The tensor functors $\C^G\to\C$ and $\D\to\D_G$} \label{Thetensorfunc}
Let $\D$ be a tensor category over $\Rep (G)$. By definition, $\D_G$ is the category of $A$-modules in $\D$ equipped with the tensor product \eqref{tens_A} (here $A\in\Rep (G)$ is the regular algebra).
So the ``free $A$-module" functor
\begin{equation}  \label{freemod}
\mbox{Free}: \D\to\D_G, \quad X\mapsto A\ot X
\end{equation}
has an obvious structure of tensor functor. If $\D$ is a {\em braided\,} tensor category over 
$\Rep (G)$ then this tensor functor is braided. 

On the other hand, if $\C$ is a tensor (resp. braided tensor) category with $G$-action then the forgetful functor $\Phi :\C^G\to \C$ is a tensor (resp. braided) functor.

Now let $\D=\C^G$ and let $F:\C\iso\D_G$ be the tensor equivalence from the proof of
Proposition~\ref{indeed_deeq}.

\begin{lemma}  \label{nasty}
The tensor functor $F\Phi : \D\to\D_G$ is canonically isomorphic to the tensor functor
\eqref{freemod}.  \hfill\qedsymbol
\end{lemma}

\noindent {\bf Reformulation of the lemma.} 
{\em If one identifies $\C$ with $\D_G$ using the tensor equivalence
$F:\C\iso\D_G$ then the forgetful functor $\C^G\to\C$ identifies, as a
tensor functor, with \eqref{freemod}.}

\begin{remark}  \label{pseudotensrem}
Recall that the functor $\mbox{Free}: \D\to\D_G$ and the forgetful functor $\C^G\to\C$
have right adjoints, namely, the forgetful functor $\D_G\to\D$ and the induction functor
$\Ind :\C\to\C^G$ (see formula~\eqref{induction} and \S\ref{Thecanonfunc}). These right adjoints
are {\em pseudotensor\,} functors in the sense of \cite[\S1.1.5 and \S1.1.6(ii)]{CH}. 
E.g., the structure of pseudotensor functor on $\Ind :\C\to\C^G$ is given by the morphism
\begin{equation}   \label{pseudotenstruct}
\Ind (X)\ot \Ind (Y)\to \Ind (X\ot Y), \quad X,Y\in\C
\end{equation}
that kills $\alpha (X)\ot\beta (Y)\subset\Ind (X\ot Y)$ if 
$\alpha,\beta\in G$ are not equal and maps $\alpha (X)\ot\alpha (Y)\subset\Ind (X\ot Y)$ to 
$\alpha (X\ot Y)\subset\Ind (X\ot Y)$ in the obvious way. If $G\ne \{ 1\}$ the morphism of functors
\eqref{pseudotenstruct} is not an isomorphism, so $\Ind$ is not a tensor functor.
\end{remark}

\subsubsection{Dimensions}\label{equidim}
From now on we assume that $\C$ is a fusion category with $G$-action and $\D$ is
a fusion category over $\Rep (G)$ such that the functor $\Rep (G)\to\D$ is fully faithful
($\C$ and $\D$ are not assumed to be braided).
The following result shows that dimensions are well behaved under equivariantization and
de-equivariantization.

\begin{proposition}\label{equidimprop}
 We have $\dim(\C^G)=|G|\cdot\dim(\C)$, {\em $\FPdim(\C^G)=|G|\cdot\FPdim(\C)$}.
Equivalently, $\dim(\D_G)=\frac1{|G|}\dim(\D)$ and {\em $\FPdim(\D_G)=\frac1{|G|}\FPdim(\D)$}.
\end{proposition}

\begin{proof} The equivalence of the formulas involving $\C$ and $\D$ follows immediately from
Proposition~\ref{indeed_deeq}. We will prove only the formulas involving $\C$.

Let $X$ be a simple object of $\C$ and 
$H\subset G$ the stabilizer of its isomorphism class. We have a central extension 
$0\to k^{\times}\to \hat H\to H\to 0$,
where $\hat H$ is the set of pairs consisting of an element $h\in H$ and an isomorphism
$F_h(X)\iso X$. The set of isomorphism classes of simple objects of $\C^G$ whose
image under the forgetful functor $\C^G\to \C$ contains $X$ is in natural bijection 
with the set $\text{Irr}^1(\hat H)$ of irreducible finite-dimensional representations of
$\hat H$ such that each $\lambda\in k^{\times}\subset\hat H$ acts as multiplication by $\lambda$. 
For $V\in \text{Irr}^1(\hat H)$
let $X_V$ denote the corresponding simple object of $\C^G$. We will show that
\begin{equation}\label{proj1}
|X_V|^2=(G:H)^2\cdot \dim(V)^2\cdot |X|^2
\end{equation}
\begin{equation}\label{proj}
\FPdim(X_V)=(G:H)\cdot \dim(V)\cdot \FPdim(X).
\end{equation}
To prove \eqref{proj1}, notice that there exists a pair $(\psi,\psi^G)$ of
functorial (but not necessarily tensor) isomorphisms
$\psi_X: X\to X^{**},\;  X\in \C$ and $\psi_Y^G: Y\to Y^{**},\;
Y\in \C^G$ such that $F(\psi^G)=\psi$ where $F: \C^G\to \C$ is the forgetful functor (such pairs
are in bijection with $G$-invariant functions $\O(\C)\to k^\times$, see \S \ref{squared}). For such a pair
 we clearly have $d_\pm(X_V)=d_\pm(F(X_V))=(G:H)\cdot\dim(V)\cdot d_\pm(X)$, so 
\eqref{proj1} follows from \eqref{dim1}. Similarly, \eqref{proj} follows from the formula
$\FPdim(X_V)=\FPdim(F(X_V))$, which is a consequence of the uniqueness of the Frobenius-Perron
dimension, see \S \ref{fpd}.  

Now the proposition follows easily from \eqref{proj1}, \eqref{proj}, and the formula
$$\sum_{V\in \text{Irr}^1(\hat H)}\dim(V)^2=|H|.$$
\end{proof}

\begin{corollary} \label{equiintegral}
The category $\C^G$ is integral (resp. weakly integral) if and only if $\C$ is. \hfill\qedsymbol
\end{corollary} 

\begin{proof}
In the weakly integral case use Proposition \ref{equidimprop}. In the integral case use 
formula~\eqref{proj}.
\end{proof}

The following corollary is an analog of Proposition~\ref{dimgradi} for categorical dimensions.

\begin{corollary}     \label{dimgradcat}
Let $G$ be a finite group and let $\C =\oplus_{g\in G}\, \C_g$ be a faithfully $G$-graded fusion category. 
Then $\dim(\C) = |G|\cdot\dim(\C_1)$.
\end{corollary}
\begin{proof}
By \cite[Proposition 2.9(ii)]{ENO3}, $\Z(\C)$ contains 
a Tannakian subcategory $\E =\Rep(G)$
such that $\E' \bt_\E \Vec \cong \Z(\C_1)$. Computing the categorical dimensions of both sides
using    \eqref{dimZ}, Theorem~\ref{Muger2}(i), and Proposition~\ref{equidimprop} we get
$\dim(\C)^2/ |G|^2 = \dim(\C_1)^2$.
\end{proof}

\begin{corollary} \label{relativedim}
Let $\C$ be a fusion category over a symmetric fusion category $\E$. Set
$\dim_{\E}(\C):=\dim (\C)/\dim (\E)$, $\FPdim_{\E}(\C):=\FPdim (\C)/\FPdim (\E)$.
Then $\dim_{\E}(\C)$ and $\FPdim_{\E}(\C)$ are algebraic integers.
\end{corollary} 

\begin{proof}
If $\E$ is Tannakian this follows from Proposition~\ref{equidimprop} and the statements from
\S\ref{fpd} and \S\ref{pivotal}.
Now suppose that $\E$ is not Tannakian. For any $n\in\BN$ let $\C^{\bt n}$ denote the tensor product
of $n$ copies of $\C$. Let $\E_n\subset\E^{\bt n}$ be the maximal Tannakian subcategory. 
Theorem~\ref{Deligne1}(iii) implies that
$\dim (\E_n)=\frac{1}{2}\dim (\E^{\bt n})$, so considering $\C^{\bt n}$ as a fusion category over 
$\E_n$ we see that $2(\dim_{\E}(\C))^n$ is an algebraic integer for any $n\in\BN$. Therefore
$\dim_{\E}(\C)$ is an algebraic integer. A similar argument works for $\FPdim_{\E}(\C)$.
\end{proof}

\subsubsection{Fusion subcategories and gradings}
Let $\D$ be a fusion category over $\E=\Rep (G)$
such that the functor $\E\to\D$ is fully faithful.
Then the fiber category $\C:=\D\bt_{\E}\Vec$
is a fusion category with $G$-action, and $\D =\C^G$.

\begin{proposition}    \label{lattice isomorphism}
\begin{enumerate}
\item[(i)] There is an isomorphism between the lattice of fusion subcategories $\tilde\D\subset\D$ 
containing $\E$ and the lattice  of $G$-stable fusion subcategories $\tilde\C\subset\C$:
namely, $\tilde\D=\tilde\C^G$, $\tilde\C=\tilde\D\bt_{\E}\Vec$.
\item[(ii)] There is a bijection between the set of $G$-invariant
gradings of $\C$ and the set of gradings of $\D$ trivial on $\E$: namely, to  a 
grading $\C =\oplus_{b\in B}\,\C_b$ such that each $\C_b$ is $G$-stable
one associates the grading of $\D$ by $B$ defined by $\D_b:= (\C_b)^G$, $b\in B$.
This bijection preserves faithfulness of the gradings.
\item[(iii)] Let $\D$ be a braided fusion category over $\E$ (so
$\C:=\D\bt_{\E}\Vec$ is braided). Then the isomorphism from (i)
commutes with taking the centralizer, i.e.,
$(\tilde\C^G)'=(\tilde\C')^G$, $(\tilde\D\bt_{\E}\Vec )'=\tilde\D'\bt_{\E}\Vec$.
\end{enumerate}
\end{proposition}
\begin{proof} 
Statement (i) is clear. 
To prove (ii), we will construct the inverse map 
\[
\{\mbox{Gradings of }\D\mbox{ trivial on }\E \} \to \{G\mbox{-invariant gradings of }\C\}.
\]
Namely, given a grading $\D =\oplus_{b\in B}\,\D_b$ trivial
on $\E$ we set 
\[
\C_b:=(\D_b)_G:=\{\mbox{the category of }A\mbox{-modules in }\D_b\},
\]
where $A$ is the regular algebra of $\E$. We have to show that if $M_1\in\C_{b_1}$
and $M_2\in\C_{b_2}$ then $M_1\ot_A M_2\in\C_{b_1b_2}$. This follows from the fact
that $M_1\ot_A M_2$  is a direct summand of  $M_1\ot M_2$.

Finally, let us prove (iii).
Let $\tilde\C\subset\C$ be a $G$-stable fusion subcategory.
Clearly $(\tilde\C')^G$ centralizes $\tilde\C^G$, so 
$(\tilde\C')^G\subset (\tilde\C^G)'$. Let $\Phi :\C^G\to\C$ be the forgetful
functor. To prove that $(\tilde\C^G)'\subset (\tilde\C')^G $, we have to show that if 
$X\in\C^G$ centralizes $\tilde\C^G$ then $\Phi (X)$ centralizes each object $Y\in\tilde\C$. 
By Lemma \ref{eqtztion}(iii), $Y$ is a direct summand of $\Phi (Z)$ for some $Z\in\tilde\C^G$. 
Since $X$ centralizes $Z$ we see that $\Phi (X)$ centralizes $\Phi (Z)$. So $\Phi (X)$ centralizes $Y$.
\end{proof}

\begin{corollary} \label{newc}
(i) The braided fusion category $\D\bt_{\E}\Vec$ is non-degenerate if and only if $\D'=\E$.

(ii) $\D\bt_{\E}\Vec$ is symmetric (resp. Tannakian)  if and only if $\D$ is.

(iii) The isomorphism from Proposition \ref{lattice isomorphism} maps the set of symmetric (resp. Tannakian) subcategories of $\D$ containing $\E$ onto the set of $G$-stable symmetric (resp. Tannakian)  subcategories of $\D\bt_{\E}\Vec$. 
\end{corollary} 
\begin{proof} 
Let $\C:=\D\bt_{\E}\Vec$, so $\D =\C^G$.
By Theorem \ref{Muger1}, $\C$ is non-degenerate if and only if $\C'=\Vec$. 
On the other hand, by Lemma \ref{eqtztion}(iii),  $\C'=\Vec$ if and only if $(\C')^G$ equals $\Vec^G$,
which is the same as $\Rep(G)$. Now statement (i) follows from Proposition \ref{lattice isomorphism}(ii). 

To prove (ii) in the symmetric case, use Proposition \ref{lattice isomorphism}(ii).
In the Tannakian case use Lemma \ref{tan-eq}. Clearly (iii) follows from (ii).
\end{proof}

\subsubsection{Some canonical equivalences}  \label{catdiam}

\begin{lemma}
\label{3rd iso}
Let $\D$ be a fusion category over the Tannakian category $\E:=\Rep (G)$
such that the functor $\E\to\D$ is fully faithful
and let  $\F$ be a fusion subcategory of $\E$.
Let  $\tilde{\D} := \D \bt_\F \Vec$ and
$\tilde{\E} := \E \bt_\F \Vec$.
\begin{enumerate}
\item[(i)] There is a $G$-equivariant  tensor equivalence
\begin{equation}
\label{3rd equiv}
\tilde{\D}\bt_{\tilde{\E}} \Vec \cong \D\bt_{\E} \Vec.
\end{equation}
\item[(ii)] If $\D$ is a braided fusion category over  $\E$ 
(so both sides of  \eqref{3rd equiv}  are braided)
then \eqref{3rd equiv} is an equivalence of  braided fusion categories. 
\end{enumerate}
\end{lemma}
\begin{proof}
We will derive the statements of the lemma from the corresponding results about
equivariantizations.  Namely, let $\C$ be a fusion category with an action of a group $G$
and let $H$ be a normal subgroup of $G$.  The quotient group  $G/H$ acts
on $\C^H$ and there is an equivalence
\begin{equation} 
\label{2step equiv}
\C^G \cong (\C^H )^{G/H}  
\end{equation}
of fusion categories over $\Rep(G)$.  Clearly, if $\C$ is a {\em braided} fusion category and 
the action of $G$ respects its braiding then \eqref{2step equiv} is an equivalence of 
braided categories over $\Rep(G)$.

Set $\D =\C^G,\, \F =\Rep(G/H)$, then
$\D \bt_\F \Vec = \C^H$. Taking de-equivariantizations 
of both sides of \eqref{2step equiv} we obtain a $G$-equivariant 
equivalence \eqref{3rd equiv}.
\end{proof} 

\begin{remark}    \label{Psifunctor}
In the situation of Lemma~\ref{3rd iso} one gets a canonical $G$-equivariant tensor functor
\begin{equation} \label{functorPsi}
 \Psi : \D \bt_\F \Vec\to\D\bt_{\E} \Vec
 \end{equation}
by composing the
tensor functor $\D \bt_\F \Vec=\tilde{\D}\to\tilde{\D}\bt_{\tilde{\E}} \Vec$
with the equivalence~\eqref{3rd equiv}. If $\D$ is a braided fusion category over  $\E$ 
then $ \Psi$ is braided.
\end{remark}

\begin{proposition} \label{diamond2isomorphism} 
\begin{enumerate}
\item[(i)] Let $G$ be a finite group and $\D$ a fusion category over the Tannakian category 
$\E:=\Rep (G)$ such that the functor $\E\to\D$ is fully faithful. Let $\D_1\subset\D$ be a fusion subcategory. Then there is a canonical $G$-equivariant tensor equivalence 
\begin{equation}
\label{diaomd2eqn}
\D_1\boxtimes_{\D_1\cap \E } \Vec \to (\D_1 \vee \E)\boxtimes_\E \Vec.
\end{equation}
\item[(ii)] Suppose that $\D$ is a {\em braided\,} fusion category over $\E$ (so both sides of 
\eqref{diaomd2eqn} are braided). Then the equivalence \eqref{diaomd2eqn} is braided.
\end{enumerate}
\end{proposition}

\begin{rems}     \label{notrem}
\begin{enumerate}
\item[(i)] The Tannakian subcategory $\D_1\cap \E\subset\E$ equals $\Rep (G_1)$ for some quotient group $G_1$ of $G$.
\item[(ii)] Recall that $(\D_1 \vee \E)\boxtimes_\E\Vec$ is the notation for the de-equivariantized
category $(\D_1 \vee \E)_G$. Similarly, $\D_1\boxtimes_{\D_1\cap \E } \Vec:=(\D_1)_{G_1}$.
So $G$ acts on both sides of \eqref{diaomd2eqn} (the action on the r.h.s. factors through $G_1$).
\end{enumerate}
\end{rems}

In the proof below we will be using the notation from Remarks \ref{notrem}.

\begin{proof}
Define the functor \eqref{diaomd2eqn} to be the composition
\begin{equation}  \label{eqstar}
\D_1\boxtimes_{\D_1\cap \E } \Vec {\buildrel{\Phi}\over{\longrightarrow}}
(\D_1 \vee \E)\boxtimes_{\D_1\cap \E } \Vec 
{\buildrel{\Psi}\over{\longrightarrow}}\ (\D_1 \vee \E)\boxtimes_\E \Vec, 
\end{equation}
where $\Phi$ comes from the embedding $\D_1\hookrightarrow\D_1 \vee \E$ of
fusion categories over $\D_1 \cap \E$ and $\Psi$ is the functor~\eqref{functorPsi}.
Both $\Phi$ and $\Psi$ are tensor functors, and in the situation of (ii) they are braided.
Clearly $\Phi$ is $G_1$-equivariant and $\Psi$ is $G$-equivariant, so $\Psi\circ\Phi$ is
$G$-equivariant. It remains to prove that $\Psi\circ\Phi$ is an equivalence.

First consider the easy case where $\D_1\cap \E=\Vec$. Then we have to show that the composition
$\D_1\hookrightarrow\D_1\vee\E\to (\D_1\vee\E )\bt_{\E}\Vec$ is an equivalence. Since $\D$ is
over $\E$ we have a canonical tensor functor $\D_1\bt\E\to\D_1\vee\E$ given by
the tensor product of $\D$. It is obviously surjective and it is fully faithful 
by Lemma~\ref{tens fun from product}.
So it is an equivalence and $(\D_1\vee\E )\bt_{\E}\Vec=(\D_1\bt\E )\bt_{\E}\Vec$. But $\D_1\bt\E=\D_1^G$
(here we consider the trivial action of $G$ on $\D_1$). 
So $(\D_1\bt\E)\bt_{\E}\Vec=(\D_1^G)_G=\D_1$. 

Now consider the general case. Set $\tilde{\D}_1 := \D_1\boxtimes_{\D_1\cap \E } \Vec$
and $\tilde{\E}:= \E \boxtimes_{\D_1\cap \E } \Vec$. Then 
$\tilde{\D}_1 \cap \tilde{\E} =({\D}_1 \cap {\E}) \bt_{{\D}_1 \cap {\E}} \Vec =\Vec$,
so we are in the easy situation considered above. Therefore  the composition
\begin{equation} \label{eqsta2}
\tilde\D_1\to\tilde\D_1 \vee \tilde\E\to
(\tilde\D_1 \vee \tilde\E)\boxtimes_{\tilde\E} \Vec
\end{equation}
is an equivalence. By Proposition~\ref{lattice isomorphism}(i),
$\tilde\D_1 \vee \tilde\E= (\D_1 \vee \E) \boxtimes_{\D_1\cap \E} \Vec$.  
Clearly $\D_1 \vee \E$ is a fusion
category over  $\E$, so applying  Lemma~\ref{3rd iso} to
the Tannakian subcategory $\F:=\D_1\cap \E\subset\E$ we identify 
$(\tilde\D_1 \vee \tilde\E)\boxtimes_{\tilde\E} \Vec$ with 
$(\D_1 \vee \E) \boxtimes_\E \Vec$. Thus diagram \eqref{eqsta2} identifies with diagram \eqref{eqstar}.
So the composition in \eqref{eqstar} is an equivalence.
\end{proof}

\begin{corollary}    \label{diamondcor}
Let $\D$ be a braided fusion category, $\D_1\subset\D$ a fusion subcategory,
and $\E\subset\D$ a Tannakian fusion subcategory equipped with a fiber functor $F: \E \to \Vec$.  
Let $G_F: = \Aut (F)$, so $\E:=\Rep (G_F)$. Then
\begin{enumerate}
\item[(i)]
there is a canonical $G_F$-equivariant tensor equivalence 
$$\D_1\boxtimes_{\D_1\cap \E } \Vec \iso (\D_1 \vee \E)\boxtimes_\E \Vec,$$
where $\boxtimes_\E \Vec$ and $\boxtimes_{\D_1\cap \E } \Vec$ correspond
to the fiber functors $F: \E \to \Vec$ and $F|_{\D_1\cap\E}: \D_1\cap\E\to \Vec$, respectively;
\item[(ii)]
if  $\E$ centralizes $\D_1$ then this equivalence is braided.
\end{enumerate}
\end{corollary}

\begin{proof}
$\D$ has a structure of fusion category over $\E$:
namely, define the braided functor $\E\to\Z (\D )$ to be the composition 
$\E\hookrightarrow\D{\buildrel{\Phi}\over{\longrightarrow}}\Z(\D)$,
where $\Phi$ comes from the braided structure on $\D$. 
So statement (i) immediately follows from Proposition~\ref{diamond2isomorphism}(i).
If  $\E$ centralizes $\D_1$ then 
$\D_1 \vee \E$ is a braided fusion category over $\E$, so statement (ii) follows from 
Proposition~\ref{diamond2isomorphism}(ii).
\end{proof}

\subsection{Tensor categories $\D$ equipped with a tensor functor $\Rep(G)\to\D$} \label{sketchy}
Fix a finite group $G$.  By Theorem~\ref{princprinc}(i), the 2-category of Karoubian tensor categories $\D$ equipped with a {\em central\,} functor $F:\Rep (G)\to\D$ is 2-equivalent to the 2-category of Karoubian tensor categories $\C$ equipped with $G$-action. The equivalence is given by equivariantization and 
de-equivariantization, i.e., $\C=\D_G\,$, $\D=\C^G$.

In this subsection 
we briefly discuss the situation where $F:\Rep (G)\to\D$ is not a central functor but merely a tensor functor. Let $\K_1$ be the 2-category of Karoubian tensor categories $\D$ equipped with a tensor
functor $F:\Rep (G)\to\D$. Let $\K_2$ be the 2-category of Karoubian tensor categories $\cP$ 
equipped with a tensor functor $\Phi:\Vec_G\to\C$. 

Let $A$ denote the regular algebra in $\Rep (G)$. If $\D\in\K_1$ then $A$-bimodules in $\D$
form a Karoubian tensor category $\cP$. It is easy to see that the tensor category of 
$A$-bimodules in $\Rep (G)$ identifies with $\Vec_G\,$, so we get a tensor functor $\Phi:\Vec_G\to\cP$.
Thus we have defined a 2-functor $\K_1\to\K_2$.

On the other hand, let $B$ denote the group algebra $k[G]$ viewed as an algebra in $\Vec_G$.
If $\cP\in\K_2$ then $B$-bimodules in $\cP$ form a Karoubian tensor category $\D$. It is easy to see that the tensor category of $B$-bimodules in $\Vec_G$ identifies with $\Rep (G)$, so we get a tensor functor $F:\Rep (G)\to\D$. Thus we have defined a 2-functor $\K_2\to\K_1$.

\begin{theorem}   \label{Last_principle}
These 2-functors $\K_1\to\K_2$ and $\K_2\to\K_1$ are mutually inverse 2-equivalences.
\end{theorem}

This theorem is proved in \cite{ENO4}. In the following remarks we explain the idea of the proof.
Theorem~\ref{Last_principle} will not be used in this article.

\begin{rems}
\begin{enumerate}
\item[(i)] 
If $R_1$ and $R_2$ are Morita-equivalent rings then the category of
rings $L$ equipped with a homomorphism $R_i\to L$ does not depend on $i$ (because it can be
interpreted as the category of algebras in the tensor category of $R_i$-bimodules and this tensor
category does not depend on $i$).

\item[(ii)] M.~M\"uger \cite{Mu-I} introduced a notion of Morita equivalence for fusion 
categories\footnote{In this situation M\"uger used the name ``weak monoidal Morita equivalence".}
parallel to the classical notion of Morita equivalence for rings 
(modules over rings are replaced by module categories over fusion categories, see \cite{Mu-I, O1}).
It is known that the fusion categories $\Rep (G)$ and $\Vec_G$ are Morita equivalent 
(this is a reformulation of Theorem~\ref{eqdepre}). There are also other examples of Morita equivalent fusion categories (see \cite{EO, Mu-I, Mu4}).

\item[(iii)] In Theorem~\ref{Last_principle} one can replace $\Rep (G)$ and $\Vec_G$ by
any pair of fusion categories related by a Morita equivalence. This generalization of 
Theorem~\ref{Last_principle} is proved in \cite{ENO4} using essentially the argument from Remark~(i) above.
\end{enumerate}
\end{rems}

\begin{remark}   \label{bimodules&deeq}
Let us explain the relation between the de-equivariantization 2-functor $\D\mapsto\D_G$ from 
\S\ref{deeqfus} and the 2-functor $\K_1\to\K_2$ from Theorem~\ref{Last_principle}.
Suppose that $\D$ is a Karoubian tensor category over $\Rep (G)$ (i.e., one has a 
{\em central\,} functor $\Rep (G)\to\D$). Let $\C:=\D_G$ be its de-equivariantization (i.e., $\C$ is the category of $A$-modules in $\D$). Then $\C$ is equipped with $G$-action, and it is easy to show
that the category of $A$-bimodules in $\D$ canonically identifies with the semidirect product 
$\Vec_G\ltimes\C$.
\end{remark}

\subsection{Braided categories containing $\Rep(G)$ and braided $G$-crossed categories}\label{orbi}
In this subsection we consider only those tensor categories which are fusion.
However, it is not hard to formulate and prove analogs of the definitions and results of
\S\ref{subsubsec-containing}-\ref{deeqcrossed} for tensor categories which are not necessarily fusion (in the spirit of Theorem~\ref{princprinc}).

\subsubsection{Braided categories containing a symmetric category}  \label{subsubsec-containing}
\begin{definition}  \label{containingdef}
Let $\E$ be a symmetric fusion category. A {\em braided fusion category containing $\E$\,} is a braided fusion category $\D$ equipped with a fully faithful braided functor $\E\to\D$.
\end{definition}

\subsubsection{Main principle}     \label{princ3}
Let $G$ be a finite group. Theorem~\ref{princprinc}(ii) describes those braided fusion categories
$\D$ containing $\Rep(G)$ which satisfy the additional property $\Rep(G)\subset\D'$;
namely, it says that such $\D$ is essentially the same as a braided fusion category with a braided
$G$-action. A.~Kirillov, Jr. and M.~M\"uger \cite{Ki,Mu3} discovered a similar description of
{\em arbitrary\,} braided fusion categories $\D$ containing $\Rep(G)$. Their main principle says
that {\em the datum of a braided fusion category $\D$ containing $\Rep(G)$ is equivalent to the datum of 
a braided $G$-crossed category $\C$.} The latter notion (due to
V.~Turaev \cite{Tuor}) is recalled in \S\ref{brcrossed} below. The above two types of data are 
related via the mutually inverse equivariantization and de-equivariantization constructions, 
which are described in \S\ref{equiv3tion} and \S\ref{deeqcrossed} below. For a precise
statement see Theorem \ref{princcross} below.

\subsubsection{Braided $G$-crossed fusion categories}  \label{brcrossed}
\begin{definition}\label{crossed}
 A {\em braided $G$-crossed fusion category} is a fusion category $\C$
equipped with the following structures:

\begin{enumerate}
\item[(i)] an action of $G$ on $\C$;
\item[(ii)] a (not necessarily faithful) grading $\C=\bigoplus\limits_{g\in G}\C_g$;
\item[(iii)] isomorphisms 
\begin{equation} \label{G-braid}
c_{X,Y}: X\ot Y\iso g(Y)\ot X, \quad g\in G, \, X\in \C_g, \, Y\in \C,
\end{equation}
which are called the {\em $G$-braiding\,} isomorphsims.
 \end{enumerate}
These data should satisfy the following conditions:
\begin{enumerate}
\item[(a)]$g(\C_h)\subset \C_{ghg^{-1}}$  for all $g,h\in G$;
\item[(b)] the isomorphisms $c_{X,Y}$ are functorial in $X$ and $Y$;
\item[(c)] the isomorphisms $c_{X,Y}$ are compatible with $G$-action, i.e.,
\begin{equation} \label{G-compat}
\gamma (c_{X,Y})=c_{\gamma (X),\gamma (Y)}, \quad \gamma\in G;
\end{equation}
\item[(d)] the following diagrams commute for all $g,h\in G$, $X\in \C_g$, and $Y\in \C_h$:
\begin{equation}   \label{G-hex1}
\xymatrix{X\ot Y\ot Z\ar[rr]^{c_{X,Y\ot Z}} \ar[d]_{c_{X,Y}\ot \id_Z
}&&g(Y\ot Z)\ot X\ar[d]^{||}\\ g(Y)\ot X\ot Z\ar[rr]^{\id_{g(Y)}\ot c_{X,Z}}&&g(Y)\ot g(Z)\ot X}
\end{equation}
\begin{equation}   \label{G-hex2}
\xymatrix{X\ot Y\ot Z\ar[rr]^{c_{X\ot Y,Z}} \ar[rd]_{\id_X\ot c_{Y,Z}
}&&gh(Z)\ot X\ot Y\\ &X\ot h(Z)\ot Y\ar[ru]_{c_{X,h(Z)}\ot \id_Y}&}
\end{equation}
\end{enumerate}
\end{definition}

\begin{rems}   \label{crossedrems}
\begin{enumerate}
\item[(i)] A braided $G$-crossed fusion category concentrated in degree $1\in G$
(i.e., such that $\C_g=0$ for $g\ne 1$) is the same as a braided fusion category
with a braided $G$-action.
\item[(ii)] According to Proposition~\ref{ptcr} below, a pointed fusion category $\C$ with
$\O (\C)=G$ has a natural structure of  braided $G$-crossed fusion category.
\item[(iii)] In \S\ref{Bra} we introduce the notion of a braided $G$-crossed category $\A$ in the
setting of abstract (rather than linear) monoidal categories, and in \S\ref{crcatcrmod} we explain
how to construct examples of such categories in which all objects and morphisms are
invertible. Let $\A$ be such a category. If the automorphism group
of the unit object of $\A$ equals $k^{\times}$ and both $G$ and the group of isomorphism classes
of objects of $\A$ are finite then the $k$-linear hull of $\A$ is a braided $G$-crossed fusion category.
(Here ``linear hull" means the category of all contravaraint functors $F:\A\to\Vec$ such that for
every object $X\in\A$ each $\lambda\in k^{\times}=\Aut (X)$ acts on $F(X)\in\Vec$ as multiplication
by $\lambda$.)
\item[(iv)] If $\C$ is a braided $G$-crossed fusion category and $H\subset G$ is a normal  subgroup
then $\C^H$ has a natural structure of braided $(G/H)$-crossed fusion category 
(in \S\ref{equiv3tion} below we define this structure in the case $H=G$, and the general definition is 
similar). E.g., $\Vec_G$ is a braided $G$-crossed fusion category by Remark (ii) above, so 
$(\Vec_G)^H$ is a braided $(G/H)$-crossed fusion category.
\end{enumerate}
\end{rems}

\subsubsection{Equivariantization}  \label{equiv3tion} 
Let $\C$ be a braided $G$-crossed fusion category. In particular, $\C$ is
a fusion category with $G$-action, so by Theorem~\ref{princprinc}(iii), applying the equivariantization
construction from \S\ref{equiv2tion} one gets a fusion category $\C^G$ over $\Rep (G)$.
Now we will define a braided structure on~$\C^G$.

If $X\in \C_g$ and $Y\in \C^G$ define a $\C$-isomorphism $\tilde c_{X,Y}: X\ot Y\iso Y\ot X$ to be
the composition $X\ot Y\xrightarrow{c_{X,Y}}g(Y)\ot X
\xrightarrow{u_g\ot \id_X} Y\ot X$, where $u_g$ comes from the structure of
equivariant object on $Y$. Extend $\tilde c_{X,Y}$  to all
objects $X\in \C$ by linearity. It follows from \eqref{G-compat} that if $X$ has a structure
of $G$-equivariant object then $\tilde c_{X,Y}$ respects the equivariant structures.
So for $X,Y\in \C^G$ we get a $\C^G$-isomorphism $\tilde c_{X,Y}: X\ot Y\to Y\ot X$.
It is functorial in $X,Y\in\C^G$. To show that $\tilde c$ is a braiding on the fusion category $\C^G$
it remains to check the two hexagon axioms. They follow from the commutativity of
\eqref{G-hex1} and \eqref{G-hex2}. It is easy to check that $\tilde c$ restricts to the standard
symmetric braiding on $\Rep(G)=\Vec^G\subset \C^G$. So $\C^G$ is a braided fusion category
containing $\Rep (G)$.

\begin{rems} \label{my_remarks}
\begin{enumerate}
\item[(i)] The isomorphisms $\tilde c_{X,Y}: X\ot Y\iso Y\ot X$ for $Y\in \C^G$, $X\in \C$ define
a braided functor $\C^G\to\Z (\C )$ whose composition with the forgetful functor $\Z (\C )\to\C$
equals the forgetful functor $\C^G\to\C$. In other words, {\em the forgetful functor 
$\C^G\to\C$ has a canonical structure of  central functor\,} in the sense of Definition~\ref{centraldef}.

\item[(ii)] Restricting the above central functor $\C^G\to\C$ to $\Rep (G)=\Vec^G\subset\C^G$
we get a central functor $\Phi :\Rep (G)\to\C$. As a tensor functor, $\Phi$ is just the composition of
the forgetful functor $\Rep (G)\to\Vec$ and the embedding $\Vec\hookrightarrow\C$. But the
structure of central functor on $\Phi$ is more interesting because it encodes the $G$-grading of $\C$.
Namely, if $X\in\C_g$ and $\rho :G\to\Aut V$ is a finite-dimensional representation then the
isomorphism $X\ot \Phi (V,\rho )\iso\Phi (V,\rho )\ot X $ equals the composition
\begin{equation}   \label{centralstructure}
X\ot V\xrightarrow{\id_X\ot\rho (g)} X\ot V=V\ot X.
\end{equation}

\end{enumerate}
\end{rems}

\subsubsection{Main Theorem} \label{mainth}
Let $\K_1$ denote the $2$-category of braided $G$-crossed fusion categories.
Let $\K_2$ denote  the $2$-category of braided fusion categories containing $\E:=\Rep (G)$.
In \S\ref{equiv3tion} we constructed a 2-functor $\K_1 \to \K_2$, namely $\C\mapsto\C^G$.

\begin{theorem}  \label{princcross} 
This 2-functor $\K_1 \to \K_2$ is a 2-equivalence.
\end{theorem}

The proof of the theorem is given in \S\ref{trickyproof} below. A different proof was given by 
A.~Kirillov, Jr. and M.~M\"uger \cite{Ki,Mu3}; their description of the inverse 2-functor 
$\K_2 \to \K_1$ is given in \S\ref{deeqcrossed} below. Neither proof is transparent enough.
One could try to find a transparent proof using Theorem~\ref{Last_principle}
and Remark~\ref{bimodules&deeq}.

\subsubsection{Proof of Theorem~\ref{princcross}}  \label{trickyproof}

To construct the inverse 2-functor $\K_2 \to \K_1$, we will first define a new $2$-category $\K_3$.
\begin{definition}   \label{K3def}
An object of $\K_3$ is a semisimple rigid tensor category $\A$ (in the sense of \S\ref{semis} and
\S\ref{terminology}) 
equipped with the following structures:
\begin{enumerate}
\item[(i)] an action of  $G\times G$ on $\A$;
\item[(ii)] a $(G\times G)$-equivariant decomposition 
\begin{equation}
\label{Eg}
\be =\bigoplus_{g\in G}\, E_g
\end{equation}
of the unit object of $\A$ into a sum of simple objects (here $G\times G$ acts on $G$ by
two-sided translations, i.e., $(g_1,g_2)\in G\times G$ takes $g\in G$ to $g_1gg_2^{-1}$);
\item[(iii)] a $(G\times G)$-equivariant relative braiding between the forgetful tensor functors
 $F_1:\A^{G\times \{1\}}\to \A$ and $F_2:\A^{\{1\}\times G}\to \A$; by a  {\em relative braiding\,}  we mean
 a functorial collection of isomorphisms
\begin{equation}  \label{relbraiding}
c_{X,Y}:F_1(X)\ot F_2(Y)  \xrightarrow{\sim} F_2(Y)\ot F_1(X),
\qquad X\in\A^{G\times \{1\}},\, Y\in \A^{\{1\}\times G},
\end{equation}
satisfying the compatibility axioms analogous to the hexagon axioms of a braided category.
\end{enumerate}
One defines $1$-morphisms and $2$-morphisms in $\K_3$ in the obvious way.
\end{definition}

\begin{rems}   \label{Sh}
\begin{enumerate}
\item[(i)] By \cite[Proposition XI.2.4]{Kass}, the ring $\End (\be )$ is commutative, so
$\Hom (E_{g_1}, E_{g_2})=0$ for $g_1\ne g_2$.
One also has $E_{g_1}\ot E_{g_2}=0$ for $g_1\ne g_2$ (this follows from the equality
$A\ot\id_{\be}=\id_{\be}\ot A$, $A\in\End (\be )$, see {\em loc.\,cit.}). So the isomorphism 
$\be\ot\be\iso\be$ induces isomorphisms $E_g\ot E_g\iso E_g$, $g\in G$. 
\item[(ii)] Let $Sh(G)$ denote the tensor category of sheaves of finite-dimensional $k$-vector spaces 
on $G$. By the previous remark, the decomposition \eqref{Eg} defines a fully faithful tensor functor 
\begin{equation}   \label{Sheavesembed}
Sh(G)\hookrightarrow\A, \quad \F\mapsto \bigoplus_{g\in G} (\F_g\ot E_g).
\end{equation}
\end{enumerate}
\end{rems}

We will construct the inverse $2$-functor $\K_2\to \K_1$ as a composition 
$\K_2\iso \K_3\iso\K_1$. The following lemma yields a $2$-functor $\K_3\to \K_2$.

\begin{lemma}  \label{K3toK2}
Let $\A\in\K_3$. Then
\begin{enumerate}
\item[(i)] the tensor category $\A^{G\times G}$ is fusion;
\item[(ii)] the relative braiding on $\A$ induces a braiding on $\A^{G\times G}$;
\item[(iii)] the functor \eqref{Sheavesembed} induces a fully faithful braided functor
\begin{equation}
\label{sheave embed}
\Rep (G)=Sh(G)^{G\times G}\hookrightarrow\A^{G\times G}.
\end{equation}
\end{enumerate}
\end{lemma}

\begin{remark}
In (iii) we used the equivalence
$Sh(G)^{G\times G}\iso\Rep (G), \quad \F\mapsto\F_e$,
where $e\in G$ is the unit.
\end{remark}

\begin{proof}
(ii) and (iii) are clear. It is also clear that $\A^{G\times G}$ is a rigid tensor category. By
Proposition~\ref{semisprop}, it is semisimple. By (iii), the ring of endomorphisms of the unit object of
$\A^{G\times G}$ equals $k$, so $\A^{G\times G}$ is fusion.
\end{proof}

\begin{remark}
\label{no putanitsa}
Let $\D:=\A^{G\times G}$.
The embedding $\Rep(G) \hookrightarrow \D$ from  \eqref{sheave embed} 
yields a  braided functor $\Rep(G) \bt \Rep(G)  \hookrightarrow \D \bt \D^\op$.
Combining it with  the braided functor   $ \D \bt \D^\op  \hookrightarrow  \Z(\D)$ 
from \S\ref{centre} one gets a functor $\Rep(G)\bt \Rep(G)   \to \Z(\D)$.
It is straightforward to check that this functor equals the braided functor
$\Rep(G\times G) \to \Z(\D)$ that comes from the fact that $\D$ is the category
of $(G\times G)$-equivariant objects.
\end{remark}

\begin{lemma}   \label{doubledeequivariantization}
The 2-functor $\K_3\to \K_2$ constructed in Lemma~\ref{K3toK2} is a 2-equivalence.
\end{lemma}

\begin{proof}
We will construct the inverse 2-functor.
Let $\D$ be a braided fusion
cate\-gory containing $\E:=\Rep(G)$.
One has a braided functor $\D \bt \D^{\oP}\to\Z (\D )$ (see Proposition~\ref{Muger1prime}(iii) and
Remark~\ref{newrems}(a)). Composing it with the embedding 
$\E\bt\E= \E\bt \E^{\oP}\hookrightarrow \D \bt \D^{\oP}$  one gets a braided functor $\E\bt\E\to\Z (\D )$.
So $\D$ is a fusion category over $\E\bt\E =\Rep (G\times G)$. 
Let $\A := \D_{G\times G}$ be its de-equivariantization in the sense of \S\ref{equi}.
Since $\A$ contains $\E_{G\times G}= Sh (G)$ we conclude that
$\A$ has structures (i) and (ii) from Definition~\ref{K3def}.

To construct the relative braiding, we will use the following observation due to Joyal and Street \cite{JS}:
for any tensor categories $\B_1,\,\B_2,\,\B$ the category of pairs of
tensor functors $F_i:\B_i\to \B,\,i=1,2$, equipped with a relative braiding between $F_1$ and $F_2$
is canonically equivalent to the category of tensor functors $\B_1\bt\B_2\to\B$.
Applying this for $\B_1=\B_2=\B=\D$, $F_1=F_2=\id_{\D}$ we see that
the braided structure on $\D$ defines a tensor functor $\D\bt\D\to\D$. This is
the multiplication functor
\begin{equation}\label{JS functor}
M: \D\boxtimes \D\to \D ,\quad  X\bt Y \mapsto X\ot Y
\end{equation}
with the tensor structure 
$M((X_1\bt Y_1)\ot (X_2\bt Y_2))\iso M(X_1\bt Y_1)\ot M(X_2\bt Y_2)$
being 
$\id_{X_1}\ot c_{X_2,Y_1}\ot \id_{Y_2}:X_1\ot X_2\ot Y_1\ot Y_2 \iso X_1\ot Y_1\ot X_2\ot Y_2$
(here $c$ denotes the braiding on $\D$).
Recall that $\D$ is a fusion category over $\E\bt \E$. Consider 
$\D\boxtimes \D$ as a fusion category over $\E\bt \E$ via
the embedding
\begin{equation}\label{choose2from4}
\E \bt \E \hookrightarrow \Z(\D)\bt \Z(\D) =\Z(\D\bt\D) : 
U\bt V \mapsto  (U,\, c_{U,-}^{-1})\bt (V,\, c_{-,V}).
\end{equation}
Then it is easy to check\footnote{Since $\D \bt \D$ and $\D$ are fusion categories over
$\E \bt \E$ we have central functors $\Phi_1:\E \bt \E \to \D \bt \D$ and $\Phi_2:\E \bt \E \to \D$.
Moreover, the tensor functors $M\circ\Phi_1:\E \bt \E \to \D$ and $\Phi_2:\E \bt \E \to \D$
are canonically isomorphic (both identify with the tensor functor $A\bt B\mapsto A\ot B$).
So it remains to check that $M$ is compatible with the {\em central\,} structures on $\Phi_1,\Phi_2$.
Explicitly, one has to check that for all $X,Y\in \D$, $A,B\in \E$ the following diagram commutes:
\begin{equation*}
\xymatrix{(X\ot Y)\ot (A\ot B)\ar[rr]\ar[d]&&(A\ot B)\ot (X\ot Y)\ar[d]\\
(X\ot A)\ot (Y\ot B)\ar[rr]&&(A\ot X)\ot (B\ot Y)}
\end{equation*}
(here the vertical arrows come from the tensor structure on the functor \eqref{JS functor}
and the horizontal ones come from the central structures on $\Phi_1,\Phi_2$).
This is left to the reader.}
that the functor \eqref{JS functor} is a tensor functor over $\E\bt \E$.

Taking its de-equivariantization we get a tensor functor
\begin{equation}\label{JSdeeq}
M_{G\times G} :  \D_{\{1\}\times G}\boxtimes \D_{G\times \{1\}} 
= \A^{G\times \{1\}}  \boxtimes \A^{\{1\}\times G}
\to \D_{G\times G}  =\A
\end{equation}
whose restrictions to $\A^{G\times \{1\}}$ and $\A^{\{1\}\times G}$ are
the forgetful tensor functors $F_1:\A^{G\times \{1\}}\to \A$ and $F_2:\A^{\{1\}\times G}  \to \A$. 
Again using the Joyal-Street observation, we get a
$(G\times G)$-equivariant relative braiding between $F_1$ and~$F_2$.
\end{proof}

Now we have to construct a $2$-functor $\K_3\to \K_1$. We will define it by
$\A\mapsto\A^{G\times \{1\}}$, but to introduce the structure of braided $G$-crossed category
on $\A^{G\times \{1\}}$ we need some lemmas.

\begin{lemma}   \label{matrices}
Let $\A\in\K_3$.
\begin{enumerate}
\item[(i)] $\A=\bigoplus\limits_{g_1,g_2\in G}\A_{g_1,g_2}$, where $\A_{g_1,g_2}:=E_{g_1}\ot\A \ot E_{g_2}$.
\item[(ii)] If $X\in\A_{g_1,g_2}$ and $Y\in\A_{g_3,g_4}$ then $X\ot Y\in\A_{g_1,g_4}$ and
$X\ot Y\ne 0$ only if $g_2=g_3$.
\item[(iii)] $\A$ is equipped with $G$-gradings
\begin{equation}  \label{firstgrading}
\A=\bigoplus\limits_{g\in G}\A_g, \quad \A_g:=\bigoplus\limits_{x\in G} \A_{gx,x}
\end{equation}
\begin{equation}  \label{secondgrading}
\A=\bigoplus\limits_{g\in G}\A_{(g)}, \quad \A_{(g)}:=\bigoplus\limits_{x\in G} \A_{x.xg}
\end{equation}
The gradings \eqref{firstgrading} and \eqref{secondgrading} are preserved by $\{1\} \times G$ and  
$G  \times \{1\}$, respectively.
\end{enumerate}
\end{lemma}

\begin{proof}
Use Remark~\ref{Sh} (i).
\end{proof}

The next lemma says that the relative braiding defined on equivariant objects of $\A\in\K_3$
extends to a ``crossed braiding" defined on all objects of $\A$. Notation: if $\A\in\K_3$
the result of the action of $(g_1,g_2)\in G\times G$ on an object $X\in\A$ will be denoted by
${}^{g_1} {X}^{g_2^{-1}}$. 

\begin{lemma}   \label{from_relative_to_crossed}
Let $\A\in\K_3$.
\begin{enumerate}
\item[(i)] There exists a unique functorial collection of isomorphisms 
\begin{equation}  \label{crbrading}
c_{X,Y} : X \ot Y \xrightarrow{\sim}  Y^{g^{-1}} \ot {}^{h^{-1}} X,\quad X\in\A_{(g)} ,\; Y\in\A_h
\end{equation}
such that if $X$ is $(G  \times \{1\})$-equivariant and $Y$ is $(\{1\} \times G)$-equivariant then
\eqref{crbrading} equals the isomorphism \eqref{relbraiding}.

\item[(ii)] The image of $c_{X,Y}$ under any $(\alpha,\beta)\in G\times G$ equals
$c_{{}^{\alpha}X^{\beta^{-1}},{}^{\alpha}Y^{\beta^{-1}}}\,$.
\item[(iii)] The following diagram commutes for all $g_1,g_2,h\in G$, $X_i\in\A_{(g_i)}$, $Y\in\A_{h}$:
\begin{equation}  \label{firsthexagon}
\xymatrix{
X_1\ot X_2\ot Y\ar[rrrr]^{c_{X_1\ot X_2,Y}} \ar[d]_{\id_{X_1}\ot c_{X_2,Y}}&&&&
Y^{g_2^{-1}g_1^{-1}}\ot {}^{h^{-1}}(X_1\ot X_2)\ar[d]^{||}\\ 
X_1\ot Y^{g_2^{-1}}\ot {}^{h^{-1}} X_2\ar[rrrr]^{c_{X_1, Y^{g_2^{-1}}} \ot \id_{  {{}^{h^{-1}}X_2}}   }&&&&
Y^{g_2^{-1}g_1^{-1}} \ot {}^{h^{-1}}  X_1 \ot {}^{h^{-1}}   X_2.
}
\end{equation}
\item[(iv)] The following diagram commutes for all $g,h_1,h_2\in G$, $X\in\A_{(g)}$, $Y_i\in\A_{h_i}$:
\begin{equation}  \label{secondhexagon}
\xymatrix{X\ot Y_1\ot Y_2\ar[rrrr]^{c_{X,Y_1\ot Y_2}} \ar[d]_{c_{X,Y_1}\ot \id_{Y_2}}&&&&
(Y_1\ot Y_2)^{g^{-1}}\ot {}^{h_2^{-1} h_1^{-1} } X \ar[d]^{||}\\ 
Y_1^{g^{-1}}\ot {}^{h_1^{-1} } X \ot Y_2\ar[rrrr]^{\id_{ Y_1^{g^{-1}} }\ot c_{{}^{h_1^{-1} } X,Y_2} }&&&&
Y_1^{g^{-1}}\ot Y_2^{g^{-1}}\ot {}^{h_2^{-1} h_1^{-1}}  X.
}
\end{equation}
\end{enumerate}
\end{lemma}
\begin{proof}
We will only construct the isomorphism
\begin{equation}  \label{2crbrading}
c_{X,Y} : X \ot Y \xrightarrow{\sim}  Y^{g_2^{-1}g_1} \ot {}^{g_3g_2^{-1}} X,
\quad X\in\A_{g_1,g_2} ,\; Y\in\A_{g_2,g_3} 
\end{equation}
Define $\tilde X\in\A^{G\times\{1\}}$ and $\tilde Y\in \A^{\{1\} \times G}$ by 
$\tilde X:=\bigoplus\limits_{\alpha\in G}{}^{\alpha}X $, $\tilde Y:=\bigoplus\limits_{\beta\in G} Y^{\beta}$.
The $\A_{g_1,g_3}$-components of $\tilde X\ot\tilde Y$ and $\tilde Y\ot\tilde X$ equal $X\ot Y$
and $Y^{g_2^{-1}g_1} \ot {}^{g_3g_2^{-1}} X$, respectively. So the relative braiding 
$\tilde X\ot\tilde Y\xrightarrow{\sim}\tilde Y\ot\tilde X$ induces an isomorphism
$X \ot Y \xrightarrow{\sim}  Y^{g_2^{-1}g_1} \ot {}^{g_3g_2^{-1}} X$. This is $c_{X,Y}$.
\end{proof}

\begin{remark}
In formula \eqref{2crbrading} $Y^{g_2^{-1}g_1} \in\A_{g_1,g_4}$ and 
${}^{g_3 g_2^{-1}} X \in\A_{g_4,g_2}$, where $g_4=g_3g_2^{-1}g_1$. Note that $g_1,g_2,g_3,g_4$
form a ``parallelogram", i.e., $g_4g_1^{-1}=g_3g_2^{-1}$ or equivalently, $g_2^{-1}g_1=g_3^{-1}g_4$.
\end{remark}

\begin{corollary}   \label{Cconstruction}
If $\A\in\K_3$ then $\C:=\A^{G \times \{1\}}$ is a braided $G$-crossed fusion category: the $G$-action
on $\C$ comes from the $(\{ 1\} \times G )$-action on $\A$, the $G$-grading on $\C$ comes from 
the grading \eqref{secondgrading}, and the $G$-braiding on $\C$ comes from the 
iso\-mor\-phism ~\eqref{2crbrading}. \hfill\qedsymbol
\end{corollary}

Now we can finish the proof of Theorem~\ref{crossed}.
By Corollary \ref{Cconstruction}, we have a 2-functor $\K_3\to\K_1$, 
$\A\mapsto\C:=\A^{G \times \{1\} }$. It is easy to see that it is a 2-equivalence.
Recall that our 2-functor $\K_1\to\K_2$ takes $\C$ to $\C^G$, so the composition
$\K_3\iso\K_1\to\K_2$ is the 2-functor $\A\mapsto\A^{G\times G}$, which is a 2-equivalence
by Lemma~\ref{doubledeequivariantization}. Therefore the  2-functor $\K_1\to\K_2$ is a 2-equivalence.
\hfill\qedsymbol

\subsubsection{De-equivariantization}   \label{deeqcrossed} 
According to Theorem~\ref{princcross}, the 2-functor 
$\K_1 \to \K_2$ has an inverse. It will be denoted by $\D\mapsto\D_G\,$.
We will give an explicit description of the braided $G$-crossed fusion category $\D_G\,$.

Below we use the notation $\tilde c$ for the braiding on $\D$ and $c$ for the $G$-braiding on $\D_G$
to be constructed. As before, $A$ will denote the regular algebra in $\Rep (G)$, i.e., the algebra
of functions on $G$ equipped with the action of $G$ by left translations. As an algebra in $\Rep (G)$,
it carries an action of $G$, namely, the action by right translations. 

Recall that $\D$ is a fusion category over $\Rep (G\times G)$ (see the proof of
Lemma~\ref{doubledeequivariantization}). As a fusion category with $G$-action,
$\D_G$ is the de-equivariantization of $\D$ with respect to $\{ 1\}\times G\subset G\times G$
(this follows from \S\ref{deeqfus} or from the proof of Theorem~\ref{princcross}). In other words,
$\D_G$ is the category of left\footnote{Given a structure of left $A$-module on $M\in\D$ 
one can define a structure of right $A$-module on $M$ using the isomorphism 
$\tilde c_{A,M}^{-1}:A\ot M\iso M\ot A$.} $A$-modules in $\D$,
and the tensor product of $A$-modules is defined using the composition 
$\E=\E^{\oP}\hookrightarrow\D^{\oP}\to\Z (\D )$.
Explicitly, this means that if $M$ and $N$ are left $A$-modules in $\D$ then to define $M\ot_AN$ one makes $M$ into an $A$-bimodule using the isomorphism $\tilde c_{A,M}^{-1}: M\ot A\iso A\ot M$. The 
$G$-action on $\D_G$ comes from its action on $A$.

It remains to describe the $G$-grading and the $G$-braiding on $\D_G\,$. Here is the description
from \cite{Ki,Mu3}.

\begin{proposition}  \label{explicitdescr}
\begin{enumerate}
\item[(i)] A left $A$-module $M$ belongs to the graded component $(\D_G)_g$, $g\in G$, if and only if
the following diagram commutes:
\begin{equation*}
\xymatrix{A\ot M\ar[rr]^{\tilde c_{M,A} \tilde c_{A,M}} \ar[d]_{g^{-1}\ot\id_M}&&A\ot M\ar[d]\\ A\ot M\ar[rr]^{}&&M}
\end{equation*}
Here $g^{-1}$ denotes the operator $A\to A$ that takes a function $f:G\to k$ to the function 
$x\mapsto f(xg^{-1})$.

\item[(ii)] If $M\in (\D_G)_g$ and $N\in\D_G$ then the $G$-braiding
$c_{M,N}: M\ot_AN\iso g(N)\ot_AM$ is the unique morphism
such that the following diagram commutes:
 \begin{equation*}
\xymatrix{M\ot N\ar[rr]^{\tilde c_{M,N}} \ar[d]&&N\ot M\ar[d]\\ M\ot_AN\ar[rr]^{c_{M,N}}&&g(N)\ot_AM}
\end{equation*}
\end{enumerate}
\end{proposition}

\begin{proof}
By Theorem~\ref{princcross}, we can assume that $\D=\C^G$, where $\C$ is a braided $G$-crossed
fusion category. 
In this case we will check (i) and (ii) straightforwardly.

Let \,$\Ind :\C\to\C^G$ be the induction functor~\eqref{induction}. For $X,Y\in\C$ one has a canonical morphism $\Ind (X)\ot \Ind (Y)\to \Ind (X\ot Y)$, see \eqref{pseudotenstruct}.
The fusion category $\D_G=(\C^G)_G$ identifies with $\C$ so that $X\in\C$ corresponds to the
$A$-module $M=\Ind\, (X)$ (viewed as an object of $\D =\C^G$, our $A$ identifies with 
$\Ind\, (\be_{\C})$, and the $A$-module structure on $M$ is given by the canonical morphism 
$\Ind\, (\be_{\C})\ot\Ind\, (X)\to\Ind\, (X)\,$). Moreover, if $Y\in\C$ and $N=\Ind (Y)\in\D$ then the
usual morphism $M\ot N\to M\ot_A N$ identifies with the canonical 
morphism $\Ind (X)\ot \Ind (Y)\to \Ind (X\ot Y)$.

Now we see that statement (ii) is equivalent to the commutativity of the diagram
 \begin{equation}  \label{the-square}
\xymatrix{\Ind\, (X)\ot \Ind\, (Y)\ar[rr]^{\tilde c_{\Ind\, (X),\Ind\, (Y)}} \ar[d]&&\Ind\, (Y)\ot \Ind\, (X)\ar[d]\\ \Ind\, (X\ot Y)\ar[rr]^{\Ind\, (c_{X,Y})}&&\Ind\, (g(Y)\ot X)}
\end{equation}
for any $X\in\C_g$, $Y\in\C$. This commutativity immediately follows from the definition of the 
braiding $\tilde c$ in $\C^G$ (see \S\ref{equiv3tion}).

The automorphism $g^{-1}:A\to A$ from 
(i) identifies with the
automorphism $\varphi_g:\Ind (\be )\to\Ind (\be )$ that induces for each $\alpha\in G$ the
``identity map" between $\alpha (\be )\subset\Ind (\be )$ and $\alpha g (\be )\subset\Ind (\be )$.
So to prove (i), we have to check that an object $Z\in\C$ belongs to $\C_g$ if and only if 
the composition
\begin{equation}   \label{comp1}
\Ind (\be)\ot\Ind (Z)\xrightarrow{\tilde c_{\Ind (Z),\Ind (\be )}\cdot \tilde c_{\Ind (\be ),\Ind (Z)}}
\Ind (\be)\ot\Ind (Z)\to\Ind (Z)
\end{equation}
equals $u_{g,Z}$, where $u_{g,Z}$ is the composition
\begin{equation}     \label{comp2}
\Ind (\be)\ot\Ind (Z)\xrightarrow{\varphi_g\ot\id}\Ind (\be)\ot\Ind (Z)\to\Ind (Z).
\end{equation}
Since $u_{g,Z}\ne u_{g',Z}$ whenever $g\ne g'$ and $Z\ne 0$, it suffices to prove the ``only if"
statement, i.e., to show that if $Z\in\C_g$ then \eqref{comp1} equals \eqref{comp2}. To do this,
apply the commutativity of \eqref{the-square} for $X=\be$, $Y=Z$ and for $X=Z$, $Y=\be$.
\end{proof}

\subsubsection{Complements} \label{complem}
Let $G$ be a finite group and $\E :=\Rep(G)$. Let $\D$ be a braided
fusion category containing $\E$. Its de-equivariantization $\C:=\D_G$ is a braided $G$-crossed
fusion category and $\C^G=\D$. Let $\C_1\subset \C$ be the trivial component of the grading.
Recall that $\C_1$ is a braided category.

\begin{proposition}\label{crosscentralizers}
 (i) $\C_1=\E'_G=\E'\boxtimes_\E \Vec$ and $\E'=\C_1^G$.

(ii) $\D$ is non-degenerate if and only if $\C_1$ is non-degenerate and the grading on $\C$
is faithful.

(iii) $\D'\vee \E=(\C_1')^G$, where $\C_1'$ is the centralizer of $\C_1$ in $\C_1$.

(iv) Let $H\subset G$ be the image of the grading $\O(\C)\to G$. Then $H$ is a normal subgroup
and $\D'\cap \E$ equals $\Rep(G/H)\subset \Rep(G)=\E$.
\end{proposition}

\begin{proof} 
Let $X\in\C$, $X=\bigoplus\limits_{g\in G}X_g$, $X_g\in\C_g$. Let 
$Y=(V,\rho )\in\Rep (G)=\E\subset\C^G$. Then by the definition from \S \ref{equiv3tion},
$c_{Y,X}\tilde c_{X,Y}: X\ot Y\to X\ot Y$ is the endomorphism of $X\ot V$ whose restriction to
$X_g\ot V$ equals $\id_{X_g}\ot\rho (g)$. By definition, $\E'$ consists of all $X\in\C^G$ such that
$c_{Y,X}\tilde c_{X,Y}=\id_{X\ot Y}$ for all $Y=(V,\rho )\in\Rep (G)$. This happens if and only if
$X_g=0$ for $g\ne 1$.
So $\E'=\C_1^G$. This proves the
second part of (i). The first one follows by Proposition \ref{lattice isomorphism} (i).

On the other hand, $\D'\cap\E$ consists of all $Y=(V,\rho )\in\Rep (G)$ such that
$c_{Y,X}\tilde c_{X,Y}=\id_{X\ot Y}$ for all $G$-equivariant $X$. This happens if and only if
$\rho (g)=1$ for all $g\in H$. This proves (iv).

For (iii), note that $\D'\vee \E=\E''$ by Corollary \ref{Muger-2'} and $\D'\vee \E \subset \E'$. Thus the result
follows from  (i) and Proposition \ref{lattice isomorphism} (iii).

By (iii), $\C_1'=\Vec$ if and only if $\D'\subset \E$. Combining this with (iv) one gets~(ii).
\end{proof}

Recall the following definition from the introduction.

\begin{definition}
\label{Lagrangian subcategories}
Let $\D$ be a braided fusion category. We say that a Tannakian subcategory $\E \subset \D$
is {\em Lagrangian} if $\E'=\E$.
\end{definition}

\begin{proposition} \label{Lagrange}
The following properties of a Tannakian subcategory $\E \subset \D$ are equivalent:

(i) $\E$ is Lagrangian;

(ii) $\C_1=\Vec$;

(iii) $\C$ is pointed and $|\O(\C_g)|\le 1$ for every $g\in G$.
\end{proposition}

\begin{proof} It is clear from Proposition \ref{crosscentralizers}~(i) that (i) is equivalent to (ii).
It is also clear that (iii) implies (ii). 

Finally, assume that (ii) holds. Then $X\ot {\null^*X}\in \C_1=\Vec$ for every simple $X\in\C$.
On the other hand, $\dim \Hom(X\ot {\null^*X},\be)=\dim\Hom(X,X)=1$. Therefore every simple $X\in\C$
is invertible, i.e.,  $\C$ is pointed.  So $\O (\C)$ is a group. By (ii), the grading homomorphism
$\O (\C)\to G$ has trivial kernel. So it is injective, i.e., $|\O(\C_g)|\le 1$ for all $g\in G$.
Thus (ii) implies (iii).
\end{proof}

The next result is immediate from Propositions \ref{Lagrange} and \ref{crosscentralizers}~(ii).

\begin{corollary} \label{crossLag}
The following properties are equivalent:

(i) $\D$ is non-degenerate and $\E \subset \D$ is Lagrangian;

(ii) $\C$ is pointed and the grading map $\O(\C)\to G$ is bijective.
\end{corollary}

\subsubsection{Pointed categories as braided $G$-crossed categories} \label{ptbr}
\begin{remark}    \label{Adrem}
Let $\A$ be a monoidal category. Then every invertible object $X\in\A$ defines a monoidal equivalence 
\[
\Ad_X : \A\iso\A, \quad \Ad_X (Y):=X\ot Y\ot X^{-1}.
\]
If, in addition, $\A$ is $k$-linear and $\End (\be_{\A})=k$ then for every invertible object 
$X\in\A$ the homomorphism $k^{\times}=\Aut (X)\to\Aut (\Ad_X)$ is trivial. So $\Ad_X$ depend
(up to unique isomorphism) only on the isomorphism class of $X$. Thus the group of isomorphism
classes of invertible objects of $X$ acts on the monoidal category $\A$.
\end{remark}

Now let $\C$ be a pointed category and $G=\O (\C )$. Then all simple objects of $\C$ are invertible,
so $G$ is a group. According to Remark \ref{Adrem}, $G$ acts on $\C$ by conjugation, 
$g\mapsto\Ad_g$. On the other hand, $\C$ has a canonical $G$-grading 
(namely, the identity map $\O (\C )\iso G$). We claim that the $G$-action and the $G$-grading
can be combined together into a structure of braided $G$-crossed category on $\C$. Here is
a precise statement.

\begin{proposition}   \label{ptcr}
\begin{enumerate}
\item[(i)] The canonical $G$-grading on $\C$ can be upgraded to a structure of 
braided $G$-crossed category on $\C$. Such an upgrade is unique up to unique isomorphism.
\item[(ii)] The corresponding action of $G$ on $\C$ is the action by conjugation.
\item[(iii)] The braided fusion category $\C^G$ corresponding to the braided $G$-crossed category 
$\C$ canonically identifies with the center $\Z (\C )$. More precisely, the braided functor
$\C^G\to\Z (\C )$ from Remark~\ref{my_remarks}(i) is an equivalence.
\end{enumerate}
\end{proposition}

\begin{remark}    \label{upgrgroupoid}
In (i) the words ``to upgrade" mean ``to define an action of $G$ on $\C$ and a functorial collection of isomorphisms \eqref{G-braid} satisfying properties (a)--(d) from Definition \ref{crossed}.
All such ``upgrades" form a groupoid: an action is a monoidal functor 
$G\to\{\mbox{tensor functors } \C\to\C\}$, so there is a notion of isomorphism between two actions
and therefore a notion of isomorphism between two upgrades.
\end{remark}

\begin{proof}
Specifying a functorial collection of isomorphisms \eqref{G-braid} for {\em all\,} $X\in\C_g$ is the
same as specifying \eqref{G-braid} for {\em invertible\,}  $X\in\C_g$. But if $X\in\C_g$ is invertible
then \eqref{G-braid} is the same as an isomorphism $X\ot Y\ot X^{-1}\iso g(Y)$, i.e., an
isomorphism $\varphi_g :\Ad_g (Y)\iso g(Y)$ functorial in $Y\in\C$. The commutativity of
\eqref{G-hex1} means that $\varphi_g$ is an isomorphism of {\em tensor\,} functors for each
$g\in G$, and the commutativity of \eqref{G-hex2} means that the collection $\{\varphi_g\}_{g\in G}$
is an isomorphism between the two actions of $G$ on $\C$, namely, the action $g\mapsto\Ad_g$ and the action $(g,Y)\mapsto g(Y)$. There is one and essentially only one way to define an action 
$(g,Y)\mapsto g(Y)$ and a collection of isomorphisms $\varphi_g :\Ad_g (Y)\iso g(Y)$ with these properties: namely, $g(Y)=\Ad_g (Y)$, $\varphi_g=\id$. By construction, the corresponding
collection of isomorphisms \eqref{G-braid} satisfies properties (b) and (d) from Definition~\ref{crossed}.
It is easy to check that it also satisfies (a) and (c).

So we have proved (i) and (ii). Statement (iii) is checked straightforwardly by comparing
\S\ref{equiv3tion} with the definition of $\Z (\C )$ from \S\ref{center}. 
\end{proof}

\begin{remark}
In the proof of Proposition \ref{ptcr} it turned out that properties (a) and (c) from Definition~\ref{crossed}
follow from (b) and (d). This also happens in the situation of Proposition \ref{BraidCen}, which
is a generalization of Proposition \ref{ptcr}.
\end{remark}

\subsubsection{Centers of pointed categories and Lagrangian subcategories}  
\label{Centers of pointed}
Recall that the center of any fusion category is a non-degenerate braided fusion category,
see Corollary \ref{centerfactorizable}.

Recall that the center of any fusion category is a non-degenerate braided fusion category,
see Corollary~\ref{centerfactorizable}.

\begin{theorem} \label{centerhasLagr}\label{main1}
A non-degenerate braided fusion category $\D$ is braided equivalent to the
center of some pointed fusion category if and only if $\D$ contains a Lagrangian subcategory.
\end{theorem}

We will formulate and prove a more precise statement (see Theorem \ref{more precise}).
But first, note that the ``only if" part of Theorem~\ref{main1} is an immediate consequence of the
following remark.

\begin{remark}\label{rmk}
Let $\C$ be a pointed fusion category. By Corollary~\ref{centerfactorizable}, its center $\Z(\C )$ 
is a non-degenerate braided fusion category. Set $\E :=\pi^{-1}(\Vec) \subset\Z(\C )$, where
$\pi: \Z(\C )\to\C$ is the forgetful functor and $\Vec \subset \C$ is the fusion subcategory
generated by $\be_\C$. Then $\E\subset\Z(\C )$ is Tannakian and $\pi |_{\E}: \E \to \Vec$ is 
a fiber functor (in fact, if $\C =\Vec_G^\omega$ then $\E =\Rep (G)$ and $\pi |_{\E}$ is
the forgetful functor $\Rep(G)\to \Vec$).  Moreover, $\E$ is Lagrangian. Since $\E\subset\E'$
this can be proved by showing that $\dim (\E')$=$\dim (\E )$. We have $\dim (\E )=\dim (\C )$. 
By  \eqref{Dim D dim D'} and \eqref{dimZ}, 
$\dim (\E)\cdot\dim (\E' )=\dim (\Z(\C ))=\dim (\C )^2$.
So $\dim (\E')$=$\dim (\E )$.
\end{remark}

Now let $\fP$ be the 2-groupoid of pointed fusion categories (1-morphsims of $\fP$ are
tensor equivalences and 2-morphisms are isomorphisms between them). Let $\fL$ be 
the 2-groupoid of triples $(\D ,\E ,F)$, where $\D$ is a non-degenerate braided fusion category,
$\E\subset\D$ is a Lagrangian subcategory, and $F:\E\to\Vec$ is a fiber functor.

\begin{theorem}\label{more precise}
The 2-functor $\fP\to\fL$ defined in Remark~\ref{rmk} is a 2-equivalence.
\end{theorem}

\begin{proof} 
Theorem \ref{princcross} and Corollary~\ref{crossLag} identify $\L$ with the 2-groupoid of
pairs $(\C ,G)$, where $G$ is a finite group and $\C$ is a pointed braided $G$-crossed category
such that the grading map $\O (\C )\to G$ is an isomorphism. The latter 2-groupoid identifies
with $\fP$ by Proposition \ref{ptcr}.
\end{proof}

\begin{rems}  \label{morethan1}
\begin{enumerate}
\item[(i)] 
A braided fusion category $\C$ may have more than one Lagrangian subcategory.
E.g., if $\C$ is the pointed braided category $\C (G,q)$ corresponding to a metric group $(G,q)$
then Lagrangian subcategories of $\C (G,q)$ bijectively correspond to Lagrangian subgroups
of $G$ .
\item[(ii)] 
Combining the previous remark  with Theorem \ref{more precise} we see that a braided fusion 
category may have more than one representation as $\Z(\Vec_G^\omega)$. 
Given two pairs $(G,\omega )$ and $(G',\omega ')$ it is natural to ask how to describe all braided equivalences between $\Z(\Vec_G^\omega)$ and $\Z(\Vec_{G'}^{\omega '})$ in group-theoretical terms.
An answer to this question is given in \cite{NN}.
\end{enumerate}
\end{rems}

\section{The core of a braided fusion category} \label{coresec}
\subsection{Warm-up: the core of a pre-metric Lie coalgebra}  \label{warmup}
As explained in \S\ref{heursubsubsec}, non-degenerate braided categories are ``quantum"
analogs of pre-metric Lie coalgebras. Let us define the notion of the core of a
pre-metric Lie coalgebra $C$.

Let $\fg =C^*$ be the Lie algebra dual to $C$.
For each isotropic Lie subcoalgebra  $\fa\subset C$ we have the pre-metric Lie coalgebra
$C_{\fa}:=\fa^{\perp}/\fa$. Moreover, 
$C_{\fa}$ is a pre-metric Lie $\fg$-coalgebra, i.e., the coadjoint action of $\fg$ on $C$ 
gives rise to  a homomorphism
\begin{equation}   \label{Lieaction}
\fg\to\Der_m (C_{\fa}),
\end{equation}
where $\Der_m (C_{\fa})$ is the Lie algebra of derivations of $C_{\fa}$ preserving the 
inner product (here by a derivation of $C_{\fa}$ we mean its linear
endomorphism dual to a derivation of $\fg$).

An isotropic Lie subcoalgebra $\fa\subset C$ is said to be {\em maximal\,} if it is
not contained in any other isotropic Lie subcoalgebra. In this case $\fa$ contains the
kernel of the bilinear form on $C$, so $C_{\fa}$ is a {\em metric\,} Lie coalgebra.

\begin{lemma} \label{maximality1}
An isotropic Lie subcoalgebra $\fa\subset C$ is maximal if and only if the pre-metric Lie coalgebra
$C_{\fa}$ has no nonzero isotropic Lie subcoalgebras stable with respect to the action 
\eqref{Lieaction}  of $\fg$.
\hfill \qedsymbol
\end{lemma}

\begin{lemma} \label{independence}
If $\fa_1,\fa_2\subset C$ are maximal isotropic  Lie subcoalgebras 
then $C_{\fa_1}$ and $C_{\fa_2}$ are isomorphic as metric Lie $\fg$-coalgebras.
\end{lemma}

\begin{proof} 
We will show that both $C_{\fa_1}$ and $C_{\fa_2}$ are isomorphic to
$(\fa_1^\perp\cap\fa_2^\perp )/(\fa_1\cap\fa_2)$. Let us do this for $C_{\fa_1}$.

We have an obvious homomorphism 
$(\fa_1^\perp\cap\fa_2^\perp )/(\fa_1\cap\fa_2)\to\fa_1^{\perp}/\fa_1 =C_{\fa_1}$.
To prove that it is an isomorphism we have to show that
\begin{equation} \label{idty1}
\fa_1\cap\fa_2^\perp\subset\fa_2,
\end{equation}
\begin{equation}   \label{idty2}
(\fa_1^\perp\cap\fa_2^\perp )+\fa_1=\fa_1^\perp .
\end{equation}
\eqref{idty1} holds because otherwise $\fa_2+(\fa_1\cap\fa_2^\perp )$ would be an isotropic 
Lie subcoalgebra bigger than $\fa_2$. By the modular law, 
$(\fa_1^\perp\cap\fa_2^\perp )+\fa_1=\fa_1^\perp\cap (\fa_2^\perp +\fa_1)$. So to prove
\eqref{idty1} it suffices to show that $\fa_2^\perp +\fa_1\supset\fa_1^\perp$. This is
equivalent to the inclusion $\fa_2\cap\fa_1^\perp\subset\fa_1$, which is similar to
\eqref{idty1}.
\end{proof}

\begin{definition} 
The {\em core\,} of $C$ is ``the" metric Lie $\fg$-coalgebra $C_{\fa}$, where
$\fa\subset C$ is a maximal isotropic Lie subcoalgebra.
\end{definition}

\begin{remark}  \label{whyquot}
We put ``the'' in quotation marks for the following reason.
The proof of Lemma~\ref{independence} provides
a concrete isomorphism $f_{12}: \fa_1^\perp /\fa_1\stackrel{\sim}{\to} \fa_2^\perp /\fa_2$. 
But it may happen that given an  isotropic Lie subcoalgebra
$\fa_3\subset C$ one has $f_{13}\ne f_{23}\circ f_{12}$ (this happens already if $\fg$ is abelian).
\end{remark}

\subsection{Plan of this section}
The goal is to define the notion of core of a braided category $\C$.
The role of isotropic Lie subcolagebras  will be played by Tannakian subcategories $\E\subset\C$.
To a maximal  Tannakian subcategory $\E\subset\C$ we will associate in \S\ref{maxTan}
a braided category $\mbox{Core}_\E(\C)$ 
using the de-equivariantization construction from \S\ref{deeqfus}-\ref{deeqbr}.
To show that the equivalence class of $\mbox{Core}_\E(\C)$ does not depend on $\E$
we need an analog of the modular law used
in the proof of Lemma \ref{independence}. We prove it in \S\ref{lattice}.

\begin{remark}  
In the situation of \S\ref{warmup} the metric Lie coalgebra $C_{\fa}$ is 
equipped with a homomorphism \eqref{Lieaction}. Let $\fh_{\fa}$ denote the image of the
corresponding homomorphism $\fg\to\Der_m (C_{\fa})/\Int (C_{\fa})$, where
$\Int (C_{\fa})\subset\Der_m (C_{\fa})$ is the ideal of internal derivations.
In the setting of braided category we will define an analog of $\fh_{\fa}$.
Namely, the group of isomorphism classes of braided
autoequivalences of $\mbox{Core}_\E(\C)$ will be equipped with a certain subgroup
$\Gamma_{\E}$, which is analogous to $\fh_{\fa}$. 
\end{remark}

A ``quantum" analogue of Lemma~\ref{maximality1} is established in \S\ref{propofcore}.
It leads to the study of {\em weakly anisotropic} braided categories (see  Definition~\ref{def anisotropic})
in \S\ref{propofcore}-\ref{whyreasonable}.
Such categories are analogs of  metric Lie coalgebras without non-trivial isotropic 
Lie subcoalgebras stable under all metric-preserving derivations. 

\subsection{The lattice of fusion subcategories}  \label{lattice}

Let $L(\C)$  denote the lattice of fusion subcategories of a
fusion category $\C$. For any two subcategories $\mathcal{A}$ and
$\mathcal{B}$ their meet is their intersection and their join is
the category $\mathcal{A}\vee \mathcal{B}$.

\begin{lemma}
\label{Dedekind}
Let $\C$ be a fusion category such that
\begin{equation}\label{weakcom} 
X\ot Y\simeq Y\otimes X \mbox{ for }X,Y\in \O(\C).
\end{equation}
Then the lattice $L(\C)$ is modular, i.e.,
\begin{equation}
\label{modular law}
\mathcal{A} \cap (\mathcal{B} \vee \D)  = (\mathcal{A} \cap \mathcal{B}) \vee \D
\mbox { for all } \mathcal{A},\, \mathcal{B},\, \D\in L(\C) \mbox{ such that }\D\subset\mathcal{A}.
\end{equation}
\end{lemma}

\begin{proof}
The inclusion $\mathcal{A} \cap (\mathcal{B} \vee \D)\supset (\mathcal{A} \cap \mathcal{B}) \vee \D$
is obvious. Let us prove that 
$\mathcal{A} \cap (\mathcal{B} \vee \D)\subset (\mathcal{A} \cap \mathcal{B}) \vee \D$.
Property \eqref{weakcom} implies that a simple object $X\in\C$ is in $\mathcal{B} \vee \D$ if
and only if there exist simple objects $B\in\mathcal{B}$ and $D\in \D$ such that
$X$ is contained in  $B \ot D$. Now let $X$ be a simple object of 
$\mathcal{A} \cap (\mathcal{B} \vee \D)$ and let $B$ and $D$ be as above. Then 
$B$ is contained in $X \ot D^*$. But $X\in\mathcal{A}$ and $D\in \D\subset\mathcal{A}$,
so $B\in \mathcal{A}$. Therefore $B\in\mathcal{A}\cap\mathcal{B}$, 
$B\ot D\in(\mathcal{A} \cap \mathcal{B})\vee \D$, and finally 
$X\in(\mathcal{A} \cap \mathcal{B})\vee \D$.
\end{proof}

\begin{remark}
\begin{enumerate}
\item[(i)]It is easy to see that for fusion categories \eqref{weakcom} is equivalent to the commutativity
of the Grothendieck ring $K_0(\C)$.
\item[(ii)] When $\C =\Rep(G)$ is the representation category of a finite group $G$,
Lemma~\ref{Dedekind} gives a well-known property of the lattice of
normal subgroups of~$G$.
\end{enumerate}
\end{remark}

\subsection{Maximal Tannakian subcategories and the core}  \label{maxTan}
\subsubsection{Formulation of the theorem}      \label{formtheor}
Let $\E$ be a Tannakian subcategory of a braided
fusion category $\C$ and let $F: \E \to \Vec$ be a fiber functor.
By \S\ref{deeqfus}-\ref{deeqbr}, in this situation we have
the fiber category $\E' \boxtimes_\E \Vec$, which is a 
braided fusion category equipped with an action of
the group $G_F := \Aut (F)$, i.e., one has a monoidal functor 
$\underline{G_F}\to \underline{\Aut}^{br}(\E' \boxtimes_\E \Vec)$.

In particular, we have a homomorphism
\begin{equation} \label{act}
G_F\to {\Aut}^{br}(\E' \boxtimes_\E \Vec),
\end{equation}
where ${\Aut}^{br}(\E' \boxtimes_\E \Vec)$ is the group of isomorphism classes of
objects of the category $\underline{\Aut}^{br}(\E' \boxtimes_\E \Vec)$.
Let $\Gamma_{\E,F}$ be the image of \eqref{act}.

We say that a Tannakian subcategory $\E \subset \C$ 
is {\em maximal} if it is not contained 
in any other Tannakian subcategory of  $\C$.

\begin{definition}
\label{core corresponding to E}
Let $\C ,\E$ be as above and suppose that the Tannakian category $\E$  is maximal.
The pair $(\E' \boxtimes_\E \Vec , \Gamma_{\E ,F})$ is called
the {\em core of $\C$
corresponding to $\E$} and denoted by $\mbox{Core}_{\E ,F}(\C)$.
\end{definition} 

In \S\ref{Proof of canmod} we will prove the following theorem.

\begin{theorem}
\label{canonical modularization} 
Let $\C$ be a  braided fusion category and
let $\E_1$, $\E_2$ be maximal Tannakian subcategories of
$\mathcal{C}$ with fiber functors $F_1:\E_1\to\Vec$ and $F_2:\E_2\to\Vec$. Then 
\begin{equation}  \label{coreeq}
\mbox{Core}_{\E_1 ,F_1}(\C) \simeq \mbox{Core}_{\E_2 ,F_2}(\C).
\end{equation}
\end{theorem}

The precise meaning of \eqref{coreeq} is as follows: let 
$\E_1' \boxtimes_{\E_1}\Vec$ and $\E_2' \boxtimes_{\E_2}\Vec$
be the fiber categories corresponding to the fiber functors
$F_1$ and $F_2$, then there exists a braided tensor equivalence 
$H:\E_1' \boxtimes_{\E_1}\Vec\xrightarrow{\sim}\E_2' \boxtimes_{\E_2}\Vec$
such that the corresponding group isomorphism
$H_* : \Aut^{br}(\E_1' \boxtimes_{\E_1}\Vec)
\xrightarrow{\sim} \Aut^{br}(\E_2' \boxtimes_{\E_2}\Vec)$
maps $\Gamma_{\E_1,F_1}$ onto  $\Gamma_{\E_2,F_2}$.

Theorem~\ref{canonical modularization} shows that the equivalence
class of $\mbox{Core}_{\E ,F}(\C)$
is, in fact, independent of the choice of a maximal Tannakian subcategory $\E\subset\C$. 
Thus, we can call it  ``the'' {\em core} of the category $\C$ 
(the reason for using quotation marks was explained in Remark \ref{whyquot}).
``The'' core of $\C$ will be denoted by $\mbox{Core}(\C)$.

\begin{remark} \label{added_by_Vitya_on_June23}
Let $\C$ be a braided fusion category and let $\E=\Rep(G)$ be a Tannakian 
subcategory in $\C$. Recall that the de-equivariantized category $\C_G$ has a structure
of braided $G$-crossed category, see \S \ref{orbi}. We know from Proposition
\ref{crosscentralizers}~(i) that the fiber category $\E'\boxtimes_\E \Vec$ identifies with 
the trivial component $(\C_G)_1$ of the grading on $\C_G$. Moreover, it is clear that
the action of $G$ on  $\E'\boxtimes_\E \Vec$ identifies with the $G$-action on $(\C_G)_1$ 
coming from the structure of braided $G$-crossed category. 
In particular, if the  Tannakian subcategory $\E =\Rep(G)\subset\C$ is maximal then
$\mbox{Core}_{\E ,F}(\C)$ can be expressed in terms of the braided $G$-crossed category $\C_G$.
\end{remark}

\subsubsection{Proof of Theorem \ref{canonical modularization}} \label{Proof of canmod}
We will prove the following statement, which is stronger than Theorem \ref{canonical modularization}.

\begin{theorem}   \label{generalization} 
Let $\C$ be a  braided fusion category and
let $\E_1,\E_2\subset\C$ be Tannakian subcategories equipped with fiber functors
$F_1:\E_1\to\Vec$ and $F_2:\E_2\to\Vec$. Suppose that $\E_1$ and $\E_2$ are
maximal among Tannakian subcategories of  $\E_1\vee\E_2$.
Then there exists a braided tensor equivalence 
\begin{equation}   \label{theequiv}
H:\E_1' \boxtimes_{\E_1}\Vec\xrightarrow{\sim}\E_2' \boxtimes_{\E_2}\Vec
\end{equation}
such that the the corresponding group isomorphism
\begin{equation}     \label{theiso}
H_* : \Aut^{br}(\E_1' \boxtimes_{\E_1}\Vec)
\xrightarrow{\sim} \Aut^{br}(\E_2' \boxtimes_{\E_2}\Vec)
\end{equation}
maps $\Gamma_{\E_1,F_1}$ onto  $\Gamma_{\E_2,F_2}$.
\end{theorem}

To construct the equivalence \eqref{theequiv}, we
choose a fiber functor $F:\E_1\cap \E_2\to\Vec$ and isomorphisms between
$F_1|_{\E_1\cap\E_2}:\E_1\cap \E_2\to\Vec$,
$F_2|_{\E_1\cap\E_2}:\E_1\cap \E_2\to\Vec$,
and $F$. Then we get epimorphisms $G_{F_i}:=\Aut (F_i)\twoheadrightarrow \Aut (F)=:G_{F}$,
$i=1,2$. The remaining part of the construction is canonical.
We will identify the pairs $(\E'_i\boxtimes_{\E_i}\Vec , \Gamma_{\E_i ,F_i})$, $i=1,2$,
with the pair
\begin{equation}  \label{E3}
((\E_1'\cap\E_2')\boxtimes_{\E_1\cap\E_2} \Vec , \Gamma_{\E_1\cap\E_2 ,F}) . 
\end{equation}

To identify \eqref{E3} with 
$(\E'_1\boxtimes_{\E_1}\Vec , \Gamma_{\E_1 ,F_1})$, we will show that
\begin{equation}  \label{E0}
(\E_1'\cap \E_2') \cap \E_1 = \E_1\cap \E_2 ,\quad (\E_1'\cap \E_2') \vee \E_1 = \E_1' .
\end{equation}
By Corollary~\ref{diamondcor}(ii), this will yield
a $\Gamma_{\E_1 ,F_1}$-equivariant equivalence
\[
(\E_1'\cap\E_2')\boxtimes_{\E_1\cap\E_2} \Vec\iso
\E_1'\boxtimes_{\E_1} \Vec
\] 
and therefore an equivalence between the pair 
$(\E'_1\boxtimes_{\E_1}\Vec , \Gamma_{\E_1 ,F_1})$ and the pair \eqref{E3}.

It remains to prove \eqref{E0}. Clearly $(\E_1'\cap \E_2')\cap\E_1=\E_1\cap\E'_2$ and by
Lemma~\ref{Dedekind}, $(\E_1'\cap \E_2') \vee \E_1 =  (\E_2'  \vee\E_1)\cap\E_1'$. Now use the following lemma.
\begin{lemma}
If $\E_1,\E_2\subset\C$ are maximal among Tannakian subcategories of  $\E_1\vee\E_2$
then
\begin{equation}  \label{E1}
\E_1'\cap \E_2=\E_1\cap \E_2' =\E_1\cap \E_2 ,
\end{equation}
\begin{equation}  \label{E2}
\E_1\vee \E_2' =\E_1'\vee \E_2 =\E_1'\vee \E_2' .
\end{equation}
\end{lemma}

\begin{proof}
The fusion subcategory $(\E_1'\cap \E_2 )\vee\E_1$ is a Tannakian subcategory
of $\E_1\vee \E_2$ containing $\E_1$, so it equals $\E_1$. This means that 
$\E_1'\cap \E_2\subset\E_1$, i.e., $\E_1'\cap \E_2 =\E_1\cap \E_2$.
Similarly, $\E_1\cap \E_2' =\E_1\cap \E_2$. Finally, \eqref{E2} follows from \eqref{E1}
and the double centralizer theorem (Theorem \ref{Muger2}(ii)).
\end{proof}

So we have proved Theorem~\ref{generalization}. 

\begin{remark}    \label{remark_for_future}
Suppose that in the situation of Theorem~\ref{generalization} one has a fusion subcategory
$\tilde\C\subset\C$ containing $\E_1\vee\E_2$. Set $\tilde\E'_i:=\E'_i\cap\tilde\C$.
Let $\tilde\Gamma_{\E_i,F_i}$ denote the image of the homomorphism
$G_F\to {\Aut}^{br}(\tilde\E'_i \boxtimes_{\E_i} \Vec)$. It is easy to see that the equivalence 
\eqref{theequiv} constructed in the proof of Theorem~\ref{generalization} induces an equivalence
$\tilde H:\tilde\E_1' \boxtimes_{\E_1}\Vec\xrightarrow{\sim}\tilde\E_2' \boxtimes_{\E_2}\Vec$.
Moreover, the corresponding group isomorphism 
$\tilde H_* : \Aut^{br}(\tilde\E_1' \boxtimes_{\E_1}\Vec)
\xrightarrow{\sim} \Aut^{br}(\tilde\E_2' \boxtimes_{\E_2}\Vec)$
maps $\tilde\Gamma_{\E_1,F_1}$ onto  $\tilde\Gamma_{\E_2,F_2}$, and one has a canonical
commutative diagram  
\begin{equation}   \label{thesquare}
\xymatrix{  \Gamma_{\E_1,F_1}  \ar^{H_*}[rr] \ar@{>>}_{\partial_1}[d] & & \Gamma_{\E_2,F_2} \ar@{>>}^{\partial_2}[d] \\
\tilde\Gamma_{\E_1,F_1} \ar[rr]^{\tilde H_*} & & \tilde\Gamma_{\E_2,F_2}}
\end{equation}
This is because $\tilde\Gamma_{\E_i,F_i}$ equals the image of 
$\Gamma_{\E_i,F_i}\subset {\Aut}^{br}(\E'_i \boxtimes_{\E_i} \Vec,\tilde\E'_i \boxtimes_{\E_i} \Vec)$
under the homomorphsim ${\Aut}^{br}(\E'_i \boxtimes_{\E_i} \Vec,\tilde\E'_i \boxtimes_{\E_i} \Vec)\to
{\Aut}^{br}(\tilde\E'_i \boxtimes_{\E_i} \Vec)$. Here 
${\Aut}^{br}(\E'_i \boxtimes_{\E_i} \Vec,\tilde\E'_i \boxtimes_{\E_i} \Vec)\subset 
{\Aut}^{br}(\E'_i \boxtimes_{\E_i} \Vec )$ is the group of isomorphism classes of braided autoequivalences of 
$\E'_i \boxtimes_{\E_i}\Vec$ preserving 
$\tilde\E'_i \boxtimes_{\E_i} \Vec\subset\E'_i \boxtimes_{\E_i}\Vec$.
\end{remark}

\subsubsection{A generalization of Theorem~\ref{canonical modularization}}
\label{modified core}
Fix a subcategory $\C^0$ of a braided fusion category $\C$.  Assume that 
the subcategories $\E_1,\, \E_2$ in Theorem~\ref{canonical modularization} are maximal 
{\em among Tannakian subcategories of $\C^0$}. 
Let $\Gamma^0_{\E_1,F_1}\subset \Aut^{br}(\E_1'\bt_{\E_1} \Vec),\, 
\Gamma^0_{\E_2,F_2} \subset \Aut^{br}(\E_2'\bt_{\E_2} \Vec)$  denote the images 
of the groups $G_{F_1},\, G_{F_2}$.  Then
\begin{equation}
\label{modcoreeq}
(\E_1'\bt_{\E_1} \Vec,\, \Gamma^0_{\E_1,F_1}) \cong (\E_2'\bt_{\E_2} \Vec,\, \Gamma^0_{\E_2,F_2}),
\end{equation}
where the meaning of $\cong$ is the same as in \eqref{coreeq}. The proof given in \S\ref{Proof of canmod}
carries over without changes. Let 
\[H_*^0 : \Aut^{br}(\E_1'\bt_{\E_1} \Vec) \xrightarrow{\sim}
\Aut^{br}(\E_2'\bt_{\E_2} \Vec)
\]
denote the corresponding group isomorphism.

Furthermore, by Proposition~\ref{lattice isomorphism}(i) the subcategory $(\E_i'\cap \C^0)\bt_{\E_i} \Vec$
is stable under the action of $G_{F_i},\, i=1,2,$ and hence the restriction of the action of $G_{F_i}$ yields 
a group homomorphism $\partial_i: \Gamma^0_{\E_i,F_i} \to \Gamma_{\E_i,F_i},\, i=1,2,$ such that
the diagram
\begin{equation}
\label{needforrefcore}
\xymatrix{  \Gamma^0_{\E_1,F_1}  \ar^{H_*^0}[rr] \ar_{\partial_1}[d] & & \Gamma^0_{\E_2,F_2} \ar^{\partial_2}[d] \\
\Gamma_{\E_1,F_1} \ar[rr]^{H_*} & &  \Gamma_{\E_2,F_2}}
\end{equation}
is commutative.

\subsection{Properties of the core}
\label{propofcore}
\subsubsection{Nondegeneracy criterion}
\begin{proposition}  \label{nondegprop}
The core of a braided fusion category $\C$ is non-degenerate if and only if $\C'$
is Tannakian. 
\end{proposition}

\begin{proof}
By definition, $\mbox{Core}(\C)=\E' \boxtimes_\E \Vec$, where
$\E\subset\C$ is a maximal Tannakian subcategory. By Corollary \ref{newc}(i), 
$\E' \boxtimes_\E \Vec$ is non-degenerate if and only if $\E''=\E$.
By Corollary~\ref{Muger-2'}, $\E''=\E$ if and only if $\E\supset\C'$,
i.e., if and only if $\C'$ is Tannakian.
\end{proof}

\begin{corollary}     \label{core of non-deg}
The core of a non-degenerate braided fusion category is non-degenerate. \hfill\qedsymbol
\end{corollary}

\subsubsection{The core is weakly anisotropic}
\begin{definition} \label{def anisotropic}
A braided fusion category $\D$ is said to be {\em anisotropic\,} if it has no Tannakian 
subcategories different from $\Vec$. We say that $\D$ is {\em weakly anisotropic\,}
if it has no Tannakian subcategories different from $\Vec$ and stable under
all braided autoequivalences of $\D$.
\end{definition}

\begin{remark} \label{not-much}
In \S\ref{5weak&strong} we will see that the class of weakly anisotropic braided fusion categories
is bigger but {\em not much bigger\,} than the class of anisotropic ones.
\end{remark}

Let $\C$ be a braided fusion category and $\E\subset\C$ a Tannakian subcategory with a
fiber functor $F: \E \to \Vec$. Recall that $G_F:=\Aut F$ acts on the fiber category
$\E'\boxtimes_\E \Vec$, the image of $G_F$ in the group $\Aut^{br}(\E'\boxtimes_\E \Vec )$
is denoted by $\Gamma_{\E ,F}$, and if $\E$ is maximal then the pair 
$(\E'\boxtimes_\E \Vec , \Gamma_{\E ,F})$ is called the core of $\E$.
The next statement is similar to Lemma \ref{maximality1}. 

\begin{proposition}   \label{maximality2}
A Tannakian subcategory $\E \subset \C$ is maximal if and only if 
$\E'\boxtimes_\E \Vec$ has no $\Gamma_{\E ,F}$-stable Tannakian subcategories
different from $\Vec$. \hfill\qedsymbol
\end{proposition}

\begin{proof}
Use Corollary~\ref{newc}(iii).
\end{proof}

\begin{corollary}  \label{coreweaklyanis}
The core of any braided fusion category is weakly anisotropic as a braided fusion category. \hfill\qedsymbol
\end{corollary}

\begin{example}  \label{Is-ex}
Let $\mI$ be an Ising fusion category (see Appendix \ref{Is}). Then the core of the center
$\Z (\mI )$ is {\em not anisotropic,\,} see Lemma \S\ref{Is-core}. 
\end{example}

\begin{remark} \label{iteratedcores}
Let $\C$ be a braided fusion category and let $\mbox{Core} (\C)$ denote its core viewed just as a braided fusion category (we forget the subgroup of $\Aut^{br}(\mbox{Core}  (\C))$. 
If $\mbox{Core} (\C)$ is not anisotropic (as in Example \ref{Is-ex}) then 
$\mbox{Core} (\mbox{Core} (\C))\not\simeq \mbox{Core} (\C)$. However Theorem~\ref{5wapt0} and
Remark~\ref{5waprem1}(ii) below show that if $\C$ is weakly anisotropic then 
$\mbox{Core} (\mbox{Core} (\C))\simeq \mbox{Core} (\C)$. Combining this with 
Corollary~\ref{coreweaklyanis} we see that
if $\C$ is {\em any\,} braided fusion category then
$\mbox{Core} (\mbox{Core} (\mbox{Core} (\C)))\simeq \mbox{Core} (\mbox{Core} (\C))$. 
\end{remark}

\subsection{Why the notion of weak anisotropy is reasonable.}   \label{whyreasonable}
\subsubsection{Relation between anisotropy and weak anisotropy} \label{5weak&strong}
\begin{definition} \label{5def2anisotropic}
A pre-metric group $(G,q)$ is said to be  {\em weakly anisotropic\,}
if it has no nonzero isotropic subgroups stable under $\Aut (G,q)$.
\end{definition}

The pointed braided category $\C (G,q)$ corresponding to a pre-metric group $(G,q)$
is (weakly) anisotropic if and only if $(G,q)$ is. 

\begin{example}
Let $G$ be a finite-dimensional vector space over $\BF_p$ and let $Q:G\to\BF_p$ be
{\em any\,} non-degenerate quadratic form. Define $q:G\to k^{\times}$ by $q(x):=\zeta^{Q(x)}$,
where $\zeta\in k$ is a primitive $p$-th root of 1. Then it is easy to prove that $(G,q)$ is weakly
anisotropic. We prove this in Appendix A (see Lemmas \ref{wap-lm1} and \ref{wap-lm4}). 
In fact, if $|G|\ne 4,9$ then the representation of $\Aut (G,q)$ in $G$ is irreducible, which clearly
implies that $(G,q)$ is weakly anisotropic. 
\end{example}

\begin{theorem}  \label{5wapt0}
A braided fusion category $\D$ is weakly anisotropic if and only if it can be represented as
$\C (V\oplus V^*,q)\bt\D_1$, where $\D_1$ is anisotropic,  $V$ is a direct sum of finitely many cyclic groups of prime orders, $q:V\oplus V^*\to k^{\times}$ is the standard quadratic form, and 
$\C (V\oplus V^*,q)$ is the pointed braided category corresponding to the metric group
$(V\oplus V^*,q)$.
\end{theorem}

The proof will be given in a subsequent article. If $\D$ is pointed Theorem~\ref{5wapt0}
amounts to a theorem on weakly anisotropic pre-metric groups, which will be proved 
in \S\ref{wapg} 
(see Lemma \ref{wap-lem} and Theorem \ref{wap-th}).

\begin{rems}     \label{5waprem1}
\begin{enumerate}
\item[(i)] The tensor product decomposition mentioned in Theorem~\ref{5wapt0} is usually
not unique.
\item[(ii)] In the situation of Theorem \ref{5wapt0} the core of $\D$ equals $(\D_1 ,\{ 1\})$, where
$\{ 1\}$ denotes the trivial subgroup of $\Aut^{br} (\D_1)$. Therefore the braided equivalence
class of $\D_1$ is uniquely determined by $\D$. The isomorphism class of $V$ is also
uniquely determined by $\D$ (because it is uniquely determined by $|V|$ and 
$|V|^2=\FPdim (\D )/\FPdim (\D_1)$).
\item[(iii)]  Combining Theorem~\ref{5wapt0} with the classification of anisotropic pre-metric groups
from \S\ref{anisotropic-odd}, \S\ref{2groups}, and \S\ref{degensubsec} of Appendix A one gets a 
{\em complete classification\,} of pointed weakly anisotropic  braided fusion categories.
\end{enumerate}
\end{rems}

\subsubsection{The 3 types of weakly anisotropic braided fusion categories}  \label{5the3types}
The goal of this subsection is to define and motivate the adjective ``ordinary", which appears in
Theorem \ref{5ordmaintheor} below.

\begin{lemma}  \label{5l1}
Let $\C$ be a weekly anisotropic braided fusion category and $\C_1\subset\C$
a fusion subcategory stable under all braided equivalences of $\C$. Then
$\C_1$ is weakly anisotropic. \hfill\qedsymbol
\end{lemma}

\begin{lemma}  \label{5l2}
A weakly anisotropic symmetric fusion category $\C$ is braided equivalent to
$\Vec$ or to $s\Vec$.
\end{lemma}

\begin{proof}
By Deligne's theorem (see \S\ref{Deligne}), $\C =\Rep(G,z)$ for some finite group $G$ and
some central element $z\in G$ of order $1$ or $2$. Representations of $G$ on which $z$
acts trivially form a Tannakian subcategory $\widetilde{\C}\subset\C$ stable under all braided
auto-equivalences of $\C$. So $\widetilde{\C}=\Vec$, i.e., $G$ is generated by $z$.
\end{proof}

\begin{proposition}  \label{5modified}
Let $\C$ be a weakly anisotropic braided fusion category. Then
\begin{enumerate}
\item[(i)] $\C'\subset\C_{pt}\cap (\C_{pt})'$;
\item[(ii)] $\C_{pt}\cap (\C_{pt})'$ is braided equivalent to $\Vec$ or $s\Vec$.
\end{enumerate}
\end{proposition}  

\begin{proof}
Both $\C_{pt}\cap (\C_{pt})'$  and $\C'$ are symmetric fusion subcategories
of $\C$ stable under all braided equivalences of $\C$. Now use Lemmas \ref{5l1} and \ref{5l2}.
\end{proof}

\begin{corollary}  \label{53typescor}
A weakly anisotropic braided fusion category has one of the following types:
\begin{equation}  \label{5typeO}
\C_{pt}\cap (\C_{pt})'=\C'=\Vec ;
\end{equation}
\begin{equation}  \label{5typeS}
\C_{pt}\cap (\C_{pt})'=\C'=s\Vec ;
\end{equation}
\begin{equation}  \label{5typeI}
\C'=\Vec ,\quad \C_{pt}\cap (\C_{pt})'=s\Vec .
\end{equation}
\end{corollary}  

\begin{definition}  \label{53typesdef}
A weakly anisotropic braided fusion category $\C$ is said to be {\em ordinary\,} 
(or  {\em of ordinary type\,}) if \eqref{5typeO} holds. We say that $\C$ has {\em super\,} type
if \eqref{5typeS} holds. We say that $\C$ has {\em Ising\,} type if  \eqref{5typeI} holds. 
\end{definition}

\begin{remark}     \label{5various}
An Ising braided category (see Appendix \ref{Is}) has Ising type.
\end{remark}

\subsubsection{The canonical decomposition}    \label{50ordsubsec}
\begin{definition}   \label{50simplicitydef}
A braided fusion category $\C$ is said to be {\em simple\,} if it is not pointed, 
not symmetric, and has no fusion subcategories except $\C$ and $\Vec$. 
\end{definition}

\begin{theorem}     \label{5ordmaintheor}
The following properties of a braided fusion category $\C$ are equivalent:
\begin{enumerate}
\item[(a)] $\C$ is weakly anisotropic and ordinary;
\item[(b)] there exists a decomposition
\begin{equation}     \label{50decomposition}
\C =\C_0\bt\C_1\bt\ldots\bt\C_n
\end{equation}
in which $\C_1$, \ldots, $\C_n\subset\C$ are simple fusion subcategories
and $\C_0$ is pointed, weakly anisotropic, and non-degenerate.
\end{enumerate}
In this case a decomposition \eqref{50decomposition} is unique.
$\C$ is anisotropic if and only if $\C_0$ is.
\end{theorem}

The proof will be given in a subsequent article.

\begin{rems}   \label{5ordsituation}
\begin{enumerate}
\item[(i)] Theorem~\ref{5ordmaintheor} and Remark~\ref{5waprem1}(iii)
reduce classifying ordinary weakly anisotropic braided fusion categories to classifying
simple ones. Many examples of simple braided fusion categories can be constructed using
quantum groups at roots of unity.
\item[(ii)] Theorem \ref{5ordmaintheor} has an analog for  weakly anisotropic
braided fusion categories of super type and hopefully, for those of Ising type.
\end{enumerate}
\end{rems}

\section{Gauss sums}
\label{sectiongauss}
This section is devoted to an exposition of some invariants of braided fusion categories
which generalize the classical notion of Gauss sums. The Gauss sums attached to
pre-metric groups are reviewed in \S \ref{gauss2}. In \S \ref{gauss} we recall the basic
definitions and properties of Gauss sums and central charge for pre-modular categories
following mainly \cite{BK}. In the case of pointed pre-modular categories these Gauss sums
specialize to the classical Gauss sums from \S \ref{gauss2}.

 In \S \ref{TheGauss} we prove Theorem \ref{tau after mod} which relates the
Gauss sums of a modular category with the Gauss sums of its fiber category. This is
one of the main technical results in this section.  

The value of Gauss sums depends on two structures on a fusion category: the braided structure
and the spherical structure. In \S \ref{elementxc} we restrict ourselves to the case of weakly integral
braided fusion categories and define Gauss-Frobenius-Perron sums for them. These
invariants are closely related with Gauss sums from \S \ref{gauss} but their advantage is
that no spherical structure is needed for definition. Gauss-Frobenius-Perron sums
will be used in the proof of Lemma~\ref{Is-core1} and especially in the study of nilpotent braided fusion categories.

\subsection{Classical Gauss sums and isotropic subgroups}\label{gauss2}
The {\em Gauss sums} of a pre-metric group $(G, q)$ are defined by the classical formula
\begin{equation}\label{goodold}
\tau^\pm(G,\, q) =\sum_{a\in G}\, q(a)^{\pm 1}.
\end{equation}
Let us recall the basic properties of Gauss sums. Clearly 
$$\tau^-(G,\, q)=\overline{\tau^+(G,\, q)}.$$
It is also clear that if $(G,\, q)$ is the
orthogonal direct sum of $(G_1,\, q_1)$ and $(G_2,\, q_2)$ then 
$\tau^\pm(G,\, q)=\tau^\pm(G_1,\, q_1)\tau^\pm(G_2,\, q_2)$.

Let $(G, q)$ be a pre-metric group. A subgroup $H\subset G$ is said to be 
 {\em isotropic} if $q|_H =1$. In this case $H\subset H^\perp$, where $H^\perp$ is the
 orthogonal complement of $H$ with respect to the bicharacter $b:G\times G\to k^{\times}$  
 associated to $q$. Moreover, the restriction of $q$ to $H^\perp$ is the pullback of a
 quadratic form $\tilde q:H^\perp/H\to k^{\times}$.
 
 \begin{proposition} \label{subquot} 
 In this situation $\tau^\pm(G,\, q)=\tau^\pm(H^\perp/H,\, \tilde q)\cdot |H|$.
\end{proposition}

 \begin{proof}
 Let $g\in G$. Then $\sum\limits_{h\in H}q(gh)=q(g)\sum\limits_{h\in H}b(g,h)$ equals 0 if 
 $g\notin H^\perp$ and $q(g)\cdot |H|$ if $g\in H^\perp$.
 \end{proof}

An isotropic subgroup $H$ of a pre-metric group $G$ is said to be {\em Lagrangian} if $H^\perp = H$.

\begin{corollary}   \label{absval} 
 Let $(G,\, q)$ be a pre-metric group and $H\subset G$ a Lagrangian subgroup.
 Then $\tau^\pm(G,\, q)=|H|$. \hfill\qedsymbol
 \end{corollary}

\begin{corollary}   \label{hyp} 
 Let $(G,\, q)$ be a metric group. Then $\tau^+(G,\, q)\tau^-(G,\, q)=|G|$. 
\end{corollary}

\begin{proof}
Consider the orthogonal direct sum of $(G,\, q)$ and $(G,\, q^{-1})$. The diagonally embedded
$G\subset G\times G=G\oplus G$ is a Lagrangian subgroup. Now apply Corol\-lary~\ref{absval}.
 \end{proof}

A pre-metric group is said to be {\em hyperbolic\,} if it has a Lagrangian subgroup. 
If $G$ is a metric group and $H\subset G$ is a Lagrangian subgroup then $|H|=|G|^{1/2}$, so
the order of a hyperbolic metric group is a square.

Corollary \ref{hyp} implies that if a pre-metric group $(G,\, q)$ is hyperbolic then 
$\tau^\pm(G,\, q)$ is a positive integer. The following partially converse
statement is well known. 

\begin{proposition}  \label{basis} 
Let $(G,\, q)$ be a metric group such that $\tau^+(G,\, q)$ is a positive
integer and $|G|$ is a power of a prime $p$. Then $(G,\, q)$ is hyperbolic.
\end{proposition}

A proof of the proposition is given in \S\ref{A-2} of Appendix A.

\begin{rems}
\begin{enumerate}
\item[(i)] The following example explains why in Proposition \ref{basis} $G$ is
required to be a $p$-group. Take two different primes $p_1$ and $p_2$. 
There exist metric groups $(G_1,q_1)$ and  $(G_2,q_2)$ such that $|G_i|$ is a power of 
$p_i$ and $\tau^+(G_i,\, q_i)$ is a {\em negative\,} integer (see Appendix A, Proposition \ref{dimtwo}).
Let $(G,\, q)$ be the orthogonal direct sum of $(G_1,\, q_1)$ and $(G_2,\, q_2)$.
Then $\tau^+(G,\, q)$ is a positive integer, but since $G_1$ and $G_2$ have no
Lagrangian subgroups neither does~$G$.

\item[(ii)]  Proposition \ref{basis} remains valid for pre-metric $p$-groups. 
This generalization easily follows from Proposition \ref{basis} itself.
\end{enumerate}
\end{rems}
 
\subsection{Gauss sums and central charge in pre-modular categories}\label{gauss}

Let $\C$ be a pre-modular category. The
{\em Gauss sums} of  $\C$ are defined by
\begin{equation}
\label{Gauss sums}
\tau^{\pm}(\C) = \sum_{X\in \O(\C)}\, \theta_X^{\pm 1}  d(X)^2.
\end{equation}

\begin{example}   \label{mod&class}
Let $\C$ be the pointed pre-modular category $\C(G,q)$ from \S\ref{thetaq}. Then
$\O(\C)=G$, $d(X)=1$ for all $X\in\O(\C)$, and $\theta_X=q(X)$ by \eqref{tw}.
So $\tau^{\pm}(\C)$ equals the classical Gauss sum $\tau^\pm(G,\, q)$ defined by \eqref{goodold}.

If $\C$ is the general pointed pre-modular category $\C(G,q,\chi)$ from \S\ref{genpointed} then 
$\tau^{\pm}(\C)=\tau^\pm(G,\, \theta )$, where $\theta (g)=q(g)\chi (g)$.
\end{example}

Below we summarize some basic properties of twists and Gauss sums
(see e.g., \cite[Section 3.1]{BK} for proofs).

\begin{proposition} \label{tauprop}
(i) For any pre-modular categories $\C_1$ and $\C_2$ 
\begin{equation}
\label{multiplicativity of tau}
\tau^\pm(\C_1 \bt \C_2)  = \tau^\pm(\C_1) \tau^\pm(\C_2).
\end{equation}

(ii) Each $\theta_X, \, X\in \O(\C),$ is a root of unity; in particular $\tau^{\pm}(\C)\in k_{cyc}$.

(iii) For a modular category $\C$ we have 
\begin{equation}
\label{mod tau =dimC}
\tau^+(\C)\tau^-(\C) =\dim(\C)\; \; \mbox{and}\; \; \tau^-(\C)=\overline{\tau^+(\C)}.
\end{equation}

(iv) For a modular category $\C$ the element $\frac{\tau^+(\C)}{\tau^-(\C)}\in k$ is a root of unity.
\end{proposition}

Statements (ii) and (iv)  were
obtained in the works \cite{AM} by  Anderson and Moore and  \cite{V} by Vafa.

Recall that $\dim(\C)$ is a totally positive element of the maximal cyclotomic subfield
$\kcyc\subset k$. If we fix an embedding $\kcyc\hookrightarrow\mathbb C$ we can talk about the {\em multiplicative central charge} 
\begin{equation}
\label{central charge}
\xi(\C) = \frac{\tau^+(\C)}{\sqrt{\dim(\C)}},
\end{equation}
where $\sqrt{\dim(\C)}$ is the {\em positive\,} square root. One has
$\xi(\C)^2=\tau^+(\C)/\tau^-(\C)$, so 
$\xi(\C)$ is a root of unity by Proposition \ref{tauprop}(iv).

\begin{rems}   \label{chargerems}
\begin{enumerate}
\item[(i)] If $\dim (\C)$ is a square of an integer
then $\xi(\C)$ does not depend on the embedding $\kcyc\hookrightarrow\mathbb C$. 
In particular, this happens if $\FPdim (\C)$ is a square of an integer because 
in this case $\dim (\C )= \FPdim (\C)$ by Proposition \ref{FPint1}(i).
\item[(ii)]  Bakalov and Kirillov \cite{BK} use a different convention: according to
 \cite[Definition 5.7.9]{BK} and \cite[Theorem 5.7.11]{BK}, their multiplicative central charge
  equals $\xi(\C)^2=\tau^+(\C)/\tau^-(\C)$. As explained in loc. cit., the quantity
 $\tau^+(\C)/\tau^-(\C)$ measures the ``anomaly" of the modular functor corresponding to $\C$.
\item[(iii)] If $k=\mathbb C$ many authors use the ``Virasoro central charge"
$c\in{\mathbb Q}/8{\mathbb Z}$, which is related to the multiplicative central charge $\xi$
by $\xi=e^{\pi ic/4}$.
\end{enumerate}
\end{rems}

\begin{example}\label{charge1}
The center $\Z(\C)$ of any spherical fusion category $\C$  is a modular category
of dimension $\dim (\C)^2$ (see \S\ref{center}). According to \cite[Theorem 1.2]{Mu4},
$\tau^+(\Z(\C))=\tau^-(\Z(\C))=\dim(\C)$. Since $\dim(\C)\in\kRcyc$ is totally positive
(see \S\ref{squared}) this implies that the multiplicative central charge of $\Z(\C)$ 
corresponding to any embedding $\kcyc\hookrightarrow\mathbb C$ equals $1$.
\end{example}

\subsection{The Gauss sums of a fiber category and invariance of the central charge}
\label{TheGauss}

If $G$ is a pre-metric group
and $H\subset G$ is an isotropic subgroup then 
there is a simple relation between the Gauss sum
of $G$ and that  of $H^{\perp}/H$ (see Proposition \ref{subquot}).
The goal of this subsection is to prove a similar relation
for Gauss sums of pre-modular categories 
(see Theorem~\ref{tau after mod} below). This relation 
will be crucial in the sequel.
Recall that the Gauss sums and central charge 
of a pre-modular category were introduced in \S\ref{gauss}.

The next Lemma is known, see \cite[Lemma 3.1.5]{BK}. We include its proof
for the reader's convenience.
\begin{lemma}
\label{december lemma}
Let $\C$ be a pre-modular category. Then for any $Y\in \O(\C)$ we have:
\begin{equation}
\label{december formula}
\sum_{X\in \O(\C)}\, \theta_X d(X) s_{XY} =
d(Y) \theta_Y^{-1} \tau^+(\C). 
\end{equation}
\end{lemma}
\begin{proof}
Using formula~\eqref{sij} we compute
\begin{eqnarray*}
\sum_{X\in \O(\C)}\, \theta_X d(X) s_{XY} 
&=& \theta_Y^{-1} \sum_{X\in \O(\C)}\, d(X) \sum_{Z\in \O(\C)}\, N_{XY}^Z \theta_Z d(Z) \\
&=& \theta_Y^{-1}  \sum_{Z\in \O(\C)}\,  
\left( \sum_{X\in \O(\C)}\, N_{ZY^*}^X d(X) \right) \theta_Z d(Z) \\
&=& \theta_Y^{-1}  d(Y) \sum_{Z\in \O(\C)}\,  \theta_Z d(Z)^2 = d(Y) \theta_Y^{-1} \tau^+(\C),
\end{eqnarray*}
as required.
\end{proof}

In the next Proposition we extend the multiplicative property 
\eqref{multiplicativity of tau} of Gauss sums.

\begin{proposition}
\label{stronger  multiplicativity of tau}
Let $\D$ be a fusion subcategory of a pre-modular category $\C$ and let $\D'$
be its centralizer in $\C$. Then 
\begin{equation}
\label{gen mult}
\tau^\pm(\C) \tau^\mp(\D)  = \dim(\D) \tau^\pm(\D'). 
\end{equation}
\end{proposition}
\begin{proof}
Multiplying formula \eqref{december formula} in Lemma~\ref{december lemma}
by $d(Y)$ and taking sum over $Y\in \O(\D)$ we have:
\begin{eqnarray*}
\tau^+(\C) \tau^-(\D) 
&=& \sum_{Y\in \O(\D)}\, \sum_{X\in \O(\C)}\, \theta_X d(Y) d(X) s_{XY} \\
&=& \sum_{X\in \O(\C)}\, \left( \sum_{Y\in \O(\D)}\,  s_{XY} d(Y) \right) \theta_X d(X) \\
&=& \sum_{X\in \O(\D')}\, \theta_X d(X)^2 \dim(\D)  =  \dim(\D) \tau^+(\D'),
\end{eqnarray*}
where in  the third equality we used Proposition~\ref{AB criterion}. 
Applying the complex  conjugation we obtain the second equality.
\end{proof}

\begin{remark}
Our proof of Proposition~\ref{stronger  multiplicativity of tau}
is similar to the proof of Proposition~\ref{subquot}. In the case that $\C$ is pointed
the two proofs are identical.
\end{remark}

\begin{definition}    \label{isotrdef}
A fusion subcategory $\D$ of a pre-modular category $\C$ is said to be
{\em isotropic} if $\theta_X = \id_X$ for all $X\in \D$.
\end{definition}

\begin{remark}
By \eqref{balancing ax}, {\em any isotropic subcategory 
is symmetric.} On the other hand, if $\D$ is symmetric then $\theta$ is a tensor automorphism of
$\id_{\D}$ and $\theta^2=\id$, so {\em any symmetric subcategory of odd dimension
is isotropic.}
\end{remark}

\begin{corollary}
\label{orto-isotropic}
Let $\C$ be a modular category and let $\E$ be an isotropic
subcategory of $\C$.  
Then $\tau^\pm(\E') =\tau^\pm(\C)$.
\end{corollary}
\begin{proof}
Take $\D =\E$ in \eqref{gen mult} and observe that
$\tau^\pm(\E)=\dim(\E)$ since $\E$ is isotropic.
\end{proof}

Now let $\C$ be a pre-modular category and let $\E\subset\C$ be 
isotropic and Tannakian; in other words, $\E$ is Tannakian and the
spherical structure on $\C$ is the standard one. Let $\E'\boxtimes_{\E}\Vec$ 
be the fiber category, i.e., the category of $A$-modules in $\E'$,
where $A$ is the regular algebra in $\E$. As explained in \S\ref{deeqfus}-\ref{Thetensorfunc}, 
$\E'\boxtimes_{\E}\Vec$ is a braided fusion category, and the ``free module"
functor $F: \E' \to \E'\boxtimes_{\E}\Vec$ defined by $F(X):=A\otimes X$ 
is braided. Moreover, according to \cite{Br, Mu1, KiO},
$\E'\boxtimes_{\E}\Vec$ {\em has a unique spherical structure such that
$F$ is spherical.}
This means that the twists $\phi_Y$, $Y\in \O(\E'\boxtimes_{\E}\Vec)$, agree with the twists 
$\theta_X, X\in\O(\C)$, i.e.,
$\phi_Y =\theta_X \mbox{ whenever } \Hom_\C(X, Y)\neq 0$.
Here we abuse notation and view $Y$ as an object of $\C$ via the forgetful functor.

Thus $\E'\boxtimes_{\E}\Vec$ is a pre-modular category.
By Corollary \ref{newc}(i), if $\C$ is modular then so is $\E'\boxtimes_{\E}\Vec$. In this case
$\E'\boxtimes_{\E}\Vec$ is what Brugui\`{e}res \cite{Br} and  M\"uger \cite{Mu2}
call the {\em modularization\,} of~$\E'$. 

Let $d$ and $\bar{d}$ denote the dimension functions in $\C$ and $\E'\boxtimes_{\E}\Vec$,
respectively. For any object $X$ in $\E'\boxtimes_{\E}\Vec$ one has
\begin{equation}
\label{dims}
\bar{d}(X) = \frac{d(X)}{d(A)}
\end{equation}
and, hence,
\begin{equation}
\label{dims1}
\dim(\E'\boxtimes_{\E}\Vec)=\frac{\dim(\C)}{\dim(\E)},
\end{equation}
cf.\ \cite[Theorem 3.5]{KiO}, \cite[Proposition 3.7]{Br} (note that $d(A)=\dim(\E)\neq 0$).

\begin{theorem}
\label{tau after mod}
Let $\C$ be a pre-modular category and let $\E$ be an isotropic Tannakian
subcategory of $\C$.  
Then $\tau^\pm (\C)  = \tau^\pm (\E'\boxtimes_{\E}\Vec)\cdot\dim(\E)$.
\end{theorem}

\begin{remark}
The theorem implies that the multiplicative central charges of $\C$ and $\E'\boxtimes_{\E}\Vec$ 
corresponding to any embedding $\kcyc\hookrightarrow\mathbb C$ are equal.
\end{remark}
\begin{proof}
We compute the Gauss sums of  $\E'\boxtimes_{\E}\Vec$: 
\begin{eqnarray*}
\lefteqn{ \dim(\E) \tau^\pm(\E'\boxtimes_{\E}\Vec) = d(A) \sum_{Y\in\O(\E'\boxtimes_{\E}\Vec)}\, \phi^{\pm 1}_Y \bar{d}(Y)^2 =  }\\
&=&  \sum_{Y\in\O(\E'\boxtimes_{\E}\Vec)}\, \phi^{\pm 1}_Y  d(Y) \bar{d}(Y) \\
&=& \sum_{Y\in\O(\E'\boxtimes_{\E}\Vec)}\, \phi^{\pm 1}_Y  \left( \sum_{X\in\O(\E')}\, \dim_k\Hom_{\E'}(X,\, Y) d(X) 
    \right) \bar{d}(Y) \\
&=&  \sum_{X\in\O(\E')}\, \theta^{\pm 1}_X d(X)  \left( \sum_{Y\in\O(\E'\boxtimes_{\E}\Vec)}\, 
      \dim_k\Hom_{\E'}(X,\, Y) \bar{d}(Y) \right) \\
&=&  \sum_{X\in\O(\E')}\, \theta^{\pm 1}_X d(X)  \left( \sum_{Y\in\O(\E'\boxtimes_{\E}\Vec)}\, 
      \dim_k\Hom_{\E'\boxtimes_{\E}\Vec}(A \ot X,\, Y) \bar{d}(Y) \right) \\
&=&  \sum_{X\in\O(\E')}\, \theta^{\pm 1}_X d(X)  \bar{d}(F(X)) = \tau^\pm(\E') = \tau^\pm(\C),
\end{eqnarray*}
where we used the relation \eqref{dims}, the standard adjunction for $F$, 
the fact that $F$ is compatible with the spherical structures, and  Corollary~\ref{orto-isotropic}.
\end{proof}

\begin{corollary}
\label{zariad}
Let $\E \subset \C$ be an isotropic Tannakian subcategory of a
pre-modular category $\C$.  Then
\begin{equation}
\xi(\E'\boxtimes_{\E}\Vec) = \xi(\C),
\end{equation}
i.e., the central charge does not change after passing to the fiber category.
\end{corollary}
\begin{proof}
Note that $\dim(\E'~\boxtimes_{\E}~\Vec) =
\dim(\C)/\dim(\E)^2$ by  equation \eqref{dims1} 
and Theorem~\ref{Muger2}(i) and combine this with
Theorem \ref{tau after mod}.
\end{proof}
 
\subsection{Some invariants of weakly integral non-degenerate braided categories}\label{elementxc}
Let $\C$ be a weakly integral non-degenerate braided fusion category. 
We will associate to it an element $\x_{\C}\in\q$ and the ``Gauss-Frobenius-Perron
sums" $\T^{\pm}(\C )\in k^{\times}$. 

\subsubsection{The element $\x_{\C}\in\q$}
As usual, let $\O (\C )$ denote the set of isomorphism classes of the simple objects of $\C$.
Let $\Gamma\subset\q$ be the image of the grading $\deg :\O (\C )\to\q$ from Proposition \ref{FPint1}(iii).
Let $\C^{int}\subset C$ be the trivial component of this grading, i.e., the maximal
integral subcategory of $\C$. Set $\A:=(\C^{int})'$. By definition, $\C^{int}\supset\C_{ad}$, so
$\A\subset\C_{pt}$ by Corollary~\ref{adpt}. Since $\C_{pt}\subset\C^{int}$ we see that 
$\A\subset\A'$, i.e., $\A$ is symmetric. For every simple $a\in\A$ the braiding
$c_{a,a}:a\ot a\iso a\ot a$ acts as multiplication by some $\chi (a)\in k^{\times}$, and
since $\A$ is symmetric the function
\begin{equation} \label{thecharacter}
\chi :\O (\A )\to k^{\times}
\end{equation}
is a character with values $\pm 1$.

On the other hand, the perfect pairing \eqref{thepairing}
provides an isomorphism $\alpha: \O(\C_{pt})\iso\Hom (U_{\C},k^{\times})$, 
and $\O(\A )\subset\O(\C_{pt})$ is the preimage of 
$\Hom (\Gamma ,k^{\times})\subset\Hom (U_{\C},k^{\times})$. So the character
\eqref{thecharacter} defines an element  of $\Gamma\subset\q$, which
will be denoted by $\x_{\C}$. In other words, $\x_{\C}\in\Gamma$ is defined by the equality 
\begin{equation} \label{defequality}
\langle a, \x_{\C}\rangle =\chi (a) \;\mbox{ for all }\; a\in\O (\A ),
\end{equation}
where $\langle , \rangle$ is the pairing between $\O (\A )$ and $\Gamma$ induced by 
the pairing \eqref{thepairing}. Note that since
$\Gamma\subset\q$ we can consider $\x_{\C}$ as an element of $\q$. 
Clearly
\begin{equation} 
\x_{\C_1\boxtimes \C_2}=\x_{\C_1}\cdot\x_{\C_2}.
\end{equation}

\begin{example}  \label{xIs}
Let $\C$ be an Ising braided category (see Appendix \ref{Is}). Then $\x_{\C}\ne 1$.
To see this, recall that $\C^{int}=\C_{pt}=\C_{ad}$ (see \S\ref{fusionIsing}). 
So the subcategory $\A:=(\C^{int})'$ equals $(\C_{ad})'=\C_{pt}=\C_{ad}$.
By Lemma \ref{adsuper}, $\C_{ad}\simeq s\Vec$, so the character \eqref{thecharacter} 
is nontrivial. Therefore the element $\x_{\C}\in\Gamma\subset\q$ is not equal to $1$.
In fact, $\Gamma$ has only one element not equal to $1$, namely, the image of
$2\in\BQ^{\times}_{>0}$.
\end{example}

\begin{proposition}  \label{xcenter}
If $\C$ is the center of a fusion category then $\x_{\C}=1$.
\end{proposition}

\begin{proof}
Suppose that $\C$ is the center of a fusion category $\M$. By Proposition \ref{centerweaklyintegral},
$\M$ is weakly integral.
Therefore Proposition \ref{FPint1}(iii) applies to $\M$, and we get a grading
$\deg :\O (\M )\to\q$. For each grading $\delta :\O (\M )\to k^{\times}$ which can be factored as
$\O (\M )\to \q\to k^{\times}$ we have the following object $c_{\delta}\in \Z (\M )=\C$: 
as an object of $\M$, $c_{\delta}$ equals $\be$, and for each simple $m\in\M$ the isomorphism
$m\ot\be\iso\be\ot m$ equals $\delta (m)$. It is easy to see that the fusion category
$\A:=(\C^{int})'$ is generated by the $c_{\delta}$'s. On the other hand, the braiding
$c_{\delta}\ot c_{\delta}\iso c_{\delta}\ot c_{\delta}$ equals $1$ for each $\delta$.
\end{proof}

Now let us assume that $\C$ contains a Tannakian subcategory $\E$. The fiber category
$\tilde \C=\E'\boxtimes_\E \Vec$ is also weakly integral. Indeed, $\FPdim(\tilde \C)$ is an
algebraic integer (see \S \ref{fpd}), and Corollary \ref{subwi} and  Proposition \ref{equidimprop} 
imply that  $\FPdim(\tilde \C)\in \BQ$. Hence the element
$x_{\tilde \C}\in \q$ is defined.

\begin{proposition} \label{xtildec}
We have $x_\C=x_{\tilde \C}$.
\end{proposition}

\begin{proof} 
Let $\A:= (\C^{int})'\subset\C$. Let $\chi : \O (\A)\to k^\times$ have the same meaning as in 
\eqref{thecharacter}. Similarly, one has $\tilde \A:=(\tilde \C^{int})'$ and  
$\tilde\chi : \O (\tilde \A)\to k^\times$.
Let $F: \E'\to \tilde \C$ be the canonical braided functor. 
By Corollary \ref{symint}, $\E \subset \C^{int}$, hence $\A \subset \E'$. 
Since $F( \C^{int})=\tilde \C^{int}$ one has $F(\A)\subset \tilde \A$. Clearly 
$\chi (a)=\tilde \chi (F(a))$. 

By the definition of $x_{\tilde \C}$, there
exists an object $\tilde X\in \O(\tilde \C)$ such that the image of $\FPdim(\tilde X)^2$ in $\q$ equals
$x_{\tilde \C}$ and the square of the braiding $c_{\tilde X,a}c_{a,\tilde X}$ equals $\tilde \chi (a)$ for
any $a\in \O(\tilde \A)$. Let $X\in \O(\C)$ be a simple object such that $F(X)$ contains $\tilde X$
as a direct summand. Then the image of $\FPdim(X)^2$ in $\q$ equals $x_{\tilde \C}$ and the square of the braiding $c_{X,a}c_{a,X}$ equals $\tilde \chi (F(a))=\chi(a)$ for any $a\in \O(\A)$. By definition,
this means that $x_{\tilde \C}=x_\C$.
\end{proof}

\subsubsection{A formula for the Gauss sum of $\C$}   \label{Gaussformula}
Choose a spherical structure $s$ on $\C$ (it exists by Corollary \ref{FPint3}).
Then $(\C ,s)$ is a modular category, so one can consider the Gauss sums 
\[
\tau^{\pm}(\C ,s):= \sum_{X\in \O(\C)}\, \theta_{X,s}^{\pm 1}  d_s(X)^2,
\]
where $\theta_{X,s}$ is the twist corresponding to $s$ and $d_s(X)$ is the quantum dimension.

For $\x\in\q$ define $\O_{\x}(\C)\subset\O (\C )$ to be the preimage of $\x$ 
with respect to $\deg :\O (\C )\to\q$.

\begin{proposition}  \label{onlyone}
If $s$ is chosen so that $d_s(X)=1$ for all invertible $X\in\C$ then
$\tau^{\pm}(\C ,s)= \sum\limits_{X\in \O_{\x_{\C}}(\C)}\, \theta_{X,s}^{\pm 1}  d_s(X)^2$.
\end{proposition}

\begin{proof}
It suffices to show that
\begin{equation}   \label{vanishing}
\sum\limits_{X\in \O_{\x}(\C)}\, \theta_{X,s}^{\pm 1}  d_s(X)^2=0 \;\mbox{ for }\; \x\ne\x_{\C}.
\end{equation}
By \eqref{balancing ax}, for each simple $a\in\A$ and $X\in\C$ one has 
$\theta_{a\ot X}/(\theta_a\theta_X)=c_{a,X}c_{X,a}$, where $c$ is the braiding.
By the definition of the pairing \eqref{thepairing}, for each simple $a\in\A$ one has 
$c_{a,X}c_{X,a}=\langle a, \deg X \rangle$, so 
$\theta_{a\ot X}=\theta_a\cdot\theta_X\cdot\langle a, \deg X \rangle$.
Since $\A$ is a pointed pre-modular category in which the quantum dimensions
of all simple objects equal $1$ the twists $\theta_a$ are given by formula \eqref{tw}.
In other words, $\theta_a=\chi (a)=\chi (a)^{-1}$, where $\chi$ is the character \eqref{thecharacter}.
By \eqref{defequality}, we get $\theta_a=\langle a,\x_{\C}^{-1} \rangle$ and therefore
$\theta_{a\ot X}=\theta_X\cdot\langle a,\x_{\C}^{-1} \cdot\deg X \rangle$. Combining
this with the equality  $d_s(a\ot X)^2=d_s(X)^2$ we get~\eqref{vanishing}.
\end{proof}

\subsubsection{The Gauss-Frobenius-Perron sums $\T^{\pm}(\C )$}
We keep the notation of \S\ref{Gaussformula}.
Represent $\x_{\C}\in\q$ as the class of a square-free positive number $n_{\C}\in\BZ$. 
Let $n_{\C}^{-1/2}\in\mathbb R$ denote the positive square root of $n_{\C}^{-1}$.
By definition, $\O_{\x_{\C}}(\C)$ consists of those $X\in \O (\C )$ for which the number
$n_{\C}^{-1/2}\FPdim (X)\in\mathbb R$ is an in\-te\-ger. Consider these integers as elements
of $k$ and define the {\em Gauss-Frobenius-Perron sums $\T^{\pm}(\C )$\,} by
\begin{equation}  \label{GFP}
\T^{\pm}(\C )= \sum\limits_{X\in \O_{\x_{\C}}(\C)}\, 
(n_{\C}^{-1/2}\FPdim (X))\cdot\theta_{X,s}^{\pm 1}  d_s(X),
\end{equation}  
where $s$ is any spherical structure on $\C$. 

\begin{rems}  
\begin{enumerate}
\item[(i)] It is easy to see that the product $\theta_{X,s}^{\pm 1}  d_s(X)$ does not depend
on $s$, so $\T^{\pm}(\C )$ is well-defined. In fact, $\theta_{X,s}^{-1}  d_s(X)=\Tr(u_X)$ and 
$\theta_{X,s}  d_s(X)=\Tr((u_X^*)^{-1})$ where $u_X$ is defined by \eqref{Drinfeld iso}.
\item[(ii)] Choose an embedding $\varphi :\kRcyc\to\mathbb{R}$. By
Corollary \ref{FPint3}, $\C$ has a unique $\varphi$-positive spherical structure $s_{\varphi}$.
One has $d_{s_{\varphi}}(X)=\varphi^{-1}(\FPdim (X))$. Combining this with
Proposition \ref{onlyone} we see that 
\begin{equation}  \label{GGFP}
\T^{\pm}(\C )=\varphi^{-1}(n_{\C}^{-1/2})\cdot\tau^{\pm}(\C ,s_{\varphi}).
\end{equation}
In particular, if $\C$ is integral then $\T^{\pm}(\C )$ equals the Gauss sum $\tau^{\pm}$ 
of $\C$ equipped with the positive spherical structure.
\item[(iii)] Since $\tau^{\pm}(\C ,s_{\varphi})\in k_{\cyc}^{\times}$ and $n_{\C}^{1/2}$ is cyclotomic
formula \eqref{GGFP} implies that $\T^{\pm}(\C )\in k_{\cyc}^{\times}$. It also implies that
$\T^-(\C )=\overline{\T^+(\C )}$. Moreover, 
$\T^{\pm}(\C )$ is an algebraic integer; this follows from \eqref{GFP} because the
numbers  $n_{\C}^{-1/2}\FPdim (X)$, $d_s(X)$, and $\theta_{X,s}^{\pm 1}$ are algebraic integers.
\end{enumerate}
\end{rems}  

\begin{proposition}  \label{propGFG}
Let $\C$ be a weakly integral non-degenerate braided category.
\begin{enumerate}
\item[(i)] $\T^-(\C)=\overline{\T^+(\C)}$.
\item[(ii)] If $\C=\C_1\bt\C_2$ then $\T^{\pm}(\C )=a\T^{\pm}(\C _1)\cdot\T^{\pm}(\C _2)$ for some positive $a\in\BQ$.
\item[(iii)] If $\C$ is the center of a fusion category  then $\T^{\pm}(\C )$ is a positive integer.
\item[(iv)] if $\E\subset\C$ is a Tannakian subcategory then
$\T^{\pm}(\C )=\T^{\pm}(\E'\bt_{\E}\Vec )\cdot\dim(\E)$.
\end{enumerate}
\end{proposition}

\begin{proof}
Statements (i) and (ii)  follow from \eqref{GGFP}, \eqref{mod tau =dimC},  and 
\eqref{multiplicativity of tau}. Let us prove (iii). If $\C$ is the center of a fusion category $\M$ then
$\M$ is weakly integral, see Proposition \ref{centerweaklyintegral}. 
Choose an embedding $\varphi :\kRcyc\to\mathbb{R}$ and equip $\M$ with the
$\varphi$-positive spherical structure. The corresponding spherical structure on $\C$
is the $\varphi$-positive spherical structure $s_{\varphi}$, so $\tau^{\pm}(\C ,s_{\varphi})=\dim\M$
(see Example \ref{charge1}). Since $\M$ is weakly integral $\dim\M$ is a positive integer.
Now use \eqref{GGFP} and note that $n_{\C}=1$ by Proposition \ref{xcenter}.
To prove (iv), use \eqref{GGFP}, Theorem \ref{tau after mod}, and 
Proposition \ref{xtildec}.
\end{proof}

\begin{example}  \label{TIs}
By \S\ref{braidIs}, Ising braided categories are labeled by elements $\zeta\in k$ satisfying
$\zeta^8=-1$; we denote them by $\mI_{\zeta}$. We will show that
\begin{equation}  \label{T_of_Is}
\T^{\pm}(\mI_{\zeta} )=\zeta^{\mp 1}(\zeta^2+\zeta^{-2}). 
\end{equation}  
Let us deduce \eqref{T_of_Is} from \eqref{GFP}. By Example \ref{xIs}, if
$\C =\mI_{\zeta}$ then $n_{\C}=2$. By Proposition \ref{B3},
the set $\O_{\x_{\C}}(\C):=\{Y\in\O (\C )|\FPdim (Y)/\sqrt 2\in\BZ\}$ has a single element
$X$, and $\FPdim (X)=\sqrt 2$. Therefore $\T^{\pm}(\C) =\theta_{X,s}^{\pm 1}  d_s(X)$,
where $s$ is any spherical structure on $\C$. By \S\ref{sphIs}, these structures are labeled by
$\epsilon\in\{ 1,-1\}$. If $s$ corresponds to $\epsilon$ then $\theta_{X,s}=\epsilon\zeta^{-1}$
(see Proposition \ref{B10})
and $d_s(X)=\epsilon (\zeta^2+\zeta^{-2})$
(see \S\ref{sphIs} and formula \eqref{16eq}), so we get \eqref{T_of_Is}.
\end{example}

\appendix
\section{Pre-metric groups and their Gauss sums}
\label{app on metric groups}

In this appendix we recall basic results on quadratic forms on finite
abelian groups. With a possible exception of Theorem~\ref{wap-th},
all of them are well known (see, e.g., \cite{Wa,We,MH,BLLV,L}), but we are unable to
attribute them to concrete authors.
However, we would like to mention that the description
of anisotropic metric groups given below is a particular case of C.T.C.Wall's classification
of {\it all\,} metric groups in \cite[\S6]{Wa}.

\medskip

We will use the definitions and notation from \S\S \ref{Quadra}-\ref{metra} and \S \ref{gauss2}
with the following exception: the operation in pre-metric groups will be denoted by $+$. So a
function $q: G\to k^\times$ is a quadratic form if $q(-g)=q(g)$ and the function
$b(g_1,g_2)=\frac{q(g_1+g_2)}{q(g_1)q(g_2)}$ is a bicharacter. These properties imply that
$q(ng)=q(g)^{n^2},\; n\in \BZ$. The following fact is standard.

\begin{lemma}\label{bili}
 If $|G|$ is odd then any quadratic form $q:G\to k^\times$ can be written as
$q(g)=\beta(g,g)$, where $\beta: G\times G\to k^\times$ is a bicharacter. $\square$
\end{lemma}

\subsection{Anisotropic metric groups} \label{anisotropic-odd}
A quadratic form $q: G\to k^\times$ is said to be
{\em anisotropic} if $q(g)\ne 1$ for $g\in G\setminus \{ 0\}$. In this case we also say that the
pre-metric group $(G,q)$ is anisotropic. Remark~\ref{easyrems}(iii) implies that if $|G|$ is odd then 
an anisotropic quadratic form $q: G\to k^\times$ is non-degenerate. On the other hand, the quadratic form
$q: \BZ/2\BZ \to k^\times$ defined by $q(n)=(-1)^n$ is anisotropic but degenerate.

In this subsection and in \S\ref{2groups} we classify anisotropic {\em metric\,} groups.
Degenerate anisotropic pre-metric groups
will be classified in \S\ref{degensubsec}.

\begin{proposition}\label{serre}
 (i) The norm map $N_{\BF_{p^2}/\BF_p}: \BF_{p^2}\to \BF_p$ is an
anisotropic non-degenerate quadratic form.

(ii) Let $V$ be a vector space over $\BF_p$ and $q: V\to \BF_p$ an anisotropic quadratic form.
Then $\dim V\le 2$. If $\dim V=2$ then $(V,q)$ is isomorphic to $(\BF_{p^2}, N_{\BF_{p^2}/\BF_p})$.
\end{proposition}

This is well known (e.g., see \cite[Ch IV, \S 1.7]{Se}).

\begin{proposition} \label{dimtwo}
 (i) Let $p$ be a prime and $(G,q)$ an anisotropic pre-metric group such that
$q(g)^p=1$ for all $g\in G$. Then $G$ is a vector space over $\BF_p$ of dimension at most 2.

(ii) For each prime $p$ there exists a unique up to isomorphism pair $(G,q)$, where $G$ is
a 2-dimensional vector space over $\BF_p$ and $q$ is an anisotropic quadratic form on $G$
with values in $\mu_p:=\{ x\in k^\times | x^p=1\}$. Namely, one can take $G=\BF_{p^2}$,
$q(x)=\zeta^{N_{\BF_{p^2}/\BF_p}(x)}$, where $\zeta \in k$ is a primitive $p$-th root of 1.
Such a pair $(G,q)$ is a metric group with $\tau^+(G,q)=-p$.
\end{proposition}

\begin{proof} (i) If $g\in G$ then $q(pg)=q(g)^{p^2}=1$. Since $q$ is anisotropic $pG=0$, i.e.,
$G$ is a vector space over $\BF_p$. Now use Proposition \ref{serre}.

(ii) Use Proposition \ref{serre} and the formula
$$\sum_{x\in \BF_{p^2}}\zeta^{N_{\BF_{p^2}/\BF_p}(x)}=-p.$$
To prove this formula note that for each $a\in \BF_p^\times$ there are $p+1$ elements
$x\in \BF_{p^2}$ with $N_{\BF_{p^2}/\BF_p}(x)=a$, so
$$\sum_{x\in \BF_{p^2}}\zeta^{N_{\BF_{p^2}/\BF_p}(x)}=1+(p+1)\sum_{a\in \BF_p^\times}
\zeta^a=1-(p+1)=-p.$$
\end{proof}

\begin{lemma}   \label{pG} 
Let $p$ be an odd prime.
\begin{enumerate} 
\item[(i)]  If $n\ge 2$ and $(G,q)$ is a pre-metric group such that $p^nG=0$ then the subgroup 
$p^{n-1}G\subset G$ is isotropic.
\item[(ii)]  If $G$ is an abelian $p$-group and the quadratic form $q: G\to k^\times$ is anisotropic then
$pG=0$.
\item[(iii)]  If $pG=0$ and $q: G\to k^\times$ is a quadratic form then $q(g)^p=1$ for all $g\in G$.
\end{enumerate} 
\end{lemma}

\begin{proof} 
(i) and (iii) follow from Lemma \ref{bili}. Statement (ii) follows from (i).
\end{proof}

\begin{remark} If $p=2$ then statements (ii) and (iii) are false; (i) is false if  $p=2$ and $n=2$.
Counterexamples: (i-ii) $G=\BZ/4\BZ, 
\; q(n)=\zeta^{n^2},$ where $\zeta$ is a primitive $8$-th root of 1; (iii) $G=\BZ/2\BZ, \; q(n)=i^{n^2},$
where $i^2=-1$.
\end{remark}

Now we get a complete list of anisotropic metric $p$-groups for $p\ne 2$. Fix a primitive $p$-th
root of unity $\zeta \in k$.

\begin{proposition}\label{A5}
 If $p\ne 2$ there are 3 isomorphism classes of nonzero anisotropic metric
$p$-groups $(G,q)$, namely:

(i) $G=\BF_p, \; q(a)=\zeta^{a^2}$,

(ii) $G=\BF_p, \; q(a)=\zeta^{ca^2}$, where $c\in \BF_p^\times$ is a quadratic nonresidue,

(iii) $G=\BF_{p^2}, \; q(a)=\zeta^{N_{\BF_{p^2}/\BF_p}(a)}$.
\end{proposition}

\begin{proof} Combine Lemma \ref{pG}(ii-iii) and Proposition \ref{dimtwo}.
\end{proof}

A complete list of anisotropic metric $2$-groups will be given in \S \ref{future}.

\subsection{Proof of Proposition \ref{basis}} \label{A-2}
We will prove the following statement.

\begin{proposition}\label{A6}
 Let $p$ be a prime and $(G,q)$ a metric group of order $p^{2n},\; n\in \BN$.
Then $\tau^+(G,q)=p^n\epsilon$, where $\epsilon =\pm 1$ if $p\ne 2$ and $\epsilon^8=1$ if $p=2$.
If $\epsilon =1$ then $G$ has a Lagrangian subgroup.
\end{proposition} 

Proposition \ref{basis} follows from Proposition \ref{A6} because if $\tau^+(G,q)\in \BZ$ then 
$|G|=\tau^+(G,q)\tau^-(G,q)=(\tau^+(G,q))^2$ is the square of an integer.

Let us prove Proposition \ref{A6} for $p\ne 2$. By Proposition \ref{subquot}, it suffices to show that if
$(G,q)$ is an {\em anisotropic} metric group of order $p^{2n}$ then $\tau^+(G,q)=\pm p^n$ and
$\tau^+(G,q)=p^n$ only if $n=0$. This follows from Propositions \ref{dimtwo} and \ref{A5}.

To treat the case $p=2$ we need the following lemma.

\begin{lemma}   \label{A7}
Let $G$ be an abelian 2-group. If there exists a symmetric bicharacter 
$b :G\times G\to k^{\times}$ such that the quadratic form  $b(x,x)$ is anisotropic then $|G|\le 2$.
\end{lemma}

\begin{proof}   
Let $x\in G$ and let $n\in\BN$ be the minimal number such that $2^nx=0$. If $n>1$ then
$b (2^{n-1}x,2^{n-1}x)=0$ contrary to the assumption on $b$. Thus $2G=0$. Therefore
the map $x\mapsto b (x,x)$ is a character with values $\pm 1$ and trivial kernel. So $|G|\le 2$.
\end{proof}

\begin{corollary} \label{Acor}
Let $G$ be an abelian $2$-group equipped with a symmetric bicharacter $b: G\times G\to k^\times$.
Then there exists a subgroup $L\subset G$ such that $L\subset L^\perp$ and $(L^\perp :L)\le 2$.
\end{corollary}

\begin{proof} 
It suffices to show that if $|G|>2$ then there is a nonzero subgroup $L_0\subset G$ such that
$L_0\subset L_0^{\perp}$. By Lemma \ref{A7}, $b(x,x)=0$ for some nonzero $x\in G$.
Take $L_0$ to be the subgroup generated by $x$.
\end{proof}

Now let us prove Proposition \ref{A6} for a metric group $(G,q)$ of order $2^{2n}$. 
Applying Corollary \ref{Acor} to the bicharacter
$b: G\times G\to k^\times$ corresponding to $q$ we get a subgroup $L\subset G$ such that
$L\subset L^\perp$ and $(L^\perp : L)\le 2$. But $(L^\perp : L)=\frac{|G|}{|L|^2}=\frac{2^{2n}}{|L|^2}$
is a square, so $L^\perp =L$ and $|L|=2^n$. 

Let $\chi : L\to k^\times$ be the restriction of $q: G\to k^\times$. Then $\chi$ is a character with values
$\pm 1$. Let $u\in G/L$ be the preimage of $\chi \in L^*$ under the canonical isomorphism
$G/L \stackrel{\sim}{\to} L^*$. Then $2u=0$.

For any $g\in G$ the sum 
$$\sum_{h\in L}q(g+h)=q(g)\sum_{h\in L}b(g,h)\chi(h)$$
equals $2^nq(g)$ if the image of $g$ in $G/L$ equals $u$; otherwise the sum equals 0. So
$$ \tau^+(G,q)=2^nq(g_0)$$
for any $g_0\in G$ whose image in $G/L$ equals $u$. Since $2u=0$ we have $2g_0\in L$, so
$q(g_0)^4=q(2g_0)=\chi(g_0)=\pm 1$. Therefore $q(g_0)$ is an $8$-th root of 1. Finally, if
$q(g_0)=1$ then $G$ has a Lagrangian subgroup $L'$: namely, if $u=0$ take $L'=L$ and
if $u\ne 0$ take $L'$ to be the subgroup generated by $g_0$ and 
$\Ker (L\stackrel{\chi}{\longrightarrow} \{ 1,-1\})$.

\subsection{Anisotropic metric $2$-groups}    \label{2groups} 
We begin with some examples.

\subsubsection{Metric groups of order 2}  \label{anis2}
For each $i\in k$ such that $i^2=-1$ let
$A_i$ denote the group $\BZ/2\BZ$ equipped with the quadratic form $n\mapsto i^{n^2}$.
Clearly $A_i$ is an anisotropic metric group and $\tau^+(A_i)=1+i$. It is easy to see that
there are no other metric groups of order 2.

\subsubsection{Some metric groups of order $4$}
\begin{examples}  \label{anis4}
\begin{enumerate}
\item[(i)] For each $i\in k$ such that $i^2=-1$ the orthogonal direct sum $A_i\oplus A_i$
is an anisotropic metric group of order $4$ and $\tau^+(A_i\oplus A_i)=2i$.
\item[(ii)] Let $\xi\in k$ be a primitive $8$-th root of unity. Then $\BZ/4\BZ$ equipped with the quadratic form $n\mapsto\xi^{n^2}$ is an anisotropic metric group whose Gauss sum $\tau^+$
equals $2\xi$.
\item[(iii)] Let $C$ denote the metric group from Proposition \ref{dimtwo} (ii)
for $p=2$; expli\-citly, $C$ is the group $(\BZ/2\BZ)^2$ equipped with the quadratic form
$q(x,y)=(-1)^{x^2+xy+y^2}$. Then $C$ is an anisotropic metric group of order $4$ and
$\tau^+(C)=-2$.
\end{enumerate} 
\end{examples}

In \S\ref{future} we will see that Examples \ref{anis4} give all anisotropic metric groups of order~$4$.
Now we will give a {\em uniform\,} description of the metric groups from Examples~\ref{anis4}(i-iii).
Note that each of them contains an element of order 2 on which the quadratic form takes
value $-1$.
Let $T$ denote the set of isomorphism classes of triples $(G,q,u)$ in which $(G,q)$
is a metric group of order $4$ and $u\in G$ is an element of order $2$ with $q(u)=-1$. 
Here is a description of $T$.

\begin{lemma}  \label{uniformdescr}
\begin{enumerate}
\item[(i)] For any $(G,q,u)\in T$ the values of $q$ on the two elements of $G\setminus\{ 0,u\}$
are 8-th roots of unity equal to each other.
\item[(ii)] Associating this 8-th root of unity to $(G,q,u)\in T$ one gets a bijection
\[
T\iso\mu_8:=\{ \xi\in k|\xi^8=1\}.
\]
If $(G,q,u)$ corresponds to $\xi$ then $\tau^+(G,q)=2\xi$.
\item[(iii)] Let $\xi\in\mu_8$. Set $n(\xi ):=0$ if $\xi^4=1$ and $n(\xi ):=1$ if
$\xi^4=-1$. Then the metric group $M_{\xi}$ corresponding to $\xi$ has the
following description: as an abelian group, it has generators $u,v$ with the defining
relations $2u=0$, $2v(\xi )u$, and the quadratic form $q:M_{\xi}\to k^{\times}$
is given by 
\[
q(u)=-1, \; q(v)=q(u+v)=\xi, \; q(0)=1.
\]
\end{enumerate} 
\end{lemma}  

\begin{proof}
Let $b:G\times G\to k^{\times}$ be the bicharacter corresponding to $q$.
For any $v\in G$ one has $4v=0$, so $q(v)^8=b(v,v)^4=b(4v,v)=1$.

Now suppose that $v\ne 0,u$. Since $b(u,u)=q(u)^2=1$ and $2v\in\{ 0,u\}$ we get
$b(u,v)^2=1$. As $b$ is non-degenerate $b(u,v)\ne 1$ (otherwise $u\in\Ker b$).
So $b(u,v)=-1$. Therefore $q(u+v)=q(u)q(v)b(u,v)=q(v)$.
Clearly $2vu$ for some $n\in\BZ/2\BZ$. If $q(v)=\xi$ then $\xi^4=q(u)^{n^2}=
(-1)^{n^2}=(-1)^n$, so $n(\xi )$. 

Finally, $\tau^+(G,q)=1-1+2\xi=2\xi$.
\end{proof}

\begin{rems}  \label{AutMxi}
\begin{enumerate}
\item[(i)] $M_{\xi}$ is anisotropic if and only if $\xi\ne 1$.
\item[(ii)] If $\xi$ is a primitive 8-th root of unity then $M_{\xi}$ is the metric group from
Example \ref{anis4}(ii). If $\xi =i$, $i^2=-1$, then $M_{\xi}$ is the metric group $A_i\oplus A_i$
from Example \ref{anis4}(i). If $\xi =-1$ then $M_{\xi}$ is the metric group from
Example \ref{anis4}(iii). Finally, if $\xi =1$ then $M_{\xi}$ is isomorphic to $(\BZ/2\BZ)^2$
equipped with the hyperbolic form $(x,y)\mapsto (-1)^{xy}$. 
\item[(iii)] Let $\Aut (M_{\xi},u)$ denote the stabilizer of $u$ in $\Aut (M_{\xi})$.
Clearly $|\Aut (M_{\xi},u)|=2$ and the unique nontrivial element of 
$\Aut (M_{\xi},u)$ permutes $v$ and $v+u$. This automorphism can also be written as
$x\mapsto x+m(x)u$, where $m(x)\in\BZ/2\BZ$ is such that $(-1)^{m(x)}=b(x,u)$.
Here $b:M_{\xi}\times M_{\xi}\to k^{\times}$ is the bicharacter corresponding to 
the quadratic form on $M_{\xi}$.
\item[(iv)]  If $\xi\ne -1$ then $\Aut (M_{\xi},u)=\Aut (M_{\xi})$.
If $\xi = -1$ then $\Aut (M_{\xi})=GL(2,\BF_2)\ne\Aut (M_{\xi},u)$.
\end{enumerate}
\end{rems}

\subsubsection{Classification of anisotropic metric $2$-groups}  \label{future}
\begin{proposition}    \label{A13}
\begin{enumerate}
\item[(i)] Every anisotropic metric $2$-group has order at most 8.
\item[(ii)] The anisotropic metric groups of order $2$ are the groups $A_i$
from \S\ref{anis2} corresponding to elements $i\in k$ with $i^2=-1$.
\item[(iii)]  The anisotropic metric groups of order $4$ are the metric groups $M_{\xi}$
from Lemma \ref{uniformdescr}(iii) corresponding to elements $\xi\in k$ such that $\xi^8=1$, $\xi\ne 1$.
\item[(iv)]  Every anisotropic metric group of order $8$ is isomorphic to an orthogonal
direct sum $M_{\xi}\oplus A_i$, where $\xi,i\in k$, $i^2=-1$, $\xi^8=1$, $\xi\ne 1,-i$.
\item[(v)]  $M_{\xi}\oplus A_i\simeq M_{i\xi}\oplus A_{-i}$. 
If $M_{\xi}\oplus A_i\simeq M_{\xi'}\oplus A_i$ then $\xi =\xi'$.
\end{enumerate}
\end{proposition}

\begin{proof}
(i) Let $(G,q)$ be a metric 2-group and $b:G\times G\to k^{\times}$ the bicharacter corresponding
to $q$. By Corollary \ref{Acor}, there is a subgroup $L\subset G$ such that $L\subset L^{\perp}$ and
$2\cdot |L|^2\ge |G|$. The restriction of $q$ to $L$ is a character with values $\pm 1$. Its kernel
is trivial because $q$ is anisotropic. So $|L|\le 2$ and therefore $G\le 2\cdot |L|^2\le 8$.

(ii) is clear.

(iii) If $|G|=4$ then $|L|=2$. If $u\in L$, $u\ne 0$ then $q(u)^2=b(u,u)=1$, so $q(u)=-1$.
Therefore $(G,q)$ is one of the groups $M_{\xi}$ from Lemma \ref{uniformdescr}. Recall that
$M_{\xi}$ is anisotropic if and only if $\xi\ne 1$.

(iv) If $|G|=8$ then $|L|=2$ and $|L^{\perp}|=4$. Let $v\in L$, $v\notin L^{\perp}$. Since 
$|L^{\perp}/L|=2$ one has $2v\in L$. So $q(2v)=b(2v,v)^2=1$, and since $q$ is anisotropic $2v=0$.
Let $N\subset G$ be the subgroup generated by $v$. The restriction of $b$ to $N$ is non-degenerate,
so $N$ and $N^{\perp}$ are metric groups and $G=N\oplus N^{\perp}$. By (ii) and (iii),
$N\simeq A_i$ and $N^{\perp}\simeq M_{\xi}$ for some $i$ and $\xi$. Clearly
$M_{\xi}\oplus A_i$ is anisotropic if and only if $q(x)\notin \{ 1,-i \}$ for $x\in M_{\xi}\setminus\{ 0\}$, i.e.,
if and only if $\xi\ne 1,-i$.

(v) If $M_{\xi}\oplus A_i\simeq M_{\xi'}\oplus A_i$ then $\tau^+(M_{\xi})=\tau^+(M_{\xi'})$ and therefore 
$\xi =\xi'$. To construct an isomorphism $M_{\xi}\oplus A_i\iso M_{i\xi}\oplus A_{-i}\,$, note that 
$M_{\xi}\oplus A_i$ has an element $w$ such that $q(w)=-i$ (namely, $w$ is the sum of the
element $u\in M_{\xi}$ and the nonzero element of $A_i$).
\end{proof}

\subsection{Slightly degenerate pre-metric groups} \label{slightsubsec}
Let $(G,q)$ be a pre-metric group. Let $b:G\times G\to k^{\times}$ be the
bicharacter corresponding to $q$. The restriction of $q$ to $\Ker b$ is a character.

\begin{definition}  \label{slightdef}
A pre-metric group $(G,q)$ is said to be {\em slightly degenerate\,} if $|\Ker b|=2$
and the value of $q$ at the nontrivial element of $\Ker b$ equals $-1$.
\end{definition} 
\begin{example} \label{slightex}
$G=\BZ/2\BZ , \quad q(n)=(-1)^n$. (The role of this example is clear from Corollary \ref{slightcor}
below.)
\end{example}

\begin{lemma}  \label{slightlem}
If a pre-metric group $(G,q)$ is slightly degenerate then the exact sequence of groups
$0\to\Ker b\to \G\to G/\Ker b\to 0$ splits.
\end{lemma}

\begin{proof} 
Since $\Ker b\simeq\BZ/2\BZ$ it suffices to show that $\Ker b\not\subset 2G$.
If $x\in G$ and $2x\in\Ker b$ then 
$q(2x)=q(x)^4=b(2x,x)=1$. By Definition \ref{slightdef}, $q(y)\ne 1$ if $y\in\Ker b$, $y\ne 0$.
So $2x=0$.
\end{proof} 

Let $\A$ be the following category: its objects are slightly degenerate pre-metric groups
and an $\A$-morphism from $(G_1,q_1)\in\A$ to $(G_2,q_2)\in\A$ is a morphism
of pre-metric groups which maps $\Ker b_1$ to $\Ker b_2$.

We will construct an equivalence between  $\A$ and the following category $\B$:
the objects of $\B$ are pairs $(H,\beta )$, where $H$ is a finite abelian group and 
$\beta :H\times H\to k^{\times}$ is a non-degenerate symmetric pairing; a morphism
$(H,\beta )\to (H',\beta' )$ is a homomorphism $\phi: H\to H'$ such that
$\beta' (\phi (h_1),\phi (h_2))=\beta (h_1,h_2)$ for all $h_1,h_2\in H$.

Let $(G,q)$ be a pre-metric group and $b:G\times G\to k^{\times}$ the
bicharacter corresponding to $q$. The group $\bar G:=G/\Ker b$ is equipped with
the non-degenerate symmetric pairing $\bar b:\bar G\times\bar G\to k^{\times}$ induced by $b$.
Thus we get a functor $F:\A\to\B$. 

\begin{proposition}   \label{slightprop1}
The functor $F:\A\to\B$ is an equivalence.
\end{proposition} 

\begin{proof}   
The inverse functor takes $(H,\beta )\in\B$ to the following pre-metric group $(\widetilde{H}, q)$.
Define $\widetilde{H}$ to be the set of pairs $(h,a)$ with $h\in H$, $a\in k^{\times}$, $a^2=b(h,h)$.
Define the group operation on $\widetilde{H}$ by 
$(h_1,a_1)+(h_2,a_2)=(h_1+h_2,a_1a_2\beta (h_1,h_2))$.
Finally, set $q(h,a):=a$.
\end{proof} 

\begin{proposition}   \label{slightprop2}
Let $H$ be a finite abelian group with a non-degenerate symmetric pairing
$\beta :H\times H\to k^{\times}$ and let $\widetilde{H}$ be the corresponding
slightly degenerate pre-metric group (see the proof of Proposition \ref{slightprop1}). Then
\begin{enumerate}
\item[(i)] 
the epimorphism $\pi :\widetilde{H}\to H$ admits a splitting;

\item[(ii)]  
splittings of $\pi$ bijectively correspond to quadratic forms
$\bar q:H\to k^{\times}$ such that the bicharacter $H\times H\to k^{\times}$ 
associated to $\bar q$ equals $\bar b$; namely, the splitting $H\to\widetilde{H}$
corresponding to $\bar q$ is defined by $h\mapsto (h,\bar q(h))$.
\end{enumerate}
\end{proposition} 

\begin{proof}   
Checking (ii) is straightforward, and (i) follows from Lemma \ref{slightlem}.
\end{proof}

\begin{corollary}   \label{slightcor}
Every slightly degenerate pre-metric group can be decomposed as an orthogonal
direct sum of a metric group and the pre-metric group from Example~\ref{slightex}.
\hfill\qedsymbol
\end{corollary}

\subsection{Degenerate anisotropic pre-metric groups} \label{degensubsec}
\begin{lemma} \label{degenlem1}
\begin{enumerate}
\item[(i)] 
A degenerate anisotropic pre-metric group is slightly degenerate.

\item[(ii)]  
Let $(G,q)$ be a slightly degenerate pre-metric group and $b:G\times G\to k^{\times}$ the
bicharacter corresponding to $q$. Let $\bar G:=G/\Ker b$ and $\bar b:\bar G\times\bar G\to k^{\times}$ 
the bicharacter induced by $b$. Then $q$ is anisotropic if and only if the quadratic form
$x\mapsto \bar b(x,x)$, $x\in\bar G$, is anisotropic.
\end{enumerate}
\end{lemma}

\begin{proof}   
The restriction of $q$ to $\Ker b$ is a character with trivial kernel, so we get (i). Statement (ii)
follows from the equalities $b(x,x)^2=q(x)$ and $q(x+u)=-q(x)$, where $x\in G$ and $u$ is the
nontrivial element of $\Ker b$.
\end{proof} 

\begin{proposition}     \label{degenprop}
\begin{enumerate}
\item[(i)] 
A degenerate anisotropic pre-metric group is an orthogonal direct sum of a degenerate
anisotropic pre-metric 2-group and an anisotropic metric group of odd order.
\item[(ii)]  
There are two degenerate anisotropic pre-metric 2-groups $(G,q)$, namely the one from 
Example~\ref{slightex} and the following one:
\begin{equation} \label{degeneq}
G=(\BZ/2\BZ)^2 , \quad q(m,n)=i^m(-1)^n.
\end{equation}
Here $i\in k$ is a fixed square root of $-1$.
\end{enumerate}
\end{proposition}
\begin{remark}    \label{i-minus-i}
The equality $(-i)^m(-1)^n=i^m(-1)^{m+n}$ shows that replacing $i$ by $-i$ in \eqref{degeneq}
one gets an isomorphic pre-metric group.
\end{remark}

\begin{proof}   
(i) is clear and (ii) follows from Proposition~\ref{slightprop1}, Lemma~\ref{degenlem1}, and 
Lem\-ma~\ref{A7}.
(Alternatively, one can use Corollary \ref{slightcor} and the following corollary of 
Proposition \ref{A13}: if $(H,q)$ is
an anisotropic metric 2-group such that $q$ does not take the value $-1$ then
$|H|\le 2$.)
\end{proof} 

\subsection{Weakly anisotropic pre-metric groups} \label{wapg}
According to Definition \ref{5def2anisotropic}, a pre-metric group $(G,q)$ is said to be 
{\em weakly anisotropic\,} if it has no nonzero isotropic subgroups stable under $\Aut (G,q)$.

\subsubsection{Formulation of the theorem} \label{wap-form}
\begin{lemma}   \label{wap-lem}
A pre-metric group $(G,q)$ is weakly anisotropic if and only if all its $p$-components are.
\hfill\qedsymbol
\end{lemma}

Now fix a prime $p$. Let $H$ denote $\BF_p^2$ equipped with the hyperbolic quadratic form
$(x,y)\mapsto\zeta^{xy}$, where $\zeta\in k$ is a fixed primitive $p$-th root of 1.
As usual, let $b:G\times G\to k^{\times}$ denote the symmetric pairing associated to $q$.

\begin{theorem}   \label{wap-th}
The following properties of a pre-metric $p$-group $(G,q)$ are equivalent:
\begin{enumerate}
\item[(i)] $(G,q)$ is weakly anisotropic;
\item[(ii)] $(G,q)$ is an orthogonal direct sum of several copies of $H$ and an anisotropic pre-metric
$p$-group;
\item[(iii)] for every isotropic subgroup $A\subset G$ there exists an isotropic subgroup $B\subset G$ such that the pairing $A\times B\to k^{\times}$ induced by $b:G\times G\to k^{\times}$ is
non-degenerate.
\end{enumerate}
If $p\ne 2$ these properties are equivalent to the following one:
\begin{enumerate}
\item[(iv)] $pG=0$ and $q$ is non-degenerate.
\end{enumerate}
\end{theorem}

\begin{remark}     \label{wap-rem1}
Combining Lemma~\ref{wap-lem} and Theorem~\ref{wap-th} with the classification of anisotropic 
pre-metric groups from \S\ref{anisotropic-odd}, \S\ref{2groups}, and \S\ref{degensubsec} one gets a 
{\em complete classification\,} of pointed weakly anisotropic  pre-metric groups.
\end{remark}

\subsubsection{Proof of Theorem \ref{wap-th}} \label{wap-proof}
\begin{lemma}   \label{wap-lm1}
(iii)$\Rightarrow$(i).
\end{lemma}

\begin{proof}
Let $A$ and $B$ have the same meaning as in (iii). Clearly $A\cap B$ is contained in the kernel of
the pairing $A\times B\to k^{\times}$, so $A\cap B=0$. Choose an isomorphism $f:A\iso B$
and let $\varphi :G\to G$ be the automorphism such that $\varphi |_{A}=f$, $\varphi |_{B}=f^{-1}$,
and $\varphi (x)=x$ if $x$ is orthogonal to $A\oplus B$. Then $\varphi$ preserves $q$ and
$\varphi (A)\ne A$ unless $A=0$.
\end{proof}

\begin{lemma}   \label{wap-lm2}
Property (iii) implies that all isotropic elements of $G$ are annihilated by $p$.
\end{lemma}

\begin{proof}
Suppose that $x\in G$, $q(x)=1$, $p^mx=0$, $p^{m-1}x\ne 0$, $m>1$.
Let $A\subset G$ be the subgroup generated by $p^{m-1}x$ and let $B$ be as in (iii).
Then $|B|=p$, so $B$ is orthogonal to $p^{m-1}x$, which contradicts the non-degeneracy
of the pairing $A\times B\to k^{\times}$.
\end{proof}

\begin{lemma}   \label{wap-lm3}
(iii)$\Rightarrow$(ii).
\end{lemma}

\begin{proof}
Let $A\subset G$ be the {\em maximal\,} isotropic subgroup. Let $B$ be as in (iii). Then 
$G=(A\oplus B)\oplus (A\oplus B)^{\perp}$ and $(A\oplus B)^{\perp}$ is anisotropic. 
Lemma~\ref{wap-lm2} implies that $(A\oplus B)\simeq H^n$ for some $n\ge 0$,
\end{proof}

\begin{lemma}   \label{wap-lm4}
Suppose that $pG=0$ and $q(x)^p=1$ for all $x\in G$. If $(G,q)$ is either non-degenerate or slightly
degenerate then (iii) holds.
\end{lemma}

\begin{proof}
$G$ is a vector space over $\BF_p$ and $q(x)=\zeta^{Q(x)}$, where $\zeta\in k$ is a primitive $p$-th
root of 1 and $Q:G\to\BF_p$ is a quadratic form. So at least in the non-degenerate case, the lemma
is a standard fact from linear algebra. The slightly degenerate case (which can occur only if $p=2$)
can be reduced to the non-degenerate one because any isotropic subgroup $A\subset G$ is contained
in some metric subgroup $G_1\subset G$ (note that since $\Ker b$ is anisotropic $A\cap\Ker b=0$,
then take $G_1\subset G$ to be any subgroup of index $p$ containing $A$ and not containing $\Ker b$).
\end{proof}

By Lemma \ref{wap-lm4}, condition (iv) of Theorem \ref{wap-th} implies (iii). So to prove 
Theorem \ref{wap-th} for $p\ne 2$ it remains to show that in this case (i)$\Rightarrow$(iv) and
(ii)$\Rightarrow$(iv). If $p\ne 2$ then
$\Ker b$ is isotropic, so a weakly anisotropic pre-metric $p$-group is non-degenerate.
Lemma \ref{pG}(i) implies that if $p\ne 2$ then a weakly anisotropic pre-metric $p$-group is annihilated by $p$. So (i)$\Rightarrow$(iv) and (ii)$\Rightarrow$(iv).

In the case $p=2$ we still have to prove that (i)$\Rightarrow$(iii) and (ii)$\Rightarrow$(iii).
By Lemma~\ref{wap-lm4}, it suffices to prove the next lemma.

\begin{lemma}   \label{wap-lm5}
Let $p=2$. Either condition (i) or (ii) implies that
\begin{enumerate}
\item[(a)] all isotropic elements of $G$ are annihilated by $2$;
\item[(b)] the pre-metric group $G':=\{ x\in G|2x=0, q(x)^2=1\}$ is either non-degenerate or slightly degenerate.
\end{enumerate}
\end{lemma}

\begin{proof}
By Lemma~\ref{degenlem1}(i), condition (ii) implies (b). Let us prove that (i)$\Rightarrow$(b).
Let $G'$ be as in (b). Let $b':G'\times G'\to k^{\times}$
be the restriction of $b:G\times G\to k^{\times}$. We have to show that
\begin{equation}   \label{wap-eq}
\{ y\in\Ker b' \,|\,q(y)=1\} =0.
\end{equation}
This follows from (i) because the l.h.s. of \eqref{wap-eq} is an isotropic subgroup stable under 
$\Aut (G,q)$.

Finally, let us prove that  (b)$\Rightarrow$(a).
Suppose that $x\in G$, $q(x)=1$, $2^mx=0$, $2^{m-1}x\ne 0$, $m>1$. Then
$q(2^{m-1}x)=1$ and $2^{m-1}x\in\Ker b'$, which contra\-dicts~\eqref{wap-eq}.
\end{proof}

\subsection{The Witt group of metric groups}
\subsubsection{Definition of the Witt group} \label{DefWitt}
Let $(G,q)$ be a metric group. Then for each
isotropic subgroup $H\subset G$ the quotient $H^\perp /H$ has a structure of metric group.
In this situation we will say that $H^\perp /H$ is an {\em m-subquotient\,} of $(G,q)$
(here ``m" stands for ``metric" and ``middle").

Clearly $H$ is maximal among isotropic subgroups if and only if the metric group $H^\perp /H$
is anisotropic. We will say that $H^\perp /H$ is ``the" anisotropic metric group corresponding
to $(G,q)$. This terminology is justified by the next lemma.

\begin{lemma}\label{A8} If $H_1, H_2\subset G$ are maximal isotropic subgroups then the 
metric groups $H_1^\perp /H_1$ and $H_2^\perp /H_2$ are isomorphic.
\end{lemma}

The proof given below provides a concrete isomorphism 
$a_{12}: H_1^\perp /H_1\stackrel{\sim}{\to} H_2^\perp /H_2$. But given a maximal isotropic subgroup
$H_3\subset G$ the diagram

\begin{equation*}
\label{triangle}
\xymatrix{H_1^\perp /H_1\ar[rr]^{a_{12}} \ar[dr]^{a_{13}
}&&H_2^\perp /H_2
\ar[dl]_{a_{23}}\\ &H_3^\perp /H_3&}
\end{equation*}
does not commute, in general. This is why we talk about ``the" anisotropic metric group corresponding
to $(G,q)$ (with ``the" in quotation marks).

\begin{proof} We can assume that $H_1\cap H_2=0$ (otherwise replace $G$ by 
$(H_1\cap H_2)^\perp /H_1\cap H_2$). Since $H_1^\perp /H_1$ is anisotropic and $H_2$ is
isotropic $H_2\cap H_1^\perp \subset H_2\cap H_1=0$. Similarly, $H_1\cap H_2^\perp =0$.
So the pairing $b: H_1\times H_2\to k^\times$ is non-degenerate. Therefore the restriction of $q$
to $C:=H_1\oplus H_2$ is non-degenerate. So $G=C\oplus C^\perp$. Clearly $H_1^\perp /H_1
\simeq C^\perp \simeq H_2^\perp /H_2$.
\end{proof}

Say that two metric groups are {\em equivalent\,} if ``the" corresponding anisotropic metric groups are
isomorphic. 

\begin{lemma}\label{equivequiv} 
Two metric groups are equivalent if and only if they have isomorphic m-subquotients.
\hfill\qedsymbol
\end{lemma}

Let $W$ denote the set of equivalence classes of metric groups. If $G_1, G_2, G_1', G_2'$ are metric
groups and $G_1\sim G_1',\; G_2\sim G_2'$ then the orthogonal direct sum $G_1\oplus G_2$
is equivalent to $G_1'\oplus G_2'$ (this follows from Lemma~\ref{equivequiv}). Thus $W$ is a commutative monoid. In fact, it is a group
because the orthogonal direct sum $(G,q)\oplus (G,q^{-1})$ contains the diagonally embedded
$G$ as a Lagrangian subgroup. For each prime $p$ let $W_p\subset W$ be the set of classes
of metric $p$-groups. Each $W_p$ is a subgroup, and $W=\bigoplus\limits_p W_p$.  

\begin{definition} $W$ is called the {\em Witt group of metric groups}. $W_p$ is called {\em
the Witt group of metric $p$-groups}.
\end{definition}
 
\begin{remark} $W$ identifies with the set of isomorphism classes of anisotropic metric groups.
The sum of the classes of anisotropic metric groups $G_1$ and $G_2$ is the class of ``the"
anisotropic metric group corresponding to the orthogonal direct sum $G_1\oplus G_2$.
\end{remark}

Let us formulate the universal property of $W$. The universal property of $W_p$ is formulated
and proved in a similar way.

\begin{lemma}\label{A9} Suppose we have an abelian group $\Pi$ and a rule that associates to
an isomorphism class of metric groups $G$ an element $[G]\in \Pi$ so that $[G_1\oplus G_2]=
[G_1]+[G_2]$ and if $G$ is hyperbolic then $[G]=0$. Then there exists a (unique) morphism
$W\to \Pi$ such that the class in $W$ of a metric group $G$ is mapped to $[G]$.
\end{lemma}

\begin{proof} It suffices to show that if $H\subset G$ is an isotropic subgroup then $[H^\perp /H]=[G]$.
Set $G^-:=(G,q^{-1})$, where $q$ is the original quadratic form on $G$. Then the orthogonal
direct sums $G^-\oplus (H^\perp /H)$ and $G^-\oplus G$ are hyperbolic (because the
diagonally embedded subgroups $H^\perp \hookrightarrow G^-\oplus (H^\perp /H)$ and 
$G\hookrightarrow G^-\oplus G$ are Lagrangian). So $[G^-]+[H^\perp /H]=0=[G^-]+[G]$. 
Therefore $[H^\perp /H]=[G]$.
\end{proof}

\subsubsection{The Gauss sum morphism} \label{WittGauss}
Let $\BQ^{>0}:=\{ x\in \BQ^\times | x>0\}$.
By Proposition \ref{subquot}, the image of $\tau^\pm (G,q)$ in $k_{cyc}^\times /\BQ^{>0}$ depends
only on ``the" anisotropic metric group corresponding to $(G,q)$, so we have well defined maps
$\tau^\pm : W\to k_{cyc}^\times /\BQ^{>0}$. They are homomorphisms. Since $\tau^-$ is
complex-conjugate to $\tau^+$ it suffices to study $\tau^+$. The following statement is a
reformulation of Proposition \ref{basis}.

\begin{proposition}\label{A10} For each prime $p$ the homomorphism 
$\tau^+ : W_p\to k_{cyc}^\times /\BQ^{>0}$ is injective. $\square$
\end{proposition}

\begin{corollary}\label{A11} If $(G_1,q_1)$ and $(G_2,q_2)$ are anisotropic
metric $p$-groups and
$\tau^+(G_1,q_1)=\tau^+(G_2,q_2)$ then $(G_1,q_1)\simeq (G_2,q_2)$. $\square$ 
\end{corollary}

To describe $\tau^+(W_p)$, we need some notation.  Let $\mu_n\subset k^\times$ denote 
the group of $n$-th roots of 1.
Let $C_p\subset k^{\times}$ be the following subgroup:
if $p\equiv 1\mod{4}$ then $C_p$ is generated by $-1$ and $\sqrt{p}$; if $p\equiv 3\mod{4}$
then $C_p$ is generated by $-1$ and $\sqrt{-p}$; finally, $C_2$ is generated by $\mu_8$ and
$\sqrt{2}$ (or by $\mu_8$ and $1+i$). Set $p^\BZ :=\{ p^n | n\in \BZ \}$.
Clearly $C_p/p^\BZ \subset k_{cyc}^\times /\BQ^{>0}$.

\begin{proposition}    \label{A12} 
$\tau^+(W_p)=C_p/p^\BZ$.
\end{proposition}

\begin{proof} 
If $p=2$ use Proposition \ref{A13}, Lemma \ref{uniformdescr}(ii) and \S\ref{anis2}.

Now let $p\ne 2$. Define $q:\BF_p\to k^{\times}$ by $q(x)=\zeta^{x^2}$, where $\zeta$ is a primitive 
$p$-th root of $1$. By Propositions \ref{A5} and \ref{dimtwo}(ii), $\tau^+(W_p)$ is generated by $-1$ and 
$\tau^+(\BF_p,q)$. Moreover, $\tau^+(\BF_p,q)^2$ equals the Gauss sum of the quadratic form
on $\BF_p^2$ defined by $q(x,y)=\zeta^{x^2+y^2}$. If $p\equiv 1\mod{4}$ this form is hyperbolic,
so $\tau^+(\BF_p,q)^2=p$. If $p\equiv 3\mod{4}$ the form is anisotropic,
so $\tau^+(\BF_p,q)^2=-p$ by Proposition \ref{dimtwo}. 
\end{proof}

\begin{rems} 
\begin{enumerate}
\item[(i)] Propositions \ref{A10} and \ref{A12} describe the group $W_p$. Another approach
to $W_p$ is developed in \cite{L} and \cite[Appendice A]{BLLV}. It is based on the exact
sequence 
$$0\to WQ(\BZ_p)\to WQ(\BQ_p)\to W_p\to 0$$
from \cite[Th\'eor\`eme A.25]{BLLV} and the isomorphism $WQ(\BZ_p)\stackrel{\sim}{\to}WQ(\BF_p)$
from \cite[Corollaire 2.4]{L}. Here $WQ(A)$ is the Witt group defined by using finitely generated
projective $A$-modules equipped with non-degenerate $A$-valued quadratic forms.
\item[(ii)] By Propositions \ref{A10} and \ref{A12}, for any $p$ one has $8W_p=0$. 
Here is a sketch of an alternative proof of this fact. The group $W_p$ is a module of the
Witt ring of symmetric bilinear forms over $\BZ_p$, so it suffices to show that in this ring
$8=0$. This is because for any prime $p$ the symmetric bilinear from $\sum\limits_{i=1}^8x_iy_i$ over
$\BQ_p$ has a Lagrangian subspace.
\item[(iii)] We defined $C_p$ by considering separately the cases $p\equiv 1\mod{4}$, 
$p\equiv 3\mod{4}$, and $p=2$. Here is a uniform description of $C_p$. Let $k_{p\mbox{-}cyc}\subset 
k_{cyc}$ be the subfield generated by $p^n$-th roots of unity, $n=1,2,$\ldots . Let 
$E_p\subset k_{p\mbox{-}cyc}$ be the subfield of elements fixed by 
$(\BZ_p^\times)^2\subset \BZ_p^\times =Gal(k_{p\mbox{-}cyc}/k)$. Then
$C_p/p^\BZ \subset E_p^\times /\BQ^{>0}$ is the subgroup of $2$-torsion. Note that the
inclusion $\tau^+(W_p)\subset C_p$ follows from Remark (ii) above and the
inclusion $\tau^+(W_p)\subset E_p^\times/p^\BZ$, which is clear from the definition of 
Gauss sum. 
\end{enumerate}
\end{rems}

\section{Ising categories} \label{Is}
In this Appendix we collect some well known facts about the so-called {\em Ising categories}.

\begin{definition} \label{isingdef}
We say that a fusion category $\mI$ is an {\em Ising} category if $\FPdim(\mI)=4$
but $\mI$ is not pointed. An {\em Ising braided category\,} is a braided fusion category which is 
Ising in the above sense. An {\em Ising modular category\,} is a modular category which is 
Ising as a fusion category.
\end{definition} 

\begin{remark} The modular category that appears in the description of the 2-dimensional
Ising model at criticality via Conformal Field Theory is an example of an Ising modular category. 
\end{remark}

\subsection{Ising fusion categories}  \label{fusionIsing}
\begin{proposition}  \label{B3}
An Ising category $\mI$ contains precisely 3 isomorphism classes of simple
objects: the unit object $\be$,  an invertible object $\delta$ not isomorphic to
$\be$,  and a non-invertible object $X$. One has
\begin{equation} \label{isingK0}
\delta \otimes \delta \simeq \be ;\; \delta \otimes X\simeq X\otimes \delta \simeq X;\; 
X\otimes X\simeq \be \oplus \delta , \end{equation}
$$\FPdim(\delta)=1, \; \; \FPdim(X)=\sqrt{2}.$$
\end{proposition}

\begin{proof} See e.g. \cite[Proposition 8.32]{ENO}. 
\end{proof}

It follows that the object $X$ is self-dual. Choose an isomorphism $\phi : X\iso X^*$ (it is defined
uniquely up to a scaling). Then the composition 
\begin{equation}  \label{defoflambda}
\be \stackrel{\coev_X}{\longrightarrow}X\otimes X^*\stackrel{\phi \otimes \phi^{-1}}{\longrightarrow}
X^*\otimes X\stackrel{\ev_X}{\longrightarrow}\be
\end{equation}
is independent of the choice of $\phi$. Let us denote this composition by 
$\lambda (\mI)\in \Hom(\be,\be)=k$. 

\begin{lemma} We have $\lambda (\mI)^2=2$.
\end{lemma}

\begin{proof} 
Combining Propo\-si\-tions~\ref{FPint1}(ii) and \ref{B3} we see that $|X|^2=2$,
where $|X|^2$ is the squared norm (see \S\ref{squared}). On the other hand,
let us compute $|X|^2$ using the isomorphism $\psi_X: X\to X^{**}$ defined by
$\psi_X:=(\phi^*)^{-1}\phi$ (here $\phi : X\iso X^*$ is any isomorphism). The corresponding
numbers $d_{\pm}(X)$ (see \S\ref{squared}) equal $\lambda$, so 
$|X|^2=d_+(X)d_-(X)=\lambda^2$ by  formula~\eqref{dim1}.
\end{proof}
 
\begin{proposition} \label{sqrt2}
(i) For any $\lambda \in k$ such that $\lambda^2=2$
there exists an Ising category $\mI$ with $\lambda (\mI)=\lambda$.

(ii) The number $\lambda (\mI)$ determines an Ising category $\mI$ uniquely up
to an equivalence of fusion categories.
\end{proposition}

This proposition is a particular case of \cite[Theorem 3.2]{TY} (in \cite{TY} 
multiplication tables more general than \eqref{isingK0} are studied).

\begin{rems} \label{tyf}
\begin{enumerate}
\item[(i)] It is not hard to show that any tensor autoequivalence of an Ising category is isomorphic to the identity functor.
\item[(ii)] Here is an explicit construction of an Ising category $\mI$ with $\lambda (\mI)=\lambda$,
where $\lambda$ is a fixed square root of 2 (this is a special case of \cite[Definition 3.1]{TY}).
The simple objects of $\mI$ are $X$ and $\delta_a$, $a\in\BZ/2\BZ$. The tensor product of simple
objects is defined by
\[
\delta_a\ot\delta_b=\delta_{a+b},\quad \delta_a\ot X=X=X\ot\delta_a,\quad X\ot X=\delta_0\oplus\delta_1\, .
\]
The associativity isomorphism $(\delta_a\ot X)\ot \delta_b\iso \delta_a\ot (X\ot \delta_b)$ equals
$(-1)^{ab}\cdot\id_X$\,. The associativity isomorphism $(X\ot \delta_a)\ot X\iso X\ot (\delta_a\ot X)$
is the morphism $\bigoplus\limits_{b\in\BZ/2\BZ}\delta_b\to\bigoplus\limits_{b\in\BZ/2\BZ}\delta_b$
with components $(-1)^{ab}$. 
The associativity isomorphism $(X\ot X)\ot X\iso X\ot (X\ot X)$ is the morphism
$$\bigoplus\limits_{a\in\BZ/2\BZ}(\delta_a\ot X)\to\bigoplus\limits_{b\in\BZ/2\BZ}(X\ot \delta_b)$$
with matrix elements $(-1)^{ab}\cdot\lambda^{-1}$, $a,b\in\BZ/2\BZ$ (recall that $\delta_a\ot X=X=
X\ot \delta_b$).
The remaining 
associativity isomorphisms 
\[
(\delta_a\ot\delta_b)\ot\delta_c\iso\delta_a\ot (\delta_b\ot\delta_c),
\]
\[
(\delta_a\ot\delta_b)\ot X\iso\delta_a\ot (\delta_b\ot\ X), \quad
(X\ot\delta_a)\ot\delta_b\iso X\ot (\delta_a\ot\delta_b),
\]
\[
(\delta_a\ot X)\ot X\iso\delta_a\ot (X\ot X), \quad
(X\ot X)\ot\delta_a\iso X\ot (X\ot\delta_a)
\]
equal the identities.
\item[(iii)] According to the previous remark, the maximal pointed subcategory $\mI_{pt}\subset \mI$
is equivalent to $\Vec_{\BZ/2\BZ}$. Here is an {\em a priori\,} proof.
Since the fusion category $\mI_{pt}$ is pointed it is equivalent to $\Vec_G^\omega$, see 
\S\ref{terminology}. It is clear that $G\simeq \BZ/2\BZ$. We claim that the class of $\omega$
in $H^3(G,k^\times)\simeq \BZ/2\BZ$ is trivial. To see this, note that 
the functor $\Hom(X,?\otimes X): \mI_{ad}\to \Vec$ has a structure of tensor functor.
On the other hand, a tensor functor  $\Vec_G^\omega \to \Vec$ yields a representation of
the 3-cocycle $\omega$ as a coboundary.

\item[(iv)] Keep the notation of Remark (ii). Identify $X^*$ with $X$ so that 
$\coev_X:\delta_0\to X\ot X^*$ identifies with the embedding 
$\delta_0\hookrightarrow\delta_0\oplus \delta_1=X\ot X$. Then 
$\ev_X:X^*\ot X\to\delta_0$ identifies with $\lambda\cdot\pi$, where 
$\pi : X\ot X=\delta_0\oplus\delta_1\twoheadrightarrow\delta_0$ is the projection and
$\lambda$ has the same meaning as in Remark (ii). This follows from the description of
the associativity isomorphism $(X\ot X)\ot X\iso X\ot (X\ot X)$ in Remark (ii) or from
the fact that the composition \eqref{defoflambda} equals $\lambda$.
\end{enumerate}
\end{rems}

Let $\mI$ be an Ising fusion category. 
It is clear that the adjoint category $\mI_{ad}$ is generated by objects $\be$ and $\delta$. Hence 
$\mI_{ad}$ is pointed and $\mI$ is nilpotent of nilpotency class 2, see  \cite{GN}. 
(The nilpotency of $\mI$ is also clear since  it is a $p$-category for $p=2$ 
and $p$-categories are nilpotent by \cite[Theorem 8.28]{ENO}). 
In \S\ref{gradsubsec} we defined the notion of universal grading $\deg :\O(\C)\to U_\mI$. 
In our situation it has the following explicit description:
\begin{equation}  \label{theunivgr}
U_\mI =\BZ/2\BZ ,\quad  \deg(\be)=\deg(\delta)=0\in \BZ/2\BZ ,\; \deg(X)=1\in \BZ/2\BZ .
\end{equation}
By Proposition \ref{un-grad} (iii), 
$\Aut_\otimes (\id_\mI)=\Hom (U_\mI ,k^{\times})=\BZ/2\BZ$.

\subsection{Spherical structures on Ising categories}  \label{sphIs}
It follows from Corollary \ref{FPint3} that an Ising category $\mI$ has a spherical structure. Moreover, 
since $U_\mI =\BZ/2\BZ$ it is clear from \S \ref{pivotal} that $\mI$ has precisely two such
structures. For any of them $d(X)^2=\FPdim(X)^2=2$. Since $X\otimes X\simeq \be \oplus \delta$ this implies that $d(\delta)=1$. Therefore a spherical structure on $\mI$ is completely determined by the sign 
$\epsilon (\mI)=\pm 1$ in the equality $d(X)=\epsilon (\mI)\lambda (\mI)$. 

\begin{lemma}   \label{sphlemma}
Equip $\mI$ with the spherical structure such that $\epsilon (\mI)=1$. Then the corresponding
isomorphism $X\iso X^{**}$ equals $(\phi^*)^{-1}\phi$, where $\phi :X\to X^*$ is any isomorphism.
\end{lemma}

\begin{proof} 
Let $\psi_X:=(\phi^*)^{-1}\phi$. It suffices to check that the dimension
$d_+(X)$ corresponding to $\psi_X$ (see \S\ref{squared}) equals $\lambda(\mI)$.
This follows from the commutative diagram
\begin{equation*}
\xymatrix{
\be \ar[r]^(.4){\coev_X}\ar[d]^{\id}&X\ot X^*\ar[rr]^{(\phi^*)^{-1}\phi \ot \id_{X^*}}\ar[d]^{\id}&&
X^{**}\ot X^*\ar[r]^(.65){\ev_{X^*}}\ar[d]^{\phi^*\ot \phi^{-1}}&\be \ar[d]^{\id}\\
\be \ar[r]^(.4){\coev_X}&X\ot X^*\ar[rr]^{\phi \ot \phi^{-1}}&&X^*\ot X\ar[r]^(.65){\ev_X}&\be }
\end{equation*}
whose rows compute $d_+(X)$ and $\lambda(\mI)$, see formulas
\eqref{2 traces +} and \eqref{defoflambda}.
\end{proof}

\subsection{Ising braided categories}  \label{braidIs}
The following result is well known, see e.g. \cite[Appendix D]{MS}.

\begin{proposition} \label{existenceprop}
There exists an Ising braided category.\qed
\end{proposition}

\begin{remark}\label{quantumIs} One can construct examples of braided Ising categories (or even
modular ones) using the quantum group $U_qsl(2)$ with $q$ being a primitive 8-th root
of unity,  e.g. one can use \cite[Exercise 3.3.24(iii)]{BK} with $k=2$.
\end{remark}

\begin{corollary} \label{existencecor}
Each Ising category admits a structure of braided category.
\end{corollary}

\begin{proof} This follows from Proposition~\ref{existenceprop} because by Proposition \ref{sqrt2}, all Ising categories are conjugate under the action of the group of automorphisms of $k$.
\end{proof}

\begin{lemma} \label{adsuper}
Let $\mI$ be an Ising braided category. Then $\mI_{ad}$ is braided equivalent to $s\Vec$.
\end{lemma}

\begin{proof} Since $\delta \ot X\simeq X$ it is clear that $\delta$ projectively centralizes 
$X$, see Definition \ref{projcentdef1}. So by Proposition \ref{projective}, $\delta$ centralizes 
$X\ot X^*\simeq \be \oplus \delta$. The subcategory $\mI_{ad}$ is generated by $\delta$, hence 
it is symmetric. Thus by Theorem \ref{Deligne1}, either $\mI_{ad}\simeq \Rep(\BZ/2\BZ)$ or
$\mI_{ad}\simeq s\Vec$. 

The composition $\mI_{ad}\hookrightarrow\mI\hookrightarrow\Z (\mI )$ defines on $\mI$ a structure
of fusion category over $\mI_{ad}$ (see Definition~\ref{overdef}). So
if $\mI_{ad}\simeq \Rep(\BZ/2\BZ)$ then by Theorem~\ref{princprinc}, we have an equivalence
of fusion categories $\mI \simeq \C^{\BZ/2\BZ}$ for some fusion category $\C$. 
By Proposition \ref{equidimprop}, $\FPdim (\C)=2$, i.e., 
\[
\sum\limits_{X\in\O (\C)}\FPdim (X)^2=2.
\]
Now the properties of the Frobenius-Perron dimension (see \S\ref{fpd}) imply that
$\FPdim (X)=1$ for all $X\in\O (\C)$. So $\C$ is integral, which contradicts
Corollary~\ref{equiintegral}. Hence $\mI_{ad}\simeq s\Vec$.
\end{proof}

\begin{corollary} \label{isingnondeg}
 Any Ising braided category $\mI$ is non-degenerate in the 
sense of Definition \ref{defnondeg}.
\end{corollary}

\begin{proof} 
By \eqref{isingK0}, $\delta\ot X\simeq X$. Combining this with Lemmas~\ref{adsuper} and
\ref{Dimalem} we see that $\delta\notin\mI'$. Since $X\ot X$ contains $\delta$ this implies that
$X\notin\mI'$. So $\mI'=\Vec$. By Proposition \ref{Muger1prime}, this means that
$\mI$ is non-degenerate.
\end{proof}

\begin{corollary} \label{four}
\begin{enumerate}
\item[(i)] An Ising fusion category $\mI$ admits precisely 4 braided structures.
\item[(ii)] Let $c_{UV}$, $U,V\in\mI$ be a braiding. Then  the three other braidings on $\mI$
are
\begin{equation}  \label{thethreebraidings}
\bar c_{UV}:=c_{VU}^{-1}, \; c_{UV}':=(-1)^{\deg U\cdot\deg V}c_{UV},\;
\bar c_{UV}':=(-1)^{\deg U\cdot\deg V}c_{VU}^{-1}, 
\end{equation}
where $\deg$ is the universal grading \eqref{theunivgr}.
\end{enumerate}
\end{corollary}

\begin{proof} (i) By \S \ref{centre}, a braiding on $\mI$ is the same as a section of the forgetful functor
$F:\Z(\mI)\to \mI$. Here ``section" means an isomorphism class of tensor functors
$\Phi :\mI\to\Z (\mI )$ such that $F\Phi\simeq \id$ (or equivalently, a fusion subcategory 
$\D\subset\Z (\mI )$ such that $F|_{\D}:\D\to\mI$ is an equivalence).

Choose some braided structure $c$ on $\mI$ (it exists by Corollary~\ref{existencecor}).
By Corollary~\ref{isingnondeg} and Proposition~\ref{Muger1prime}, this choice
identifies $\Z(\mI)$ with $\mI \boxtimes \mI^{\oP}$ so that $F:\Z(\mI)\to \mI$ identifies
with the tensor product functor $\mI \boxtimes \mI^{\oP}\to\mI$.
A section $\mI \to \Z(\mI)=\mI \boxtimes \mI^{\oP}$ is
uniquely determined by the image of $X$, and we have 4 possibilities for this image:
$X\bt \be$,  $\be \bt X$, $X\bt \delta$, $\delta \bt X$. They correspond, respectively, to the
original braiding $c$ and the braidings \eqref{thethreebraidings}
\end{proof}

The proof of the following proposition 
is contained in \S \ref{difficult_proofs}.

\begin{proposition}\label{braidedising}
Let $\mI$ be a braided Ising category. Assume that
the braiding $c_{XX}$ acts on the one dimensional vector space 
 $\Hom (\be ,X\otimes X)$ as multiplication by a number $\zeta \in k$. Then $c_{XX}$ acts
 on  $\Hom (\delta ,X\otimes X)$ as multiplication by  $\zeta^{-3}$ and
 \begin{equation}  \label{16eq}
\zeta^2+\zeta^{-2}=\lambda,
\end{equation}
where $\lambda=\lambda ({\mI })\in k$ is the square root of 2 corresponding to $\mI$
(see \S\ref{fusionIsing}).
\end{proposition} 

\begin{remark}   \label{4-remark}
Let $\mI$ be an Ising fusion category and $\lambda\in k$ the corresponding square root of 2.
Let $\zeta$ be the solution to the equation \eqref{16eq} corresponding to a braiding $c$ on $\mI$
(see Proposition~\ref{braidedising}).
Then the solutions to the equation \eqref{16eq} corresponding to the three other braidings 
\eqref{thethreebraidings} are equal, respectively, to $\zeta^{-1}$, $-\zeta$, and $-\zeta^{-1}$. 
Thus we get a bijection between braided structures on $\mI$ and the solutions to
the equation \eqref{16eq}.
\end{remark}

Proposition \ref{braidedising} associates to each Ising braided category $\mI$ a solution 
$\zeta =\zeta (\mI)$ to the equation
\[
(\zeta^2+\zeta^{-2})^2=2, \quad \zeta\in k,
\]
i.e., an 8-th root of $-1$ in $k$. Combining Proposition \ref{sqrt2} and Remark~\ref{4-remark} 
one gets the following statement.

\begin{corollary}\label{braidedIs2}
The map $\mI\mapsto\zeta (\mI)$ is a bijection between the set of braided equivalence classes
 of Ising braided categories and the set of 8-th roots of $-1$ in $k$.  \hfill\qedsymbol
\end{corollary} 

The Ising braided category corresponding to $\zeta\in k$, $\zeta^8=-1$, will be denoted by~
$\mI_{\zeta}$.

\begin{rems}   \label{chbraiding}
(i) Following \cite[\S 5]{Sie}, we give here the explicit formulas for the braiding on $\mI_\zeta$ using 
the notation from Remark \ref{tyf} (ii). The braiding $c_{\delta_a\delta_b}$ equals 
$(-1)^{ab}\id_{\delta_{a+b}}$ and
the braiding $c_{XX}$ equals $\zeta\cdot\id_{\delta_0}\oplus\zeta^{-3}\cdot\id_{\delta_1}$
(this follows from Lemma~\ref{adsuper} and Proposition~\ref{braidedising}). Finally, 
according to \cite[\S 5]{Sie}, the braidings $c_{\delta_aX}$ and $c_{X\delta_a}$ equal $\zeta^{4a}\cdot\id_X$.

(ii) For any choice of braided structure on $\mI$ and any $Y\in\mI$ the square of the braiding 
$Y\ot \delta\iso\delta\ot Y\iso Y\ot \delta$ equals $(-1)^{\deg Y}$. This is clear from Remark~(i)
above. On the other hand, this follows from Corollary \ref{isingnondeg} and Lemma \ref{alpha}.
\end{rems}

\subsection{Proof of Proposition \ref{braidedising}}  \label{difficult_proofs}
\begin{lemma}  \label{prelimB10}
Let $\mI$ and $\zeta$ have the same meaning as in Proposition \ref{braidedising}. 
Equip $\mI$ with the spherical structure corresponding to $\epsilon=1$ (see \S\ref{sphIs}).
Then $\theta_\be =1$, $\theta_\delta =-1$ and $\theta_X=\zeta^{-1}$.
\end{lemma}

\begin{proof}
By Lemma \ref{adsuper}, $\mI_{ad}$ is a pointed braided category corresponding
to the pre-metric group $(\BZ/2\BZ,q),\; q(n)=(-1)^n$. By \S \ref{sphIs}, the spherical structure on
$\mI_{ad}$ is positive (i.e., $d(\delta)=1$), so the twists of the objects of $\mI_{ad}$ are given
by formula~\eqref{tw}. 

Formula~\eqref{deltatheta} says that the isomorphism $u_X:X\iso X^{**}$ defined by 
$\eqref{Drinfeld iso}$ equals $\theta_X^{-1}\cdot\psi_X$, where $\psi_X:X\iso X^{**}$ comes from the spherical structure. By Lemma~\ref{sphlemma}, 
if we identify $X^*$ with $X$ using some $\phi :X\iso X^*$ 
and identify $X^{**}$ with $X^*$ accordingly then $\psi_X$ identifies with $\id_X$; on the other hand,
$u_X$ identifies with the composition
$$X\xrightarrow{\id_X\ot \coev_X}X\ot (X\ot X)\iso (X\ot X)\ot X\xrightarrow{c_{XX}\ot \id_X}
(X\ot X)\ot X \xrightarrow{\ev_X\ot \id_X}X,$$
where $\coev_X:\be\to X\ot X$ and $\ev_X:X\ot X\to\be$ depend on the choice of $\phi$.
Now choose $\phi$ as in Remark \ref{tyf}(iv). Using Remarks \ref{tyf}(ii, iv) and the definition 
of $\zeta$ given in Proposition \ref{braidedising} one checks that $u_X: X\iso X$ equals 
$\zeta\cdot\id_X$. So $\theta_X^{-1}=\zeta$.
\end{proof}
Now let us prove Proposition \ref{braidedising}.
Equip $\mI$ with the spherical structure corresponding to $\epsilon=1$ (see \S\ref{sphIs}). Assume that 
$c_{XX}=\zeta\cdot\id_{\delta_0}+\xi\cdot\id_{\delta_1}$. Then  $c_{XX}^2=\zeta^2 \cdot\id_{\delta_0}+
\xi^2\cdot \id_{\delta_1}$. 
On the other hand, \eqref{balancing ax} says that 
$c_{XX}^2=\frac{\theta_\be}{\theta_X^2}\cdot\id_{\delta_0}+
\frac{\theta_\delta}{\theta_X^2} \cdot\id_{\delta_1}$. Therefore
\begin{equation}  \label{squares_opposite}
\zeta^2=-\xi^2. 
\end{equation}

Next let us compute $\Tr(c_{XX}^{-1})$ in two different ways. Clearly 
$\Tr(c_{XX}^{-1})=\zeta^{-1} +\xi^{-1}\cdot d(\delta )$. By \S\ref{sphIs}, $d(\delta )=1$, so
$\Tr(c_{XX}^{-1})=\zeta^{-1} +\xi^{-1}$. On the other hand, formula \eqref{traceofc} says that
$\Tr(c_{XX}^{-1})=\theta_X^{-1}d(X)$.
Recall from \S\ref{sphIs} and Lemma \ref{prelimB10} 
that $d(X)=\lambda$ and $\theta_X=\zeta^{-1}$. Thus we have
\begin{equation}\label{16eqprime}
\zeta^{-1} +\xi^{-1}=\zeta \lambda
\end{equation}
Since $\lambda^2=2$ (see \S \ref{sphIs}) we get $(\zeta^{-1} +\xi^{-1})^2=2\zeta^2$.
Combining this with \eqref{squares_opposite} we get  $\xi =\zeta^{-3}$.
Now \eqref{16eq} is an immediate consequence of \eqref{16eqprime}.

\subsection{Ising modular categories} \label{modularising}
Recall that a structure of modular category on a fusion category is the same as a
non-degenerate braided structure combined with a spherical structure
(see \S\ref{premod}).
By Corollary \ref{isingnondeg},
for Ising categories nondegeneracy is automatic. Now combining the results formulated
in \S\S\ref{fusionIsing}-\ref{braidIs} we get a bijection $\mI\mapsto (\zeta (\mI),\epsilon (\mI))$ 
between the set of modular equivalence classes of Ising modular categories and the
set of pairs $(\zeta, \epsilon )$, where $\zeta$ is an 8-th roots of $-1$ in $k$ and
$\epsilon =\pm 1$. The Ising modular category corresponding to $(\zeta, \epsilon )$
will be denoted by~$\mI_{\zeta ,\epsilon}$.

\begin{remark} One can show that for modular Ising categories obtained from
$U_qsl(2)$ (see Remark \ref{quantumIs}) we have $\epsilon =-1$.
\end{remark}

\begin{proposition} \label{B10}
The twists in $\mI_{\zeta,\epsilon}$ are as follows:
\begin{equation}  \label{Isingtwists}
\theta_\be =1,\; \theta_\delta =-1,\; \theta_X=\epsilon \zeta^{-1}.
\end{equation}
\end{proposition}

\begin{proof} By \S \ref{sphIs} a spherical structure on $\mI$ restricts to the positive spherical
structure on $\mI_{ad}$. Thus $\theta_\be$ and $\theta_\delta$ do not depend on choice
of $\epsilon$ and the result follows from Lemma \ref{prelimB10}.

Both $\theta_X$ and $\epsilon$ change their signs if one changes the spherical structure.
So the formula for $\theta_X$ follows from Lemma~\ref{prelimB10}.
\end{proof}

\begin{corollary}
The $S$-matrix of $\mI_{\zeta,\epsilon}$ equals
$$\left( \begin{array}{ccc}1&1&\epsilon (\zeta^2+\zeta^{-2})\\ 1&1&-\epsilon (\zeta^2+\zeta^{-2})\\
\epsilon (\zeta^2+\zeta^{-2})&-\epsilon (\zeta^2+\zeta^{-2})&0\end{array}\right) $$
Here we assume that the simple objects of
$\mI_{\zeta,\epsilon}$ are ordered as follows: $\be ,\delta, X$. 
\end{corollary}

\begin{proof} This is immediate from \eqref{sij} and the formula $\theta_\delta =-1$
(note that $\theta_X$ is not needed for this calculation).
\end{proof}

\begin{rems} \label{IsGauss}
\begin{enumerate}
\item[(i)] 
By \S\ref{sphIs}, $d(1)=d(\delta )=1$ and $d(X)=\epsilon\lambda=\epsilon (\zeta^2+\zeta^{-2})$,
so $d(X)^2=2$.
\item[(ii)] 
Combining Remark (i) with \eqref{Gauss sums} and \eqref{Isingtwists} one finds the 
Gauss sum of~$\mI_{\zeta,\epsilon}$\,:
\begin{equation}  
\tau^+(\mI_{\zeta,\epsilon})=2\epsilon \zeta^{-1}.
\end{equation}
In particular, $\tau^+(\mI_{\zeta,\epsilon})/2$ is an 8-th root of -1, i.e., 
a primitive 16-th root of unity.
\end{enumerate}
\end{rems}

\subsection{The tensor product of two Ising braided categories}  
Consider
two Ising braided categories $\mI_\zeta$ and $\mI_{\zeta'}$. We will describe 
the maximal Tannakian and maximal symmetric subcategory of 
$\mI_\zeta \boxtimes \mI_{\zeta'}$. Then we will describe the core
of $\mI_\zeta \boxtimes \mI_{\zeta'}$ 
(see Definition \ref{core corresponding to E} for the definition of core).

\begin{lemma} \label{Is-sym1}
\begin{enumerate}
\item[(i)] The adjoint subcategory
$(\mI_\zeta \boxtimes \mI_{\zeta'})_{ad}=(\mI_\zeta)_{ad}\boxtimes (\mI_{\zeta'})_{ad}$ 
is symmetric and has Frobenius-Perron dimension 4.
It is the unique maximal symmetric fusion subcategory of $\Z (\mI)$.
\item[(ii)] The fusion subcategory $\E\subset \mI_\zeta \boxtimes \mI_{\zeta'}$ generated by 
$\delta\boxtimes\delta$ is the unique maximal Tannakian
subcategory of $\mI_\zeta \boxtimes \mI_{\zeta'}$.
\item[(iii)] The centralizer $\E'$ equals the maximal integral subcategory 
$(\mI_\zeta \boxtimes \mI_{\zeta'})_{int}$
of $\mI_\zeta \boxtimes \mI_{\zeta'}$.
\end{enumerate}
\end{lemma} 

\begin{proof}
The first part of (i) follows from Lemma \ref{adsuper}. Let us prove that any
symmetric fusion subcategory $\D\subset \mI_\zeta \boxtimes \mI_{\zeta'}$ 
is contained in $(\mI_\zeta \boxtimes \mI_{\zeta'})_{ad}$.
By Corollary \ref{symint}, $\D$ is integral, so $\D\subset (\mI_\zeta \boxtimes \mI_{\zeta'})_{int}$.
By Theorem \ref{Muger2-FPdim}, $\FPdim (\D )\le 4$. 
Since $X\boxtimes X\in (\mI_\zeta \boxtimes \mI_{\zeta'})_{int}$ 
has Frobenius-Perron dimension 2 we see that
$X\boxtimes X\notin\D$. But all simple objects of $(\mI_\zeta \boxtimes \mI_{\zeta'})_{int}$ 
except $X\boxtimes X$
are contained in $(\mI_\zeta \boxtimes \mI_{\zeta'})_{ad}$, 
so $\D\subset (\mI_\zeta \boxtimes \mI_{\zeta'})_{ad}$.

We have proved (i). Statement (ii) follows from (i) and Lemma \ref{adsuper}.
Statement~(iii) follows from 
Remark~\ref{chbraiding}(ii).
\end{proof}

For each eighth root of unity $\xi$ we defined in Lemma \ref{uniformdescr}
a certain metric group $M_{\xi}$ of order $4$. Recall that $M_{\xi}$ has a distinguished
element $u$ of order $2$ and that $q(u)=-1$, $q(x)=\xi$ for $x\in M_{\xi}\setminus\{ 0,u\}$.
Now we will show that the core of $\mI_\zeta \boxtimes \mI_{\zeta'}$ is the
pointed braided category corresponding to a certain $M_{\xi}$.

\begin{lemma} \label{Is-core1}
Let $(\C ,\Gamma )$ be the core of $\mI_\zeta \boxtimes \mI_{\zeta'}$. Then

(i) the braided category $\C$ is pointed;

(ii) the metric group corresponding to $\C$ is isomorphic to $M_{\xi}$, where
\begin{equation}      \label{xiformula}
\xi =f(\zeta)f(\zeta')/2, \quad f(\zeta ):=\zeta^{- 1}(\zeta^2+\zeta^{-2})
\end{equation}
 
(iii) the subgroup $\Gamma\subset\Aut^{br}(\C )=\Aut (M_{\xi})$ equals $\Aut (M_{\xi},u)$
(i.e., the stabilizer of $u$).
\end{lemma}

Note that by Remarks \ref{AutMxi}(iii-iv),  $|\Aut (M_{\xi},u)|=2$ and if $\xi\ne -1$ then 
$\Aut (M_{\xi},u)=\Aut (M_{\xi})$.

\begin{proof} 
Let $\E\subset \mI_\zeta \boxtimes \mI_{\zeta'}$ 
be the Tannakian subcategory from Lemma \ref{Is-sym1}(ii).
Choosing a fiber functor $\E\to\Vec$ we identify $\E$ with $\Rep (\BZ/2\BZ )$.
By definition, $\C$ is the de-equivariantized category $\E'\boxtimes_{\E}\Vec$
(so  $\BZ/2\BZ$ acts on $\C$) and $\Gamma$ is the image of the homomorphism
$\BZ/2\BZ\to \Aut^{br}(\C )=\Aut (G,q)$. The category of equivariant objects $\C^{\BZ/2\BZ}$
identifies with $\E'$.

By Lemma \ref{Is-sym1}(iii), $\E'$ is integral. So $\C =\E'\boxtimes_{\E}\Vec$ is integral.
In particular, $\C$ is not Ising. But $\FPdim (\C )=4$, so $\C$ has to be pointed.
Since the braided category $\mI_\zeta \boxtimes \mI_{\zeta'}$ is non-degenerate so is $\C$. Therefore
$\C =\C (G,q)$ for some metric group $(G,q)$ of order~$4$. 

Now consider the maximal symmetric fusion subcategory of $\mI_\zeta \boxtimes \mI_{\zeta'}$
(see Lemma \ref{Is-sym1}(i)). De-equivariantizing it one gets a 
braided subcategory of $\C (G,q)$ equivalent to the category of super-vector spaces.
So we get a $\Gamma$-invariant element $u\in G$ of order $2$ such that $q(u)=-1$.
Therefore $(G,q)$ identifies with a metric group of the form $M_{\xi}$
(see statements (ii) and (iii) of Lemma \ref{uniformdescr}). To find $\xi$, we use the
equality  $\tau^+(M_{\xi})=2\xi$ from Lemma \ref{uniformdescr}(ii). Combining it with
Proposition  \ref{propGFG} and formula \eqref{T_of_Is} we see that 
$2\xi =\tau^+(M_{\xi})=f(\zeta)f(\zeta')$. 

Since $u$ is $\Gamma$-invariant $\Gamma\subset\Aut (M_{\xi},u)$.
But $|\Aut (M_{\xi},u)|=2$, so to compute $\Gamma$ it remains to show that the 
action of $\BZ/2\BZ$ on $G=\O(\C)$ is nontrivial.
If it were trivial the category of equivariant objects $\C^{\BZ/2\BZ}$ would be pointed
(here we use  that $H^2(\BZ/2\BZ,k^\times)=0$). But in fact, 
 $\C^{\BZ/2\BZ}=\E'$ is not pointed.
\end{proof}

\subsection{The center of an Ising category}   \label{Is-center}
Consider the center $\Z (\mI)$ of an Ising fusion category $\mI$. Corollary
\ref{isingnondeg} and Proposition \ref{Muger1prime} imply that a choice of a
braiding on $\mI$ produces a braided equivalence
$\Z (\mI)\simeq\mI \boxtimes \mI^{\oP}$. By Remark \ref{4-remark}, 
$(\mI_\zeta)^{\oP}\simeq \mI_{\zeta^{-1}}$.
Thus Lemmas \ref{Is-sym1} and \ref{Is-core1} imply the following one.

\begin{lemma} \label{Is-core}
\begin{enumerate}
\item[(i)] $\Z (\mI)$ has a unique maximal symmetric fusion subcategory.
This subcategory has Frobenius-Perron dimension 4.
\item[(ii)] Let $(\C ,\Gamma )$ be the core of $\Z (\mI )$ (see Definition \ref{core corresponding to E}). Then $\C$ is the pointed braided category $\C (G,q)$ corresponding to the hyperbolic metric group
\begin{equation}  \label{thehyperb}
G=(\BZ/2\BZ )^2, \quad q(m,n):=(-1)^{mn}, \; m,n\in\BZ/2\BZ 
\end{equation}
and the subgroup $\Gamma\subset\Aut^{br}(\C )=\Aut (G,q)$ equals $\Aut (G,q)$. $\square$
\end{enumerate}
\end{lemma} 

Note that the only nontrivial element of $\Aut (G,q)$ is the automorphism $(m,n)\mapsto (n,m)$.
It permutes the two Lagrangian subgroups of $(G,q)$. 

\section{The Altsch\"uler-Brugui\`eres characterization of centralizers}
\label{appensix Alt-Br}

Let $\C$ be a braided fusion category, and let $\D\subset \C$
be a fusion  subcategory. In this appendix  we recall and extend a characterization,
due to D.~Altsch\"uler and A.~Brugui\`eres \cite{AB},  of the centralizer $\D'$
in terms of the $S$-matrix of $\C$, (Proposition~\ref{AB criterion}). 

Let us recall some notation. Let $\C$ be a fusion category 
and let
\[
\mbox{ev}_X: X^*\ot X\to \be,\qquad \mbox{coev}_X: \be \to X\ot X^* 
\]
be the evaluation and coevaluation maps. Given a natural (but not necessarily
tensor) isomorphism
$\psi_X: X  \xrightarrow{~} X^{**}$ we define new morphisms
\begin{eqnarray*}
\mbox{ev}'_X &:=& \mbox{ev}_{X^*}(\psi_X \ot \id_{X^*}) : X\ot X^*\to \be,\\
\mbox{coev}'_X &:=&  (\id_{X^*}\ot \psi_X^{-1}) \mbox{coev}_{X^*}:   \be \to X^*\ot X.  
\end{eqnarray*}
For any morphism $f: X\to X$ we defined 
in \eqref{2 traces +} and \eqref{2 traces -} two traces, $\Tr_+(f)$ and $\Tr_-(f)$ by
\begin{eqnarray*}
\Tr_+(f) &=& \mbox{ev}'_X \circ (f \ot \id_{X^*}) \mbox{coev}_X, \\
\Tr_-(f) &=& \mbox{ev}_X \circ (\id_{X^*}\ot f) \mbox{coev}'_X.
\end{eqnarray*}
We also set $d_\pm(X) =\Tr_\pm(\id_X)$. Note that
\[
\sum_{X\in \O(\C)}\, d_+(X) d_-(X) =\dim(\C), 
\]
where the latter is the categorical dimension of $\C$, see \S\ref{squared}.

Lemma~\ref{handleslide} below was proved in  \cite{AB} under the assumption 
that $\psi$ is an isomorphism of tensor functors, i.e., a pivotal structure.
Our proof is a slight modification of that in \cite{AB}.
Let $\O(\C)$ denote the set of simple objects of $\C$.

\begin{lemma}
\label{handleslide}
Let $\C$ be a  fusion category (not necessarily pivotal) 
and let $\D \subset \C$ be a fusion subcategory.
Let $V$ be a fixed object of $\C$ and let $\beta_U: U\ot V \xrightarrow{\sim}
U\ot V,\,U\in \D$  be a natural family of isomorphisms. Define
\begin{equation}
H_U: =   \sum_{Y\in \O(\D)}\, d_+(Y) (\Tr_- \ot \id_{U\ot V})\beta_{Y\ot U}.
\end{equation}
Then $H_U =\id_U \ot H_{\be}$.
\end{lemma}
\begin{proof}
For any object $Y$ in $\D$ and any $X\in \O(\D)$ let $\pi^{(X)}_Y :Y \to Y$
be the idempotent endomorphism of $Y$  corresponding to the projection 
on the maximal multiple of $X$ in $Y$.
Define auxiliary morphisms $A_X^Y,\, B_X^Y: U\ot V \to U\ot V$ by
\begin{eqnarray*} 
A_X^Y &:=& (\id_U \ot \mbox{ev}_X \ot \id_V )( \pi^{(Y^*)}_{U\ot X^*} \ot \beta_X) (\id_U \ot \mbox{coev}'_X \ot \id_V ),\\
B_X^Y &:=& (\mbox{ev}_Y \ot \id_{U\ot V}) (\id_{Y^*} \ot \beta_{Y\ot U}(\pi^{(X)}_{Y\ot U} \ot \id_V) ) (\mbox{coev}'_Y \ot \id_{U\ot V}).
\end{eqnarray*}
To prove the Lemma it suffices to check the identity
\begin{equation}
\label{AB identity}
d_+(X)  A_X^Y = d_+(Y) B_X^Y,\quad X,Y\in\O(\D),
\end{equation}
since then the result follows by taking the sum of both sides of
\eqref{AB identity} over $X,Y\in\O(\D)$.

If $X,Z$ are objects in $\C$ then there is a non-degenerate pairing
\begin{equation}
\label{XZ pairing}
\Hom(X,\,Z) \times \Hom(Z,\,X): (f,g)\mapsto \Tr_+(fg).
\end{equation}
For any basis $\{f_i\}_{i\in I}$ of  $\Hom(X,\,Z)$ let 
$\{f^i\}_{i\in I}$ be the basis of $\Hom(Z,\,X)$ dual to $\{f_i\}_{i\in I}$
with respect to the pairing \eqref{XZ pairing}.
In the case when $X$ is simple we have the following  straightforward identities:
\[
f^if_j = d_+(X)^{-1} \delta_{ij} \id_X,\qquad  \pi^{(X)}_Z = d_+(X) \sum_{i\in I}\, f_if^i.
\]
Let us take $Z= Y\ot U$ above. Using naturality of $\beta$, we obtain
\begin{eqnarray*}
\beta_{Y \ot U} (\pi^{(X)}_{Y\ot U} \ot \id_V)
&=& d_+(X) \sum_{i\in I}\, \beta_{Y \ot U} ( f_if^i \ot \id_V) \\
&=& d_+(X) \sum_{i\in I}\, ( f_i \ot \id_V)  \beta_X( f^i \ot \id_V).
\end{eqnarray*}
For every $i\in I$ consider morphisms
\begin{eqnarray*} 
g_i &:=&  (\mbox{ev}_Y \ot \id_{U\ot X^*})(\id_{Y^*} \ot f_i \ot \id_{X^*}  ) ( \id_{Y^*}  \ot \mbox{coev}_X): Y^* \to U \ot X^*, \\
g^i &:=&  (\id_{Y^*} \ot \mbox{ev}'_X) (\id_{Y^*} \ot f^i \ot \id_{X^*}  ) (\mbox{coev}'_Y \ot \id_{U\ot X^*}) : U \ot X^*\to Y^*.
\end{eqnarray*}
Using the axioms of evaluation and coevaluation we obtain
\begin{eqnarray*} 
(\id_U \ot \mbox{ev}_X) (g_i \ot \id_X) &=&  (\mbox{ev}_Y \ot \id_{U})(\id_{Y^*} \ot f_i), \\
(g^i \ot \id_X) (\id_U \ot \mbox{coev}'_X) &=&  (\id_{Y^*} \ot f^i) (\mbox{coev}'_Y \ot \id_{U}).
\end{eqnarray*}
Using definitions of $\Tr_\pm$ and properties of evaluation and coevaluation maps
one computes 
\[
\Tr_-(g^i g_j) = \frac{d_-(Y^*)}{d_+(Y)} \Tr_+(f^i f_j)= \frac{d_-(Y^*)}{d_+(Y)}  \delta_{ij}.
\]
Hence,  $g^ig_j = d_+(Y)^{-1}  \delta_{ij} \id_{Y^*}$ and, therefore,
$\sum_{i\in I}\, g_ig^i = d_+(Y)^{-1} \pi^{(Y^*)}_{U\ot X^*}$.
Using previousy obtained identities we compute
\begin{eqnarray*} 
B_X^Y 
&=&  d_+(X) (\mbox{ev}_Y \ot \id_{U\ot V}) (\id_Y \ot \sum_{i\in I}\, ( f_i \ot \id_V)  \beta_X( f^i \ot \id_V) ) (\mbox{coev}'_Y \ot \id_{U\ot V}) \\
&=&  d_+(X) (\id_U \ot \mbox{ev}_X \ot \id_V ) ( \sum_{i\in I}\, g_ig^i \ot \beta_X )  (\id_U \ot \mbox{coev}'_X \ot \id_V )\\
&=&  d_+(X) (\id_U \ot \mbox{ev}_X \ot \id_V ) (d(Y)^{-1} \pi^{(Y^*)}_{U\ot X^*} \ot \beta_X) (\id_U \ot \mbox{coev}'_X \ot \id_V ) \\
&=&  d_+(X) d_+(Y)^{-1} A_X^Y,
\end{eqnarray*}
and the result follows.
\end{proof}

\begin{remark}
Lemma~\ref{handleslide} can be thought of as a categorical analogue of the Haar
theorem. Indeed, for $\C=\Vec_G$, where $G$ is a finite group, it says that
the counting measure is invariant under translations.
\end{remark}

Let $\C$ be a braided fusion  category with braiding $c$. 
Let $\tilde{S}=\{\tilde s_{XY}\}_{X,Y\in \O(\C)}$ be the $\tilde{S}$-matrix
of $\C$ defined in \S\ref{tildeS matrix}.
Let $\D \subset \C$ be a fusion subcategory.
The next Proposition was proved by D.~Altsch\"uler and A.~Brugui\`eres 
in \cite{AB} for pre-modular categories (see \cite{Mu2} for the proof
in the special case of pseudo-unitary braided fusion categories). 

\begin{proposition}
\label{AB criterion}
Let $\C$ be a braided fusion category, let $V$ be a simple object in $\C$,
and let $\D \subset \C$ be a fusion subcategory.
We have the following dichotomy:
\begin{enumerate}
\item[(i)] $V\in \O(\D')$ if and only if  
     $\tilde{s}_{YV} = 1$ for all $Y \in \O(\D)$,
\item[(ii)] $V\not\in \O(\D')$ if and only if  
     $\sum_{Y\in \O(\D)}\, |Y|^2 \tilde{s}_{YV} =0$ and if and only if \\
$\sum_{Y\in \O(\D)}\, \tilde{s}_{VY} |Y|^2 =0$.
\end{enumerate}
\end{proposition}

\begin{proof}
Note that conditions in the statement of the Proposition are independent from
the choice of a natural isomorphism $\psi_X: X \xrightarrow{\sim} X^{**}$.

Fix $V\in \O(\C)$ and set  $\beta_U =  c_{V,U}c_{U,V}$ in Lemma~\ref{handleslide}. 
Then
\begin{eqnarray*}
H_U
&=& \sum_{Y\in \O(\D)}\, d_+(Y) (\Tr_- \ot\id_{U \ot V}) (c_{V,Y\ot U} c_{Y\ot U,V})\\
&=& \sum_{Y\in \O(\D)}\, d_+(Y) (\Tr_- \ot\id_{U \ot V}) (\id_{Y} \ot c_{V, U}) (c_{V,Y}c_{Y,V} \ot \id_U)
    (\id_{Y} \ot c_{U, V})\\
&=& (\id_U \ot H_\be) c_{V, U} c_{U, V}.
\end{eqnarray*}
We have $H_\be =  \sum_{Y\in \O(\D)}\, d_+(Y) (\Tr_- \ot\id_{V}) (c_{V,Y} c_{Y,V}): V \to V$.
Applying $\Tr_+$  we  obtain $H_\be = \sum_{Y\in \O(\D)}\, |Y|^2 \tilde{s}_{YV} \,\id_V$. 
Combining this with Lemma~\ref{handleslide} we obtain
\begin{equation}
\label{AB dichotomy}
\left(   \sum_{Y\in \O(\D)}\,  |Y|^2 \tilde{s}_{YV}  \right) (c_{V,U} c_{U,V} -\id_{U\ot V}) =0,
\qquad \mbox{ for all } U \in \O(\D).
\end{equation}
Note that when  $\tilde{s}_{YV} = 1$ for all $Y\in \O(\D)$ the first factor in \eqref{AB dichotomy}
equals $\dim(\D)\neq 0$ and hence, $c_{V,U} c_{U,V} = \id_{U\ot V}$ for all $U \in \O(\D)$,
which proves (i). If the second factor is not zero for some $U\in \O(\D)$ then the first factor must vanish, 
and the first part of (ii) follows. Finally, the second part of (ii)  is obtained by 
replacing $\C$ by its reverse category $\C^\rev$ (i.e., the category with reversed tensor
product).
\end{proof}

\section{Pointed braided categories and pre-metric groups}  \label{likbez}
As explained in \S\ref{pbcpmg}, each pointed braided category gives rise to a
pre-metric group. We will prove that each pre-metric group $(G,q)$ comes from
some pointed braided category $\C$  (this is a part of Proposition 
\ref{braided metric}). The proof below is different from those
given in \cite[\S 3]{JS} and \cite[Proposition 2.5.1]{Q}.

\subsection{Easy case}  \label{easycase}
Suppose that the quadratic form $q:G\to k^{\times}$ can be represented as $q(g)=\beta (g,g)$
for some bicharacter $\beta :G\times G\to k^{\times}$ (the existence of $\beta$ is automatic if 
$|G|$ is odd but not in general). Define a braided fusion category 
$\C_{\beta }$ as follows:

\begin{enumerate}
\item[(i)] as a fusion category, $\C_{\beta }$ equals $\Vec_G$, i.e., the category of finite-dimensional
$G$-graded vector spaces;
\item[(ii)] if $X_1,X_2\in\Vec_G=\C_{\beta }$ and $x_i\in X_i$ is homogeneous of degree $g_i\in G$
then the brading $X_1\ot X_2\iso X_2\ot X_1$ takes $x_1\ot x_2$ to $\beta (g_1,g_2) x_2\ot x_1$.
\end{enumerate}
Then $\C_{\beta }$ is a pointed braided category, and the corresponding pre-metric group
equals $(G,q)$.

\subsection{General case}   \label{generalcase}
\begin{lemma}
There exist a finite abelian group $\widetilde{G}$ and an epimorphism $\pi :\widetilde{G}\to G$
such that the quadratic form $\widetilde{q}:=q\circ\pi$ on $\widetilde{G}$
can be represented as $\widetilde{q}(x)=\widetilde{\beta} (x,x)$
for some bicharacter $\widetilde{\beta} :\widetilde{G}\times \widetilde{G}\to k^{\times}$. 
\end{lemma}

\begin{proof}
Represent $G$ as $\Gamma/H$, where $\Gamma$ is a free abelian group.
The pullback of $q$ to $\Gamma$ can be represented as $x\mapsto \alpha (x,x)$
for some bicharacter $\alpha :\Gamma\times\Gamma\to k^{\times}$ whose
values are roots of unity. There exists a subgroup $H'\subset H$ of finite index
such that $\alpha$ comes from a bicharacter 
$\widetilde{\beta} :(\Gamma/H')\times (\Gamma/H')\to k^{\times}$. Set $\widetilde{G}:=\Gamma/H'$.
\end{proof}

Choose some $\widetilde{G}$, $\pi$, and $\widetilde{\beta}$. Let $\C_{\widetilde{\beta }}$ be the pointed braided
category with $\O (\C_{\widetilde{\beta }})=\widetilde{G}$ constructed in \S\ref{easycase}.
Set $A:=\Ker\pi\subset \widetilde{G}$. Let $\E\subset\C_{\widetilde{\beta }}$ 
be the fusion subcategory such that $\O (\E )=A$.
Let $\C'_{\widetilde{\beta }}$ denote the centralizer of $\C_{\widetilde{\beta }}$
(see \S\ref{centra-drinf}).

\begin{lemma}
(i) $\E\subset\C'_{\widetilde{\beta }}$.

(ii) The symmetric fusion category $\E$ is Tannakian.
\end{lemma}

\begin{proof}
(i) One has to show that ${\widetilde{\beta }}(x,y){\widetilde{\beta }}(y,x)=1$ for
$x\in\widetilde{G}$, $y\in A$. This is clear because
${\widetilde{\beta }}(x,y){\widetilde{\beta }}(y,x)=
\widetilde{q}(x+y)\widetilde{q}(x)^{-1}\widetilde{q}(y)^{-1}$ 
and $\widetilde{q}$ is the pullback of a quadratic form on $G$.

(ii) Use that $\widetilde{q}|_A=1$.
\end{proof}

Now choose a fiber functor $\E\to\Vec$. Then the fiber category
\begin{equation}  \label{thebrcat}
\C_{\widetilde{\beta }}\boxtimes_{\E}\Vec
\end{equation}
(see \S\ref{deeqfus}-\ref{deeqbr}) is a pointed braided category, 
and the corresponding pre-metric group equals $(G,q)$.

\begin{remark}  \label{whatnext}
The groupoid of fiber functors $F:\E\to\Vec$ canonically identifies with the groupoid
$\Gamma_{\widetilde{\beta }}$ formed by central extensions
\begin{equation}     \label{theextension}
0\longrightarrow k^{\times}\longrightarrow\widehat{A} {\buildrel{p}\over{\longrightarrow}}  
A\longrightarrow 0  
\end{equation}
such that
\begin{equation}     \label{commutator}
\widehat{a}_2\widehat{a}_1=\widetilde{\beta }(a_1,a_2) \widehat{a}_1\widehat{a}_2, 
\quad \widehat{a}_i\in\widehat{A}, \; a_i:=p(\widehat{a}_i)\in A
\end{equation}
(given $F:\E\to\Vec$ one defines $\widehat{A}$ to be the group of isomorphism classes of
pairs consisting of $a\in A$ and an isomorphism $k\iso F(\delta_a)$;
here $\delta_a\in\E\subset\Vec_{\widetilde{G}}$ is the vector space $k$ placed in degree $a$).
So the construction of the braided category \eqref{thebrcat} depends on the choice
of $\widetilde{G}$, $\pi$, $\widetilde{\beta}$, and an object of $\Gamma_{\widetilde{\beta }}$.
In the next subsection we describe an explicit construction of a pointed braided category $\C$,
which depends on the same choices. In fact, $\C$ canonically identifies with \eqref{thebrcat}.
\end{remark}

\subsection{Elementary reformulation}
Let $\widetilde{G}$, $\pi$, $\widetilde{\beta}$, and $A$ have the same meaning
as in \S\ref{generalcase}. Let $\Gamma_{\widetilde{\beta }}$ be the groupoid defined in 
Remark \ref{whatnext}. Note that $\Gamma_{\widetilde{\beta }}\ne\emptyset$
(because $\widetilde{\beta}(x,x)=1$ for $x\in A$) and all objects of 
$\Gamma_{\widetilde{\beta }}$ are isomorphic to each other 
(because $k^{\times}$ is divisible and therefore $\text{Ext} (A,k^{\times})=0$). 
Choose an object of $\Gamma_{\widetilde{\beta }}$, i.e.,
a central exten\-sion~\eqref{theextension}.

Now let $\C$ be the category of finite-dimensional $\widetilde{G}$-graded $\widehat{A}$-bimodules
$M=\bigoplus\limits_{x\in\widetilde{G}}M_x$ such that
\begin{enumerate}
\item[(i)] if $x\in\widetilde{G}$, $m\in M_x$, and  $\widehat{a}\in\widehat{A}$ 
maps to $a\in A$ then
 $\widehat{a}m\in M_{ax}$ and $m\widehat{a}=\widetilde{\beta}(a,x)\widehat{a}m$,
\item[(ii)] each $\lambda\in k^{\times}\subset\widehat{A}$ acts on $M$ as multiplication
by $\lambda$.
\end{enumerate}
Clearly $\C$ is a fusion category with respect to $\ot_{\widehat{A}}:\C\times\C\to\C$.
It has a braided structure such that if $M,M'\in\C$, $m\in M_x$, $m'\in M'_{x'}$ then the
braiding $M\ot_{\widehat{A}}M'\iso M'\ot_{\widehat{A}}M$ takes $m\ot m'$ to 
$\widetilde{\beta}(x,x') m'\ot m$. In fact, $\C$ is a pointed braided category, and the 
corresponding pre-metric group equals $(G,q)$.

\begin{rems}
\begin{enumerate}
\item[(i)] One can choose $\widetilde{G}$, $\pi$, and $\widetilde{\beta}$ so that
$\widetilde{\beta }|_{A\times A}=1$ (represent $G$ as $\BZ/n_1\BZ\oplus\ldots \BZ/n_r\BZ$ 
and take $\widetilde{G}:=\BZ/2n_1\BZ\oplus\ldots \BZ/2n_r\BZ$).  In this case one can take 
\eqref{theextension} to be the trivial extension.
\item[(ii)] F.~Quinn \cite{Q} writes explicit formulas for $G=\BZ/2^k\BZ$ by taking 
$\widetilde{G}=\BZ/2^{k+1}\BZ$ and using a certain map 
$\BZ/2^{k}\BZ\hookrightarrow\BZ/2^{k+1}\BZ$ such that the composition
$\BZ/2^{k+1}\BZ\twoheadrightarrow\BZ/2^{k}\BZ$ equals the identity (see \cite[\S 2.5]{Q}).
\end{enumerate}
\end{rems}

\section{Extensions of groups by braided gr-categories}    \label{grcatsect}
\subsection{Conventions}      \label{Conv}
In this section monoidal categories are {\em not assumed to be additive.} 
Unless stated otherwise, we assume that all categories are small (i.e., objects form a set rather
than a class).

If $G$ is a monoid (e.g., a group) then the discrete category corresponding to the set $G$ is monoidal.
We use the same symbol $G$ to denote this  monoidal category.

\subsection{Generalities on monoidal categories}        \label{Mongen}
\subsubsection{$\pi_0$ and $\pi_1$}   \label{homotgroups}
The set of isomorphism classes of objects of a category $\C$ will be denoted by
$\pi_0(\C )$. If $\C$  is a monoidal category then $\pi_0(\C )$ is a monoid.

The monoid of endomorphisms of the unit object of a monoidal category $\C$ will be denoted by 
$\pi_1(\C )$. It is well known that $\pi_1(\C )$ is commutative.

\subsubsection{The category of monoidal functors}  \label{catmon}
If $\C_1$ and $\C_2$ are monoidal categories then $\Func (\C_1 ,\C_2 )$ will denote the category of monoidal functors $\C_1\to\C_2$. In $\Func (\C_1 ,\C_2 )$ there is a distinguished object, namely,
{\em the trivial functor\,} (this is the monoidal functor $G:\C_1 \to\C_2$ that takes every object of $\C_1$ 
to $\be\in\C_2$, every morphism of $\C_1$ to $\id :\be\to\be$, and the isomorphism 
$G(X\ot Y)\iso G(X)\ot G(Y)$ equals $\id :\be\to\be$ for all $X,Y\in\C_1$.) A {\em trivialization} of a 
monoidal functor $F:\C_1 \to\C_2$ is an isomorphism between $F$ and the trivial monoidal functor.

\subsubsection{Gr-categories and the adjoint action} \label{Gr-subs}
\begin{definition}   \label{Gr-def}
A {\em gr-category\,} is a monoidal groupoid $\C$ such that $\pi_0(\C )$ is a group.
\end{definition}

The name ``gr-category" (or ``Gr-category") goes back to Ho\`ang Xu\^an S\'inh, a student of Grothendieck. Gr-categories are also known as ``categorical groups" or ``2-groups".

If $\C$ is a gr-category then for every $X\in\C$ we have a functor $\Ad_X :\C\to\C$,  
$\Ad_X (Y):=X\ot Y\ot X^{-1}$. This functor has a canonical monoidal structure. Moreover,
the functor $\Ad :\C\to\Func (\C,\C )$, $X\mapsto\Ad_X$, has a canonical monoidal structure. In other words, we get an action of $\C$ on itself by monoidal auto-equivalences.
It induces an action of $\pi_0(\C )$ on $\pi_1(\C )$.

\begin{remark}   \label{adbr}
There is a one-to-one correspondence between braided structures on  a gr-category $\C$
and trivializations of the adjoint action $\Ad :\C\to\Func (\C ,\C)$ (rewrite the braiding 
$c_{X,Y}:X\ot Y\iso Y\ot X$ as $\Ad_X (Y)\iso Y$).
\end{remark}

\subsubsection{Gradings}       \label{Grad}
Let $\C$ be a monoidal category. A monoidal functor $\partial :\C\to G$ is also called a 
{\em $G$-grading\,} of $\C$. If $\partial$ is surjective the grading is said to be {\em faithful\,}
and $\C$ is said to be {\em faithfully $G$-graded.}
We can think of a grading as a decomposition $\C=\bigsqcup\limits_{g\in G}\C_g$, 
$\C_g:=\partial^{-1}(g)$. The full subcategory $\C_1:=\partial^{-1}(1)$ is monoidal; it is called
the {\em trivial component\,} of the grading.

\subsubsection{Central subcategories}           \label{Central}
\begin{definition}           \label{Centdef}
Let $\C$ be a monoidal category, $\D\subset\C$ a full monoidal subcategory, and $i:\D\to\C$
the embedding. A {\em central structure\,} on $\D$ is a lift of $i$ to a monoidal functor
$\D\to\Z (\C )$, where $\Z (\C )$ is the center of $\C$. A {\em central subcategory\,} of $\C$
is a full subcategory $\D\subset\C$ with a central structure.
\end{definition}

\begin{rems}  \label{centrex}
\begin{enumerate}
\item[(i)] Explicitly, a central structure on $\D\subset\C$ is a functorial collection of isomorphisms 
\begin{equation}  \label{centreq}
c_{Y,X}:Y\ot X\iso X\ot Y, \quad X\in \D, \; Y\in\C
\end{equation}
satisfying the hexagon axioms.
\item[(ii)] A central structure on $\D\subset\C$ defines a braided structure on $\D$ (it is induced
by the one on $\Z (\C )$ via the fully faithful embedding $\D\hookrightarrow\Z (\C )$). 
If $\D =\C$ then a central structure on $\D$ is the same as a braided structure.
\item[(iii)] Let $\C$ be a gr-category.  Rewriting \eqref{centreq} as $Y\iso \Ad_X (Y)$ we see that a central structure on $\D\subset\C$ is the same as a
trivialization of the adjoint action $\Ad :\D\to\Func (\C ,\C)$.
\item[(iv)] Suppose that $\C$ is a gr-category and $\pi_0(\D )$ is normal in $\pi_0(\C )$.
Then we have the adjoint action $\Ad :\C\to\Func (\D ,\D)$. Rewriting \eqref{centreq} 
as $\Ad_Y (X)\iso X$ we see that a central structure on $\D\subset\C$ is the same as a
trivialization of this action.
\item[(v)] If $\D\subset\C$ is a central subcategory then $\pi_0(\D )$ is contained in the center of
$\pi_0(\C )$. If, in addition, $\C$ is a gr-category then the action of $\pi_0(\C )$ on 
$\pi_1(\D )=\pi_1(\C )$ is trivial by Remark (iv). (In fact, it is easy to check that 
the triviality of the action of $\pi_0(\C )$ on $\pi_1(\C )$ is equivalent to the existence of a central
structure on the full subcategory $\{\be\}\subset\C$).
\end{enumerate}
\end{rems}

\subsection{A lemma on gr-categories}   \label{subtledef}
Let $\C$ be a gr-category with a faithful grading $\C\to G$ and let $\C_1\subset\C$ be its trivial
component. For any monoidal category $\M$, we will describe $\Func (G,\M )$ as the category
of monoidal functors $\C\to\M$ equipped with a certain additional structure.

If $F\in\Func (\C,\M )$ let $T_{\C_1}(F)$ be the set of trivializations of the restriction 
$F|_{\C_1}:\C_1\to\M$  (see \S\ref{catmon}).  The gr-category $\C$ acts on the category 
$\Func (\C_1,\M )$, namely, $X\in\C$ takes $\Phi\in\Func (\C_1,\M )$ to 
$\Ad_{F(X)}\circ\Phi\circ\Ad_X^{-1}$. Moreover, $F|_{\C_1}$ has a canonical equivariant structure with 
respect to this action. So $\pi_0(\C)$ acts on $T_{\C_1}(F)$. 

\begin{lemma}  \label{trivaction}
The action of $\pi_0(\C)$ on $T_{\C}(F)$ is trivial. 
 \hfill\qedsymbol
\end{lemma}

The lemma implies that the action of $\pi_0(\C_1)\subset\pi_0(\C)$ on  $T_{\C_1}(F)$ is trivial, 
so we get an action of $G$ on $T_{\C_1}(F)$.
Let $\Func_{\C_1}(\C,\M )$ (resp. $\Func_{\C_1}^G(\C,\M )$) 
be the category of pairs $(F,\alpha )$, where $F\in\Func (\C,\M )$ and $\alpha\in T_{\C_1}(F)$ 
(resp. $\alpha\in T_{\C_1}(F)^G$).

\begin{lemma}  \label{subtlelem}
The functor $\Func (G ,\M )\to \Func_{\C_1}(\C  ,\M )$ is fully faithful. Its essential image equals
$\Func_{\C_1}^G(\C,\M )$. \hfill\qedsymbol
\end{lemma}

\begin{proof}
A monoidal functor $G\to\M$ is the same as a monoidal functor $F:\C\to\M$ with a collection
of isomorphisms $\varphi_{c,c'}:F(c)\iso F(c')$ defined for all objects $c,c'\in\C$ with equal images in
$G$ and satisfying the equations
\begin{equation}   \label{E61}
\varphi_{c',c''}\circ\varphi_{c,c'}=\varphi_{c,c''} \, ,
\end{equation}
\begin{equation}   \label{E62}
\varphi_{c_1,c_1'}\ot\varphi_{c_2,c_2'}=\varphi_{c_1\ot c_2,c_1'\ot c_2'} \, .
\end{equation}
Such a collection is uniquely determined by the isomorphisms 
$\alpha_X:=\varphi_{\be ,X}:\be\iso F(X)$, $X\in\C_1$. Namely, \eqref{E61} -- \eqref{E62} imply that
\begin{equation}   \label{E63}
\varphi_{c,c'}=\id_{F(c)}\ot\alpha_{c^{-1}\ot c'} \, .
\end{equation}
Substituting \eqref{E63} into  \eqref{E61} -- \eqref{E62}, we get the following equations for the isomorphisms $\alpha_X$:
\begin{equation}   \label{E64}
\alpha_{X\ot Y} =\alpha_X\ot\alpha_Y\, ,\quad X,Y\in\C_1\, ,
\end{equation}
\begin{equation}   \label{E65}
\alpha_{c\ot X\ot c^{-1}}=\id_{F(c)}\ot\alpha_X\ot\id_{F(c^{-1})}\, ,\quad c\in\C, X\in\C_1\, .
\end{equation}
Equation \eqref{E64} means that the isomorphisms $\alpha_X$, $X\in\C_1$, define an element
$\alpha\in T_{\C_1}(F)$. Equations \eqref{E64} -- \eqref{E65} mean that $\alpha\in T_{\C_1}(F)^G$.
\end{proof}

\begin{remark} 
$T_{\C_1}(F)$ (resp. $T_{\C_1}(F)^G$) can also be described as the set of lifts of the 
inclusion $i:\C_1\hookrightarrow\C$ to a monoidal functor $\C_1\to\Ker F:=$\{fiber of $F$ over $\be$\}
(resp. to a monoidal functor $\C_1\to\Ker F$ equivariant with respect to the adjoint action of $\C$ on
$\C_1$ and on $\Ker F$). 
\end{remark}

\subsection{The notion of extension by a braided gr-category}    \label{thenot}
Let $G$ be a group and $\B$ a braided gr-category equipped with a braided action 
$a:G\to\Func (\B ,\B )$.

\begin{definition}   \label{extdef}
An {\em extension of $G$ by $\B$\,} is a triple consisting of a
gr-category $\C$, a faithful grading $\partial:\C\twoheadrightarrow G$ with trivial component $\B$, and an isomorphism $f$ between the monoidal functor $\Ad :\C\to\Func (\B ,\B )$ and the composition 
$\C {\buildrel{\partial}\over{\longrightarrow}}G {\buildrel{a}\over{\longrightarrow}}\Func (\B ,\B )$ such 
that the restriction of $f$ to $\B\subset\C$ equals the trivialization of $\Ad :\B\to\Func (\B ,\B )$ corresponding to the braiding on $\B$ (the correspondence was defined in Remark~\ref{adbr}).
\end{definition}

\begin{rems}
\begin{enumerate}
\item[(i)] To simplify the notation, we understand the words ``grading with trivial component $\B$" in the
{\em strict\,} sense (i.e., the trivial component $\C_1\subset\C$ {\em equals\,} $\B$). The reader
may prefer to understand them in the ``weak" sense (a monoidal equivalence $\B\iso\C_1$ as
an additional datum).
\item[(ii)] If $\pi_1(\B)=0$ then $\B$ is an abelian group and Definition \ref{extdef} gives the usual
notion of extension of a group $G$ by a $G$-module.
\end{enumerate}
\end{rems}

Definition \ref{extdef} is justified by the next proposition, which is standard if $\pi_1(\B)=0$.

\begin{proposition}
Let $\C$ be a gr-category equipped with a faithful grading $\partial:\C\to G$. Suppose that
its trivial component $\B\subset\C$ is equipped with a braided structure. Then the following conditions
are equivalent:
\begin{enumerate}
\item[(i)] for every $X\in\C$ the monoidal functor $\Ad_X :\B\to\B$ is braided;
\item[(ii)] there exists a pair $(a,f)$ consisting of an action $a:G\to\Func (\B ,\B )$ and 
an isomorphism $f$ between the monoidal functor $\Ad :\C\to\Func (\B ,\B )$ and the composition 
$\C {\buildrel{\partial}\over{\longrightarrow}}G {\buildrel{a}\over{\longrightarrow}}\Func (\B ,\B )$ 
such that the restriction of $f$ to $\B\subset\C$ equals the trivialization of $\Ad :\B\to\Func (\B ,\B )$ defined by the braiding on~$\B$. 
\end{enumerate}
If they hold then the pair $(a,f)$ is unique up to unique isomorphism
and the action $a:G\to\Func (\B ,\B )$ is braided.
\end{proposition}

The proposition says that if the equivalent conditions hold then one gets an extension in the
sense of Definition \ref{extdef} in an essentially unique way. This justifies Definition \ref{extdef}.

\begin{proof}
Applying Lemma \ref{subtlelem} to $\M=\Func (\B ,\B )$ we see that an action of $G$ on the monoidal
gr-category $\B$ is the same as an action of $\C$ on the monoidal
gr-category $\B$ whose restriction to $\C_1$ is equipped with a $G$-invariant trivialization. 
The proposition follows.
\end{proof}

\subsection{Central extensions and braided $G$-crossed gr-categories}           \label{BrCent}
\subsubsection{The notion of central extension}  
\begin{definition}   \label{centextdef}
A {\em central extension of a group $G$ by a braided gr-category $\B$\,} is an extension of $G$ by
$\B$ equipped with the trivial action $G\to\Func (\B ,\B )$.
\end{definition}

\begin{remark}   \label{centextrem}
If $\B$ is equipped with the trivial action of $G$ then the isomorphism $f$ from
Definition \ref{extdef} is a trivialization of $\Ad :\C\to\Func (\B ,\B )$, which is the same as
a central structure on $\B$ (see Remark \ref{centrex}(iv)). So Definition \ref{centextdef} says that {\em a 
central extension of $G$ by a braided gr-category $\B$ is a faithfully $G$-graded gr-category $\C$ with trivial component $\B$ and with a central structure on $\B$ extending the braided structure 
on~$\B$.} In other words, our notion of central extension is equivalent to the one introduced by
P.~Carrasco and  A.~M.~Cegarra \cite{CC}. 
\end{remark}   

\subsubsection{The notion of braided $G$-crossed category}           \label{Bra}
In \S\ref{brcrossed} this notion was defined in the setting of fusion categories. However,
it makes sense in the current setting of {\em abstract\,} monoidal categories: use 
Definition \ref{crossed} but understand the word ``grading" as a decomposition 
$\C=\bigsqcup_g\C_g$ rather than $\C=\bigoplus_g\C_g$. Thus a braided $G$-crossed category
is a monoidal category $\C$ equipped with a grading $\partial :\C\to G$, a $G$-action on $\C$, and
a collection of isomorphisms 
\begin{equation} \label{G-brai2}
c_{X,Y}: X\ot Y\iso g(Y)\ot X, \quad g\in G, \, X\in \C_g, \, Y\in \C 
\end{equation}
satisfying properties (a)--(d) from Definition \ref{crossed}. 

Note that isomorphisms \eqref{G-brai2} for $X,Y\in \C_1$ define a braided structure on $\C_1$.
Unless stated otherwise, we equip $\C_1$ with this braided structure. The monoidal
category  $\C_1$ equipped with the opposite braiding $\tilde c_{X,Y}:=c_{Y,X}^{-1}$ will be
denoted by~$\C_1^{\oP}$.

\begin{remark}     \label{centextre2}
Let $\C$ be a braided $G$-crossed category. Then by Remark \ref{centrex}(i), the isomorphisms 
\[
c_{X,Y}^{-1}:Y\ot X\iso X\ot Y, \quad X\in \C_1, \, Y\in \C 
\] 
define a central structure on the trivial component $\C_1\subset\C$. Note that the braided
structure on $\C_1$ corresponding to this central structure by Remark~\ref{centrex}(ii) is that 
of~$\C_1^{\oP}$.
\end{remark}

In the next subsection we explain how to construct examples of braided $G$-crossed gr-categories.
It can be skipped by the reader.

\subsubsection{Braided G-crossed gr-categories and crossed modules. Relation to topology}
\label{crcatcrmod}
\begin{definition}     \label{2Whitehead}
A {\em pre-crossed module} $(G,\, C)$ is a pair of groups $G$ and $C$
together with an action of $G$ on $C$, denoted $(g,\,c)\mapsto {}^g c$,
and a homomorphism $\partial: C\to G$ satisfying
\begin{equation} \label{crossed module1}
\partial({}^g c) = g \partial(c) g^{-1}, \quad c\in C, g\in G. 
\end{equation}
A {\em pre-crossed module} $(G,\, C)$ is said to be a {\em crossed module} if
\begin{equation}  \label{crossed module2}
{}^{\partial(c)} c' = cc'c^{-1}\quad  c,c'\in C, g\in G.
\end{equation}
\end{definition}

\begin{rems}    \label{2crmod from CG}
\begin{enumerate}
\item[(i)] Let $\C$ be a braided $G$-crossed gr-category. Let $C$ be the group of isomorphism 
classes of objects of $\C$. Then $(G,\, C)$ is a crossed module.
\item[(ii)] Let $\C$ be a group viewed as a category (see \S\ref{Conv}). Then a structure of
braided $G$-crossed category on $\C$ is the
same as a structure of crossed module on $(G,\, \C )$. 
\item[(iii)] The notion of crossed module is due to J.~H.~C.~Whitehead. We found the definition of
pre-crossed module in \cite{CE,BC}.
\end{enumerate}
\end{rems}

It is well known that for any group $G$ the following three types of data are essentially equivalent\footnote{To pass from (b) to (a), define $C$ to be the fiber of $F$ over $\be_{\A}$. To pass from (b) to
(c), define $X$ to be the classifying space of $\A$.}:
\begin{enumerate}
\item[(a)] a crossed module $C\to G$;
\item[(b)] a small gr-category $\A$ equipped with an essentially surjective monoidal functor $F:G\to\A$;
\item[(c)] a topological space $X$ such that $\pi_i(X)=0$ for $i\ne 1,2$ equipped with a continuous map
$BG\to X$  (here $BG$ is the classifying space of $G$).
\end{enumerate}

We expect that the following three types of data are also essentially equivalent:
\begin{enumerate}
\item[(a$'$)] a braided $G$-crossed gr-category $\C\to G$;
\item[(b$'$)] a small grouplike monoidal 2-category $\A$ equipped with an essentially surjective monoidal 2-functor $F:G\to\A$;
\item[(c$'$)] a topological space $X$ such that $\pi_i(X)=0$ for $i\ne 1,2,3$ equipped with a continuous map $BG\to X$.
\end{enumerate}

At least, one can construct
 a braided $G$-crossed gr-category $\C\to G$ given a topological space $X$ and a 
 continuous map $f:BG\to X$. Hint: an object of $\C_g$, $g\in G$, is a continuous map
 $\varphi :D^2\to BG$ such that $\varphi |_{\partial D^2}=f\circ l_g$; here $D^2$ is the standard
 disk and $l_g:\partial D^2\to X$ is the standard loop corresponding to $g$.

\subsubsection{Relation between central extensions and braided $G$-crossed gr-categories}           
\label{Brgr}
We will show that {\em a central extension of a group $G$ by a braided gr-category $\B$ is essentially the same as a braided $G$-crossed gr-category $\C$ with $\C_1=\B^{\oP}$ such that the grading 
$\C\to G$ is faithful.}

Here is a construction in one direction: if $\C$ is a braided $G$-crossed gr-category such that
the grading on $\C$ is faithful then by Remarks \ref{centextre2} and \ref{centextrem}, we get a central extension of $G$ by $\C_1^{\oP}$. Proposition \ref{BraidCen}(ii) below says that this construction
is invertible.

\begin{proposition}  \label{BraidCen}
\begin{enumerate}
\item[(i)] In the case of a faithfully $G$-graded gr-category properties (a) and (c) from the definition of
braided $G$-crossed category (Definition \ref{crossed}) follow from properties (b) and (d).
\item[(ii)] Let $\partial:\C\twoheadrightarrow G$ be a central extension of a group $G$
by a braided gr-category. Then the structure of central extension on $\C$ can be upgraded 
to a structure of braided $G$-crossed category on $\C$. Such upgrade is unique up to unique isomorphism.
\end{enumerate}
\end{proposition}

Here upgrading means defining an action of $G$ on $\C$ and a collection of isomorphisms 
\eqref{G-brai2} satisfying properties (a)--(d) from Definition \ref{crossed} so that
for $g=1$ the isomorphisms \eqref{G-brai2} come from the central structure.
The notion of isomorphism between two upgrades is clear because
actions of $G$ on $\C$ form a groupoid (cf. Remark \ref{upgrgroupoid}).

\subsubsection{Proof of Proposition \ref{BraidCen}}   \label{proofBraidCen}
Let $\C$ be a gr-category equipped with a faithful grading  $\partial:\C\twoheadrightarrow G$.
The following definition will become irrelevant as soon as we prove Proposition  \ref{BraidCen}(i).

\begin{definition}  \label{temptemp}
A structure of {\em weak braided $G$-crossed category\,} on $\C$ is a pair consisting of an
action $a:G\to\Func (\C ,\C )$ and a collection of isomorphisms \eqref{G-brai2} satisfying properties (b) 
and (d) from Definition \ref{crossed}.
\end{definition}

\begin{lemma}  \label{weakmeaning}
Such a collection is the same as an isomorphism between the monoidal functor
$\Ad :\C\to\Func (\C ,\C )$ and the composition 
$\C {\buildrel{\partial}\over{\longrightarrow}}G {\buildrel{a}\over{\longrightarrow}}\Func (\C ,\C )$.
\end{lemma}

\begin{proof}  
Rewrite \eqref{G-brai2} as an isomorphism of functors $\varphi_X :\Ad_X\iso a(\partial (X))$.
The commutativity of \eqref{G-hex1} means that $\varphi_X$ is an isomorphism of 
{\em monoidal\,} functors for each $X\in\C$, and the commutativity of \eqref{G-hex2} means that the 
isomorphism of functors $\varphi :\Ad\iso a\circ\partial$ defined by the collection $\{\varphi_X\}_{X\in\C}$
is monoidal.
\end{proof}

Combining Lemmas \ref{weakmeaning} and \ref{subtlelem} with Remark~\ref{centrex}(iii)
we get the following statement.

\begin{corollary}  \label{izlish}
A structure of weak braided $G$-crossed category on $\C$ is the same as a central structure
on $\C_1$ preserved by all monoidal auto-equivalences $\Ad_X$, $X\in\C$.
\end{corollary}

\begin{lemma}  \label{allpreserved}
Each central structure on $\C_1$ is preserved by all  monoidal auto-equivalences $\Ad_X$, $X\in\C$.
\end{lemma}

\begin{proof}
A central structure on $\C_1$ can be viewed as a trivialization of $\Ad :\C\to\Func (\C_1,\C_1)$.
Now apply Lemma \ref{trivaction} to $\M =\Func (\C_1,\C_1)$.
\end{proof}

Combining Corollary \ref{izlish} with Lemma \ref{allpreserved} we get the following
``weak" version of Proposition  \ref{BraidCen}(ii).
\begin{corollary}  \label{weaklemma}
A structure of weak braided $G$-crossed category on $\C$ is the same as 
a central structure on $\C_1$.
\end{corollary}

It remains to prove statement (i) of Proposition \ref{BraidCen}. To this end, we will reformulate
properties (a) and (c) from Definition \ref{crossed}.

Let $\C$ be a weak braided $G$-crossed gr-category. Let $\gamma\in G$ and 
$\Ad_{\gamma}:G\iso G$ the automorphism $g\mapsto\gamma g\gamma^{-1}$. Then we can introduce
a new structure of weak braided $G$-crossed gr-category on $\C$ by replacing $\partial :\C\to G$
and $a:G\to\Func (\C ,\C )$ with $\tilde\partial:=\Ad_{\gamma}\circ\partial$ and 
$\tilde a:=a \circ\Ad_{\gamma}^{-1}$. Let $\null^{\gamma}\C$ denote  the weak braided 
$G$-crossed gr-category $(\C ,\tilde\partial,\tilde a)$. 
Properties (a) and (c) from Definition \ref{crossed} mean that
$a(\gamma ):\C\iso\null^{\gamma}\C$ {\em is an equivalence of weak braided  $G$-crossed gr-categories for each $\gamma\in G$.}

We have to show that this is always true if the grading on $\C$ is faithful. Choose 
$X\in\C_{\gamma}$, then $a(\gamma )\simeq\Ad_X$, so we have to show that 
$\Ad_X:\C\iso\null^{\gamma}\C$ is an equivalence of weak braided $G$-crossed categories.
This follows from Corollary \ref{weaklemma} and Lemma \ref{allpreserved}.

\subsubsection{The action of $G$ on the trivial component} 
Let $\C$ be a braided $G$-crossed gr-category.
The $G$-action on $\C$ preserves $\C_1$, so $G$ acts on $\C_1$.
Let $\Phi :\pi_0 (\C )\to\Func (\C_1,\C_1)$ be the composition $\pi_0 (\C )\to G\to\Func (\C_1,\C_1)$.

\begin{proposition}   \label{trivialization}
$\Phi$ admits a canonical trivialization $\alpha$. Namely, for every $g\in G$, $X\in\C_g$, $Y\in\C_1$
the isomorphism $\alpha_{X,Y}:Y\iso (\Phi (X))(Y)=g(Y)$ is characterized by the following property:
the composition of $c_{Y,X}:Y\ot X \iso X\ot Y$ and $c_{X,Y}:X\ot Y \iso g(Y)\ot X$ equals
 $\alpha_{X,Y}\ot\id_X$.
\end{proposition}

\begin{remark}    \label{nontriviality}
If $X\in\C_1$ then $\Phi (X)=\id_{\C_1}$ and $\alpha_{X,Y}=b(X,Y)$, where 
$b:\pi_0(\C_1)\times\pi_0(\C_1)\to\pi_1(\C_1)$ is the symmetric pairing corresponding to the
braided structure on $\C_1$. So {\em in general, $\alpha_{X,Y}$ is nontrivial for $X,Y\in\C_1$.}
\end{remark}

\begin{proof}
The isomorphisms \eqref{G-brai2} for $Y\in\C_1$ define an isomorphism between the composition
\begin{equation}   \label{thecomposition}
\C\to G\to\Func (\C_1 ,\C_1 )
\end{equation}
and the adjoint action of $\C$ on $\C_1$. On the other hand, by Remark \ref{centrex}(iv), the
central structure on $\C_1$ defined in Remark \ref{centextre2} yields a trivialization of the
adjoint action of $\C$ on $\C_1$. Thus we get a trivialization of the composition 
\eqref{thecomposition}, which is the same as a trivialization of $\Phi$. The corresponding
isomorphism $Y\iso (\Phi (X))(Y)$ is the one described in the formulation of the proposition.
\end{proof}

\begin{corollary}    \label{weak-iso}
Let $\C$ be a braided $G$-crossed gr-category such that the grading on $\C$ is faithful.
Then the monoidal auto-equivalence of $\C_1$ corresponding to each element of $G$ is
isomorphic to $\id_{\C_1}$.
\end{corollary}

\begin{remark}    \label{nontrivialit2}
Let $\C$ be as in Corollary \ref{weak-iso}. Then the action of $G$ on $\C_1$ is not necessarily
trivial. It can be described as follows. By Corollary \ref{weak-iso}, the image of the functor
$G\to\Func (\C_1 ,\C_1 )$ is contained in the full subcategory of monoidal functors isomorphic to
$\id_{\C_1}$, which identifies with the category of $\Aut (\id_{\C_1})$-torsors. Remark \ref{nontriviality}
implies that the action of $G$ on $\C_1$ equals the composition 
\[
G{\buildrel{\alpha}\over{\longrightarrow}}\{\pi_0(\C_1)\mbox{-torsors\}}
{\buildrel{\beta}\over{\longrightarrow}}\{\Aut (\id_{\C_1})\mbox{-torsors\}},
\]
where $\alpha$ corresponds to the extension $0\to\pi_0(\C_1)\to\pi_0(\C)\to G\to 0$
(i.e., $\alpha (g)$ is the pre-image of $g$ in $\pi_0(\C)$) and $\beta$ is induced by the
homomorphism $\pi_0(\C_1)\to\Hom (\pi_0(\C_1), \pi_1(\C_1))=\Aut (\id_{\C_1})$
corresponding to the pairing $b:\pi_0(\C_1)\times\pi_0(\C_1)\to\pi_1(\C_1)$.
\end{remark}

\section{Complexified Grothendieck ring and fusion subcategories}  \label{fusionring}
In this Appendix we give a proof of the result mentioned
in \S \ref{geneneralities}: any fusion subcategory of a fusion category is rigid. 
Its idea was explained to us by P.~Etingof.

\subsection{Recollections on $*$-algebras}
A {\em $*$-algebra\,} over $\BC$ is a $\BC$-algebra $A$ with a map $*:A\to A$ such that
$(ab)^*=b^*a^*$, $a^{**}=a$, and $(\lambda a)^*=\bar \lambda a^*$ for $\lambda\in\BC$.  
A {\em positive trace\,} on $A$ is a linear functional $l: A\to \BC$ such that
$l(ab)=l(ba)$, $l(a^*)=\overline{l(a)}$, $l(aa^*)>0$ for $a\ne 0$.

\begin{example}    \label{*algexample}
Let $A=Mat(n,\BC)$. For $a\in Mat(n,\BC)$ let $a^*$ be the usual Hermitian 
conjugate matrix. Define $l: A\to \BC$ by $l(a):=p\Tr(a)$, $p>0$. Then $A$ is a $*$-algebra 
and $l$ is a positive trace.
\end{example}

Let $A_1$, \ldots, $A_r$ be $*$-algebras and $A=A_1\times\ldots A_r$.
Let $l_i:A_i\to\BC$ be a positive trace, $i=1,\ldots ,r$. Then the linear
functional $l:A\to\BC$ defined by $l(a_1,\ldots ,a_r)=l_1(a_1)+\ldots +l_r(a_r)$
is a positive trace on $A$. In this situation we write $(A,l)=\prod\limits_{i=1}^r(A_i,l_i)$.
The following proposition is classical.

\begin{proposition} Any finite-dimensional $*$-algebra with a positive trace is isomorphic 
to $\prod\limits_{i=1}^r(A_i,l_i)$, where each $(A_i,l_i)$ is as in Example \ref{*algexample}.
\end{proposition}

\begin{proof} 
Let $\End A$ denote the algebra of linear operators $A\to A$. Equip $\End A$ with the
$*$-algebra structure corresponding to the Hermitian scalar product $(a,b)=l(ab^*)$ on $A$.
The left action of $A$ on itself defines an injective $*$-homomorphism $A\hookrightarrow\End A$.
Thus $A$ identifies with a $*$-subalgebra of the algebra of operators in a finite-dimensional
Hilbert space. So the $*$-algebra $A$ is isomorphic to a product of matrix algebras 
$A_1$, \ldots, $A_r$ equipped with their standard anti-involutions. For each $i\in\{ 1,\ldots ,r\}$
the restriction $l|_{A_i}$ is proportional to the usual matrix trace, and the coefficient is
positive because $l$ is.
\end{proof}

\begin{corollary} \label{idem} 
Let $A$ be a finite-dimensional $*$-algebra with a positive trace $l$.
If $e\in A$, $e^2=e$, $e\ne 0$ then $l(e)>0$. \qed
\end{corollary}

\begin{lemma}\label{nilpota}
Let $A$ be a finite-dimensional $*$-algebra with a positive trace $l$.
If $a\in A$ and $l(a^n)=0$ for all $n\in \BN$ then $a$ is nilpotent.
\end{lemma}

\begin{proof} If $a$ is not nilpotent then there exists a nonzero idempotent $e$ of the form
$\sum\limits_{i=1}^N\lambda_ia^i$. Then $l(e)=0$, which contradicts Corollary \ref{idem}.
\end{proof}

 \begin{rems}
 \begin{enumerate}
\item[(i)] Corollary~\ref{idem} is not an immediate consequence of the inequality $l(ee^*)>0$ 
because we do not assume that $e^*=e$.
\item[(ii)] One can show that in Corollary~\ref{idem} the assumption $\dim A<\infty$ is essential.
\end{enumerate}
\end{rems}

\subsection{Fusion subcategories}
Let $\C$ be a fusion category and $A=K_0(\C)\ot_\BZ \BC$. We extend the involution 
$X\mapsto X^*$ to $A$ by anti-linearity. Let $l: A\to \BC$ be the linear functional such that 
$l([X])=\dim \Hom (\be,X)$ for any $X\in \C$. Then $(A,*)$ is a $*$-algebra and 
$l: A\to \BC$ is a positive trace.

\begin{lemma} \label{Xpower}
Let $\C$ be a fusion category. For any $X\in \O(\C)$ there exists a positive
integer $N$ such that $\Hom(\be, X^{\otimes N})\ne 0$.
\end{lemma}

\begin{proof} It is obvious that $[X]\in K_0(\C)\subset K_0(\C)\ot_\BZ \BC$ is not nilpotent.
The result follows from Lemma \ref{nilpota}.
\end{proof}

\begin{corollary} \label{fusion sub}  
A fusion subcategory of a fusion category is rigid.
\end{corollary}

\begin{proof} 
Let $\D$ be a fusion subcategory of a fusion category $\C$.
We have to show that if $X\in \O(\D)$ then $X^*\in \O(\D)$.
Clearly $X^{\otimes n}\in\D$ for all $n\ge 1$.
For $N$ as in Lem\-ma~\ref{Xpower} we get $\Hom(X^*,X^{\otimes N-1})\ne 0$.
So $X^*\in \O(\D)$.
\end{proof}

\bibliographystyle{ams-alpha}

\begin{thebibliography}{ABC}

\bibitem[AB]{AB} D.~Altsch\"uler, A.~Brugui\`eres,
\textit{Handleslide for a sovereign category},
preprint, 2001. 

\bibitem[AG]{AG} S.~Arkhipov,  D.~Gaitsgory,
\textit{Another realization of the category of modules over the small quantum group,}
Adv. Math. {\bf 173} (2003), no. 1, 114--143.

\bibitem[AM]{AM}  G.~Anderson and G.~Moore,
\textit{Rationality in conformal field theory}, 
Comm. Math. Phys. {\bf 117} (1988), no. 3, 441--450.

\bibitem[BLLV]{BLLV} J.~Barge,  J.~Lannes, F.~Latour, P.~Vogel, 
\textit{$\Lambda $-spheres,}
Ann. Sci. \'Ecole Norm. Sup. (4) {\bf 7} (1974), no. 4, 463--505. 

\bibitem[Be]{Be} R.~Bezrukavnikov,
\textit{On tensor categories attached to cells in affine Weyl groups},
Representation theory of algebraic groups and quantum groups,
69--90, Adv. Stud. Pure Math., 40, Math. Soc. Japan, Tokyo, 2004.

\bibitem[Br]{Br} A.~Brugui\`{e}res,
\textit{Cat\'egories pr\'emodulaires, modularization et invariants
des vari\'et\'es de dimension $3$}, Mathematische Annalen,
\textbf{316} (2000), no. 2, 215-236.

\bibitem[BD]{CH} A.~Beilinson and V.~Drinfeld, \textit{Chiral algebras},
American Mathematical Society, Providence, Rhode Island, 2004. 

\bibitem[BC]{BC} H.J.~Baues, D.~Conduch\'e, 
\textit{The central series for Peiffer commutators in groups with operators},
J. Algebra {\bf 133} (1990), no. 1, 1--34.

\bibitem[BK]{BK} B.~Bakalov, A.~Kirillov Jr.,
\textit{Lectures on Tensor categories and modular functors},
AMS, (2001).

\bibitem[CC]{CC} P.~Carrasco and  A.~M.~Cegarra, \textit{Extensions centrales de 
${\rm Gr}$-cat\'egories par des ${\rm Gr}$-cat\'egories tress\'ees},  C. R. Acad. Sci. Paris S\'er. I Math.  
\textbf{321}  (1995),  no. 5, 527--530.

\bibitem[CE]{CE} D.~Conduch\'e, G.J.~Ellis, \textit{Quelques propri\'et\'es homologiques des 
modules pr\'ecrois\'es}, J. Algebra {\bf 123} (1989), no. 2, 327--335.

\bibitem [DMNO]{MNO} A.~Davydov, M.~M\"uger, D.~Nikshych, and V.~Ostrik,
\textit{\'Etale algebras in braided fusion categories}, in preparation.

\bibitem[De]{De} P.~Deligne, 
{\em Cat\'egories tensorielles},  Mosc. Math. J.  {\bf 2} (2002),  no. 2, 227--248.

\bibitem[Dr]{Dr} V.G.~Drinfeld,
\textit{Quasi-Hopf algebras}, Leningrad Math. J. \textbf{1}
(1990), no. 6, 1419-1457.

\bibitem[ENO]{ENO} P.~Etingof, D.~Nikshych, V.~Ostrik,
\textit{On fusion categories}, Annals of Mathematics
\textbf{162} (2005), 581-642.

\bibitem[ENO2]{ENO2} P.~Etingof, D.~Nikshych, V.~Ostrik,
\textit{An analogue of Radford's $S^4$ formula for finite tensor categories}, 
International Mathematics Research Notices {\bf 54} (2004),
2915--2933.

\bibitem[ENO3]{ENO3} P.~Etingof, D.~Nikshych, V.~Ostrik,
\textit{Weakly group-theoretical and solvable fusion categories}, 
math.0809.3031.

\bibitem[ENO4]{ENO4} P.~Etingof, D.~Nikshych, V.~Ostrik,
\textit{Fusion categories and homotopy theory},  arXiv:0909.3140.

\bibitem[EO]{EO} P.~Etingof, V.~Ostrik,
\textit{Finite tensor categories},
Moscow Math.\ Journal \textbf{4} (2004), 627-654.

\bibitem[FG]{FG} E.~Frenkel, D.~Gaitsgory,
\textit{Geometric realizations of Wakimoto modules at the critical level},
Duke Math. J. {\bf 143} (2008), no. 1, 117--203. 

\bibitem[G]{G} D.~Gaitsgory,
\textit{The notion of category over an algebraic stack},
arXiv:math/0507192.

\bibitem[GN]{GN} S.~Gelaki, D.~Nikshych,
\textit{Nilpotent fusion categories}, 
Advances in Mathematics, \textbf{217} (2008), no. 3, 1053--1071.

\bibitem[JS]{JS} A.~Joyal, R.~Street,
\textit{Braided tensor categories}, Adv.\ Math.,
\textbf{102}, 20-78 (1993).

\bibitem[Kass]{Kass} C. Kassel, \textit{Quantum groups}, Graduate
Texts in Mathematics {\bf 155}, Springer, New York.

\bibitem[KaO]{KaO} I.~Kath, M.~Olbrich,
\textit{Metric Lie algebras and quadratic extensions},
Transform.\ Groups  \textbf{11}  (2006),  no. 1, 87--131.

\bibitem[Ki]{Ki} A.~Kirillov, Jr.
\textit{ Modular categories and orbifold models II},
arXiv:math/0110221.

\bibitem[KiO]{KiO} A.~Kirillov Jr., V. Ostrik,
\textit{On a $q$-analogue of McKay correspondence and ADE classification of
$\hat{\mathfrak{s}\mathfrak{l}}_2$  conformal field theories},
Adv. Math. {\bf 171} (2002), no. 2, 183--227.

\bibitem[L]{L} J.~Lannes, 
\textit{Formes quadratiques d'enlacement sur l'anneau des entiers d'un corps de nombres.,}
Ann. Sci. \'Ecole Norm. Sup. (4)  {\bf8}  (1975), no. 4, 535--579. 

\bibitem[MH]{MH} J.~Milnor, D.~Husemoller, \textit{Symmetric bilinear forms},  
Ergebnisse der Mathematik und ihrer Grenzgebiete, Band 73, Springer-Verlag, New York-Heidelberg, 1973. 

\bibitem [Mu1]{Mu1} M.~M\"uger, 
\textit{Galois theory for braided tensor categories
and the modular closure}, Adv. Math. {\bf 150} (2000), no. 2, 151--201.

\bibitem [Mu2]{Mu2} M.~M\"uger, \textit{On the structure of modular categories},
Proc. Lond. Math. Soc., \textbf{87} (2003), 291-308.

\bibitem [Mu3]{Mu3} M.~M\"uger,
\textit{Galois extensions of braided tensor categories and 
braided crossed G-categories}, J. Algebra {\bf 277} (2004), 
no. 1, 256--281.

\bibitem [Mu4]{Mu-I} M.~M\"uger,
\textit{From subfactors to categories and topology. I. 
Frobenius algebras in and Morita equivalence of tensor categories.},
J.\ Pure Appl.\ Algebra  \textbf{180}  (2003),  no. 1-2, 81--157.

\bibitem [Mu5]{Mu4} M.~M\"uger,
\textit{From subfactors to categories and topology. II. 
The quantum double of tensor categories and subfactors},
J.\ Pure Appl.\ Algebra  \textbf{180}  (2003),  no. 1-2, 159--219.


\bibitem [MS]{MS} G.~Moore, N.~Seiberg, \textit{Classical and Quantum Conformal Field Theory},
Commun. Math. Phys. \textbf{123} (1989), 177-254.

\bibitem [MW]{MW} S.~Montgomery, S.~Witherspoon,
\textit{Irreducible representations of crossed products},
J.\ Pure Appl.\ Algebra,  \textbf{129}  (1998),  no. 3, 315--326.

\bibitem[NN]{NN} D.~Naidu and D.~Nikshych,
\textit{Lagrangian subcategories and braided tensor equivalences
of twisted quantum doubles of finite groups}, 
Comm.\ Math.\ Phys.\ \textbf{279} (2008), 845-872.

\bibitem [O1]{O1} V.~Ostrik,
\textit{Module categories, weak Hopf algebras and modular invariants},
Transform. Groups, \textbf{8} (2003), 177-206.

\bibitem [O2]{O2} V.~Ostrik,
\textit{Module categories over the Drinfeld double of a finite group},
Int.\ Math.\ Res.\ Not. (2003) no.\ 27, 1507-1520.

\bibitem [O3]{O3} V.~Ostrik,
\textit{On formal codegrees of fusion categories}, Math. Res. Lett. {\bf 16} (2009), no. 5, 897-903.

\bibitem [P]{P} B.~Pareigis, \textit{On braiding and dyslexia}, J. Algebra {\bf 171} (1995),
no. 2, 413--425.

\bibitem[Q]{Q} F.~Quinn, \textit{Group categories and their field theories},
Proceedings of the Kirbyfest (Berkeley, CA, 1998), 407--453
(electronic), Geom. Topol. Monogr., 2, Geom. Topol. Publ.,
Coventry, 1999.

\bibitem[Se]{Se} J.-P.~Serre, 
\textit{A course in arithmetic}, Graduate Texts in Mathematics, No. 7, Springer-Verlag, New York-Heidelberg, 1973. 

\bibitem[S]{Sie} J.~A.~Siehler, 
\textit{Braided Near-group Categories}, math.QA/0011037.


\bibitem[TY]{TY} D.~Tambara, S.~Yamagami, \textit{Tensor categories with
fusion rules of self-duality for finite abelian groups}, J.~Algebra
\textbf{209} (1998), no. 2, 692--707.

\bibitem [T]{T} V.~Turaev,
\textit{Quantum invariants of knots and $3$-manifolds}, W.\ de Gruyter (1994).

\bibitem [T1]{Tuor} V.~Turaev,
\textit{Homotopy field theory in dimension 3 and crossed group-categories},
arXiv:math/0005291.

\bibitem[V]{V} C.~Vafa,
\textit{Toward classification of conformal theories}, Phys.\ Lett.\ B.\ \textbf{199} (1987), 195-202.

\bibitem[Wa]{Wa} C. ~T.~C.~Wall,  
\textit{Quadratic forms on finite groups, and related topics},  Topology  \textbf{2} (1963), 281--298. 

\bibitem[We]{We} A.~Weil, \textit{Sur certains groupes d'operateurs unitaires},
Acta Math. \textbf{111} (1964), 143--211.

\end{thebibliography}

\end{document}